%% file: 421.tex
\input t3totex.tex
  \input macros.sty
\input moredefs.sty

\documentclass{amsart}
\ifx\shlhetal\undefinedcontrolsequence\let\shlhetal\relax\fi
\ifx \shlhetal\relax
      \else
      \fi

 \usepackage{amssymb}
  \usepackage{amsthm}

\let\mathbb\mathbb
\let\mathbf\mathbf
\let\mathscr\cal

\let\mathfrak\mathfrak
\input  setup.sty
\input atalya.sty

\usepackage{hyperref}

\begin{document}

  \input markers.sty


\title[Kaplansky test problems]{Kaplansky Test Problems for $R$-Modules in ZFC}

\author[M. Asgharzadeh]{Mohsen Asgharzadeh}

\address{Mohsen Asgharzadeh, Hakimiyeh, Tehran, Iran.}

\email{mohsenasgharzadeh@gmail.com}

\author[M.  Golshani]{Mohammad Golshani}

\address{Mohammad Golshani, School of Mathematics, Institute for Research in Fundamental Sciences (IPM), P.O.\ Box:
19395--5746, Tehran, Iran.}

\email{golshani.m@gmail.com}

\author[S. Shelah]{Saharon Shelah}

\address{Saharon Shelah, Einstein Institute of Mathematics, The Hebrew University of Jerusalem, Jerusalem,
91904, Israel, and Department of Mathematics, Rutgers University, New Brunswick, NJ
08854, USA.}

\email{shelah@math.huji.ac.il}

\thanks{The second author’s research has been supported by a grant from IPM (No. 1400030417). The
	third author’s research has been partially supported by Israel Science Foundation (ISF) grant no:
	1838/19. This is publication 421 of third author.}

\begin{abstract}
Fix a ring $ R $ and look at
the class of
 left $ R $-modules
and naturally we restrict ourselves to the
case of
rings such that this class is not too similar to the case $ R $
is a field.

 We shall solve Kaplansky test problems, all three of
which say that
we do not have decomposition theory,
e.g.,  if
the square of one  module
is isomorphic to the cube
of another it does not follows that they are isomorphic.
Our results are in the ordinary
set theory, ZFC.
For this we look at bimodules,
i.e., the structures
 which are simultaneously left $ R $-modules and
right $ S $-modules with reasonable associativity,
over a
commutative ring
$ T $   included in the centers of $ R $ and of $ S $.

Eventually
we shall
choose $ S $ to help solve each of the test questions.
But first
we analyze what can be the smallest endomorphism ring of an $
(R,S)$-bimodule. We construct such bimodules by using a simple version of  the black
box.
\end{abstract}
\maketitle
\setcounter{section}{-1}
\medskip

\tableofcontents
	
\section{ Introduction}
Throughout this paper $\ringR$ is an   associative ring  with $1=1_{\ringR}$, which is not necessarily commutative nor Noetherian. An
$\ringR$-module is a left $\ringR$-module, unless otherwise specified.    We shall prove here
for example (see   Theorem \ref{3.8}):

\begin{Theorem}
	\label{0.1}
	For every ring $\ringR$, one of the following conditions hold:
	\begin{enumerate}
		\item[(a)] all $\ringR$-modules are direct sum of countably generated
		indecomposable $\ringR$-modules (such rings are called left pure semisimple
		rings), moreover such a representation is unique up to isomorphism,
	\end{enumerate}
	or
	\begin{enumerate}
		\item[(b)] for every cardinal $\lambda>|\ringR|$ satisfying
		$\lambda=\left(\mu^{\aleph_0}
		\right)^+$, and $0<m(\ast),\ell(\ast)<\omega$,
		there is an $\ringR$-module  $ \modM $
		of cardinality
		$\lambda$
		such that $$M \cong M^n\ \Leftrightarrow\
		n\in \{m(\ast)k+\ell(\ast)
		:k<\omega\}.$$
	\end{enumerate}
\end{Theorem}

We shall prove this theorem in $\mbox{ZFC}$,  the Zermelo--Fraenkel set theory with the axiom of choice.
Theorem \ref{0.1} is new even if we assume Jensen's combinatorial principle
\textit{diamond}.

\begin{corollary}
	Let $0<m_1<m_2-1$ and assume that $\ringR$  is not pure semisimple. Then there is an $\ringR$-module of cardinality $\lambda$ such
	that:
	$$
	\modM^{n^1}\cong \modM^{n^2}\mbox{  iff } \quad m_1<n^1\ \&\ m_1\leq n^2\ \&\ [(m_2-m_1)|
	(n^1-n^2)]
	$$
	( $|$ means divides).
\end{corollary}
This result extends several known statements starting with Corner \cite{Cor63}, and  a special case of it, when $\ringR$ is  the ring of integers, was obtained recently
by G\"{o}bel-Herden-Shelah
\cite[Corollary 9.1(iii)]{gobel-shelah}.

Given an algebraic theory, one of the major objectives is to find some structure theorems for the objects under consideration under
various invariants, say for example up to isomorphism.
In 1954, Kaplansky formulated three test problems (see \cite[page 12]{Kap54}), where in his opinion a structure theory can be satisfactory
only if it can solve these problems. He formulated his questions
in the context of Abelian groups as follows:
	\begin{enumerate}
	\item[(I)]  If ${\bf G}$ is isomorphic to a direct summand of ${\bf H}$ and ${\bf H}$ is
	isomorphic to a direct summand of ${\bf G}$, are ${\bf G}$ and ${\bf H}$ necessarily
	isomorphic?
	\item[(II)] If ${\bf G}\oplus {\bf G}$ and ${\bf H}\oplus {\bf H}$ are isomorphic, are ${\bf G}$ and ${\bf H}$
	isomorphic?
	\item[(III)] If ${\bf F}$ is finitely generated and $ {\bf F} \oplus {\bf G}$ is isomorphic to
	${\bf F} \oplus {\bf H}$, are ${\bf G}$ and ${\bf H}$ isomorphic?
\end{enumerate}
He also noted that
 these problems can be formulated for very general mathematical
systems, and  he also
mentioned that problem (I) has an affirmative answer in set theory, namely the
Cantor-Schr\"{o}der-Bernstein theorem.
It may be worth to note that the first two problems were asked earlier, in 1948, independently by Sikorski \cite{Si} and Tarski \cite{T}  in the context of Boolean algebras, where they answered these problems affirmatively
for countably complete Boolean algebras.
Also, Sierpinski asked
the \textit{cube problem} (see below for its statement)  in the context  of linear orders, see \cite{S}.

There is a great deal of work on Kaplansky's test problems over different algebraic structures.
In the sequel we summarize some of them.

For Boolean algebras, the first negative result was proved by Kinoshita \cite{kinoshita}, who answered test question (I)
in negative for countable Boolean algebras. Hanf \cite{hanf} answered question (II) in negative for the same class of structures
and
Ketonen \cite{K} solved, among many other interesting things, the   Tarski cube problem by producing a
countable Boolean algebra isomorphic to its cube but not its square.
A solution to the Schroeder-Bernstein problem for Banach spaces
 is the subject of \cite{gow} by  Gowers.
Very recently,  Garrett \cite{G} proved that  every linear order which is isomorphic to its cube,  is isomorphic to its square as well,
thus answering Sierpinski's cube problem for linear orders.

It is known that the problems have positive answer for many classes of Abelian groups, like finitely generated groups, free groups, divisible groups and so on, see \cite{Fuc73}. J\'{o}nsson \cite{jonsson}
gave negative answers to test problems (I) and (II) for countable centerless non-commutative groups
and then in \cite{jonsson2}, he gave a negative answer to problem (II)
for the class of torsion-free Abelian groups of rank 2. In 1961 Sasiada \cite{sasiada} answered the first problem in negative for the class
of torsion-free groups of rank $2^{\aleph_0}$.
Corner, in his breakthrough work \cite{Cor63}, proved that  any countable torsion-free
reduced  ring can be realized as an endomorphism ring of a torsion-free Abelian
group and deduced from this a negative answer to the Kaplansky's test problems (I) and (II) for countable torsion-free
reduced  groups.
Later, Corner \cite{Cor64} constructed a countable  torsion-free Abelian
group ${\bf G}$, which is isomorphic to ${\bf G}^3$ but not to ${\bf G}^2$, thereby giving a negative answer to the cube problem, which asks  if ${\bf G}$ is isomorphic to ${\bf G}^3$, does it follows that ${\bf G}$ is isomorphic to ${\bf G}^2$?
As another contribution, Corner proved that for any positive integer $r$ there exists a countable torsion-free Abelian
group ${\bf G}$ such that  ${\bf G}^m\cong  {\bf G}^n$ if and only if $m \equiv_r n$. This property of ${\bf G}$ is called the
\textit{Corner pathology.} In particular, the ring of integers has a pathological object.
For countable separable $p$-groups,  Elm's theorem provides an affirmative solution to Kaplansky  problems (I) and (II), and
Crawley \cite{crawley}, in 1965, showed that this does not extend to the uncountable case by giving a negative answer to these problems for the case of uncountable separable $p$-groups.
Given any $n \geq 2,$ Eklof   and  Shelah \cite{ESH1} have constructed  a locally free  Abelian group $G$ of cardinality $2^{\aleph_0}$ such that $G\oplus \mathbb{Z}^n\cong G$.   Very recently, Richard \cite{Rich} has presented more friendly examples of  Abelian groups
	equipped with the pathological property.
Corner's ideas were used by a number of mathematicians, to give negative answers to the first two test questions for some  classes of Abelian groups. For instance, see
see \cite{DG82}, \cite{DG82b}, \cite{DG82c} \cite{ESH}, \cite{Rich}, and \cite{eklof-mekler}.

Shelah's work \cite{Sh:381}
deals with the first test problem of Kaplansky in the category
of modules over a  general  ring,
however his results were obtained under set theoretic assumptions beyond ZFC.
In fact, Shelah applied
Jensen's diamond principle, a kind of prediction principle whose truth is independent of ZFC,  to present  rigid-like modules which violate Kaplansky test problem.

 Thom\'{e} \cite{tom} and Eklof-Shelah
\cite{ESH} constructed an $\aleph_1$-separable   Abelian
group  $\modM$ in  ZFC  such that the Corner’s ring  is  algebraically
closed in $\End(\modM)$. Consequently $\modM$ breaks down the Kaplansky's test problem,
i.e., it is isomorphic to its cube but not to its square.

 In sum, Kaplansky test problems remained widely open for the category of modules over general rings.

 In this paper, we are interested in the category
of modules over a general ring which is not necessarily commutative nor countable.
Working in ZFC,  we answer all three
test problems.
This continues \cite{Sh:381}, but can be read independently. Let us stress that all of our
results are in ZFC (without any extra set theoretic axioms), and that we  obtain even stronger results than those of \cite{Sh:381}. For example, see   Theorem \ref{0.1}.
We will do this  by using a simple case of  the ``\textit{Shelah's black
box}'', see Lemma \ref{shelah black box} and  Theorem \ref{nice construction lemma}.

The  black boxes were introduced  by Shelah in \cite{Sh:172} and in a more general form in \cite{Sh:227},
where he showed,
following some ideas from \cite[Ch.VII]{Sh:a}, that they follow from ZFC.
We can consider black boxes
as  a general method to generate a class of diamond-like principles provable in ZFC.

In \cite{Sh54}  Shelah proved that any ring $\ringR$ satisfies one of the following two possibilities:
	\begin{enumerate}
	\item[(i)] All modules are direct sums of countably generated modules, or
	\item[(ii)] 	For any cardinal $\lambda>|\ringR|$ there is an $\ringR$-module of cardinality $\lambda$ which is not the direct sum of $\ringR$-modules of cardinality $\leq \mu$ for some $\mu<\lambda$.
\end{enumerate}
Shelah's work in \cite{Sh:381}  extends  (ii) in terms of endomorphism algebras. From this and in ${\bf V}={\bf L}$, he constructed an $\ringR$-module $\modM$ with prescribed endomorphisms modulo an ideal of  small endomorphisms.
Then Shelah found a connection from (ii) to the Kaplansky's test problems.

Here, we remove the extra assumption ${\bf V}={\bf L}$ and show that it is possible to get clause $(b)$ of  Theorem \ref{0.1} from the condition (ii).

Here, we lose the $\lambda$-freeness (this is unavoidable, even for
Abelian groups- see Magidor and Shelah \cite{MgSh:204}).
In particular we prove here
that for each $m(\ast)$, there is an $\ringR$-module $\modM$,
such that $\modM \cong \modM^n$ if and only if
$n$  divides $m(\ast)$ and we also answer the other Kaplansky's test problems which were
promised
in \cite{Sh:381}. Also here we prove explicitly that the theorems apply to
elementary (= first order) classes of modules which are not superstable.

In the course of the proof of Theorem \ref{0.1},  we develop general  methods that allow us to prove the following
results in ZFC. These results provide negative solutions to the  Kaplansky's test problems (I) and (II).

\begin{theorem}
\label{0.3}
  Let $\ringR$ be a ring which is not pure semisimple and let $\lambda>|\ringR|$ be a regular cardinal of the form $\left(\mu^{\aleph_0}
		\right)^+$.
	 Then there are $\ringR$-modules $\modM_1$
		and $\modM_2$ of cardinality $\lambda$ such that:
		\begin{enumerate}
			\item[i)]  	$\modM_1$, $\modM_2$ are not isomorphic,

			\item[ii)] 	$\modM_1$ is isomorphic to a direct summand of $\modM_2$,
			\item[iii)] 	$\modM_2$ is isomorphic to a direct summand of $\modM_1$.	
		\end{enumerate}
\end{theorem}

\begin{theorem}
\label{0.4}
Let $\ringR$ be a ring which is not pure semisimple and let
$\lambda=\left(\mu^{\aleph_0}
		\right)^+> |\ringR|$ be a regular cardinal.
 Then there are
		$\ringR$-modules $\modM$,
		$\modM_1$, $\modM_2$ of \power\ $\lambda$ such that
		\begin{enumerate}
			\item[i)]	$\modM\oplus \modM_1\cong \modM\oplus \modM_2$,
			\item[ii)] 	$\modM_1\not\cong \modM_2$.
	\end{enumerate}
\end{theorem}
In particular,  theorems \ref{0.3} and \ref{0.4} improve the results from \cite{Sh:381}, by removing the
use of diamond principle.

As Kaplansky test problems and also our results are related to the indecomposability of modules, let us also
 review some background in this direction.
Fuchs  \cite{Fuc73}, with some help of Corner, proved, by induction on $\lambda$, the existence of an indecomposable
Abelian group in many cardinals $\lambda$ (e.g., up to the first strongly
inaccessible cardinal), and even of a rigid system of $2^\lambda$ Abelian
groups of power $\lambda$.
It was
conjectured at that time  that for some ``large cardinal''
(e.g., supercompact) this may fail.
Corner \cite{Cor69}
reduced the number of primes to five and
later, G\"obel and May \cite{GoMa90}
to four, in the following sense.
Suppose $R$ is an algebra over a commutative ring $A$,
$\lambda$ is an infinite cardinal, and assume that  $ R$, as an  $A$-algebra,  can be generated by no more than $\lambda$
elements. Let $M =\oplus_{\lambda} R$ and suppose in addition that there is an embedding $R\hookrightarrow \End_A(M)$ by scalar multiplication. G\"obel and May introduced
four  $A$-submodules $M_0, M_1, M_2, M_3$ of $M$ such that $$R = \{\varphi\in \End_{A}(M): \varphi (M_i) \subset
	M_i \emph { for all } 0 \leq  i \leq 3\}.$$
Shelah \cite{Sh:44} then proved the existence in every $\lambda$ (using
stationary sets) of an indecomposable Abelian groups of cardinality $\lambda$
(and even a rigid system of $2^\lambda$ ones).
Lately, G\"obel and Ziegler \cite{gobel-ziegler} generalized this to
$\ringR$--modules for ``$\ringR$ with five ideals''.
Answering a
question of Pierce, Shelah \cite{Sh:45} constructed
essentially decomposable\footnote{${\bf G}$ is essentially decomposable  if ${\bf G}={\bf G}_1 \oplus {\bf G}_2$ implies ${\bf G}_1$
or ${\bf G}_2$ is bounded.}
 Abelian
$p$-groups for any cardinal  $\lambda$, which is strong
limit of uncountable cofinality.

Eklof and Mekler \cite{EM77}, using diamond  on some
non-reflecting stationary set of ordinals $<\lambda$ of cofinality $\aleph_0$,
got a $\lambda$-free indecomposable Abelian
group of power $\lambda$. Shelah \cite{Sh:140} continued this  where he showed that the diamond can be
replaced by weak diamond on a non-reflecting stationary subset of
$$
S^\lambda_{\aleph_0}:=\{\delta<\lambda:\cf(\delta)=\aleph_0\}
$$
(so for $\lambda=\aleph_1$, $2^{\aleph_0}<2^{\aleph_1}$ suffices).

Dugas \cite{Dug81} continuing \cite{EM77} proved, under the assumption ${\bf V}={\bf L}$, the existence of a strongly
$\kappa$-free Abelian group with endomorphism ring ${\mathbb Z}$  and then, by using $p$-adic rings, G\"obel \cite{G80} realized to a larger family of
rings.

Dugas and G\"obel \cite {DG82}, continuing \cite{Dug81}, \cite{G80} and
\cite{Sh:140}, proved that for a Dedekind domain $\ringR$ which is not a field,
there are  arbitrary large indecomposable $\ringR$-modules, there are arbitrarily large $\ringR$-modules which do not satisfying the Krull-Schmidt cancellation property and also related their work to the Kaplansky's test problems,
by showing that  $\ringR$ is not a complete discrete valuation ring if and only if there are
$\ringR$-modules of arbitrary high rank which do not satisfy Kaplansky's test problems.
They also showed that every cotorsion-free ring is the endomorphism ring of some Abelian group.
 Then in \cite{DG82b}, they
characterized
the rings which can be represented as ${\rm End}(\modM)$ modulo ``the small
endomorphism'' for some Abelian $p$-group, but as it continues
\cite{Sh:45}
(which dealt with the case when we want the smallest such ring) the
representation of a ring $\ringR$ is by an Abelian group $\modM$ of a
power strong
limit cardinal of cofinality $>|\ringR|$. The situation is similar in Dugas and
G\"obel \cite{DG82c} where the results of \cite{DG82} and more are obtained in
such cardinals.

Black box enables us to get the results of \cite{DG82},
\cite{DG82b} in more and smaller cardinals, e.g., $\lambda=\left(
|\ringR|^{\aleph_0}\right)^+$.
Corner and G\"obel \cite{CG85}  continued this, and using ideas of Shelah have presented a detailed treatment of construction of Abelian groups (or modules) having preassigned endomorphism rings, and satisfying additional constraints.

For more information,
see
the books Eklof-Mekler \cite{EM02}  and G\"{o}bel-Trlifaj
\cite{GT}. The books of Prest \cite{P} and \cite{P2} talk about
model theory of modules, and in particular on pure semisimple rings.

\section{Outline of the proof of Theorem \ref{0.1}}
\label{outlineofproof}
In this section we give an overview of the proof of Theorem \ref{0.1}.
To this end, we deal with the problem of choosing a ring $\ringS$
(essentially
the ring of
endomorphisms
we would like),
and try to build an
$\ringR$-module which ``has the endomorphisms for
$s\in \ringS$ but not many more''.

In Section \ref{Developing the framework of the construction}, we develop  the setting of our construction and to explain it, let us
first present a simple case. Let $\ringR$ be a ring (with unit
$1_{\ringR}$) which is
not left purely semisimple (i.e., as in  Theorem \ref{0.1}(b)).
Then we consider formulas $\varphi(x)$ of the form
$$( 
\exists y_0,\ldots,y_{k_m} ) 
[\bigwedge\limits_{m<m(\ast )} a_mx=
\sum\limits_{i<k_m} b_{m,i}y_i]$$ where $\{a_m,
b_{m,i}\}$
are elements of the ring $\ringR$. For
any  left $\ringR$--module $\modM$ the set $\varphi(\modM)=\{x:
\modM\models\varphi(x)\}$ is not
necessarily a sub-module (as $\ringR$ is not
necessarily commutative), but it is an additive subgroup of $\modM$.
Now, our assumption implies that for some \sq\
$\bar{\varphi}=\langle \varphi_n(x):n<\omega\rangle$, each $\varphi_n$ is as
above. For every $n$ and $\modM$ we have
$\varphi_{n+1}(\modM)\subseteq\varphi_n(\modM)$, and for some
$\modM$, the sequence
$\langle\varphi_m(\modM):m<\omega\rangle$
is strictly increasing.  \Wlog
 $~ k_n < k_{n+1}$ and:
\[\varphi_n(x)=(
\exists\; y_0,\ldots y_{k_{n-1}-1})
(\bigwedge\limits_{m<n}
a_mx=\sum\limits_{i<k_m} b_{m,i}y_i).\]
Let $\modN_n$ be the $\ringR$--module generated by $x^n,y^n_i$ ($i<k_m$) freely,
except the following relations $$a_m x^n=\sum\limits_{i<k_n} b_{m,i} y^n_i \quad\forall m<n.$$ Let $g_n$ be the homomorphism from $\modN_n$ into
$\modN_{n+1}$ mapping
$x^n$ to $x^{n+1}$ 
and $y^n_i$ to  $ y^{n+1}_i $ 
(for $i<k_n$).
Let $${\mathfrak e}=\langle \modN_n,
x_n,g_n:n<\omega\rangle=
\langle \modN^{\mathfrak e}_n , x^{ {\mathfrak e}}_n, g^ {\mathfrak e}_n 
:n<\omega\rangle.$$ Next, we define the following subgroups of $\modM$:
\[\varphi_n(\modM)=\varphi^{\mathfrak e}_n (\modM):=
\{f(x): f\mbox{ is a homomorphism from }\modN_n
\mbox{ to }
\modM\}.\]
In Definition
\ref{nice construction modified} we will  define the notion of ``semi-nice construction'', which in particular gives a sequence $\bar{\modM}$ of $\ringR$-modules, and in Theorem
\ref{nice construction lemma}
we show that the existence of a semi-nice construction follows from
 a suitable version of
black boxs.  It gives us
an $\ringR$--module $\modM=\bigcup\limits_{\alpha}
\modM_\alpha$
such that every endomorphism ${\fucF}$
of $\modM$ is in a suitable sense trivial: on
general
grounds, ${\fucF}$
maps $\varphi_n(\modM)$ into itself, hence
it maps $\varphi^{\mathfrak
	e}_\omega(\modM)=:
\bigcap\limits_{n<\omega}\varphi_n(\modM)$ into itself. So it
induces an additive endomorphism
$\hat{\fucF}_n={\fucF}
\restriction (\varphi_n(\modM)/\varphi^{\mathfrak
	e}_\omega(\modM))$ of the Abelian group
$\varphi_n(\modM)/\varphi^{\mathfrak
	e}_\omega(\modM)$.
By the construction, for some $n=n(\fucF)$,
the mapping
$\hat{\fucF}_n$
is of the ``unavoidable'' kind: multiplication by some
$a$ from the center of $\ringR$.

The point is that for undesirable ${\fucF}$,
we will be able to find $x_n$ and $x$
such that $x-x_n\in\varphi_n(\modM)$ but for no $y\in \modM$ do we have:
\[n<\omega\quad\Rightarrow\quad y-{\fucF} 
( x_n)\in\varphi_n(\modM).\]
So, omitting countable types helps.
But, we would like to have more: not
to generalize ``the Abelian group
$\modM$ has no
endomorphism
except multiplication by $a\in {\mathbb Z}$'', but to generalize
``the Abelian group $\modM$ has a
prescribed endomorphism  ring $\ringS$''. For this, we consider a fixed pair
$(\ringR,\ringS)$ of
rings and a commutative subring $\ringT$ of the
center of $\ringR$ and
the center of $\ringS$, and work with
$(\ringR,\ringS)$-bimodules, where by an $(\ringR,\ringS)$-bimodule, we mean a left $\ringR$-module  $\modM$ which  is
a right $\ringS$-module  and satisfies  the following equations:
$\forall r\in \ringR$, $\forall s\in \ringS$,~$(rx)s=r(xs)$
 and  $\forall t\in \ringT$,~$tx=xt$. We
would like to
build a bimodule $\modM$ such that all of  its
$\ringR$--endomorphisms (i.e., endomorphisms as an $\ringR$--module) are,
in a sense,
equal to left multiplication by a member of $\ringS$. To be able to
construct such an $\modM$, we
need to have ${\mathfrak e}=\langle
\modN^{\mathfrak e}_n, x^{\mathfrak e}_n,
g^{\mathfrak e}_n: n<\omega\rangle$, but now
$\modN^{\mathfrak e}_n$ is a
bimodule, $x^{\mathfrak e}_n\in \modN^{\mathfrak e}_n$ and $g^{\mathfrak
	e}_n$ is a
 bimodule  homomorphism from $\modN^{\mathfrak e}_n$ to
$\modN^{\mathfrak e}_{n+1}$
mapping $x^{\mathfrak e}_n$ to $x^{\mathfrak e}_{n+1}$. As a  first
approximation, let $\varphi^{\mathfrak e}_n(\modM)$ be
\[\{x:\mbox{ there is an $\ringR$-homomorphism from }\modN^{\mathfrak
	e}_n\mbox{ to
} \modM\mbox{ mapping }x^{\mathfrak e}_n\mbox{ to }x\}.\]
Of course,  $\varphi^{\mathfrak e}_n(\modM)$ is an additive subgroup of
$\modM$. We define
$\psi^{\mathfrak e}_n (\modM)$ similarly but using bimodule homomorphism.
In fact, we consider a set $\calE$ of such
${\mathfrak e}$'s.

However, $\modN^{\mathfrak e}_n$ is not necessarily
finitely generated as an
$\ringR$-module. So let us restrict ourselves to bimodules such that, locally
they look like direct sum of $\ringR$-modules from some class ${\mathcal K}$. This property is
denoted by $0\leq_{\aleph_0} \modM$ (essentially it can be represented
as being
${\mathbb L}_{\infty,\aleph_0}$-equivalent to such a sum; more fully $0
\leq^{{\rm ads}}_{{\mathcal K},\aleph_0}\modM$, see Definition \ref{ads definition}).
Now, if $0\leq_{\aleph_0} \modM$, we let
\[\begin{array}{r}
\varphi^{\mathfrak e}_n(\modM):=\{x\in \modM:\mbox{ there is an
	$\ringR$--homomorphism from $\modN^{\mathfrak e}_n$ into $\modN$ }\quad\\
\mbox{mapping $x^{\mathfrak e}_n$ to $x$}\},
\end{array}\]
and
\[\begin{array}{r}
\psi^{\mathfrak e}_n(\modM):=\{x:\mbox{ there is a bimodule homomorphism from
	$\modN^{\mathfrak e}_n$ into $\modN$ }\quad\\
\mbox{mapping $x^{\mathfrak e}_n$ to $x$}\}.
\end{array}\]
Also our complicated set ${\mathfrak e}$ makes us to define the set
$\modL^{\mathfrak e}_n$ of elements of $\modN^{\mathfrak e}_n$ whose
image under
bimodule homomorphism is determined by the image of $x^{\mathfrak e}_n$.

In addition, we
would like to include in our framework
the class of $\ringR$-modules of
a fix
first order complete theory $T$;
this is fine for $T$ not superstable, but
we need to replace the requirement
$0\leq_{\aleph_0} \modM$ by ``$\modM^\ast
\leq_{\aleph_0} \modM$'' and
choose ${\mathcal K}$ and $\modM^ \ast $
appropriate for $T$. For example, $\modM^ \ast $
can be any $\aleph_1$-saturated model of
$T$ and
$${\mathcal K}:=\{\modN: \modN
\mbox{ is an $\ringR$--module such that
} \modM^\ast 
\prec_{{\mathcal L} (\tau_\ringR)}
\modM^\ast  
\oplus
\modN\},$$
where ${\mathcal L} (\tau_\ringR)$ is the language of $\ringR$-modules (see Definition \ref{rmodlanguage}).
Also,   $ \modN $ is not too large,
	e.g.  it has size at most $\Vert \ringR \Vert + \Vert \ringS \Vert.$

Let   $\bar \modM=\langle \modM_\alpha: \alpha \leq \kappa \rangle$ be a semi-nice construction
and let $\modM_\kappa$ be the module defined by $\bar \modM$.
 For any ${\bf f}:\modM_\kappa\to\modM_\kappa$, $\alpha < \kappa$
and $n<\omega$, we  consider the following principle:
 For any   bimodule homomorphism
${\bf h}$	from
	$\modN^{\mathfrak e}_{n}$ into $\modM_\kappa$, and for
 every $\ell<\omega$ we have
$$
{\bf f}({\bf h}(x_n^{\mathfrak e}))\in \modM_\alpha+\varphi^{\mathfrak e}_\ell(\modM_\kappa)+\Rang({\bf h}).
$$
	
We refer to this property by saying that the statement
	$(\Pr)^{n}_\alpha{[}\fucF,{\mathfrak e}{]}$ holds.
In Lemma \ref{prnalphafe} we prove that every
$\ringR$-endomorphism ${\fucF} $
of the module
$\modM_\kappa$ constructed in Section \ref{Developing the framework of the construction} is ``somewhat definable'' and
specifically satisfies
$(\Pr)^{n(\ast)}_{\alpha} [{\fucF},
{\mathfrak e}]$
(for some
$\alpha<\kappa$, $n(\ast)<\omega$ and for all the $\mathfrak e$'s we have taken care
of). We show that $(\Pr)^{n(\ast)}_{\alpha} [{\fucF},
{\mathfrak e}]$ implies a  stronger version, denoted by
$(\Pr 1)^{n(\ast
	)}_{\alpha,z} [ {\fucF},
{\mathfrak e}   ] $
where
$ z \in \modL^{\mathfrak e}_n$.
More explicitly, if $h$ is a bimodule homomorphism
from $\modN^{\mathfrak e}_{n}$ into $\modM_\kappa$, $(\Pr 1)^{n(\ast
	)}_{\alpha,z} [ {\fucF},
{\mathfrak e}   ] $ implies that
$${\bf f}(h(x^{\mathfrak e}_{n}))-h(z) \in \modM_\alpha+\varphi^{\mathfrak
	e}_\omega(\modM_\kappa).$$
This is  the subject of  Lemma \ref{prnalphafez}.

In Section \ref{More specific rings and families}
 we try to say more: in $\modM_\kappa$ every endomorphism is in some suitable
sense equal to one in a ring $\ringdE$, whose definition is given   in 	Lemma \ref{pr rings}(iii). Here, we
remark that
the ``in some sense equal'' means: for
each $n$ we restrict ${\fucF} $
to an additive sub-group
$\varphi^{\mathfrak e}_n(\modM_{\kappa})$
(closed under ${\fucF}
$),
divide by another
$ \varphi^{\mathfrak e}_\omega(\modM_{\kappa})$
and take direct
limit;
on the top of this we have an ``error term'': we have to   divide by a
``small'' sub-module of $\modM_\kappa$, which means of cardinality
$<\lambda=\| \modM_{\kappa} \|$. A better result is: 
we divide the ring of such
endomorphisms by the ideal of those with ``small'' range and then even
``compact ones'' which are essentially of cardinality
$\leq\sup\{\|\modM\|:\modM\in  {\cal K}\}$.
We close
Section \ref{More specific rings and families} by proving the following result.
\begin{corollary}\label{1.2}
	Suppose $\ringS$ is a ring extending $\mathbb Z$ such that $(\ringS,+)$ is free
	and let $\ringR$ be a ring which is not pure semisimple. Let   ${\bf D}$ be a
	field such that  $p:=\Char({\bf D})|\Char(\ringR)$  and set
	${\mathbb Z}_p:={\mathbb Z}/ p{\mathbb Z}$ if $p>0$
	and  ${\mathbb Z}_p:={\mathbb Z}$ otherwise.
Let $\Sigma$ be a set of
equations over $\ringS$ not solvable in $\fieldD\mathop{{\otimes}}
_{{\mathbb Z}_p}(\ringS/ p\ringS)$.	Then for $\modM$ which is strongly nicely constructed,
	$\Sigma$ is not
	solvable in $\End(\modM)$.
\end{corollary}

Let us now explain how the above constructions can be applied to prove Theorem \ref{0.1}
(for more details see Section \ref{Kaplansky test problems}). We first introduce
a ring	  $\ringS_0$,  it is incredibly easy compared to  $\ringS$.
To this end, let $\ringT$ be the subring of
$\ringR$ which 1
generates. Let $\ringS_0$ be the ring extending $\ringT$
 generated by $\{{\cal X}_0,\ldots, {\cal X}_{m(\ast)-1},{\cal W},
{\cal Z}\}$ freely except:
\begin{description}
	\item[$(\ast)_1$]
 ${\cal X}^2_\ell={\cal X}_\ell$,\\
	 ${\cal X}_\ell {\cal X}_m=0$ ($\ell\neq m$),\\
	$1={\cal X}_0+\ldots+ {\cal X}_{m(\ast)-1}$,\\
	${\cal X}_\ell {\cal W}{\cal X}_m=0$ for $\ell+1\neq m\ \mod\ m(\ast)$,\\
	${\cal W}^{m(\ast)}=1$,\\
	${\cal Z}^2=1$,\\
	${\cal X}_0{\cal Z}(1-X_0)={\cal X}_0{\cal Z}$,\\
	$(1-{\cal X}_0){\cal Z}{\cal X}_0=(1-{\cal X}_0){\cal Z}$.
\end{description}

For an integer $m$ let $[-\infty,m):=\{n:n \mbox{ is an integer and }~n<m\}$. To each
$\eta\in {}^{[-\infty,m)}\omega$ we assign $\eta\restriction k:=\eta
\restriction [-\infty,\Min\{m,k\})$.
We look at
\[\begin{array}{lll}
W_0:=&\big\{\eta:&\eta\mbox{ is a function with domain of the form }[-\infty,
n)\\
& &\mbox{and range a subset of }\{1,\ldots,m(\ast)-1\}\mbox{ such that}\\
& &\mbox{for every small enough }m\in {\mathbb Z},\ \eta(m)=1\big\}
\end{array}\]
and $W_1:=W_0\times \{0, \ldots ,m(\ast)-1\}$.
Let $\bf D$ be a field  \st\
\[\ringT\mbox{ is finite }\quad \Rightarrow\quad \Char({\bf D}) \mbox{
 divides } |\ringT|.\]
So, ${\bf D} \otimes \ringS$ is the ring extending $\bf D$ by adding
${\bf D},{\cal X}_0,\ldots,{\cal X}_{m(\ast)-1}, {\cal W}, {\cal Z}$
as non commuting variables freely except satisfying the equations in $(\ast)_1$.
Let $\modM^\ast={}_{\bf D}\modM^\ast$ be the left
${\bf D}$-module freely generated by
$\{x_{\eta,\ell}: \eta\in W_0,\ \ell<m(\ast)\}.$
We make ${}_{\bf D}\modM^\ast$ to a right
$({\bf D}\mathop{{\otimes}} \ringS_0)$-module by defining $x z$ for $x\in
{}_{{\bf D}}\modM^*$ and $z\in\ringS_0$, so it is enough to deal with
$z\in
\{{\cal X}_m:m<m(*)\}\cup\{{\cal Z},{\cal W}\}$. Let $x=
\sum\limits_{\eta,\ell} a_{\eta,\ell} x_{\eta,\ell}$ where
$(\eta,\ell)$ vary on
$W_0$ and $a_{\eta,\ell}\in{\bf D}$ and $\{(\eta,\ell):a_{\eta,\ell}\neq 0\}$
is finite and we let $\sum\limits_{\eta,\ell} a_{\eta,\ell}x_{\eta,\ell})z=\sum\limits_{\eta,
	\ell}a_{\eta,\ell}(x_{\eta,\ell}z),$ where
\[\begin{array}{rcl}
x_{\eta,\ell}{\cal X}_m&:=&\left\{\begin{array}{ll}
x_{\eta,\ell}&\mbox{if }\ell=m,\\
0 &\mbox{if }\ell\neq m,
\end{array}\right.\\
\\
x_{\eta,\ell}{\cal Z} &:=&\left\{\begin{array}{ll}
x_{\eta\conc\langle\ell\rangle,0} &\mbox{if }\ell>0,\\
x_{\eta\restriction [-\infty,n-1),\eta(n-1)}&\mbox{if }\ell=0,\mbox{
	and } (-\infty,n)=\Dom(\eta),
\end{array}\right.\\

 and\\

x_{\eta,\ell}{\cal W} &:=&x_{\eta,m}\ \mbox{ when }m=\ell+1\ \mod\ m(\ast).
\end{array}\]
We get a $({\bf D},\ringS)$-bimodule. Let ${}_{\bf D}\modM^\ast_\ell=\{\sum\limits_\eta d_{\eta,\ell} x_{\eta, \ell}: \
\eta\in W_0$ and $d_{\eta, \ell}\in {\bf D}\}$, so ${}_{\bf D}\modM^\ast=
\mathop{{\bigoplus}}\limits_{\ell=0}^{m(\ast)-1} {}_{\bf D} \modM^\ast_\ell$.

Now we are ready to  define $\ringS$, but in addition $\sigma=0$
when $\modM^\ast_{{\bf D}}\sigma=0$
for every ${\bf D}$ and every $({\bf D},\ringS_0)$-bimodule ${}_{\bf D}\modM^*$  as
defined above.

 We will prove that   $\ringS$
is a free $\ringT$-module. This allow us to apply Corollary \ref{1.2}.
 Look at $\modP$, the bimodule induced by the semi-nice construction equip with the strong property presented in above. Let $\modP_\ell=\modP{\cal X}_\ell$.  It follows easily that    $
\mathop{{\bigoplus}}\limits^{m(\ast)-1}_{\ell=1}{}_\ringR\modP_\ell\cong
({}_\ringR\modP_0)^{m(\ast)-1}.
$ It is enough to show $  {}_\ringR
\modP^k_0\not\cong {}_{\ringR}\modP_0$ for all $1<k<m(\ast)-1$.
Assume $k$ is a counterexample.
We apply Corollary \ref{1.2} to find
a field ${\bf D}$ and  ${\cal Y}\in
{\bf D}\mathop{{\otimes}}\limits_\ringT \ringS$ satisfying the
following equations:
\begin{description}
	\item[$(*)_2$ ] ${\cal Y}\restriction {}_{\bf D}\modM^\ast_0$ is an
	isomorphism from
	${}_{\bf D}\modM^\ast_0$ onto $\mathop{{\bigoplus}}\limits_{\ell=1}^k
	{}_{\bf D}\modM^\ast_\ell$,
	
	${\cal Y}\restriction \mathop{{\bigoplus}}\limits_{\ell=1}^k
	{}_{\bf D}\modM^\ast_\ell$ is an
	isomorphism from $\mathop{{\bigoplus}}\limits_{\ell=1}^k
	{}_{\bf D}\modM^\ast_\ell$ onto ${}_{\bf D}\modM^\ast_0$,
	
	${\cal Y}\restriction\mathop{{\bigoplus}}\limits^{m(\ast)}_{\ell=k+1}
	{}_{\bf D}\modM^\ast_\ell$ is the identity
	
	${\cal Y}^2=1$.
\end{description}
Then, we look at the following sets:

	\[\begin{array}{ll}
	w_{\eta,\ell}&:=\{(\nu,m)|(\nu,m)=(\eta,\ell)\mbox{ or }
	\eta^\frown\langle\ell\rangle\trianglelefteq\nu~ \mbox{and}~ m<m(*)\}\\u_\ell
	&:=\{(\eta,m)| m=\ell,\ \eta\in w_0\}\\w^m_{\eta,\ell}
	&:=\ w_{\eta,\ell}\cap u_m\\w^{[1,n]}_{\eta,\ell}
	&:=w_{\eta,\ell}\cap\bigcup\limits_{m\in [1,n]}u_m.
	\end{array}\]
For any $u\subseteq \{(\eta,\ell):\eta\in W_0,\ \ell<m(\ast)\}$, we let $\modN_u$ be
the sub ${\bf D}$-module of ${}_{\bf D}\modM^\ast$ generated by
$\{x_{\eta,\ell}:(\eta,\ell)
\in u\}$. For every large enough finite subset $v\subseteq w_{\eta,\ell}$,
 we show the following is well defined:
$$
{\bf n}_{\eta,\ell}:=\dim\left(\frac{\modN_{w^0_{\eta,\ell}}}{\modN_{w^0_{\eta,\ell}\setminus v}}\right)-
\dim\left(\frac{\modN_{w^{[1,
			k]}_{\eta,\ell}}}{\modN_{w^0_{\eta,\ell}\setminus v} {\cal Y} }\right).
$$
These integers are independent from the first factor: Suppose $\eta$, $\nu\in w_0$ and $\ell\in
\{0,1,\ldots,m(\ast)-
1\}$ then ${\bf n}_{\eta,\ell}={\bf n}_{\nu,\ell}$ so we shall write
${\bf n}_\ell$.
These numbers are such that they satisfy the following equations:
$$
\ideallI_\ell =\left\{\begin{array}{ll}
0 &\mbox{if }  \ell\in [1,k]  \\
\ideallI_0+\ideallI_1+\ldots+\ideallI_{m(\ast)-1} &\mbox{if }  \ell\in[k+1,m(\ast))\emph{  or  }\ell=0
\end{array}\right.
$$
Finally, we use this equation to derive the following  contradiction $$\sum^{m(\ast)-1}_{\ell=1} \ideallI_\ell= \frac{k}{m(\ast)-1}.$$

\section{Preliminaries from algebra and logic}
\label{Preliminaries from algebra and logic}

In this section we provide some preliminaries from algebra and
logic that are needed for the rest of the paper. Let us start
 by fixing some notations  that we will use through the paper.
Recall that $\ringR$ is a ring, not necessary commutative, with $1=1_\ringR$.
\begin{convention}
By $\Cent(-)$ we mean the center of a ring $(-)$.
The rings
$\ringS$ and $\ringR$ are with $1$, and we always assume that $\ringT:=\Cent(\ringR)\cap\Cent(\ringS)$ is a commutative
ring.\footnote{Indeed, in our applications  we have $\ringT=\ringR \cap \ringS$ (see Section 6, below).}\end{convention}

\begin{Definition}
 An $(\ringR,\ringS)$-bimodule $\modM$ is a left $\ringR$-module and right
$\ringS$-module such that for all $r \in \ringR, x \in \modM$ and $~s \in \ringS$ we have
$(rx)s=r(xs)$ and that $tx=xt$ for all $t \in \ringT$. When the pair $(\ringR,\ringS)$ is clear from the context, we refer to  $\modM$ as a bimodule.
\end{Definition}
\begin{convention}
We use $\modK$, $\modM$,
$\modN$ and $\modP$
to denote bimodules (or left
$\ringR$-modules).\end{convention}

\begin{Definition}Let ${\bf f}: \modM \rightarrow \modN$ be a bimodule
	homomorphism.
\begin{enumerate} 
\item  The kernel,
 $\Ker({\bf f}):=\{x\in \modM: {\bf f}(x)=0\}$
 is a  sub-bimodule of $\modM$.
  \item  The  image,
  $\Rang({\bf f}):=\{{\bf f}(x):x\in
\modM\}$
 is a sub-bimodule of $\modN$.

\item  If $\modN$  is  a sub-bimodule of $ \modM$ then
$
\modM/\modN:=\{x+\modN:x\in \modM\}$
is a homomorphic image of $\modM$. The mapping $x\mapsto
x+\modN$ is a
homomorphism with  kernel $\modN$.

\item  For a bimodule $\modM, {\rm End}_\ringR (\modM)$
is the endomorphism ring of
$\modM$ as  a
left $\ringR$-module.
\end{enumerate}

\end{Definition}
In this paper we also consider modules and bimodules as logical structures, so let us fix  a language for them.  We only sketch the  basic
 definitions and results which are needed here, and refer to \cite{B} and \cite{D} for further information.
\begin{Definition}
\label{rmodlanguage}
\begin{enumerate}
\item  An $\ringR$-module $\modM$ is considered as a $\tau_\ringR$-structure
where $\tau_\ringR=\{0,+,-\}\cup \{{}_{r}H:r\in \ringR\},$ where the universe
is the set
of elements of $\modM$, and $0,+,-$
are interpreted naturally and ${}_{r}H$ is
interpreted as (left) multiplication by $r$ (i.e., ${}_{r}H(x):=rx$).
Then
\begin{enumerate}
  \item Terms of $\tau_\ringR$ are expressions of the form $\sum_{i<m}r_ix_i$, where $m< \omega,$ $r_i$ is in $\ringR$
  and $x_i$ is a variable.
  \item Atomic formulas of $\tau_\ringR$ are equations of the form $t_1=t_2$, where $t_1, t_2$ are terms.
\end{enumerate}
\item
An $(\ringR,\ringS)$-bimodule
is similarly interpreted as a $\tau_{(\ringR,\ringS)}$- structure
where
$\tau_{(\ringR,\ringS)}=\{0,+,-\}\cup
\{{}_{r}H:r\in \ringR\}\cup
\{ H_s:s\in \ringS\}  $
and $ {}_r H, H_s$ for $ r \in \ringR , s \in  \ringS $
are unary function symbol, which will be interpreted as
follows: $ {}_{r}H $ as left multiplication by $r$ (i.e., ${}_{r}H(x):=rx$)
and $ H_s$ as
right multiplication by $s$ (i.e., $H_s(x)=xs$).
Then
\begin{enumerate}
  \item Terms of $\tau_{(\ringR,\ringS)}$ are expressions of the form $$\sum_{i<d}\sum_{h<h_d}r_{i,h}x_is_{i,h}$$where $d,h_d< \omega,$  $r_{i,h}\in\ringR$
  $s_{i,h}\in\ringS$
  and $x_i$ is a variable.
  \item Atomic formulas of $\tau_{(\ringR,\ringS)}$ are equations of the form $t_1=t_2$, where $t_1, t_2$ are terms.
\end{enumerate}
\end{enumerate}
\end{Definition}
The next lemma is trivial.
\begin{lemma}
\label{bimodules are variety}
The class of $(\ringR, \ringS)$-bimodules is a variety, i.e., there are equations in the language of $\tau_{(\ringR, \ringS)}$ such that a
$\tau_{(\ringR, \ringS)}$-structure is an $(\ringR, \ringS)$-bimodule iff it satisfies these equations.
\end{lemma}
Let us now introduce the infinitary languages for modules and bimodules. They will play essential roles in the sequel.
\begin{Definition}
 Suppose $\kappa$ and $\mu$
are infinite cardinals, which we allow to be $\infty$, and let $\tau$
be one of $\tau_{\ringR}$ or $\tau_{(\ringR,\ringS)}$. The infinitary language $\mathcal{L}_{\mu, \kappa}(\tau)$
is defined so as its vocabulary is the same as $\tau,$ it has the same terms and atomic formulas as in $\tau,$ but we also allow conjunction and disjunction of length less than $\mu$ (i.e., if $\phi_j,$ for $j<\beta < \mu$ are formulas, then so are $\bigvee_{j<\beta}\phi_j$ and $\bigwedge_{j<\beta}\phi_j$) and quantification over less than $\kappa$ many variables (i.e., if $\phi=\phi((v_i)_{i<\alpha})$, where $\alpha < \kappa$, is a formula, then so are $\forall_{i<\alpha}v_i \phi$ and $\exists_{i<\alpha}v_i\phi$).
\end{Definition}
Note that $\mathcal{L}_{\omega, \omega}(\tau)$ is just the first order logic with vocabulary $\tau.$
Given $\kappa$, $\mu$ and $\tau$ as above, we are sometimes interested in some special formulas from $\mathcal{L}_{\mu, \kappa}(\tau)$.
\begin{Definition}
\begin{enumerate}
  \item $\mathcal{L}^{\text{cpe}}_{\mu, \kappa}(\tau)$, the class of conjunctive positive existential formulas, consists of those formulas of $\mathcal{L}_{\mu, \kappa}(\tau)$ which in their formulation
      only $\wedge, \bigwedge_{j<\beta} , \exists x$ and  $\exists_{i<\alpha}v_i$ are used (where $\beta < \mu$ and $\alpha < \kappa$).
  \item $\mathcal{L}^{\text{pe}}_{\mu, \kappa}(\tau)$, the class of positive existential formulas, is defined similarly where we also allow $\vee$ and $\bigvee_{j<\beta}$
  \item $\mathcal{L}^{\text{p}}_{\mu, \kappa}(\tau)$, the class of positive formulas, is defined similarly where we allow $\vee$, $\bigvee_{j<\beta}$  and also
  the universal quantifiers $\forall x$ and $\forall_{i<\alpha}v_i$.

  \item By a simple formula of $\mathcal{L}_{\mu, \kappa}(\tau)$, we mean a formula of the form
\[
\phi=\exists_{i<\alpha}v_i [\bigwedge_{j<\beta}\phi_j],
\]
where each $\phi_j=\phi_j((v_i)_{i<\alpha})$ is an atomic formula.
\end{enumerate}
\end{Definition}

\begin{lemma}
\label{preserving pe formulas}The following assertions are valid:
\begin{enumerate}
	\item  Suppose ${\bf f}: \modM \to \modN$ is an $(\ringR, \ringS)$-bimodule homomorphism. Then ${\bf f}$ preserves $\mathcal{L}_{\infty, \infty}^{\text{pe}}(\tau_{(\ringR, \ringS)})$-formulas.
	\item  Suppose in addition to 1) that $f$ is surjective. Then ${\bf f}$ preserves $\mathcal{L}_{\infty, \infty}^{\text{p}}(\tau_{(\ringR, \ringS)})$-formulas.
\end{enumerate}
\end{lemma}

\begin{proof}
(1). Suppose $\phi((v_i)_{i<\alpha}) \in \mathcal{L}_{\infty, \infty}^{\text{pe}}(\tau_{(\ringR, \ringS)})$ and let $a_i \in \modM,$ where $i<\alpha.$ We need to show:
	\[
	\modM \models \phi((a_i)_{i<\alpha}) \Rightarrow \modN \models \phi(({\bf f}(a_i))_{i<\alpha}).
	\]
This can be proved by induction on the complexity of the formulas, and we leave its routine check to the reader.

(2). This can be proved in a similar way; let us only consider the case of the universal formula to show how the surjectivity of the function $\bf f$
is used. Thus suppose that $\phi((v_i)_{i<\alpha})=\forall x \psi((v_i)_{i<\alpha}, x)$ is in  $\mathcal{L}_{\infty, \infty}^{\text{pe}}(\tau_{(\ringR, \ringS)})$ and let $a_i \in \modM,$ where $i<\alpha.$ We may assume that the lemma holds for $\psi.$ Suppose $\modM \models \phi((a_i)_{i<\alpha})$. We show that $\modN \models \phi(({\bf f}(a_i))_{i<\alpha})$. Thus let $b \in \modN$. As $\bf f$ is surjective, there exists $a \in \modM$ such that ${\bf f}(a)=b$ and by the assumption, $\modM \models \psi((a_i)_{i<\alpha}, a)$. By the induction hypothesis, $\modN \models \psi(({\bf f}(a_i))_{i<\alpha}, {\bf f}(a))$, i.e., $\modN \models \psi(({\bf f}(a_i))_{i<\alpha}, b)$. Since $b$ was arbitrary we have $\modN \models \phi(({\bf f}(a_i))_{i<\alpha})$.
\end{proof}

\begin{remark}
Note that there is no need for ${\bf f}$ to preserve $\neg$ formulas. But if ${\bf f}$ is an isomorphism, then it preserves all formulas.
\end{remark}
The next lemma shows that under suitable conditions, we can replace $\rm cpe$-formulas by simple formulas.
\begin{lemma}\label{reducing to simple formula}	
	 Let  $\tau$ be either $\tau_{\ringR}$ or $\tau_{(\ringR, \ringS)}$,  and suppose $\mu_1\geq \mu,\kappa$  is
	regular.
 If $\varphi(\bar{x}) \in
\mathcal{L}^{\rm cpe}_{\mu,\kappa}(\tau)$, then
we can find as equivalent simple formula in
$\mathcal{L}_{\mu_1,\mu_1}(\tau)$.
\end{lemma}
We are also interested in definable subsets of (bi)modules.
\begin{Definition}
Let $\tau$
be one of $\tau_{\ringR}$ or $\tau_{(\ringR,\ringS)}$. Given a  $\tau$-structure $\modM$ and a formula $\phi(x_0, \cdots, x_{n-1})$ in $\mathcal{L}_{\infty, \infty}(\tau)$, let
\[
\phi(\modM) = \{\langle a_0, \cdots, a_{n-1}\rangle \in~^n{\modM}: M \models \phi(a_0, \cdots, a_{n-1})    \}.
\]
\end{Definition}
Here, by $\lg(-)$ we mean the length function.
\begin{lemma}
\label{formulas and direct sum}
	 Let  $\tau$ be either $\tau_{\ringR}$ or $\tau_{(\ringR, \ringS)}$, and let
  $\varphi(\bar{x})\in \mathcal{L}_{\mu,\kappa}^{\rm p}(\tau)$.
Suppose $\bar{z}_\ell \in {}^{\lg(\bar{x})} \modM_\ell$, for $\ell=1, 2$ and
$\bar{z}_\ell= \langle z^\ell_i :i < \lg (\bar{x}) \rangle$,
$\modM=\modM_1\oplus \modM_2$,
and $\bar{z}= \langle z_i:i < \lg (\bar{x} )
 \rangle$ where
$\modM \models z_i=z^1_i+z^2_i$.
Then
\[\modM \models\varphi(\bar{z})\quad\Leftrightarrow\quad
\bigwedge\limits_{\ell=1}^{2} \modM_\ell \models\varphi(\bar{z}_\ell).\]
Furthermore, if
  for $\ell=1,2$ and $i < \lg (\bar{x}), z^\ell_i= 0_{M_\ell}$, then $\modM_\ell \models \varphi
(\bar{z}_\ell)$
 and
$\bar{z}=\bar{0}_{\lg (\bar{z})}$.
\end{lemma}
\begin{proof}
The desired claim follows by an easy induction on the complexity of the formula
$\varphi(\bar{x})$.
\end{proof}
The above lemma implies if $\varphi(x)\in \mathcal{L}_{\mu,\kappa}^{\rm cpe}(\tau),$ where
 $\tau$ is $\tau_{\ringR}$ or $\tau_{(\ringR, \ringS)}$ and if $\modM=\modM_1\oplus \modM_2$, then
 $\varphi(\modM)=\varphi(\modM_1)\oplus \varphi(\modM_2)$.
\begin{lemma}
\label{definable subgroups by formulas}
For each  bimodule $\modM$, the following assertions hold:
\begin{enumerate}
  \item If $\varphi(x)\in \mathcal{L}^{{\rm cpe}}_{\infty,
\infty}(\tau_{(\ringR, \ringS)})$, then $\varphi(\modM)$
is a subgroup of
$\modM$.
  \item If
$\varphi(x)\in \mathcal{L}^{{\rm cpe}}_{\infty,
\infty}(\tau_{\ringR})$,
then $\varphi(\modM)$ is a sub-right
$\ringS$-module, but not
necessarily
a  sub-$\ringR$-module. Furthermore, if $\ringR$ is commutative, then
 it is a sub-$ \ringR $-module as well.
\end{enumerate}
\end{lemma}
\begin{proof}
This follows  by induction on the complexity
of $\varphi(x).$ The straightforward details are leave to  the reader.
\end{proof}
Another notion from model theory which is of importance for us is the notion of omitting types:\begin{discussion}Suppose $\tau$ is a countable languages
 and suppose $\cal M$ is a $\tau$-structure.
 \begin{enumerate}
 	\item
 Given $A \subseteq \cal M,$ by an $n$-type over $A$ we mean a set
$p(v_1, \cdots, v_n)$  of formulas, whose free variables are $v_1, \cdots, v_n$ such that every finite $p_0 \subseteq p$
is realized, i.e., there are $x_1, \cdots, x_n$ in $\cal M$ such that for all $\varphi \in p_0,~\cal M \models \varphi(x_1, \cdots, x_n)$.
\item  A type is complete if it is maximal under inclusion. By the axiom of choice each type can be extended into a complete type.

\item The type $p$ is isolated by some formula $\psi(v_1, \cdots, v_n) \in p$ if for every $\varphi \in p,$
\[
\cal M \models \forall x_1 \cdots \forall x_n (\psi(x_1, \cdots x_n) \rightarrow \varphi(x_1, \cdots, x_n)).
\] \end{enumerate}\end{discussion}
It is clear that if $p$ is  isolated by some formula $\psi(v_1, \cdots, v_n) \in p$, then it is realized. The omitting types theorem says that the converse is also true.
\begin{lemma}
\label{omitting types theorem}
Let  $\tau$ be a first order countable vocabulary and let $T$  be a complete $\tau$-theory.
If $p$ is a complete type which is not isolated, then there is a countable $\tau$-structure $\cal M\models T$  which omits (i.e., does not realize) $p$.
\end{lemma}

\begin{Definition}
		Suppose  that $\modM_0,\modM_1$ and $\modM_2$ are
	(bi)modules and $\bold{g}_i : \modM_0 \to \modM_i$  are (bi)module homomorphisms
	for $i = 1, 2$. The amalgamation of $\modM_1$ and $\modM_2$ along  $\modM_0$
	is $\modM:=\frac{\modM_1\oplus \modM_2}{(\bold{g}_1(m),-\bold{g}_2(m):m\in \modM_0)}$.
	There are natural maps $\bold{f}_i:\modM_i\to \modM$ such that the following diagram is commutative:
	
	\begin{picture}(300,150)
	\put(40,60){$ $}
	\put(120,0){\begin{picture}(150,150)
		\put(0,25){$\modM_0$}
		\put(5,45){\vector(0,1){35}}
		\put(-0,85){$\modM_1$}
		\put(30,90){\vector(1,0){45}}
		\put(30,90){\rm  }
		\put(90,85){$\modM$}
		\put(95,45){\vector(0,1){35}}
		\put(90,25){$\modM_2$}
		\put(45,10){$ $}
		\put(110,60){$ $}
		\put(30,28){\vector(1,0){45}}
		\end{picture}}
\end{picture}

It may be nice to note that the common notation for this is  $\modM_1+_{\modM_0} \modM_2$.
\end{Definition}

\begin{Definition}
	Let $\modM$ be an  $\ringR$-module and  $\ringS$
	be any subring of $\End_\ringR(\modM)$. We say that $\ringS$ is algebraically
	closed in $\End_\ringR(\modM)$ if every finite system of  equations  in several variables over  $\ringS$
	which has a solution in $\End_\ringR(\modM)$,  also
	has a solution in $\ringS$
\end{Definition}


Let us recall the definition of pure semisimple rings.
\begin{Definition}
	A ring $\ringR$ is left pure semisimple if every  left $\ringR$-module is pure-injective.
\end{Definition}
The next theorem gives several equivalent formulations for the above notion.
\begin{theorem}(See \cite[Theorem 4.5.7]{P2})
	The following conditions on a ring $\ringR$ are equivalent:
	\begin{enumerate}
		\item[(a)] $\ringR$ is  left pure semisimple;	
		\item[(b)] every    left $\ringR$-module is a direct sum of indecomposable modules;
		\item[(c)] there is a cardinal number $\kappa$ such that every  left $\ringR$-module is a direct sum
		of modules, each of which is of cardinality less than $\kappa$;
		\item[(d)] there is a cardinal number $\kappa$ such that every  left $\ringR$-module is a pure
		submodule of a direct sum of modules each of cardinality less than $\kappa$.
	\end{enumerate}
\end{theorem}

The following result of Shelah gives a model theoretic criteria for rings which are not pure semisimple,
and plays an important role in this paper.

\begin{theorem}(Shelah, \cite[8.7]{Sh54})
\label{shelahpure}
	Suppose $\ringR$  is not left pure semisimple. Then  for some bimodule $\modM$ and a sequence $\bar{\varphi}=\langle \varphi_n(x):n<\omega\rangle$, each $\varphi_n$ is
a conjunctive positive existential formula and the sequence $\langle \varphi_n(\modM) : n <\omega\rangle$
	is strictly decreasing.
\end{theorem}

\section{Developing the framework of the construction}
\label{Developing the framework of the construction}

In this section we develop some part of the theory that we need for our construction. The main result of this section is
Theorem \ref{nice construction lemma}, which gives, in $\text{ZFC}$, a semi-nice construction, that plays an essential role in the next sections.
\begin{Definition}
\label{classcale}
\begin{enumerate}
\item Let $ \calE_{(\ringR,\ringS)}=
\calE{(\ringR,\ringS)}$
be the class of all
$${\mathfrak e}=\langle \modN_n,x_n,g_n:n <\omega\rangle$$
 such that:
\begin{enumerate}
\item $\modN_{n}$ is an $(\ringR,\ringS)$-bimodule,
\item $x_{n}\in \modN_{n}$,
\item $g_{n}$ is a bimodule homomorphism from $\modN_{n}$ to
$\modN_{n+1}$ mapping $x_{n}$ to $x_{n+1}$.
\end{enumerate}
Given ${\mathfrak e} \in \calE_{(\ringR,\ringS)}$ we denote it by ${\mathfrak e}=\langle \modN_n^{{\mathfrak e}},x_n^{{\mathfrak e}},g_n^{{\mathfrak e}}: n <\omega\rangle$.

\item We call ${\mathfrak e}\in \callE_{(\ringR,\ringS)}$
non-trivial
if for every $n$ there is no homomorphism from
$\modN^{\mathfrak e}_{n+1}$ to $\modN^{\mathfrak e}_n$ as $\ringR$-modules
mapping $x^{\mathfrak e}_{n+1}$ to $x^{\mathfrak e}_n$.

\item Let ${\mathfrak e}=\langle \modN_n,x_n,g_n:n <\omega\rangle \in \calE_{(\ringR,\ringS)}$. For each $n \leq m$, we define: $$g_{n,m}:=g_n\circ g_{n+1}\circ\ldots\circ g_{m-1}.$$ We set  $g_{n,n}:=
{\rm id}_{\modN_n}$ and note that
$n_0 \leq n_1 \leq n_2$ implies $g_{n_1,n_2} \circ
g_{n_0,n_1} = g_{n_0,n_2}$.
We denote $g_{n, m}$ by $g^{\mathfrak
e}_{n,m}.$

\item For ${\mathfrak e} \in \calE_{(\ringR,\ringS)}$ and an infinite
${\mathcal U}\subseteq \omega$, let ${\mathfrak e}'=:{\mathfrak e}
\restriction {\mathcal U}$,
the restriction of $ {\mathfrak e}$ to  ${\mathscr U },$  be defined by
\[\modN^{{\mathfrak e}'}_\ell=:\modN^{\mathfrak e}_{m(\ell)},\quad
x^{{\mathfrak
e}'}_\ell=: x^{\mathfrak e}_{m(\ell)},\quad g^{{\mathfrak e}'}_\ell=
g^{\mathfrak e}_{m(\ell),m(\ell+1)},\]
where $m(\ell)$ is the $\ell$--th member of $\mathcal U$. Clearly ${\mathfrak e}'$ is in
$\calE_{(\ringR,\ringS)}$.
\end{enumerate}
\end{Definition}
Let us define a special sub-class of $\callE_{(\ringR,\ringS)}$.
\begin{Definition}
Let  ${\callE}_{\mu,\kappa}$ be the class
of $\mathfrak e \in\callE_{(\ringR,\ringS)}$ such
that for each $n<\omega$, $\modN^{\mathfrak e}_n$, as a bimodule, is  generated
 by
$<\kappa$
elements freely except satisfying $<\mu$ equations.
This means, there are $(x_i)_{i< \alpha<\kappa}$ and atomic formulas $t_j=0,$ for $j< \beta<\mu$, such that
$\modN^{\mathfrak e}_n = (\bigoplus_{i<\alpha}\ringR x_i \ringS) /\modK$, where $\modK$ is the bimodule generated by $\langle  t_j: j<\beta  \rangle$.
\end{Definition}
We now define what it means for a sequence of formulas to be adequate with respect to an element of $\calE_{(\ringR,\ringS)}$:
\begin{Definition}
\label{adequate definition}Let  ${\mathfrak e}:= \langle \modN_n,x_n,g_n:n
<\omega\rangle$ and  $\bar{\varphi}:=\langle\varphi_n:n<\omega\rangle$.
\begin{enumerate}
\item  The sequence $\bar{\varphi}$ is called $(\mu,
\kappa)$-adequate with respect to ${\mathfrak e}$
if the following conditions satisfied:
\begin{enumerate}
\item[$(\alpha)$] $\varphi_n=\varphi_n(x)$ is a formula from $
\mathcal{L}^{\rm cpe}_{\mu,\kappa}(\tau_{\ringR})$,
\item[$(\beta )$] $\varphi_{n+1}(x)\vdash\varphi_{n}(x)$ (for the class of
$\ringR$-modules)\note{we can restrict  ourselves
to those we actually use, i.e.
the ones from $\cl({\cal K})$.},
\item[$(\gamma )$] $\modN_n\models\varphi_n(x_n)\ \&\
\neg\varphi_{n+1}(x_n)$.
\end{enumerate}
Also, we say ${\mathfrak e}$ is $(\mu,\kappa)$-adequate  with respect to ${\mathfrak e}$.


\item We say $\mathfrak e$ is explicitly $(\mu,\kappa$)-adequate with respect to ${\mathfrak e}$ if
$ \mathfrak e$ is $ (\mu,\kappa)$-adequate   with respect to ${\mathfrak e}$ and
$\mathfrak e \in \callE_{\mu,\kappa}$.

\item For simplicity, $\kappa$-adequate means $(\infty,\kappa)$-adequate  with respect to ${\mathfrak e}$. Also,   adequate is referred to
$\aleph_0$-adequate  with respect to ${\mathfrak e}$.

\item We say that $\bar{\varphi}$ is $(\mu,\kappa)$-adequate
if $\bar{\varphi}$ is $(\mu,\kappa)$-adequate with respect to some ${\mathfrak e} \in \calE_{(\ringR,\ringS)}$.

\item We say $\bar{\varphi}$ is adequate
 with respect to $\callE \subseteq \calE_{(\ringR,\ringS)}$ if it is adequate with respect to some
${\mathfrak e}\in \calE$.
\end{enumerate}
\end{Definition}
It follows from Lemma \ref{preserving pe formulas} that any $(\mu,\kappa)$-adequate ${\mathfrak e} \in \calE_{(\ringR,\ringS)}$ is non-trivial.

\begin{Definition} \label{fisi}Let ${\mathfrak e}\in {\callE}$, $n<\omega$ and $\kappa$ be given. We define
\begin{enumerate}
\item $\varphi^{{\mathfrak e},\kappa}_n:=\bigwedge\left\lbrace\varphi(x):\varphi\in
\mathcal{L}^{{\rm cpe}}_{\infty,\kappa}
(\tau_{\ringR}) ~{\rm and~  }\modN^{\mathfrak
	e}_n\models\varphi (x^{\mathfrak e}_n) \right\rbrace\footnote{This is not
	a formula being a conjunction on a class, not a set, but
	when we deal with an $ \ringR $-module $ \modM $
	it is enough to
	restrict ourselves to $\mathcal{L}^{{\rm cpe}}_{\lambda,\kappa}$ for
	$\lambda=(||\modM
	||^{<\kappa})^+$.}$,
 \item $\psi^{{\mathfrak e},\kappa}_n:=\bigwedge\left\lbrace\psi(x):\psi\in
\mathcal{L}^{{\rm cpe}}_{\infty,\kappa}
(\tau_{(\ringR,\ringS)}) ~{\rm and~ }\modN^{\mathfrak e}_n\models\psi(x^{\mathfrak e}_n)
\right\rbrace$,
\item
$\bar{\varphi}^{{\mathfrak e},\kappa}:=\langle \varphi^{{\mathfrak e},\kappa}_n:
n<\omega\rangle$,
\item$ \bar{\psi}^{{\mathfrak e},\kappa}:=\langle
\psi^{{\mathfrak e},\kappa}_n:n<\omega\rangle,$
\item
$\varphi^{{\mathfrak e},\kappa}_\omega
(x):=\bigwedge\limits_{n<\omega}\varphi^{{\mathfrak e},\kappa}_n(x)$,
\item
$\psi^{{\mathfrak e},\kappa}_\omega(x):=\bigwedge\limits_{n<\omega}\psi^{{\mathfrak e},\kappa}_n(x)$.
\end{enumerate}
\end{Definition}

\begin{remark} In Definition \ref{fisi}, we omit the index  $\kappa$, when it is $\kappa(\mathfrak e)$ for $\psi$ or $\kappa_{\ringR}(\mathfrak e)$ for $\varphi$, where $\kappa(\mathfrak e)$ and $\kappa_{\ringR}(\mathfrak e)$ are defined in Definition \ref{sehaft}, see below.
\end{remark}

The next lemma shows the relation between different $\varphi^{{\mathfrak e},\kappa}_n$'s and  $\psi^{{\mathfrak e},\kappa}_n$'s.
\begin{lemma}
\label{related to n}
For $n\geq m$, $\varphi^{{\mathfrak e},\kappa}_n
(x)\vdash \varphi^{{\mathfrak e},
\kappa}_m (x)$,  $\psi^{{\mathfrak e},\kappa}_n
(x)\vdash \psi^{{\mathfrak e},
\kappa}_m (x)$ and  $\psi^{{\mathfrak e},\kappa}_n
(x)\vdash \varphi^{{\mathfrak e},
\kappa}_n (x)$.
\end{lemma}
\begin{proof}
The lemma follows easily by applying Lemma \ref{preserving pe formulas} to $g^{\mathfrak e}_{m, n}$.
\end{proof}

\begin{Definition}\label{sehaft}
\begin{enumerate}
\item For ${\mathfrak e}\in \calE_{\ringR,\ringS}$ we define:
\begin{enumerate}
\item[(a)]
$\kappa({\mathfrak e})$ as the first infinite cardinal $\kappa$ \st\ for each
$n<\omega$, the bimodule $\modN^{{\mathfrak e}}_n$ is generated by a
 set of $<\kappa$ elements.
\item[(b)]
$\kappa_{\ringR }
({\mathfrak e})$ be the first infinite cardinal $\kappa$ \st\ for each
$n<\omega$, the bimodule $\modN^{{\mathfrak e}}_n$
as an $ \ringR $-module 
is generated by a
 set of $<\kappa$ elements.
\end{enumerate}
\item For each  $\calE\subseteq \calE_{(\ringR,\ringS)}$, we let
\begin{enumerate}
\item[(a)] $\kappa(\calE):=\sup \{\kappa ({\mathfrak e}):{\mathfrak e}\in \calE\}$
and
\item[(b)] $\kappa(\calE)_{\ringR}:=\sup\{\kappa_{\ringR} ({\mathfrak e}):
{\mathfrak e}\in \calE\}$.\end{enumerate}
\end{enumerate}
\end{Definition}
We frequently use the following lemma.
\begin{lemma}
\label{formula vs hom}
\begin{enumerate}
\item Assume $\modN$ is a bimodule (resp. $\ringR$-module)  generated by $<\kappa$
members freely except satisfying $<\mu$ equations and $x\in \modN$.
Then there is a simple formula \space
$\varphi
\in \mathcal{L}_{\mu,\kappa}(\tau_{(\ringR, \ringS)})$ (resp. $\varphi\in
\mathcal{L}_{\mu,\kappa}(\tau_{\ringR})$)
such that for any bimodule
(resp. $\ringR$-module) $\modM$ and $y\in \modM$, the following are
equivalent:
\begin{enumerate}
\item   $\modM \models \varphi(y)$,
\item  for some homomorphism ${\bf f}$ from  $ \modN $
into $\modM$,  ${\bf f}(x)= y$
(if $\modN$ is an $\ringR$-module this means ${\bf f}$ is an
$\ringR$-homomorphism).
\item for some $\modN'$ extending $\modN$ and a homomorphism
${\bf f}$ from
$\modN'$ into $\modM$ we have ${\bf f}(x)=y$.
\end{enumerate}
\item If $\modN=\modN' \oplus \modN''$
and $x \in N'$,
then it suffices that
$\modN'$ is generated by $< \kappa$ members freely except satisfying
 $<\mu$ equations.
\item Clause (1) applies if $\modN=\modN^{\mathfrak e}_n$, $x=x_n^{\mathfrak e}$ and
$\varphi= \psi^{{\mathfrak e},\kappa}_n$, when $\modN^{\mathfrak e}_n$ is generated
 by
$<\kappa$ members freely except satisfying $<\mu$ equations (as a bimodule). Similarly for
$\varphi=\varphi^{{\mathfrak e},\kappa}_n$ for $\ringR$-modules.
\end{enumerate}
\end{lemma}
\begin{proof}
(1). We prove the lemma for the case of bimodules, exactly the same proof works for $\ringR$-modules. Suppose $\modN$ is a bimodule    generated by $<\kappa$
members  freely except  satisfying $<\mu$ equations and $x\in \modN$. So it has the form $\modN=\bigoplus_{i<\alpha}\ringR x_i \ringS / \modK$,
where $\modK$ is the bimodule generated by $\langle  t_j: j<\beta  \rangle$, where each $t_j=0$ is an atomic formula, $\alpha < \kappa$ and $\beta < \mu$. Then
\begin{itemize}
\item
$x= \sum_{i<d}\sum_{h<h_d}r_{i,h}x_is_{i,h}$, where $d,h_d< \omega,$  $r_{i,h}\in\ringR$ and
$s_{i,h}\in\ringS$. We should remark that  for all but finitely many of them we have $ r_{i,h}x_is_{i,h}=0.$
\item $t_j =\sum_{i<d}\sum_{h<h_d}r^j_{i,h} x_i s^j_{i,h},$ where $r^j_{i,h} \in \ringR$ and $ s^j_{i,h} \in \ringS$. Again, for all but finitely many of them  we have $ r^j_{i,h}x_i s^j_{i,h}=0.$
\end{itemize}
Now consider the formula
\[
\varphi(v)=\exists_{i<\alpha}v_i [\bigwedge_{j<\beta} \sum_{i<\alpha}\sum_{h<h_d}r^j_{i,h}  v_i s^j_{i,h}=0 \wedge v=\sum_{i<\alpha}\sum_{h<h_d}r_{i,h}  v_i s_{i,h} ].
\]
It is clearly a simple formula in $\mathcal{L}_{\mu,\kappa}(\tau_{(\ringR, \ringS)})$. Now suppose  $\modM \models \varphi(y)$. Thus we can find
$y_i,$ for $i<\alpha$, such that in $M$, $$\sum_{i<\alpha}\sum_{h<h_d}r^j_{i,h}  y_i s^j_{i,h}=0,$$ and $$y=\sum_{i<\alpha}\sum_{h<h_d}r_{i,h}  y_i s_{i,h}.$$ Define ${\bf f}: \modN \to \modM$ so that for each $i<\alpha, f(x_i)=y_i.$ Then $f$ is as required.

(2). This is clear and (3) follows from the fact that $\psi^{{\mathfrak e},\kappa}_n \vdash \varphi$ (because $\modN^{\mathfrak e}_n \models \varphi(x_n^{\mathfrak{e}})$).
\end{proof}
The next lemma shows that each non-trivial ${\mathfrak e}\in {\callE}_{\mu,\kappa}$ has a canonical adequate
sequence.
\begin{lemma}
If ${\mathfrak e}\in {\callE}_{\mu,\kappa}$ is non-trivial and
$\mu\geq\kappa\geq\aleph_0$
then $\bar{\varphi}^{{\mathfrak e},\kappa}$ is $(\mu,
\kappa)$-adequate with respect to $\mathfrak e$.
\end{lemma}
\begin{proof}
It is easily seen, by the structure of each $\modN_n^{\mathfrak e}$ that $\varphi^{\mathfrak e, \kappa}_{n}$ is equivalent to
a formula of $\mathcal{L}^{{\rm cpe}}_{\infty,\kappa}
(\tau_{\ringR})$.  According to  Lemma \ref{related to n}   $\varphi^{\mathfrak e, \kappa}_{n+1}  \vdash \varphi^{\mathfrak e, \kappa}_{n}.$
By the definition of  $\varphi^{\mathfrak e, \kappa}_{n}$, we have $\modN_n^{\mathfrak e} \models \varphi^{\mathfrak e, \kappa}_{n}(x_n^{\mathfrak e})$. It remains to show that
$\modN_n^{\mathfrak e} \models \neg\varphi^{\mathfrak e, \kappa}_{n+1}(x_n^{\mathfrak e})$. Suppose not. It follows from  Lemma
\ref{formula vs hom}(3) that there is a bimodule homomorphism ${\bf h}: \modN_{n+1}^{\mathfrak e} \to \modN_n^{\mathfrak e}$
with ${\bf h}(x_{n+1}^{\mathfrak e})=x_n^{\mathfrak e},$ contradicting the non-triviality of
$\mathfrak e$.
\end{proof}

\begin{Definition}
\label{context definition}
A $\lambda$-\emph{context} is a tuple
$$\frakss:=({\mathcal K},\modM_\ast ,
\calE, \ringR, \ringS, \ringT)=({\mathcal K}^\frakss,\modM_\ast
^\frakss,\calE^\frakss,
\ringR^\frakss, \ringS^\frakss, \ringT^\frakss)$$
where
\begin{enumerate}
	\item  $\ringR, \ringS$ and $\ringT$ are as usual.
\item ${\cal K}$ is a  set
of $(\ringR, \ringS)$-bimodules.

\item ${\callE}$ is a subset of
${\callE}_{\infty,\lambda}$
closed under restrictions
(i.e., if ${\mathfrak e} \in {\callE}$
 and ${\cal U}  \subseteq \omega $ is
infinite then
${\mathfrak e} \restriction {\cal U} \in {\callE}$).

\item If ${\mathfrak e}\in {\callE}$, then
$\ \modN^{\mathfrak e}_n\in {\cal K}$
for each $n<\omega$.
\item
$\modM_\ast $
is a bimodule.
If $\modM_\ast $
is omitted we mean the zero bimodule.
\end{enumerate}
Let   $\chi({\cal K})$ be
the minimal cardinal $ \chi
\geq\aleph_0$ such that
every $\modN\in {\cal K}$ is generated, as a bimodule, by a set of $<\chi$
members and set
$$||{\mathfrak m}||=\sum   \{ \Vert \modN  \Vert : \modN \in {\mathscr K} \}
+\Vert {\mathfrak E}   \Vert + ||\modM_\ast || ||\ringR||+||\ringS||+\aleph_0.$$
\end{Definition}

\begin{remark}
	We omit $\lambda$ from the $\lambda$-context, if it is clear from the context and then we mean for the minimal such $\lambda$.
	\end{remark}

\begin{convention}
From now on suppose that
${\frakss}=({\cal K}, \modM_\ast ,
{\callE}, \ringR, \ringS,\ringT)$
is a
context.\end{convention}

\begin{definition}
\begin{enumerate}
	\item Given a context $\frakss=({\mathcal K},\modM_\ast ,
\calE, \ringR, \ringS, \ringT)$, we set:	\begin{itemize}
		\item  $c\ell_{\rm is}({\cal K})$ be the
		class of bimodules isomorphic to
		some $\modK\in {\cal K}$.

		\item $c\ell^\kappa_{\rm is}({\cal K})$ be the
		class of bimodules of the form
		$\modM=\bigoplus\limits_{i<j}\modM_i$, where
		$j<\kappa$, and for $i<j, \modM_i\in c\ell_{\rm is}({\cal K})$.
		
		\item $c\ell({\cal K})=c\ell_{\rm ds}({\cal K})$ be the class of
		bimodules isomorphic to a direct sum of bimodules from
		$c\ell_{\rm is}({\cal K})$
		(so $c\ell_{\rm is}(c\ell({\cal K}))=c\ell({\cal K})$).

		\item $\kappa(\frakss)=\kappa (\calE^\frakss)$.
		\item $ \kappa_\ringR(\frakss)=
		\kappa_\ringR (\calE^\frakss)$.
	\end{itemize}

\item Following Definition \ref{adequate definition}, let us say that
$\bar{\varphi}$
is adequate
for a context $\frakss$ if it is adequate with respect to $\calE^\frakss$.

\end{enumerate}
\end{definition}

\begin{Definition}By a ${\cal
K}$-bimodule we mean a bimodule from $c\ell_{\rm is}({\cal K})$. Also, we say $\modM_1$ is a ${\cal K}$-direct summand of
$\modM_2$ if
$\modM_2=\modM_1\oplus \modK$
for some
$\modK\in c\ell_{\rm ds} ({\cal K})$.\end{Definition}

\begin{Definition}
\label{ads definition}
We say $\modM_1$ is an almost direct ${\cal K}$-summand of
$\modM_2$ with respect to $\kappa$
denoted $\modM_1\leq^{\rm ads}_{{\cal K},\kappa} \modM_2$,
if player II has a winning strategy in the following game $\Game_{{\cal K},\kappa}^{\modM_1, \modM_2}$ of length $\omega$:
in the $n^{\rm th}$ move player I chooses a subset $A_n\subseteq \modM_2$ of
cardinality $<\kappa$, and then player II chooses a sub-bimodule $\modK_n\subseteq
\modM_2.$
Player II wins iff:
\begin{enumerate}
\item[(a)] $A_n\subseteq \modM_1+\sum\limits_{\ell\leq n}\modK_\ell$,
\item[(b)] $\modK_n$ is in $c\ell^\kappa_{\rm is}({\cal K})$,
\item[(c)]
$\modM_1+\sum\limits_{\ell<\omega}\modK_{\ell}=\modM_1\oplus
\bigoplus\limits_{\ell<\omega} \modK_{\ell}$.
\end{enumerate}
We usually write $\leq_{\kappa}$ or $\leq^{\rm ads}_\kappa$
 instead of $\leq^{\rm ads}_{{\cal K},\kappa}$ if
 ${\mathcal K}$ is clear. We also
may write $\leq_{{\frakss},\kappa}$
or $\leq^{\rm ads}_{{\frakss},\kappa}$, when $\cal K = \cal K^{\frakss}$.
\end{Definition}
We now define another order between bimodules and then connect it to the above defined notion.
\begin{Definition}
\label{pr order relations} Let $\modM_1$ and $\modM_2$ be two bimodules.
\begin{enumerate}
\item $\modM_1\leq^{\rm pr}_\varphi \modM_2$ means that
$\modM_1\subseteq \modM_2$ and
if $\psi=\psi(\bar{x})$ is a sub-formula of $\varphi$ and
$\bar{z}\in {}^{\lg
 ( \bar{x} ) }
\modM_1$ then
$$\modM_1\models\psi(\bar z) \iff
\modM_2\models\psi(\bar z).$$

\item By $\modM_1 \leq^{\rm qr}_\varphi \modM_2$ we mean that
$\modM_1 \subseteq \modM_2$
and if $\psi=\psi(\bar{x})$ is a sub-formula of
$\varphi$ and $\bar{z} \in
{}^{ (  \lg \bar{x} )   }\modM_1$, then
$$\modM_2 \models \varphi (\bar{z})
\Rightarrow
\modM_1 \models \psi (\bar{z}).$$

\item Assume $\mathcal{F}$ is a set of formulas. The  notation  $\modM_1\leq^{\rm pr}_{\mathcal{F}} \modM_2$ (resp. $\modM_1\leq^{\rm qr}_{\mathcal{F}} \modM_2$) means
$\modM_1\leq^{\rm pr}_{\varphi}\modM_2$  (resp. $\modM_1\leq^{\rm qr}_{\varphi}\modM_2$) for all $\varphi \in \mathcal{F}.$ In the special case
$\bar\varphi:=\langle \varphi_n: n<\omega \rangle$, the property
$\modM_1\leq^{\rm pr}_{\bar\varphi} \modM_2$  holds iff
$\modM_1\leq^{\rm pr}_{\varphi_n}\modM_2$ for each $n$.
\end{enumerate}
\end{Definition}
For instance,  let $\varphi$ be the following   simple formula
\[\varphi=\varphi(x)=(\exists y_0,\ldots,y_i \ldots)_{i<\alpha}
\bigwedge\limits_{j<\beta} [a_j x=\sum\limits_{i<\alpha} b_{j,i}
y_i]\]
where $a_j,b_{j,i}\in \ringR$. Then
$\modM_1\leq^{\rm pr}_\varphi \modM_2$ means that
if for some $x\in \modM_1$ and  $y_i\in \modM_2$
 we have
\[a_j x-\sum_{i<\alpha} b_{j,i} y_i=0\quad \forall j<\beta\]
then there are $y'_i\in \modM_1$
such that
\[a_j x-\sum\limits_{i< \alpha} b_{j,i}y'_i=0\quad \forall j<\beta.\]
\begin{remark}
Note that if $\varphi$ is existential (e.g. a simple formula), then
$\leq^{\rm qr}_\varphi$
 is equal to $\leq^{\pr}_\varphi$.
\end{remark}

 The next lemma connects the above defined relations
 to each other.
 \begin{lemma}Let $\modM_1$ and $\modM_2$ be two bimodules such that $\modM_1$ is  ${\cal K}$-nice. The following assertions are true:
 \label{comparing two orders}
 \begin{enumerate}
\item  $\modM_1 \leq_\kappa \modM_2$ iff
$\modM_1 \leq^{\rm pr}_{\mathcal{L}_{\infty,\kappa}(\tau_{(\ringR, \ringS)})}\modM_2$.

\item If $\modM_1\leq_\kappa \modM_2$ and $\varphi(x) \in \mathcal{L}^{\rm cpe}_{\infty,\kappa}(\tau_{(\ringR, \ringS)})$,
then $\varphi(\modM_1)=\modM_1\cap\varphi(\modM_2)$.

\end{enumerate}
 \end{lemma}
 \begin{proof}
Clause (2) follows from (1) and the definition of $\leq^{\rm pr}$, so let us prove (1). Let $\modM_0$
witness $\modM_1$ is $\mathcal K$-nice.

Suppose  $\modM_1 \leq_\kappa \modM_2$. We prove, by induction on the complexity of the formula $\varphi \in \mathcal{L}_{\infty,\kappa}(\tau_{(\ringR, \ringS)}),$ that $\modM_1 \leq^{\rm pr}_{\varphi}\modM_2$.

 The only non-trivial case where the assumption $\modM_1 \leq_\kappa \modM_2$ is used is the case of existential quantifier\footnote{The case of universal quantifier then follows easily as $\forall \equiv \neg \exists \neg.$}, so let us suppose $\varphi(v_i)_{i<\alpha}$ is of the form
 $$\exists_{j<\beta}w_j[\psi((w_j)_{j<\beta}, (v_i)_{i<\alpha})],$$ the claim holds for $\psi=\psi((w_j)_{j<\beta}, (v_i)_{i<\alpha})$ and suppose $(x_i)_{i<\alpha} \in \modM_1$. We want to show that
 \[
 \modM_1 \models \varphi((x_i)_{i<\alpha}) \iff \modM_2 \models \varphi((x_i)_{i<\alpha}).
 \]
 As $\modM_1$ is a submodule of $\modM_2,$ by the induction hypothesis, if $\modM_1 \models \varphi((x_i)_{i<\alpha})$, then $\modM_2 \models \varphi((x_i)_{i<\alpha}).$
 Conversely suppose that $\modM_2 \models \varphi((x_i)_{i<\alpha}).$ Then we can find $(y_j)_{j<\beta} \in \modM_2$ such that
 $\modM_2 \models \psi((y_j)_{j<\beta}, (x_i)_{i<\alpha}).$

 Now consider the game $\Game_{\mathcal K, \kappa}^{\modM_1, \modM_2}$,
 in which player I plays $A_n=\{y_j: j<\beta \}$ at each step  $n$. At the end, we have bimodules  $\{\modK_n: n<\omega\}$
 such that:
 \begin{itemize}
\item $A_n\subseteq \modM_1+\sum\limits_{\ell\leq n}\modK_\ell$,
\item $\modK_n$ is in $c\ell^\kappa_{\rm is}({\cal K})$,
\item
$\modM_1+\sum\limits_{\ell<\omega}\modK_{\ell}=\modM_1\oplus
\bigoplus\limits_{\ell<\omega} \modK_{\ell}$.
\end{itemize}
Set $\modK=\bigoplus\limits_{\ell<\omega} \modK_{\ell}.$ It then follows that $\modM_1\oplus
\modK \models \varphi((x_i)_{i<\alpha})$. In view of Lemma \ref{formulas and direct sum},
 $\modM_1\models \varphi((x_i)_{i<\alpha}).$

 Conversely, suppose that $\modM_1 \leq^{\rm pr}_{\mathcal{L}_{\infty,\kappa}(\tau_{(\ringR, \ringS)})}\modM_2$. Let us first state a few
 simple facts:
  \begin{enumerate}
\item[(i)]  Suppose $\alpha < \kappa$
 and $\bar{b} \in$$^{\alpha}\modM_2$. Then there exists
  $\bar{a} \in$$^{\alpha}\modM_1$ such that
  \[
  (\modM_1, \bar{a}) \equiv_{\mathcal{L}_{\infty,\kappa}(\tau_{(\ringR, \ringS)})} (\modM_2, \bar{b}).
  \]
 To see this, let  $\alpha < \kappa$ and
 $\bar{b} \in$$^{\alpha}\modM_2$. Suppose by contradiction that  there is no $\bar{a} \in {}^{\alpha}\modM_1$ such that $(\modM_1, \bar{a}) \equiv_{\mathcal{L}_{\infty,\kappa}(\tau_{(\ringR, \ringS)})} (\modM_2, \bar{b}).$  Thus for each $\bar{a} \in {}^{\alpha}\modM_1$ we can find
 $\phi_{\bar{a}}(\bar{\nu}) \in \mathcal{L}_{\infty,\kappa}(\tau_{(\ringR, \ringS)})$ such that
 \[
 \modM_1 \models \neg \phi_{\bar{a}}(\bar{a}) ~~ \& ~~ \modM_2 \models  \phi_{\bar{a}}(\bar{b}).
 \]
Let $\phi(\bar{\nu})= \bigwedge\{ \phi_{\bar{a}}(\bar{\nu}):   \bar{a} \in {}^{\alpha}\modM_1      \} \in \mathcal{L}_{\infty,\kappa}(\tau_{(\ringR, \ringS)}).$ Then $\modM_1 \models~\forall \bar{\nu} \neg \phi(\bar{\nu})$,
while $\modM_2 \models \phi(\bar{b})$, a contradiction to our assumption.

\item[(ii)] Suppose  $\bar{a}_1 \in {}^{\lg
 ( \bar{a}_1 )}\modM_1, \bar{a}_2 \in {}^{\lg
 ( \bar{a}_2) }\modM_1$ with $\bar{a}_1 \unlhd \bar{a}_2$ and $\bar{b}_1 \in {}^{\lg
 ( \bar{a}_1 ) }\modM_2$ are such that $ (\modM_1, \bar{a}_1) \equiv_{\mathcal{L}_{\infty,\kappa}(\tau_{(\ringR, \ringS)})} (\modM_2, \bar{b}_1).$
 There there is  some $\bar{b}_2 \in {}^{\lg
 ( \bar{a}_2 ) }\modM_2$ such that $\bar{b}_1 \unlhd \bar{b}_2$ and $ (\modM_1, \bar{a}_2) \equiv_{\mathcal{L}_{\infty,\kappa}(\tau_{(\ringR, \ringS)})} (\modM_2, \bar{b}_2).$

 \item[(iii)] Suppose $\alpha < \kappa, \bar{a} \in {}^{\alpha}\modM_1$, $\bar{b} \in {}^{\alpha}\modM_2$ and
 $(\modM_1, \bar{a}) \equiv_{\mathcal{L}_{\infty,\kappa}(\tau_{(\ringR, \ringS)})} (\modM_2, \bar{b})$. Let $\bar{c}= \langle b_i - a_i: i<\alpha     \rangle$ and let $\modK$ be the submodule of $\modM_2$ generated by $\bar{c}$. Then $\bar{b} \in \modM_1 + \modK$ and
 $\modM_1+ \modK = \modM_1 \oplus \modK.$

 \item[(iv)] Suppose $\langle K_\ell: \ell< \omega \rangle$ is an increasing sequence of submodules of $\modM_2$ such that for
 each $\ell<\omega, \modM_1 + \modK_\ell = \modM_1 \oplus \modK_\ell$. Then $\modM_1+\sum\limits_{\ell<\omega}\modK_{\ell}=\modM_1\oplus
\bigoplus\limits_{\ell<\omega} \modK_{\ell}$.
 \end{enumerate}
 We are now ready to define a winning strategy for the game $\Game_{{\cal K},\kappa}^{\modM_1, \modM_2}$. Fix a well-ordering $<^*$ of $\modM_2$.
  To start set $A_{-1}=\emptyset$ and $\bar{b}_{-1}=\langle \rangle.$ At stage $n$, suppose
 player I has chosen $A_n \subseteq \modM_2$ of size $< \kappa.$ We may assume that $A_n \supseteq A_{n-1}.$  Let $\bar{a}_n= \langle a_n(i): i< \lg(\bar{a}_n)      \rangle$ enumerate $A_n$ in the $<^*$-order in such a way that $A_{n-1}$ is an initial segment of the enumeration.
 By (i) and (ii), choose $\bar{b}_n= \langle b_n(i): i< \lg(\bar{a}_n)      \rangle \in \modM_2$ such that  $\bar{b}_n \unrhd \bar{b}_{n-1}$ and   $ (\modM_1, \bar{a}_n) \equiv_{\mathcal{L}_{\infty,\kappa}(\tau_{(\ringR, \ringS)})} (\modM_2, \bar{b}_n).$ Let $\modK_n$ be the submodule of $\modM_2$ generated by $\{b_n(i)-a_n(i): i< \lg(\bar{a}_n)   \}.$ Now (iii) and (iv) guarantee that items (a)
 and (c) of Definition \ref{ads definition} are satisfied. This completes the proof.
 \end{proof}

\begin{Claim}
\label{pr properties}
Let $\varphi $ be  in
$\mathcal{L}^{\rm cpe}_{\infty,\infty}(\tau_{(\ringR, \ringS)})$. The following holds:
\begin{enumerate}
\item Assuming $\varphi=\varphi(\bar x)$
we can find a simple formula $\varphi_\ast
(\bar z)\in
\mathcal{L}_{\infty,\infty}(\tau_{(\ringR, \ringS)})$ satisfying:
\begin{enumerate}
\item if
$\modM_1\leq^{\rm \pr}_{\varphi_\ast }
\modM_2$ and $\bar x\in \modM_2$ then
$\langle x_i+\modM_1: i< \lg(\bar x)\rangle \in\varphi(\modM_2/\modM_1)$
iff $\bar x\in\varphi(\modM_2)+\modM_1$.
\item if $\varphi(\bar x)\in \mathcal{L}^{\rm cpe}_{\mu,\mu}(\tau_{(\ringR, \ringS)})$
 and $ \mu $ is regular,
then
$\varphi_\ast
(\bar z)\in \mathcal{L}_{\mu,\mu}(\tau_{(\ringR, \ringS)})$.
\end{enumerate}

\item Assume  $\modM_2\in c\ell_{\rm ds}({\cal K})$ and $\modM_1$ is  ${\cal K}$-nice. If $\modM_1\oplus \modM_2=\modM_3$, then
$\modM_1\leq^{\rm pr}_\varphi \modM_3$.
\item If $\langle \modM_i:i<\delta\rangle$ is increasing,
$\modM_i\leq^{\rm pr}_\varphi \modM$ for all $i<\delta$  and $\varphi\in
\mathcal{L}_{\infty,
\cf(\delta)}(\tau_{(\ringR, \ringS)})$, then $\bigcup\limits_{i<\delta} \modM_i
\leq^{\rm pr}_{\varphi} \modM$ and $ \modM_j \leq^{\rm pr}_\varphi
\bigcup\limits_{i<\delta} \modM_i$  for all $j<\delta$.
\item Assume $\varphi$ is a simple formula of
$\mathcal{L}_{\kappa,\kappa}(\tau_{(\ringR, \ringS)})$.
Then
$\modM_1 <^{\rm pr}_\varphi \modM_2$  iff  for every $<\kappa$-generated submodule
$\modN\subseteq \modM_2$ we have
$\modM_1\leq^{\rm pr}_{\varphi }
\modM_1+\modN$.

\item If $\modM=\bigoplus\limits_{t\in I}\modM_t$ and
$\varphi (x) \in \mathcal{L}_{\infty,\infty}^p(\tau_{(\ringR, \ringS)})$,  then
$\varphi(\modM)= \bigoplus\limits_{t\in I} \varphi(\modM_t)$.
\end{enumerate}
\end{Claim}
\begin{proof}
(1). For simplicity, suppose that $\varphi=\varphi(x)$. By Lemma \ref{reducing to simple formula}, we can assume the formula
$\varphi$ is simple, so
let it be of the form
$$\varphi=\exists_{i<\alpha}y_i \bigwedge_{j<\beta} \varphi_j
(\bar{y}, x)$$
 where each
$\varphi_j (\bar{y}, x)$ is an atomic formula.
But the only atomic relation is equality, so by moving to one side,
\wolog\ $\varphi_j(\bar{y}, x)$ is of the form
$$\sigma_j (\bar{y}, x)=0$$
 for some term $\sigma_i$.
Let $\bar{z}=\langle z_j:j<\beta\rangle$ and
$$\varphi_\ast
(\bar{z})=(\exists x,\bar{y})
[\bigwedge_{j<\beta}\sigma_j (\bar{y}, x)=z_j].$$
We show that $\varphi_\ast
(\bar{z})$ is as required.
Clause (b) clearly holds. In order to prove clause  (a), take $\modM_1\leq^{\rm \pr}_{\varphi_\ast }
\modM_2$ and $x\in \modM_2$.
First,
assume that $x+\modM_1\in\varphi(\modM_2/\modM_1)$. Then
\[
\modM_2/\modM_1 \models \exists\bar y \bigwedge_{j<\beta} \sigma_j (\bar y, x+\modM_1)=0.
\]
We can find $\bar b=(b_i)_{i<\alpha} \in \modM_2$ such that
 $\modM_2/\modM_1 \models~\sigma_j(\bar b+\modM_1, x+\modM_1)=0 $ for all $j<\beta$.
 \footnote{By $\bar b+\modM_1$ we mean $\langle b_i+\modM_1: i<\alpha  \rangle$.}
By definition, $\sigma_j(\bar b+\modM_1, x+\modM_1)=c_j+\modM_1 $  for some $c_j \in M_1$.
We set  $\bar c:= \langle c_j: j<\beta  \rangle$. Then
$\modM_2 \models \varphi_*(\bar c)$. Hence,
\[
M_2 \models \sigma_j(\bar b, x)=c_j
\]  for all $j<\beta$.  We apply our   assumption to see
$\modM_1 \models \varphi_*(\bar c)$, e.g., there are $\bar b'=(b'_i)_{i<\alpha}$ and $x'$ in $\modM_1$
such that
\[
M_1 \models \sigma_j(\bar b', x')=c_j
\] for all $j<\beta$.
Then $\modM_2$ satisfies the same formulas.  In particular,
\[
M_2 \models \sigma_j((b_i-b'_i)_{i<\alpha}, x-x')=0,
\]
which implies $M_2 \models \varphi(x-x').$ Thus $x \in \varphi(\modM_2)+\modM_1$.

Conversely, suppose that $x\in\varphi(\modM_2)+\modM_1.$
Let $y \in \modM_1$ be such that $x -y\in \varphi(\modM_2)$. Then
$\modM_2 \models \varphi(x -y)$, and consequently there is $\bar b=((b_i)_{i<\alpha}) \in \modM_2$ such that
\[
M_2 \models \sigma_j(\bar b, x-y)=0
\] for all $j<\beta$.
Set $\bar z :=\bar 0.$ Then $\modM_2 \models \varphi_*(\bar z)$, and hence by our assumption $\modM_1 \models \varphi_*(\bar z)$.
It follows that for some $\bar b'=(b'_i)_{i<\alpha}$ and $x'$ in $\modM_1$ and
\[
M_1 \models \sigma_j(\bar b', x')=0
\] for all $j<\beta$.
Thus, $\modM_2$ satisfies the same formulas. Therefore,
\[
M_2 \models \sigma_j((b_i - b'_i)_{i<\alpha}, x-y-x')=0
\] for all $j<\beta.$
Since $\bar b', y+x'$ are in $\modM_1$, we have
\[
\modM_2/\modM_1 \models \sigma_j(\bar b+\modM_1, x+\modM_1)=0
\]  for all $j<\beta$.
This implies $\modM_2 / \modM_1 \models \varphi(x+\modM_1)$. Hence, $x+\modM_1 \in \varphi(\modM_2 / \modM_1 )$.

(2). This is in  Lemma \ref{formulas and direct sum}.

(3). Let us first  show that $\bigcup_{i<\delta} \modM_i
\leq^{\rm pr}_{\varphi} \modM$. To see this, let $\psi(\bar x)$ be a subformula of $\varphi$
and $\bar z \in \bigcup_{i<\delta} \modM_i.$ There is  $i_*<\delta$ such that  $ \bar z \in \modM_{i_*}$. Then for all $i_* \leq i < \delta$,
$\modM_i \models \psi(\bar z)$ iff $\modM \models \psi(\bar z)$.
It follows that
$$\bigcup_{i<\delta }\modM_i \models \psi(\bar z) \iff \modM \models \psi(\bar z).$$

Let $j<\delta$, $\psi(\bar x)$ a  subformula  of $\varphi$ and   $\bar z \in \modM_j$. We conclude from the above argument
that
\[ \modM_j \models \psi(\bar z) \iff \modM \models \psi(\bar z)  \iff
\bigcup_{i<\delta} \modM_i \models \psi(\bar z).
\]
Thus $\modM_j \leq^{\rm pr}_\varphi
\bigcup_{i<\delta} \modM_i$.

(4) It is easily seen that if $\modM_1 \leq^{\rm pr}_{\varphi} \modM_2$ then $\modM_1\leq^{\rm pr}_{\varphi }
\modM_1+\modN$  for each $<\kappa$-generated submodule
$\modN$ of
$\modM_2$.

Conversely, suppose that the above condition holds.
 Let $\psi(\bar x)$ be a subformula of $\varphi$. Then it is a simple formula as well, thus it is of the form
\[
\exists_{i<\alpha}v_i \bigwedge_{j<\beta} [\sigma_j(\bar x, \bar v)=0],
\]
for some terms $\sigma_j.$ Let $\bar x \in \modM_1$. It is obvious that
if $\modM_1 \models \psi(\bar x)$, then  $\modM_2 \models \psi(\bar x)$.
Now let  $\modM_2 \models \psi(\bar x)$. Thus we can find $\bar b \in \modM_2$ such
that
$$\modM_2 \models \bigwedge_{j<\beta} [\sigma_j(\bar x, \bar b)=0].$$
Let $\modN$ be the submodule generated by $\bar b$. Then
$\modM_1+\modN \models \bigwedge_{j<\beta} [\sigma_j(\bar x, \bar b)=0]$. We combine this
along with
 our assumption to conclude that $\modM_1 \models \psi(\bar x)$.

(5). We proceed by  induction on the complexity of the formula $\varphi$ (see also Lemma \ref{formulas and direct sum}), and we leave the routine check to the reader.
\end{proof}



Let $\kappa$ be an infinite cardinal and let $\frakss=({\cal K}, \modM_\ast ,
{\callE}, \ringR, \ringS,\ringT)$ be a context.
Set  $ {\rm \Mod}_{\frakss, \kappa} := \{ \modM:
\modM_{\ast }
\leq^{\rm ads}_{{\cal K},\kappa} \modM \} $.

\begin{lemma}
\label{modm is AEC}
Let $\frakss$ be a context as above and let $\kappa$ be an infinite cardinal.
\begin{enumerate}
\item If $\modM_1 \leq_\kappa \modM_2,$ then $\modM_1$ is a sub-module of $\modM_2.$

\item $\leq_\kappa$ is a partial order on ${\rm \Mod}_{\frakss, \kappa}.$

\item If $\kappa\leq\theta$ and $\modM\leq_\theta
\modN$ then
$\modM \leq_\kappa \modN$.

\item If $\langle \modM_i: i<\delta \rangle$ is  a $\leq_\kappa$-increasing and continuous in ${\rm \Mod}_{\frakss, \kappa}$, then
\begin{enumerate}
  \item $\bigcup\limits_{i<\delta}\modM_i \in  {\rm \Mod}_{\frakss, \kappa},$
  \item For each $j<\delta, \modM_j \leq_\kappa \bigcup\limits_{i<\delta}\modM_i$.
\end{enumerate}
\item If $\modM_1, \modM_2, \modM_3 \in {\rm \Mod}_{\frakss, \kappa}$ are such that
$\modM_1, \modM_2 \leq_\kappa \modM_3$ and $\modM_1 \subseteq \modM_2$, then $\modM_1 \leq_\kappa \modM_2$.

\item If $\modM_1\leq_\kappa \modM_2$ and $A\subseteq \modM_2$
has cardinality
$<\kappa$ then for some $\modN\in c\ell^\kappa_{is}({\cal
K})$, we
have $A\subseteq \modM_1+\modN=\modM_1\oplus \modN\leq_\kappa
\modM_2$.
\item The pair
 $ ( {\rm Mod}_ {\frakss, \kappa} , \leq_\kappa ) $
 has the amalgamation property.

\item There is a cardinal $\chi$ such that if $\modM_1 \subseteq \modN$ are in ${\rm \Mod}_{\frakss, \kappa}$, then there is $\modM_2 \in {\rm \Mod}_{\frakss, \kappa}$ such that $\modM_1 \subseteq \modM_2 \leq_\kappa \modN$ and $|\modM_2| \leq |\modM_1|+\chi$.

\end{enumerate}
\end{lemma}
\begin{proof}
 The lemma follows easily from Lemmas \ref{comparing two orders}  and \ref{pr properties}.
\end{proof}

\begin{definition}
\label{free bimodule}
We say the bimodule
 $\modN$ is free as an
$\ringS$--module
if
\begin{enumerate}
  \item [(a)] As an $ \ringR $-module, it can be written as
 $\mathop{{\bigoplus}}\limits_{i<\alpha}
\modN_i$ where each $ \modN_i $ is a sub-$ \ringR $-module
of $\modN$.
  \item [(b)] If  $\modM$ is a bimodule, $i<\alpha$ and
$g:\modN_i\longrightarrow \modM$ is an $\ringR$-homomorphism,
 then there is a unique
bimodule homomorphism $h:\modN\longrightarrow \modM$
extending  $g$.
\end{enumerate}
\end{definition}

The next lemma can be proved as in Lemma \ref{formula vs hom}.
\begin{lemma}
\label{aleph0 formula vs hom}
Let ${\mathfrak e}\in {\callE}$.
\begin{enumerate}
\item   There exists  a formula
$\varphi^{\mathfrak e}_n \in \mathcal{L}_{\infty, \infty}^{cpe}(\tau_{\ringR})$ such that
for an $\ringR$-module
$\modM$,
$\modM\models \varphi^{\mathfrak
e}_n(x)$ iff
 for some $\modM'$ with
$\modM\leq_{\aleph_0}\modM'$,
$\modM'\models \varphi^{{\mathfrak e},\infty}_n(x)$;
equivalently there is an $ \ringR $-module
homomorphism from $ \modN^{\mathfrak e} _n $
into $ \modM' $  mapping
$ x^ {\mathfrak e}   _ n $ to $ x $.

\item There exists $\psi^{\mathfrak e}_n \in \mathcal{L}_{\infty, \infty}(\tau_{\ringR})$ such that
for a bimodule
$\modM$,
$\modM\models \psi^{\mathfrak
e}_n(x)$ iff
 for some $\modM'$ with
$\modM\leq_{\aleph_0}\modM'$,
$\modM'\models \psi^{{\mathfrak e},\infty}_n(x)$.
\end{enumerate}
\end{lemma}
The above lemma suggests the following: \begin{definition}
By $\varphi^{\mathfrak e}_\omega(x)$ we mean $
\bigwedge\limits_{n<\omega}\varphi^{\mathfrak
e}_n(x)$ and  $\overline{\varphi}^{\mathfrak e}:=
\langle \varphi^{\mathfrak
e}_n: n<\omega\rangle$. Similarly, we define
$\psi^{\mathfrak e}_\omega(x):=
\bigwedge\limits_{n<\omega}\psi^{\mathfrak
e}_n(x)$ and let $\overline{\psi}^{\mathfrak e}:=
\langle \psi^{\mathfrak
e}_n: n<\omega\rangle$.\end{definition}
\begin{Definition}
\label{closed and m-nontrivial}
\begin{enumerate}
\item A bimodule $\modM$ is called ${\callE}$-closed if for
every ${\mathfrak e}\in {\callE}$,  $n<\omega$ and every
$x\in \modM$ we have
$$\modM\models\psi^{\mathfrak
e}_n(x) \iff  \modM\models \psi^{{\mathfrak e},\infty}_n(x).$$
\item We say that $\mathfrak e$ is non-trivial for ${\frakss}$
if for every $n<\omega$ there exists  $\modM\in {\cal K}^\frakss$ such that
$\modM_*^\frakss \oplus \modM\models(\exists x)
[\varphi^{\mathfrak e}_n(x)
\wedge\neg\varphi^{\mathfrak e}_\omega(x)]$.
\end{enumerate}
\end{Definition}
\begin{fact}It is easily seen that if ${\mathfrak e} \in {\mathfrak E}^{\frakss}$
 is adequate and each
$\modN^{\mathfrak e}_n$ is
finitely generated (as
 an $  \ringR $-module), then ${\mathfrak e}$ is
 non-trivial for ${\frakss}$.\end{fact}

\begin{convention}
	From now on
we fix a
context $\frakss=({\mathcal K},\modM_*,
\calE, \ringR,\ringS,\ringT)$
which
is non-trivial which means:
\begin{enumerate}
\item[(a)] ${\mathcal K}\neq\emptyset$,
\item[(b)] $\calE\neq\emptyset$,
\item[(c)] every ${\mathfrak e}\in \calE_{\infty,\kappa}$ is non-trivial for $\frakss$.
\end{enumerate}
\end{convention}

\begin{Definition}
\label{subgroups Ln}
Let
${\mathfrak e}\in {\callE}$ and $\bar\varphi$ is adequate with respect to
$\mathfrak e$.
\begin{enumerate}
\item Suppose
 $\modM$ is a bimodule and ${\bf h_1}, {\bf h_2}:\modN^{\mathfrak
 	e}_n\to\modM$ are bimodule homomorphisms.  We define
\[\modL^{{\mathfrak
e},\bar{\varphi},{\bf h_1},{\bf h_2}}_n(\modM):=\{z\in\varphi_n(\modN^{\mathfrak e}_n): {\bf h_1}(z)={\bf h_2}(z)\ \mod \
\varphi_\omega(\modM)\}.\] For simplicity, and if there is no danger of confusion, we set $$\modL^{{\mathfrak
		e},\bar{\varphi},{\bf h_1},{\bf h_2}}_n:=\modL^{{\mathfrak
		e},\bar{\varphi},{\bf h_1},{\bf h_2}}_n(\modM).$$
If $\bar{\varphi}=\bar{\varphi}^{{\mathfrak e},\kappa}$ we may write
$\modL^{{\mathfrak e},
\kappa, {\bf h_1},{\bf h_2}}_n$ and if $\kappa=\aleph_0$ we may omit it.

\item Let
$\modL^{\mathfrak e}_n [\frakss]:=
\bigcap\big\{\modL^{{\mathfrak e},{\bf h_1}, {\bf h_2}}_n:\text{~for
some~}
\modM\in {\cal K}\cup \{\modM_\ast\}, {\bf h_1}, {\bf h_2}:\modN^{\mathfrak
	e}_n\to\modM  \text{~are  bimodule }$
 $\text{homomorphisms and~}
{\bf h_1}(x_n^{\mathfrak e})={\bf h_2}(x_n^{\mathfrak e})\big\}.$
We may write $\modL^{\mathfrak e}_n:=\modL^{\mathfrak e}_n [\frakss]$, when $\frakss$
is clear from the context.

\item $\modL^{\mathfrak e}_n [\cal K]$, is defined similarly but
we only require $\modM\in {\cal K}$.
\end{enumerate}
\end{Definition}

\begin{lemma}
	\label{Pe bimodules defined}
Let ${\mathfrak e}\in {\callE}$. Then, there are bimodules $\modP^{\mathfrak e}$,
	$\modP^{\mathfrak e}_n$ and $\modK_n$, embeddings ${\bf h}^{\mathfrak e}_n:\modN^{\mathfrak
		e}_n\longrightarrow
	\modP^{\mathfrak e}$, an element $x=x_{\mathfrak e}$, an
	embedding ${\bf f}^ {\mathfrak e}:\modN^{\mathfrak e}_0\longrightarrow
	\modP^{\mathfrak e}$
	and
	$x_n:=x_{{\mathfrak e},n} \in \modP^{\mathfrak e}$   such that:
	\begin{enumerate}
		\item[(a)] $\Rang({\bf h}^{\mathfrak e}_n)\cap\sum_{m\neq n}\Rang({\bf h}^{\mathfrak e}_m)
		=\{0\}$,
		\item[(b)] For each $\modK_n$ we have:
		$$\modP^{\mathfrak e}=\left(\sum_{\ell<n}\Rang({\bf h}^{\mathfrak
			e}_\ell)\right)+\modK_n=
		{\bigoplus}_{\ell<n}\Rang({\bf h}^{\mathfrak e}_\ell)\oplus
		\modK_n,$$
		and $\modK_n$ is a direct sum of copies of $\modN^{\mathfrak e}_m$'s.
		\item[(c)] $\sum_{n<\omega}\Rang({\bf h}^{\mathfrak e}_n)$ is not a direct
		summand of $\modP^{\mathfrak e}$; moreover, for some
		embeddings ${\bf f}^{\mathfrak e}_n:\modN^{\mathfrak e}_n\longrightarrow
		\modP^{\mathfrak e}$
		satisfying:
		\begin{enumerate}
			\item[$(\ast)_1$] $x_n=
			\sum_{\ell<n} {\bf h}^{\mathfrak e}_\ell\big(x^{\mathfrak e}_\ell\big) \in\sum_{\ell <n}\Rang(h^{\mathfrak e}_\ell)$,
			\item[$(\ast)_2$] $x 
			-x_n\in\varphi^{\mathfrak e}_n(\modP^{\mathfrak e})$,
			\item[$(\ast)_3$]  ${\bf f}^{\mathfrak e}={\bf f}^{\mathfrak e}_0$,
			
			\item[$(\ast)_4$] ${\bf f}^{\mathfrak e}\big(x^{\mathfrak e}_{0}\big)=x$,
			\item[$(\ast)_5$] $x\notin\sum_{n<\omega}\Rang({\bf h}^{\mathfrak e}_n)+
			\varphi^{\mathfrak e}_\omega (\modP^{\mathfrak e})$.
		\end{enumerate}
		\item[(d)] $\modP^{\mathfrak e}$ is the direct sum of the copies
		$\Rang({\bf f}^{\mathfrak e}_n)$ of $\modN_n^{\mathfrak e}$.
	\end{enumerate}
\end{lemma}

\begin{proof} Set $\modP^{\mathfrak e}:=\mathop{{\bigoplus}}\limits_{n< \omega}
\modN^{\mathfrak e}_n$ and denote the natural embedding from $ \modN^{\mathfrak e}_n$ into 	$\modP^{\mathfrak e}$
by 	${\bf f}^{\mathfrak e}_n$.
In particular, 	
	 $$\modP^{\mathfrak e}=\mathop{{\bigoplus}}\limits_{n< \omega}
	\Rang({\bf f}^{\mathfrak e}_n),$$
i.e.,  (d) holds. We define ${\bf h}^{\mathfrak
		e}_n:\modN^{\mathfrak e}_n\longrightarrow \modM$,
	for all $n<\omega$, as follows,  where $g^{\mathfrak e}_{n,n+1}$ is defined as in
Definition \ref{classcale}(3):
	\[
	{\bf h}^{\mathfrak e}_n(y)= {\bf f}^{\mathfrak e}_n(y)-{\bf f}^{\mathfrak e}_{n+
		1}(g^{\mathfrak e}_{n,n+1}(y)) \mbox{ for every } y\in \modN^{\mathfrak e}_n.\]
 As $\modP^{\mathfrak e}=\Rang({\bf f}^{\mathfrak
		e}_n)\oplus\mathop{{\bigoplus}}\limits_{\ell\neq n}\Rang({\bf f}^{\mathfrak e}_\ell)$,
	we observe that ${\bf h}^{\mathfrak e}_n$ is a bimodule embedding. Set also
\[
\modP^{\mathfrak e}_n: =\sum\limits_{\ell<n}\Rang({\bf h}^{\mathfrak e}_\ell).
\]
The next claim can be proved easily.
\begin{claim}\label{this}Adopt the above notation. Then,
for  each $n$,
		$$\Rang({\bf f}^{\mathfrak e}_n)\oplus\Rang({\bf f}^{\mathfrak e}_{n+1})=
		\Rang({\bf h}^{\mathfrak e}_n)\oplus
		\Rang({\bf f}^{\mathfrak e}_{n+1}).$$
In particular, $\modP^{\mathfrak e}=\mathop{{\bigoplus}}\limits_{\ell<n}\Rang({\bf h}^{\mathfrak e}_n)\oplus
		\mathop{{\bigoplus}}\limits_{\ell\geq n}\Rang({\bf f}^{\mathfrak e}_\ell)$ for all $n$.
\end{claim}
	
We set $\modK_n:=\mathop{{\bigoplus}}\limits_{\ell\geq n}\Rang({\bf f}^{\mathfrak e}_\ell)$.	
In view of Claim \ref{this}, the items $(a)$, $(b)$ and $(d)$ of  Lemma \ref{Pe bimodules defined} are hold.
	Next we shall show that $x:=
	{\bf f}^{\mathfrak e}_0((x^{\mathfrak e}_0))$ and
	$x_n:=\sum\limits_{\ell<n}  {\bf h}^{\mathfrak
		e}_\ell(x^{\mathfrak e}_\ell)$ are as required in Lemma \ref{Pe bimodules defined}(c).
	This implies the first
	statement of $(*)_1$ in item
	(c).
	Trivially $(*)_3$ and $(\ast)_4$ are true.
	Now,
	
	\[\begin{array}{ll}
	x  &= {\bf f}^{\mathfrak e}_0(x^{\mathfrak e}_0)\\
	&={\bf h}^{\mathfrak e}_0(x^{\mathfrak e}_0) +
	{\bf f}^{\mathfrak e}_1(g^{\mathfrak e}_{0,1}(x^{\mathfrak e}_0)) \\
	&={\bf h}^{\mathfrak e}_0(x^{\mathfrak e}_0)+
	{\bf f}^{\mathfrak e}_1(x^{\mathfrak e}_1)\\
	&={\bf h}^{\mathfrak e}_0(x^{\mathfrak e}_0)
	+{\bf h}^{\mathfrak e}_1 (x^{\mathfrak e}_1)
	+
	 {\bf f}^{\mathfrak e}_2(x^{\mathfrak e}_2).
	\end{array}\]
	
	By induction on
	$n$ we have
	$$x=\sum\limits_{\ell<n}{\bf h}^{\mathfrak
		e}_\ell(x^{\mathfrak e}_\ell) +{\bf f}^{\mathfrak e}_n(x^{\mathfrak e}_n)
	=x_n+{\bf f}^{\mathfrak e}_n(x^{\mathfrak e}_n).$$
	Clearly, $x_n\in\mathop{{\bigoplus}}\limits_{\ell<n + 1}
	\Rang ({\bf h}^{\mathfrak e}_\ell) \subseteq
	\mathop{{\bigoplus}}\limits_{\ell<\omega}
	\Rang({\bf h}^{\mathfrak e}_\ell)$ and by the choice of ${ \bf f}^{\mathfrak e}_n$
	the second term is in $\varphi^{\mathfrak
		e}_n(\modP^{\mathfrak
		e})$,
	so $(*)_1+(*)_2$ of clause (c)
	holds.
	We shall now show that
	$x\notin\sum\limits_{\ell<
		\omega }\Rang({\bf h}^{\mathfrak e}_\ell)+
	\varphi^{\mathfrak e}_\omega(\modP^{\mathfrak e})$,
	i.e.,~$(*)_5$ holds.

	Suppose by contradiction that there is an
	$ n \geq 1 $
	such that $x\in y+
	\sum\limits_{\ell<n}\Rang({\bf h}^{\mathfrak e}_\ell)$,
	where $y\in\varphi^{\mathfrak
		e}_\omega (\modP^{\mathfrak e})$.

We now define  a bimodule endomorphism ${\bf f}_n\in\End(\modP^{\mathfrak e})$. As $\modP^{\mathfrak
	e}=\mathop{{\bigoplus}}
\limits_{ \ell <\omega }
\Rang({\bf f}^{\mathfrak e}_\ell)$, it
	 is clearly
	enough to define each ${\bf f}_n\restriction
	\Rang(f^{\mathfrak e}_\ell)$, for $n \geq 1,$ separately.
	Recall that ${\bf f}^{\mathfrak e}_\ell$ is
	one-to-one.  Let $z\in \modN^{\mathfrak e}_\ell$. We define
	$$
	 {\bf f}_n({\bf f}^{\mathfrak e}_\ell(z))=\left\{\begin{array}{ll}
	{\bf f}^{\mathfrak e}_{n}(g^{\mathfrak e}_{\ell,n}(z)) &\mbox{if } \ell\leq n,\\
	0 &\mbox{otherwise }
	\end{array}\right.
	$$Now, we bring the following claim:
	
\begin{claim}
			\begin{enumerate}
		\item[1)]  If   $\ell\not= n$, then	
		${\bf f}_n\restriction\Rang({\bf h}^{\mathfrak e}_\ell)$
		is identically zero.
		\item[2)]  If   $\ell= n$, then	
	${\bf f}_n(x)
=
	{\bf f}^{\mathfrak e}_n(x^{\mathfrak e}_n).$
		\end{enumerate}
	\end{claim}	

\begin{proof}
1): For $\ell>n$ this is trivial, so suppose that $\ell<n$. Let
	$z\in \modN^{\mathfrak e}_\ell$. Then ${\bf h}^{\mathfrak e}_\ell(z)= {\bf f}^{\mathfrak
		e}_\ell(z)-
	{\bf f}^{\mathfrak e}_ { {\ell} + 1 } (g^{\mathfrak e}_ { {\ell} , {\ell} + 1 }(z))
	$. So,

	\[\begin{array}{ll}
	{\bf f}_n({\bf h}^{\mathfrak e}_\ell(z)) &= {\bf f}_n \left({\bf f}^{\mathfrak
		e}_\ell(z)-
	{\bf f}^{\mathfrak e}_ { {\ell} + 1 } (g^{\mathfrak e}_ { {\ell} , {\ell} + 1 }(z))\right)\\
	&={\bf f}_n({\bf f}^{\mathfrak
		e}_\ell(z))-
	{\bf f}_n({\bf f}^{\mathfrak e}_ { {\ell} + 1 } (g^{\mathfrak e}_ { {\ell} , {\ell} + 1 }(z)))  \\ &= {\bf f}^{\mathfrak e}_{n}(g^{\mathfrak e}_{\ell,n}(z))-
	 {\bf f}^{\mathfrak e}_n(g^{\mathfrak e}_{\ell,\ell+1}(
	g^{\mathfrak e}_{\ell+1,n}(z))) \\
	&={\bf f}^{\mathfrak e}_{n}(g^{\mathfrak e}_{\ell,n}(z))- {\bf f}^{\mathfrak e}_{n}(g^{\mathfrak e}_{\ell,n}(z))  \\
	&= 0.
	\end{array}\]
The third equation holds because of clause (e) of Lemma \ref{Pe bimodules defined}.

2):	It is enough to recall ${\bf f}_n(x)
=
{\bf f}_n({\bf f}^{\mathfrak e}_0(x^{\mathfrak e}_0))
=
{\bf f}^{\mathfrak e}_n(g^{\mathfrak e}_{0,n } ({\bf f}^{\mathfrak e}_0(x^ {\mathfrak e}_0))
=
{\bf f}^{\mathfrak e}_n(x^{\mathfrak e}_n).$
\end{proof}

Since	$x^{\mathfrak e}_n \notin
	\varphi_\omega (\modN^{\mathfrak e}_n)$, we have ${\bf f}_n(x) =
	{\bf f}^{\mathfrak e}_n(x^{\mathfrak e}_n )
	\neq 0$. Indeed, we have
 $$x^{\mathfrak e}_n\notin
	\varphi^{\mathfrak e}_\omega(\Rang({\bf f}^{\mathfrak e}_n))=
	\varphi^{\mathfrak e}_\omega(\modP^{\mathfrak e})\cap\Rang({\bf f}^{\mathfrak e}_n).$$
	But $x-y\in \sum\limits_{\ell<n}\Rang({\bf h}^{\mathfrak e}_\ell)$ and
	${\bf f}_n\rest \Rang (h^{\mathfrak e}_n)$ is zero, so ${\bf f}_n(x-y)=0$. Hence
	${\bf f}_n(x)={\bf f}_n(y)$. However ${\bf f}_n(x)\notin \varphi^{\mathfrak e}_\omega
	(\modP^{\mathfrak e})$. Thus,
	${\bf f}_n(y)\notin\varphi^{\mathfrak e}_\omega (\modP^{\mathfrak e})$.
	Recall that $y\in \varphi^{\mathfrak e}_\omega
	(\modP^{\mathfrak e})$. Therefore, ${\bf f}_n(y)\in \varphi^{\mathfrak e}_\omega
	(\modP^{\mathfrak e})$. This   contradiction completes the proof of Lemma \ref{Pe bimodules defined}.
\end{proof}

\begin{lemma}
\label{1.12}
\begin{enumerate}
\item Let $x\in \modM_1$,
$ {\mathfrak e}   \in {\mathfrak E}$ and
$\varphi\in\{\tilde\varphi^{\mathfrak e}_\alpha,
\tilde\psi^{\mathfrak e}_\alpha:\alpha\leq
\omega\}$.   If  $\modM_1\leq_{\aleph_0}\modM_2$,     then
$$\modM_1\models\varphi(x) \iff
\modM_2\models\varphi(x).$$
\item If $\modM_\ast\leq_{\aleph_0}
\modM_1$ and ${\callE}'\subseteq
{\callE}$,
then there is an $\modM_2$ with the following properties:
\begin{enumerate}
\item $\modM_1\leq_{\aleph_0}\modM_2$,
\item $\|\modM_2\|\leq\|\modM_1\|+
\Vert \frakss \Vert
+|{\callE}'|$,
\item if $x\in \modM_1$, $n<\omega $ and ${\mathfrak e}\in {\callE}'$ then
 $$\modM_2 \models\varphi^{{\mathfrak e},\infty}_n(x)
\iff \modM_1\models\varphi^{\mathfrak e}_n[x],$$
\item $\modM_2$ is the free sum of $\{\modM_{2,i}:i<i^\ast\} \cup
\{\modM_1\}$ where each
$\modM_{2,i}$ isomorphic to some member of ${\cal K}$,
\item For each $\modN \in \cal K,$ there are
$||\modM_2||$-bimodules from  $\{\modM_{2,i}:i<i^\ast\}$
each of them being
isomorphic to $\modN$.
\end{enumerate}
\end{enumerate}
\end{lemma}
\begin{proof}
  Clause (1) can be proved easily. To prove (2), let $\kappa:=|\modM_1\|+
\Vert \frakss \Vert
+|{\callE}'|$ and set
\[
\modM_2:= \modM_1 \oplus \bigoplus_{\modN \in \cal K} \bigoplus_{i<\kappa} \modM_i^{\modN},
\]
where each $\modM_i^{\modN}$ is isomorphic to $\modN.$ It is easily seen that $\modM_2$
is as required.
\end{proof}
We now introduce the notion of (semi) nice construction. This concept plays an important role for the
proof of our main results.
\begin{Definition}
\label{w.s. nice definition}
Suppose $\frakss=(\mathcal K, \modM_*, \calE, \ringR, \ringS, \ringT)$ is a context,
$\lambda \geq\kappa=\cf(\kappa)  > \kappa(\calE)$. Also, $ \lambda=\lambda^{\kappa}$ and for all $\alpha<\lambda$,
$ |\alpha|^{\aleph_0}<\lambda$. Here,
$\bar{\gamma}^*=\langle \gamma^*_\alpha:
\alpha\leq \kappa\rangle$
is increasing continuous
and $S\subseteq S^\kappa_{\aleph_0}$\footnote{where $S^\kappa_{\aleph_0}=\{\alpha < \kappa: \cf(\alpha)=\aleph_0     \}$.} is stationary such that $S^\kappa_{\aleph_0}\setminus S$ is stationary as well. Let also $\langle \mathfrak e_\alpha: \alpha < |\calE|    \rangle$ be a fixed enumeration of $\calE$ and $\emph{rep}(\cal K)=\langle   \modN^\beta: \beta < |\mathcal K/\cong|       \rangle$
be an enumeration of $\cal K$ up to isomorphism.

A \emph{weakly semi-nice construction}
$\mathscr A$ for $(\lambda, {\frakss},S,\bar \gamma^*)$
is a sequence $\bar \modM=\langle
\modM_\alpha:\alpha\leq\kappa \rangle$ together with the
other objects mentioned below satisfying the following
conditions
\begin{enumerate}
\item[(A)] $\modM_{\alpha} $ is a bimodule whose universe is the ordinal
$\gamma_\alpha^*$ for $\alpha\leq\kappa$.

\item[(B)] $\alpha<\beta\ \Rightarrow\ \modM_\alpha\subseteq
\modM_\beta$.
\item[(C)] if $\delta\leq\kappa$ is a limit ordinal, then
$\modM_\delta=\bigcup\limits_{
\alpha<\delta} \modM_\alpha$.
\item[(D)] $\modM_0=\modM_\ast$.
\item[(E)] for every $\alpha\in\kappa
\setminus S$, $\modM_{\alpha+1}=\modM_\alpha\oplus
\mathop{\bigoplus}\limits_{t\in J_\alpha} \modN^\alpha_t$
where
\begin{itemize}
  \item [(a)] $|J_\alpha|\leq\lambda$,
  \item [(b)] $J_\alpha \cap (\bigcup\limits_{\beta< \alpha}J_\beta) =\emptyset,$
  \item [(c)] $\cf (\gamma^*_{\alpha+1})<\lambda
\Leftrightarrow |J_\alpha|<\lambda$,
\item [(d)]  for every $t \in J_\alpha$ there is $\beta < |\mathcal K/\cong|$
such that $\modN_t^\alpha \cong \modN^\beta$, i.e., every $\modN_t^\alpha$ is isomorphic to some member of $\cal K,$

\item [(e)] for every $\beta < |\mathcal K/\cong|$ there are $||\modM_\alpha||$-many $t \in J_\alpha$ such that  $\modN_t^\alpha \cong \modN^\beta$.
\end{itemize}
By the assumption (b), for each $t \in \bigcup\limits_{\beta \leq \alpha}J_\beta$ there is a unique
$\beta \leq \alpha$ such that $t \in J_\beta$; so we may replace $\modN^\beta_t$ by $\modN_t.$ Also given
$t \in J_\alpha$, there is a unique $\modN_t^{\rm so} \in \emph{rep}(\cal K)$ such that $\modN_t^{\rm so} \cong \modN_t^\alpha$. Also, we use the notations
$h_t$ and  $h_{\alpha, t}: \modN_t^{\rm so} \stackrel{\cong}\longrightarrow \modN_t^\alpha$ for the mentioned isomorphism.

\item[(F)]
for $\delta\in S$, either the demand in
clause\footnote{
we can ignore this
first possibility as we can
just shrink $S$ if the result is stationary.}
 (E) holds
or else there are $\gamma^{**}_\delta<\delta$ and
\[
\langle ({\mathfrak e}_s, {\mathbb P}_s, \bar{\alpha}_s,
\bar{t}_s, \bar{g}_s, h_s, q_s) :s\in J_\delta\rangle
= \langle {\mathfrak e}^\delta_s,{\mathbb P}^\delta_s,
\bar{\alpha}^\delta_s,\bar{t}^\delta_s,\bar{g}^\delta_s,
h^\delta_s,q^\delta_s):s\in J_\delta\rangle
\]
such that  $J_\delta \cap (\bigcup_{\alpha < \delta}J_\alpha)=\emptyset$ and
\begin{enumerate}
\item[(a)]
 ${\mathfrak e}_s\in \calE$ and in fact
${\mathfrak e}_s\in \{{\mathfrak e}_\beta:\beta<\gamma^{**}_\delta$
and $\beta<|\calE|\}$,
\item[(b)]
$\bar{\alpha}_s=\langle \alpha_{s,n}:n<\omega\rangle$,
\item[(c)]
$\bar{t}=\langle t_{s,n}:n<\omega\rangle$,
\item[(d)] $\langle \alpha_{s,n}:n<\omega\rangle$ is an increasing sequence of ordinals bigger
than  $\gamma^{**}_\delta$ such that $\alpha_{s, n} \notin S$,

\item[(e)]
$\delta=\sup\{\alpha_{s,n}:n<\omega\}$,

\item[(f)]
$t_{s,n}\in J_{\alpha_{s,n}}$,

\item[(g)]
$\modN^{so}_{t_{s,n}}= \modN^{{\mathfrak e}_s}_n$,

\item[(h)]
if $s_1\not= s_2$ are in  $J_\delta$ then the
sets
$$
\{t_{s_1,n}:n<\omega\}, \{t_{s_2,n}:n<\omega\}
$$
are tree-like, i.e.,
\begin{itemize}
\item $\{t_{s_1,n}:n<\omega\} \cap \{t_{s_2,n}:n<\omega\}$
is finite,

\item
if $t_{s_1,n_1}=t_{s_2,n_2}$ then $n_1=n_2$ and
$\bigwedge\limits_{n<n_1} t_{s_1,n}=t_{s_2,n}$,
\end{itemize}
\item[(i)]
$\bar{g}_s=\langle g_{s,n}:n<\omega\rangle$,
where each $g_{s,n}$ is a (bimodule) \homo\ from
$\modN ^{{\mathfrak e}_s}_n$
into $\modM_{\gamma^{**}_\delta}$,

\item[(j)]
for $s\in J_\delta$, $ h_s$ is a bimodule homomorphism
from $\modP^{{\mathfrak e}_{s}}$\footnote{where $\modP^{{\mathfrak e}_{s}}$   is as in Lemma \ref{Pe bimodules defined}.} onto $\modP^\delta_s$, where $\modP^\delta_s$ is a
sub-bimodule of $\modM_{\delta+1}$,

\item[(k)] the following diagram commutes for $n< \omega$, $s\in
J_\delta$ and ${\mathfrak e}={\mathfrak e}_{s}$
(where ${\mathbb P}^{\mathfrak e}, {\bf h}^{\mathfrak e}_n$ are from Lemma \ref{Pe bimodules defined}
and the $ h_{t_{s,n } } $ are from clause (E)):

\begin{picture}(300,150)
\put(40,60){$(h_{t_{s,n}}+g^\delta_{s,n})$}
\put(120,0){\begin{picture}(150,150)
\put(0,25){$\modN^{\mathfrak e}_{n}$}
\put(5,45){\vector(0,1){35}}
\put(-15,85){$\modM_\delta\cap \modP^\delta_s$}
\put(30,90){\vector(1,0){45}}
\put(50,100){\rm id}
\put(90,85){$\modP^\delta_s$}
\put(95,45){\vector(0,1){35}}
\put(90,25){$\modP^{\mathfrak e}$}
\put(45,10){${\bf h}^{\mathfrak e}_n$}
\put(110,60){$h^\delta_s$}
\put(30,28){\vector(1,0){45}}
\end{picture}}
\end{picture}

\item[(l)] for $s \in J_\delta$ we have
\begin{itemize}
\item[(a)] $\modP^\delta_s\subseteq \modM_{\delta+1}$,

\item[(b)] $\modM_{\delta+1}$ is generated by
$\modM_\delta\cup\bigcup\limits_{s\in J_\delta}
\modP^\delta_s$,

\item[(c)] $\modP^\delta_s$, $\langle \modM_\delta\cup
\bigcup\limits_{s'\neq s}
\modP^\delta_{s'}\rangle_{\modM_{\delta+1}}$
are amalgamated freely over
$$\sum\limits_{n<\omega}\Rang(h^\delta_s  \circ
{\bf h}^{{\mathfrak e}_s}_n),$$
\end{itemize}

\item[(m)] for $s\in J_\delta$ the type $ q_s$ with the free variable $y$ has the form
$\{\varphi^{\mathfrak e}_n (y-z_{s,n}):n<\omega\}$ where
${\mathfrak e}\in \callE$ and
$z_{s,n}\in \modM_\delta$,
\item[(n)] $q_s$ is omitted by
$\modM_\alpha$ for $\alpha\in [\delta,\kappa]$,
i.e. there is no $ y \in \modM_\alpha$
such that for all $n<\omega,~\modM_\alpha \models \varphi^{\mathfrak e}_n (y-z_{s,n}).$
\end{enumerate}
\end{enumerate}
\end{Definition}
In abuse of notation we sometimes use
$\langle \modM_\alpha:
\alpha\leq\kappa\rangle$ as being
$ {\mathscr A} $,
we of course may write
$\modM_\alpha^ {\mathscr A}
, J^  {\mathscr A}
_\alpha$
etc. and even $\kappa=\kappa^ {\mathscr A},
\gamma_\alpha=
\gamma_\alpha^{\mathscr A},
\lambda=\lambda^{\mathscr A}$.
We now define various refinements  of the above concept:
\begin{Definition}
\label{nice construction modified}
Using the notation of Definition \ref{w.s. nice definition},
\begin{enumerate}
\item
\begin{enumerate}
\item[(a)] By a weakly semi-nice construction for $(\lambda,
{\mathfrak m},S,\kappa)$ we mean a weakly semi-nice construction for
$(\lambda,{\mathfrak m},S,\bar \gamma^*)$
where
one of the following occurs:

$ \bullet $   $\kappa=\cf(\lambda),~
(\forall \alpha<\kappa)
(\gamma^*_  \alpha
<\lambda)$ and if $ \alpha < \beta < \kappa$ then $\gamma^*_{\alpha +1 } - \gamma^*_\alpha
\le \gamma^*_{\beta  +1 } - \gamma^*_\beta$,

$ \bullet $  $\kappa\not= \cf(\lambda)$ and
$(\forall \alpha<\lambda)$ we have
$\gamma^*_\alpha
+\lambda\leq \gamma^*_{\alpha+1}<(\lambda^{\aleph_0})^+$.

\item[(b)]
if we omit $\kappa$, i.e., write $(\lambda,{\mathfrak m},
S)$ we mean $\kappa=\cf(\lambda)$.
\end{enumerate}
\item
We omit ``semi'' from ``weakly semi-nice construction
${ \mathscr A} $'' if whenever $\delta \in S$ and $s_1$, $s_2\in
J_\delta$  then $\bar g^\delta_{s_1}=
\bar g^\delta_{s_2}$ and $
{\mathfrak e}_{s_1}={\mathfrak e}_{s_2}$.

\item
We omit ``weakly'' from ``weakly semi-nice construction
${ \mathscr A} $'' if
we add
\begin{enumerate}
\item[$(G)_1$] if ${\bf f}$ is an endomorphism of $\modM_\kappa$ as an
$\ringR$-module,
${\mathfrak e}\in \callE^\frakss, \gamma
<\kappa$ and
${\bf g_n}$ is a
(bimodule)
homomorphism from $\modN^{\mathfrak e}_n$
into $\modM_\gamma$ for each $n=1,2,\ldots$, then there are
$y\in \modM_\kappa$, $z\in \modP^{\mathfrak e}$ and
$\{z_{n, i, \ell}
:n, \ell<\omega, i<2\}$ with $z_{n,i, \ell} \in \dom ({\bf g_\ell})=\modN^{\mathfrak e}_\ell$
such that for each $n$ and all large enough $\ell, z_{n, i, \ell}=0$ and
for  all large enough $n$,
\begin{enumerate}
\item[(a)]
$z \in \sum\limits_{\ell < \omega}{\bf h}_\ell^{\mathfrak e}(z_{n, 1,\ell}-z_{n, 0, \ell})+\varphi_n^{\mathfrak e}(\modP^{\mathfrak e}).$

\item[(b)]
$\sum\limits_{\ell < n}{\bf f}({\bf g}_\ell(x_\ell^{\mathfrak e})) \in y +\sum\limits_{\ell<\omega} {\bf g}_{\ell}(z_{n,1, \ell}) +\varphi_n^{\mathfrak e}(\modM_{\kappa}).$

\end{enumerate}

\item[$(G)_2$] if ${\mathfrak e}\in \calE^\frakss$ and
for $n<\omega$ and $\alpha\in \kappa\smallsetminus S$,
${\bf h}_{\alpha,n}$
is an embedding of $\modN^{\mathfrak e}_n$ into
$\modM_\kappa$
\st\ $\modM_\alpha+\Rang ({\bf h}_{\alpha,n})
=\modM_\alpha \oplus \Rang
({\bf h}_{\alpha,n})\leq_{\aleph_0} \modM_\kappa$,
then there are an embedding ${\bf h}$ of
$\modP^{\mathfrak e}$ into $\modM_\kappa$ and ordinals $\alpha_n
\in \kappa\smallsetminus S$ for $n=1,2,\ldots$ \st\ $${\bf h}_{\alpha_n, n}={\bf h} \circ {\bf h}^{\mathfrak e}_n$$ for each such
$n$.
\end{enumerate}
\item $\bar\modM = \langle\modM_\alpha: \alpha \leq \kappa\rangle$ is strongly semi-nice construction for $(\lambda, \frakss, S, \kappa)$ if ${\bf f}$ is an $\ringR$-endomorphism of $\modM_\kappa$
  and $\mathfrak e \in {\mathfrak E}   ^{{\frakss}},$ then there are $n(\ast) < \omega, \alpha < \lambda$ and $z \in \modN_{n(\ast)}^{\mathfrak e}$ such that for every bimodule homomorphism $h: \modN_{n(\ast)}^{\mathfrak e} \to \modM_\kappa,$
  $${\bf f}(h(x_n^{\mathfrak e})) - h(z) \in \modM_\alpha+\varphi_\omega^{\mathfrak e}(\modM_\kappa).$$

\item We say
 $\modM$ is   weakly semi-nice (resp. semi-nice) for
$(\lambda,S,\frakss,\kappa)$ if  for some weakly semi-nice (resp. semi-nice) construction
$\bar \modM= \langle \modM_\alpha: \alpha \leq \kappa \rangle$ for $(\lambda,S,\frakss,\kappa)$
we have
$\modM=\modM_\kappa$. If we omit $S$ we mean for some $S \subseteq S^\kappa_{\aleph_0}$,
similarly for $\lambda$ and $\kappa.$
\end{enumerate}
\end{Definition}

\begin{lemma}
\label{properties of w.s.n.c}
Let $\langle \modM_\alpha:\alpha\leq\kappa\rangle$ be a semi-nice
construction for
$(\lambda,{\frakss},S,\bar \gamma^*)$
and $ \kappa ( {\mathfrak E}
 ^ {{\frakss}} ) = {\aleph_0}$,
(not used in part (1)).
\begin{enumerate}
\item if $\alpha\notin S$ and $\alpha
\leq \beta\leq\kappa$, then
$\modM_\alpha\leq_{\aleph_0}
\modM_\beta$, i.e.,
$\modM_\alpha$ is an almost direct ${\cal K}$-summand of
$\modM$ with respect to $\aleph_0$.

\item  If $\alpha\leq \beta\leq\kappa$ and $n<\omega$ then
$\varphi^{\mathfrak e}_n (\modM_\alpha)=\modM_\alpha\cap
\varphi^{\mathfrak e}_n (\modM_\beta)$.
\item Assume
$ {\mathfrak e}   \in {\mathfrak E}   ^{{\frakss}}$,
$\delta\in S$ and clause (F) of Definition \ref{w.s. nice definition} holds for
$\delta$. Suppose $s_0,\ldots, s_{k(*)-1}$ are distinct members of
$J_\delta$. Then
\begin{enumerate}
\item[(a)]
for each $n$, the set
 $\varphi^{\mathfrak e}_n (\langle \modM_\delta
\cup \bigcup
\{\modP_{s_k}^{\mathfrak e}:k<k(*)i  \}
\rangle_{\modM_{\delta+1}})$ is equal to $$
\langle \modM_\delta\cup \bigcup\{\modP^{\mathfrak e}_{s_k}:k<k(*)\}
\rangle_{\modM_{\delta+1}}
)
\cap \varphi^{\mathfrak e}_n (\modM_\kappa)\rangle.$$

\item[(b)]
 Suppose $z_\ell \in \modP^{\mathfrak e}$ for $\ell<n, z \in \modM_\delta$,
and $z+\sum\{h^\delta_{s_k}(z_k):k<k(*)\}\in \varphi_n
(\modM_{\delta+1})$. Then, there is  $z'_k\in \sum
\{\rRang(h^{\mathfrak e}_\ell):\ell<\omega\}$ \st\
the following two properties hold:
\begin{enumerate}
	\item[ $(b.1)$]:
$z_k-z'_k\in \varphi_n (\modP^{\mathfrak e})$ and
\item[$(b.2)$ ]: $z+\sum \{
h^\delta_{s_k}(z'_k):k<k(*)\}\in \varphi_n (\modM_{\delta+1})$.
\end{enumerate}
\item[(c)]
 $\langle \modM_\delta
\cup\{ \modP^{\mathfrak e}_{s_k}:k
<k(*)\}\rangle_{\modM_{\delta+1}} \leq^{\rm pr} _{\varphi_n^{\mathfrak e}}
 \modM_{\delta+1}$.
\item[(d)]
$\varphi^{\mathfrak e}_n (\modM_{\delta+1})=
\varphi^{\mathfrak e}_n (\modM_\delta)+\sum \{\varphi^{\mathfrak e}_n
(\modP^{\mathfrak e}_s):s\in J_\delta\}$.
\end{enumerate}
\end{enumerate}
\end{lemma}
\begin{proof}
  (1). We proceed  by induction on $\beta \geq \alpha$. The straightforward details leave to  reader.

 (2). If $\alpha \notin S,$ then the conclusion follows from (1) and Lemma \ref{comparing two orders}. Now suppose that $\alpha \in S.$ By Definition \ref{w.s. nice definition}(F), the result holds if $\beta=\alpha+1.$ But then for any $\beta> \alpha$ (and since $\alpha+1 \notin S$),
 we have
 \[
 \modM_\alpha\cap
 \varphi^{\mathfrak e}_n (\modM_\beta)= \modM_\alpha \cap (\varphi^{\mathfrak e}_n (\modM_\beta) \cap \modM_{\alpha+1})=\varphi^{\mathfrak e}_n (\modM_{\alpha+1}) \cap \modM_\alpha = \varphi^{\mathfrak e}_n (\modM_\alpha).
 \]
 (3). Items (a), (c) and (d) can be proved easily by Definition \ref{w.s. nice definition}(F). Let us prove (b). To this end, take $z_\ell \in \modP^{\mathfrak e}$ for $\ell<k(*), z \in \modM_\delta$,
 and $z+\sum\{h^\delta_{s_k}(z_k):k<k(*)\}\in \varphi_n
 (\modM_{\delta+1})$.

In the light  of (c) we observe that $$z+\sum\{h^\delta_{s_k}(z_k):k<k(*)\}\in \varphi^{\mathfrak e}_n (\langle \modM_\delta
 \cup\{ \modP^{\mathfrak e}_{s_k}:k
 <k(*)\}\rangle_{\modM_{\delta+1}}).$$ Let $k< k(*)$. We now apply the items $(a)$, $(c)$ and $(d)$
 to   find $y \in \varphi^{\mathfrak e}_n (\modM_\delta)$
 and $y_k \in   \varphi^{\mathfrak e}_n(\modP^{\mathfrak e}_{s_k})$   such that
 \[
 z+\sum\{h^\delta_{s_k}(z_k):k<k(*)\}=y+\sum\{y_k:k<k(*)\}.
 \]
 Since $ \varphi^{\mathfrak e}_n(\modP^{\mathfrak e}_{s_k})$ is the image of $\varphi^{\mathfrak e}_n(\modP^{\mathfrak e})$ under $h_{s_k}^{\delta}$, we have
 $y_k= h_{s_k}^\delta(x_k)$ for some $x_k \in \varphi^{\mathfrak e}_n(\modP^{\mathfrak e}).$
 Then
 \[
 \sum\{h^\delta_{s_k}(z_k-x_k):k<k(*)\}=y-z \in \varphi^{\mathfrak e}_n(\modM_\delta).
 \]
 As $s_0, \cdots, s_{k(*)-1}$ are chosen distinct, we deduce that
 \[
 h^\delta_{s_k}(z_k-x_k) \in \Rang(h^\delta_{s_k}) \cap \varphi^{\mathfrak e}_n(\modM_\delta)= \varphi^{\mathfrak e}_n(\Rang(h^\delta_{s_k})),
 \] where $k<k(*).$
Hence, $z_k-x_k \in \varphi^{\mathfrak e}_n(\modP^{\mathfrak e})$.

Let $h<\omega$ be such that for each $k<k(*)$,
 \[\begin{array}{ll}
 x_k &\in \varphi^{\mathfrak e}_n\left(\bigoplus\limits_{\ell\leq h}(\Rang({\bf f}^{\mathfrak e}_n))\right)\\
 &= \bigoplus\limits_{\ell\leq h} \varphi^{\mathfrak e}_n(\Rang({\bf f}^{\mathfrak e}_\ell))\\
 &=\bigoplus\limits_{\ell < h} \varphi^{\mathfrak e}_n(\Rang({\bf h}^{\mathfrak e}_\ell))
 \oplus \varphi^{\mathfrak e}_n(\Rang({\bf f}^{\mathfrak e}_h)).
 \end{array}\]
 Thus, for some $z'_k \in \bigoplus\limits_{\ell < h} \varphi^{\mathfrak e}_n(\Rang({\bf h}^{\mathfrak e}_\ell))$, we have
 \[
 x_k - z'_k \in \varphi^{\mathfrak e}_n(\Rang({\bf f}^{\mathfrak e}_h)) \subseteq \varphi^{\mathfrak e}_n(\modP^{\mathfrak e}).
 \]
 It then follows that
 \[
 z_k-z'_k=(z_k- x_k)+ (x_k - z'_k) \in \varphi^{\mathfrak e}_n(\modP^{\mathfrak e}) +\varphi^{\mathfrak e}_n(\modP^{\mathfrak e})=\varphi^{\mathfrak e}_n(\modP^{\mathfrak e}).
 \]
 We now show that $w:=z+\sum\{h^\delta_{s_k}(z'_k):k<k(*)\} \in \varphi_n
 (\modM_{\delta+1}).$ Indeed,
 \[\begin{array}{ll}
 w&=z+\sum\{h^\delta_{s_k}(z'_k-z_k):k<k(*)\}+\sum\{h^\delta_{s_k}(z_k):k<k(*)\}\\
 &=(z+   \sum\{h^\delta_{s_k}(z_k):k<k(*)\})+\sum\{h^\delta_{s_k}(z'_k-z_k):k<k(*)\}\\
 &\in \varphi_n
 (\modM_{\delta+1})+ \sum\{ \varphi^{\mathfrak e}_n(\modP^{\delta}_{s_k}): k < k(*)  \}\\
 &=\varphi_n
 (\modM_{\delta+1}).
 \end{array}\]
 This completes the proof of (b) and hence of the lemma.
\end{proof}
The main result of this section is that under suitable  assumptions on  $(\lambda, {\frakss},S,\bar \gamma^*)$,
there is, in $\text{ZFC},$ a  semi-nice construction for it. Before, we state and prove our result, let us show that
under extra set theoretic assumptions we can get a stronger result.

\begin{Definition}
\label{3.3bis}
Suppose $\lambda=\cf(\lambda)>|\ringR|+|\ringS|+\aleph_0$, $S\subseteq S^\lambda_{\aleph_0}$ is stationary and non-reflecting\footnote{$S$ is non-reflecting if for all limit ordinals $\xi < \lambda$ of uncountable cofinality, $S \cap \xi$ is non-stationary.}
such that $S^\lambda_{\aleph_0} \setminus S$ is stationary as well.
Then $\bar\modM=\langle \modM_\alpha:\alpha\leq\lambda\rangle$ is called very nicely
constructed for $(\lambda,\frakss, S)$ if:
\begin{enumerate}
\item Clauses (A)-(E) of Definition \ref{w.s. nice definition} holds,

\item Each $J_\alpha$
is singleton,
\item For each $\modN\in {\cal K}$, for stationary many
$\alpha\in \lambda\setminus S$, $\modN$ appears as one of the summands of $\modM_{\alpha+1}$,

\item Replace clause (F) of Definition \ref{w.s. nice definition} by the following modification:\\
For $\delta\in S$, $\modM_{\delta+1}$
is defined either as in
(E), or else
there are  an infinite $\cal U \subseteq \omega$ and the sequences $\langle\alpha_n: n<
\omega\rangle$ and $\langle \beta_m(\alpha_n): m \in \cal U \rangle$ and $\langle  h_{\alpha_n, n}: n \in \cal U    \rangle$  such that
\begin{enumerate}
\item each  $\alpha_n\in \lambda \setminus S$,

\item $\langle\alpha_n: n<
\omega\rangle$ is increasing and cofinal in $\delta,$

 \item ${\mathfrak e}_\delta={\mathfrak e}$,

\item for each $n<\omega$ and $m \in \cal U$, $\alpha_n <\beta_m(\alpha_n) < \alpha_{n+1}$ is in $\lambda \setminus S$,

 \item for $n \in \cal U$, $h_{\alpha_n, n}: \modN_n^{\mathfrak e} \to \modM_{\beta_n(\alpha_n)}$ is a bimodule homomorphism and
 $\modM_{\alpha_n}+\Rang(h_{\alpha_n, n})=\modM_{\alpha_n}\oplus \Rang(h_{\alpha_n, n}) \leq_{\aleph_0}\modM_{\beta_m(\alpha_n)}$.

\item Set $\modP_{\cal U}: = \bigoplus_{n \in \cal U}\Rang({\bf f}_n^{\mathfrak e})$ and $\modP_{{\cal U}, n}:=\sum\limits_{\ell \in \cal U \cap n}\Rang({\bf h}^{\mathfrak e}_\ell)$. Recall from  Lemma \ref{Pe bimodules defined} the bimodule homomorphisms
 ${\bf h}^{\mathfrak e}_n:\modN^{\mathfrak
	e}_n\longrightarrow
\modP^{\mathfrak e}$ and ${\bf f}^ {\mathfrak e}:\modN^{\mathfrak e}_0\longrightarrow
\modP^{\mathfrak e}$.
We define $\modN^\ast_\delta:=\sum\limits_{n\in {\cal U}} h_{\alpha_n, n}(\modN_n)$.
Then
 $\modP^{\mathfrak e}$ is isomorphic to $\modP_{\cal U}$ by
an isomorphism $h_\delta$ such that the following  diagram  commutes (here, $n=i(m)=m$-th member of
$\cal U$):
\[
\begin{picture}(150,150)
\put(0,60){${\bf h}^{\mathfrak e}_m$}
\put(10,0){\begin{picture}(150,150)
\put(0,25){$\modP_{\cal U}$}
\put(5,75){\vector(0,-1){35}}
\put(0,85){$\modN_n$}
\put(30,90){\vector(1,0){35}}
\put(40,100){$h_{\alpha_n, n}$}
\put(80,85){$h_{\alpha_n, n}(\modN_n)$}
\put(95,75){\vector(0,-1){35}}
\put(90,25){$\modP^{\mathfrak e}$}
\put(45,10){$h_\delta$}
\put(110,60){$\mbox{ id }$}
\put(75,28){\vector(-1,0){45}}
\end{picture}}
\end{picture}
\]
\item $\modP^{\mathfrak e}_{n}=h_\delta'' (\modP_{{\cal U}, n})$.
So in $\modM_{\delta+1}$, $\modP^{\mathfrak e}\cap \modM_\delta= \modN^\ast_\delta$.

\end{enumerate}
\end{enumerate}
\end{Definition}
Recall that for a stationary set $S \subseteq \lambda,$ Jensen's diamond $\diamondsuit_\lambda(S)$ asserts the existence
of a sequence $\langle  S_\alpha: \alpha \in S   \rangle$
such that for every $X \subseteq \lambda$ the set $\{\alpha \in S: X \cap \alpha=S_\alpha    \}$
is stationary. It is easily seen that $\diamondsuit_\lambda(S)$ implies $2^{<\lambda}=\lambda$. One the one hand, due to a well-known
theorem of Jensen, $\diamondsuit_\lambda(S)$ holds in G\"{o}del's constructible universe ${\text L}$ for all uncountable regular cardinals
$\lambda$ and all  stationary sets $S \subseteq \lambda$. On the other hand, if $2^{<\lambda}$ is forced to be above $\lambda,$
then  $\diamondsuit_\lambda(S)$ fails for all stationary sets $S \subseteq \lambda$.
The next theorem of Shelah shows that under diamond principle, we can get very nicely
$\langle \modM_\alpha: \alpha\leq\lambda\rangle$.

\begin{lemma}
\label{3.5x}
Suppose $\frakss$ is a non-trivial context, $\lambda=\cf(\lambda)>|\ringR|+|\ringS|+\kappa(\frakss)+||\frakss||$, $S \subseteq S^\lambda_{\aleph_0}$ is a stationary non-reflecting subset of $\lambda$ and $\diamondsuit_\lambda(S)$ holds.
Then
there is a very nicely
construction for $(\lambda, \frakss, S)$  such that
\begin{enumerate}
  \item $(\lambda, \frakss, S)$ is strongly very nice.
  \item If $\cf(\alpha)\in \kappa\smallsetminus
S$ and $\alpha<\beta\leq\lambda$, then $\modM_\alpha$ is a ${\cal K}$-
directed summand of $\modM_\beta$.
\end{enumerate}
\end{lemma}
\begin{proof}
See \cite[Section 2]{Sh:381}.
\end{proof}
We would like to prove a similar result in $\text{ZFC}$, thus we have to avoid the use of diamond
in the above lemma.  To this end, we use Shelah's Black Box. This leads us to get a weaker conclusion than Lemma \ref{3.5x},
but as we will see later, it is sufficient to get our desired results.

\begin{notation} For the rest of this section we use the following:
\begin{itemize}
  \item $\lambda >\kappa$ be infinite cardinals.
  \item $\mathcal H_{<\theta}(X)$ be the least set $Y$ such that $Y \supseteq X$ and if $y \subseteq Y$
  and $|y|< \theta$,then $y \in Y.$
  \item $\langle \tau_n: n<\omega \rangle$ be an increasing sequence of vocabularies, each of size $\leq \kappa$ such that for each $n<\omega$,
  there exists some unary predicate $P_n$ in $\tau_{n+1} \setminus \tau_n$.
\end{itemize}\end{notation}
\begin{Definition}
\label{family fn}
For $n<\omega$ let $\cal F_n$ be the family of sets of the form
\[
\{(\cal  N_\ell, f_\ell): \ell \leq n        \}
\]
satisfying the following conditions:
\begin{enumerate}
  \item[(a)] $f_\ell: \kappa^{\leq \ell} \to \lambda^{\leq \ell}$ is a tree embedding, i.e.,
  \begin{enumerate}
    \item[(1)] for each $\eta \in \kappa^{\leq \ell},~\eta$ and $f_\ell(\eta)$
    have the same length,
    \item[(2)] for $\eta \lhd \nu$ in $\kappa^{\leq \ell},~f_{\ell}(\eta) \lhd f_{\ell}(\nu)$,
     \item[(3)]   $~f_{\ell}$ is one-to-one,
  \end{enumerate}
  \item[(b)] for $\ell+1 \leq n,~f_{\ell+1}$ extends $f_\ell,$
  \item[(c)] for some $\tau'_\ell \subseteq \tau_\ell,~\cal  N_\ell$ is a $\tau'_\ell$-structure
  of size $\leq \kappa$ and the universe of $\cal  N_\ell$, denoted by $N_\ell$, is a subset of
  $\cal H_{<\kappa^+}(\lambda)$.

  \item[(d)] $\tau'_{\ell+1} \cap \tau_\ell = \tau'_\ell$ and $\cal  N_{\ell+1} \restriction \tau'_\ell$
  extends $\cal N_\ell$.

  \item[(e)] if $P_m \in \tau'_{m+1}$, then $P_m^{ \cal N_\ell} = N_\ell.$

  \item[(f)] if $x, y \in N_\ell,$ then $\{x, y  \} \in N_\ell$ and $\emptyset \in N_\ell.$

  \item[(g)] $\Rang(f_\ell) \subseteq \cal N_\ell$.
\end{enumerate}
\end{Definition}
\begin{Definition}
\label{family fomega}
Let $\langle  \cal F_n: n< \omega \rangle $ be as in Definition \ref{family fn}. Then $\cal F_\omega$ is the family consisting of all
pairs $(\cal N, f)$ such that for some sequence $\langle (\cal N_\ell, f_\ell): \ell < \omega    \rangle$, we have
\begin{enumerate}
  \item[(a)] for each $n<\omega$, $\{ (\cal N_\ell, f_\ell): \ell \leq n  \}$ belongs to $\cal F_n,$
  \item[(b)] $f= \bigcup\limits_{\ell < \omega}f_\ell$ and $\cal N = \bigcup\limits_{\ell < \omega} \cal N_\ell.$
\end{enumerate}
\end{Definition}
We may note that if for each $m<\omega, P_m \in \tau'_{m+1}$, then for each
$(\cal N, f) \in \cal F_\omega$, there is a unique sequence $\langle (\cal N_\ell, f_\ell): \ell < \omega    \rangle$
witnessing this.
\begin{Definition}
\label{branch of fomega}
Let $(\cal N, f) \in \cal F_\omega$. A branch of $f$ is any $\eta \in \lambda^\omega$
such that for each $n<\omega,~\eta \restriction n \in \Rang(f)$. We use $\lim(f)$ for the set of all branches of $f$.
\end{Definition}
Given $W \subseteq \mathcal{F}_\omega$, we define two games  $\Game_W$ and  $\Game'_W$ as follows.
\begin{Definition}
\label{game for fomega}
Suppose $W \subseteq \mathcal{F}_\omega$.
\begin{enumerate}
\item The game $\Game_W$ of length $\omega$ is defined as follows.
In the $n^{\rm th}$ move, player I chooses $\cal N_n$ such that $\{(\cal N_\ell, f_\ell): \ell \leq n    \} \in \cal F_n$ (noting that $f_0$ is determined),
and player II chooses a tree embedding $f_n: \kappa^{\leq n} \to \lambda^{\leq n}$ extending
$\bigcup\limits_{\ell < n}f_\ell$ such that $\Rang(f_n) \setminus (\bigcup\limits_{\ell<n}\Rang(f_\ell))$
is disjoint to $\bigcup\limits_{\ell< n}N_\ell.$
Player II wins  if $(\bigcup\limits_{\ell<\omega}\cal N_\ell, \bigcup\limits_{\ell<\omega}f_\ell) \in W.$

\item  The game $\Game'_W$ of length $\omega$ is defined as follows.
In the zero move, player I chooses $k
<\omega$,   $\{(\cal N_\ell, f_\ell): \ell \leq k    \} \in \cal F_k$, and $X_0 \subseteq \lambda^{<\omega}$ of size less than $\lambda$.
For $n>0,$
in the $n^{\rm th}$ move, player I chooses $\cal N_{k+n}$ and $X_n$ such that $\{(\cal N_\ell, f_\ell): \ell \leq k+n    \} \in \cal F_{k+n}$
and $X_n \subseteq \lambda^{<\omega}$ is of size less than $\lambda$.
Then player II chooses a tree embedding $f_{k+n}: \kappa^{\leq k+n} \to \lambda^{\leq k+ n}$ extending
$\bigcup\limits_{\ell < k+ n}f_\ell$ such that $\Rang(f_{k+n}) \setminus (\bigcup\limits_{\ell<k+n}\Rang(f_\ell))$
is disjoint to $\bigcup\limits_{\ell< n}N_\ell \cup \bigcup\limits_{\ell<n}X_\ell.$
Player II wins  if $(\bigcup\limits_{\ell<\omega}\cal N_\ell, \bigcup\limits_{\ell<\omega}f_\ell) \in W.$
\end{enumerate}
\end{Definition}

\begin{Definition}
\label{barrier}
Suppose $W \subseteq \mathcal{F}_\omega$. Recall that
\begin{enumerate}
  \item $W$ is called a barrier if player I does not win $\Game_W$ or even $\Game'_W$.
  \item $W$ is  called a strong barrier if player II wins $\Game_W$ and even $\Game'_W$.
  \item $W$ is  called disjoint if for distinct $(\cal N^1, f^1), (\cal N^2, f^2)$ in $W$, $f^1$ and $f^2$ have
  no common branch.
\end{enumerate}
\end{Definition}
We are now ready to state the version of Shelah's black box theorem that is  needed
in this paper.
\begin{lemma}[The Black Box theorem]
\label{shelah black box}
Suppose $\lambda> \kappa$ are infinite cardinals, $\cf(\lambda)> \aleph_0,~\lambda^{\aleph_0}=\lambda^\kappa$
and $S \subseteq S^\lambda_{\aleph_0}$ is stationary. Then there is a sequence
\[
W=\{ (\cal N^\alpha, f^\alpha): \alpha < \alpha_*          \} \subseteq \cal F_\omega
\]
and a non-decreasing function
\[
\zeta: \alpha_* \to S
\]
such that the following properties are satisfied:
\begin{enumerate}
  \item $W$ is a disjoint strong barrier,

  \item every branch of $f^\alpha$ is an increasing sequence converging to $\zeta(\alpha) \in S$,

  \item  each $\cal N^\alpha$ is transitive,

  \item if $\zeta(\beta)=\zeta(\alpha), \beta+\kappa^{\aleph_0} \leq \alpha < \alpha_*$
  and $\eta$ is a branch of $f^\alpha$, then for some $k<\omega, \eta \restriction k \notin \cal N^\beta,$

  \item if $\lambda=\lambda^\kappa,$ we can also demand that if $\eta$ is a branch of $f^\alpha$
  and $\eta \restriction k \in \cal N^\beta$ for all $k<\omega$, then $\cal N^\alpha \subseteq \cal N^\beta$,
  and even for all $n<\omega, \cal N^\alpha_n \in \cal N^\beta.$
\end{enumerate}
\end{lemma}

\begin{proof}
 	This is in \cite[2.8]{Sh:227}.
\end{proof}

We now state and prove the main result of this section.
\begin{theorem}
\label{nice construction lemma}
Assume ${\frakss}$ is a non-trivial
context. Suppose $\lambda > ||\frakss||$ is such that $\cf(\lambda)\geq \aleph_1+ \kappa(\frakss)$, for $\mu<\lambda, \mu^{\aleph_0} < \lambda$,
$\kappa=\cf(\lambda) > \kappa(\frakss)$ and $S \subseteq S^\kappa_{\aleph_0}$ is such that $S$ and $S^\kappa_{\aleph_0} \setminus S$
are stationary in $\kappa.$ Then
\begin{enumerate}
\item
There is a semi-nice construction
for $(\lambda,\frakss, S, \kappa)$.

\item If in addition $\lambda$  is regular,  then there is a  nice construction
for $(\lambda,\frakss, S, \kappa)$.

\item In part (1) if we  omit the assumption $\cf(\lambda)=\kappa$, and
 let  $\bar \gamma^*=\langle \gamma^*_\alpha: \alpha \leq \kappa   \rangle$
be such that  $\gamma^*_0=||\modM_*||$ and for all $\alpha < \kappa$
$\gamma^*_{1 + \alpha}=\gamma^*_0  +
\lambda\cdot \alpha$. Then there is a
semi-nice
construction for $(\lambda,\frakss,S,\bar \gamma^*).$

\item In part (3) if $\kappa=\lambda$, then there is a
nice
construction for $(\lambda,\frakss,S,\bar \gamma^*).$
\end{enumerate}
\end{theorem}
\begin{proof}
We prove  (1) and (2).  Clauses (3) and (4)
can be proved in a similar way. We start by fixing some notation and facts:
\begin{itemize}

\item Without loss of generality for ${\mathfrak e}\in \calE$ and $n<\omega$, the universe of
$\modN^{\mathfrak e}_n$ is a cardinal.

\item As $\cf(\lambda)=\kappa,$ let $\bar{\gamma}=\langle \gamma_\alpha^0:\alpha<\kappa\rangle$ be
an \incr\ and \cont\ sequence
cofinal in $\lambda$.

\item Let $\bar{ \gamma }^\ast=\langle  \gamma^\ast_\alpha: \alpha \leq \kappa \rangle $
 be an increasing and continuous sequence of ordinals with limit $\lambda$ such that
$\gamma^*_0>||
\frakss||$ and for each $\alpha < \kappa, \gamma^*_{\alpha+1}-\gamma^*_\alpha$ is $>0$ and divisible by
$||\frakss||^{\aleph_0}+ \gamma^\ast_\alpha$~. We further assume that for each $\alpha<\kappa, \gamma^*_\alpha \geq \gamma_\alpha.$

\item Let $\langle \modN^*_\alpha: \alpha < |\cal K|  \rangle$
be an enumeration of elements of $\cal K.$
\end{itemize}
  Set
  \[
  \bar S := \{   \gamma^\ast_\alpha: \alpha \in S\}.
  \]
Then $\bar S$ is  a stationary subset of $\lambda$. For each $n<\omega$, let $\tau_n$ be the countable vocabulary
$\tau_n:=\{\in, F, G, P_0, \cdots, P_{n-1}, c_0, c_1, \cdots, c_i, \cdots     \}$, where $F, G$ are unary function symbols, $P_0, \cdots, P_{n-1}$
are unary predicate symbols and $c_i$'s, for $i<\omega$ are constant symbols. We now apply the black box theorem (see Lemma \ref{shelah black box})
to get a sequence
\[
W=\{ (\cal N^\alpha, f^\alpha): \alpha < \alpha_*          \} \subseteq \cal F_\omega
\]
and a non-decreasing function
\[
\zeta: \alpha_* \to \bar S.
\]
By induction on $\epsilon\leq\kappa$ we will construct
${{\mathscr A} }_\epsilon$ which is going to be a semi-nice construction up to
$\epsilon$ in the following sense:

\begin{enumerate}
\item[(a)]
${\mathscr A} _\epsilon$ consists of
\begin{enumerate}
  \item[(a-1)] $\langle
\modM_\xi:\xi\leq\epsilon\rangle$,

  \item[(a-2)] $\langle \modN_t, h_t:t\in J_\xi,
\xi \in \epsilon \setminus S_{\epsilon}\rangle$,

  \item[(a-3)] $S_\epsilon \subseteq S \cap \epsilon,$

  \item[(a-4)] $\cal T_\epsilon,$

  \item[(a-5)]  $\langle({\mathfrak e}_s,{\mathbb P}_s,
\bar{\alpha}_s,\bar{t}_s,
\bar{g}_s,h_s,q_s):s
\in J_\delta$ and  $\delta\in S_\epsilon\rangle$,

  \item[(a-6)] $\langle \gamma^*_\xi:\xi\leq\epsilon\rangle$.
\end{enumerate}

\item[(b)]
$ {\mathscr A} $
satisfies all the  relevant parts of Definition \ref{w.s. nice definition}.
\item[(c)] we have
\begin{enumerate}
  \item[(c-1)] if
$\xi<\epsilon$, then $S_\xi=S_\epsilon
\cap \xi$,

\item[(c-2)] for $\delta\in S \cap\epsilon$,
$$J_\delta\subseteq\{\beta<
\alpha_*: \zeta(\beta)=\gamma_\delta\},$$

\item[(c-3)] if $ \epsilon \notin S$, then
 $\epsilon\notin S_{\epsilon+1}$ and
 $$ J_\epsilon=\{ \beta: \exists \eta \in \cal T_\epsilon,~\eta^{\frown}\langle\beta \rangle \in {\cal T}_{\epsilon+1}\}.$$
\end{enumerate}
\item[(d)]
Let $\langle \gamma'_\epsilon: \epsilon < \kappa  \rangle$ be an increasing and continuous sequence of ordinals cofinal in $\lambda$
with $ \gamma'_0= \Vert
  \modM^{{\frakss}}_ * \Vert  $. The sequence is such that
  \begin{enumerate}
 \item[(d-1)]  $\langle ( n_ \gamma , {\mathfrak e}   _ \gamma,f_\gamma ):
\gamma<
\gamma'_\epsilon\rangle$ is a collection  of triples $(n,{\mathfrak e},f)$ where
$n<\omega,{\mathfrak e}\in \callE$ and $f$ is a bimodule  homomorphism from
$\modN^{\mathfrak e}_n$ into $\modM_\epsilon$,

  \item[(d-2)] every triple  $(n,{\mathfrak e},f)$ as above appears as some $( n_ \gamma , {\mathfrak e}   _ \gamma,f_\gamma ),$
  for some $\gamma < \gamma'_\epsilon$, for some $\epsilon < \kappa$ large enough.
  \end{enumerate}

   \item[(e)] we have
   \begin{enumerate}
\item[(e-1)] $\langle {\cal T}_\epsilon: \epsilon < \kappa \rangle$ is increasing and
continuous,

\item[(e-2)] ${\cal T}_0
=\{<>\}$,

\item[(e-3)]  if $
 {\cal T}_{\epsilon+1} \setminus {\cal T}_\epsilon \neq \emptyset$,
then
$${\cal T}_{\epsilon+1}
= \{\eta^{\frown} \langle \beta \rangle: \eta \in {\cal T}_{\epsilon} \text{~and~} \beta \in  [ \gamma'_\epsilon , \gamma'_{\epsilon + 1 } ) \}.$$
\end{enumerate}
\end{enumerate}
We now start defining $\cal A_\epsilon$ for $\epsilon \leq \kappa.$

{\underline{$\epsilon=0$}.} Set $\modM_0=\modM_*$ and let $\gamma^*_0$ be the universe of $\modM_\ast$.

{\underline{$\epsilon$ is a limit ordinal.}}  Then set $\modM_\epsilon=\bigcup\limits_{\xi<\epsilon}\modM_\xi$, and for everything else, just take
the union of the previous ones.

{\underline{$\epsilon=\xi+1$ and $\xi \notin S$}.} In the same vein as of Definition \ref{w.s. nice definition}(E), we define $\modM_\epsilon$, so
\[
\modM_\epsilon = \modM_\xi \oplus \bigoplus\limits_{t \in J_\xi}\modN_t,
\]
where $J_\xi$ and $\modN_t$, for $t \in J_\xi$ are defined similar to  Definition \ref{w.s. nice definition}(E). Everything else in the definition
of $\cal A_\epsilon$ is clear how to define.

{\underline{$\epsilon=\delta+1$, $\delta \in S$ and $\gamma^*_\delta \neq \gamma^0_\delta$}.} Then define $\modM_\epsilon$
as in the previous case (i.e., as in Definition \ref{w.s. nice definition}(E)).

Note that in all of the above cases, if $\delta < \epsilon$ is in $S$, $s \in J_\delta$ and in stage $\delta$ the type $q_s=q^\delta_s$
is defined,  then in view of Lemma \ref{formulas and direct sum} (and its natural generalization to arbitrary direct sums), $q_s$
continues to be omitted at $\epsilon$.

{\underline{$\epsilon=\delta+1$, $\delta \in S$ and $\gamma^*_\delta = \gamma^0_\delta$}.}
This is the most important part of the induction which corresponds to the case (F) of Definition \ref{w.s. nice definition}. Set
\[
\alpha^*_\delta=\sup\{\alpha+1: \zeta(\alpha) \leq \gamma^0_\delta  \}.
\]
By induction on $\alpha \leq \alpha^*_\delta$ we define
\begin{itemize}
  \item $\modM_{\delta,\alpha}$,
  \item $w_{\delta,\alpha},$
  \item $\eta_\beta$ for $\beta \in w_{\delta, \alpha}$,
  \item ${\mathfrak x}_\beta=({\mathfrak e}_\beta,\modP_\beta,\bar \alpha_\beta,\bar
t_\beta,\bar g_\beta,h_\beta,
q_\beta)$ and $  \eta_\beta$ for
$\beta\in w_{\delta,\alpha}$,
\end{itemize}
such that:
\begin{enumerate}
\item[$(\alpha)$] $w_{\delta,\alpha}\subseteq \{\beta<\alpha:
\zeta(\beta)=\gamma^0_\delta\}$ and for $\beta<\alpha$, $w_{\delta,\alpha}
\cap \beta=w_{\delta, \beta}.$
\item[$(\beta)$] $\eta_\beta\in \rlim (f^\beta)$ for $\beta\in w_{\delta,\alpha}$.
\item[$(\gamma)$] clause (F) of Definition \ref{w.s. nice definition} is satisfied when we consider
\[
(M_{\delta, \alpha}, w_{\delta, \alpha}, \langle {\mathfrak x}_\beta: \beta \in w_{\delta, \alpha} \rangle )
\]
instead of
\[
(\modM_{\delta+1}, J_\delta, \langle  ({\mathfrak e}_s,{\mathbb P}_s,
\bar{\alpha}_s,\bar{t}_s,
\bar{g}_s,h_s,q_s)       : s \in J_\delta  \rangle).
\]
\end{enumerate}
{\underline{$\alpha=0$}.} Then set $\modM_{\delta, 0}=\modM_\delta$ and $w_{\delta, 0}=\emptyset$. There is nothing else to define.

\underline{$\alpha$ is a limit ordinal.} Then set $M_{\delta, \alpha}=\bigcup\limits_{\xi<\alpha}M_{\delta, \xi}$,
$w_{\delta, \alpha}=\bigcup\limits_{\xi<\alpha}w_{\delta, \xi}$ and for $\beta \in w_{\delta, \alpha}$ pick some $\xi < \alpha$
such that $\beta \in w_{\delta, \xi}$ and then define $\eta_\beta$ and ${\mathfrak x}_\beta$ as defined at step
$\xi$ of the induction.

\underline{$\alpha=\beta+1$ and $\zeta(\beta) \neq \gamma^0_\delta$.} Then set $M_{\delta, \alpha}=M_{\delta, \beta}$
and $w_{\delta, \alpha}=w_{\delta, \beta}$.

\underline{$\alpha=\beta+1$ and $\zeta(\beta) = \gamma^0_\delta$.} In this case we decide if we add $\beta$ to our final $J_\delta$
or not, where $J_\delta$ is going to be $w_{\delta, \alpha^*_\delta}$. We consider two cases.

\par\noindent
\underline{Case 1.} There are $i<2$
and $\eta\in \rlim(f^\beta)$
\st\ $(\modM_{\delta, \beta}^{\eta, i}, w_{\delta, \beta+1}, \langle {\mathfrak x}_\nu: \nu \in w_{\delta, \beta+1} \rangle )$
satisfies the conclusion of clause (F) of Definition \ref{w.s. nice definition}, where $(\modM_{\delta, \beta}^{\eta, i}, w_{\delta, \beta+1}, \langle {\mathfrak x}_\nu: \nu \in w_{\delta, \beta+1} \rangle )$
is defined as follows:
\begin{enumerate}
  \item[($\delta$)] $w_{\delta, \alpha}:=w_{\delta, \beta} \cup \{ \beta \},$
  \item[($\epsilon$)] $\eta_\beta:=\eta,$
  \item[($\varepsilon$)] ${\mathfrak x}_\beta:=({\mathfrak e}_\beta,\modP_\beta,\bar \alpha_\beta,\bar
t_\beta,\bar g_\beta,h_\beta,
q_\beta)$, where this sequence is defined as follows:
\begin{enumerate}
  \item[($\varepsilon$-1)] $\bar t_\beta := \langle t_n: n < \omega \rangle,$ where for each $n<\omega, t_n = G^{\cal N^\beta}(\eta(n))$,
  \item[($\varepsilon$-2)] for some $\alpha_n < \delta, t_n \in J_{\alpha_n}$ (note that such an $\alpha_n$ is unique if it exists),
  \item[($\varepsilon$-3)] $\bar\alpha_\beta:=\langle \alpha_n: n< \omega  \rangle$ is an increasing sequence of ordinals in $\delta \setminus S$ cofinal in $\delta$,
  \item[($\varepsilon$-4)] for each $n, m<\omega, \mathfrak e_{t_n}=\mathfrak e_{t_m}.$ Set $\mathfrak e := \mathfrak e_{t_0}$,

  \item[($\varepsilon$-5)] $\{ x_n^{\mathfrak e}, h_{t_n}, t_n, \alpha_n    \} \in N_{n+1}^\beta,$

  \item[($\varepsilon$-6)] $\bar g_\beta :=\bar g_{\beta}^{\eta, i}:= \langle  g_{\beta, n}^{\eta, i}: n< \omega    \rangle$ where $g_n=g_{\beta, n}^{\eta, i}: \modN_n^{\mathfrak e} \to \modM_{\gamma_\delta^{**}},$ for some $\gamma^{**}_\delta$, is defined by $g_n = i \cdot f_{c_n^{\cal N^\beta}}$,

  \item[($\varepsilon$-7)] $z_{\beta, n}^{\eta, i}:=F^{\cal N^\beta} \left(\sum\limits_{\ell < n} (g_\ell+h_{t_\ell})(x_\ell^{\mathfrak e})\right)$,

  \item[($\varepsilon$-8)] $q_\beta:=q_{\delta, \beta}^{\eta, i} := \{\varphi_n^{\mathfrak e}(y- z_{\beta, n}^{\eta, i}): n< \omega   \}$,

  \item[($\varepsilon$-9)] $h_\beta:=h_{\delta, \beta}^{\eta, i}: \modP^{\mathfrak e} \to \modP_\beta$ is onto,

  \item[($\varepsilon$-10)]  $\modM_{\delta, \beta}^{\eta, i}$ is generated by $\modM_{\delta, \beta} \cup \modP_\beta$,
\end{enumerate}
\end{enumerate}
In this case we set
\begin{itemize}
  \item $\modM_{\delta, \beta+1}= \modM_{\delta, \beta}^{\eta, i}$,
  \item $w_{\delta, \beta+1}=w_{\delta, \beta} \cup  \{ \beta \}$,

  \item $\eta_\beta=\eta,$

   \item for  $\nu \in w_{\delta, \beta+1}$ we let $\mathfrak x_\nu$ be as above,
\end{itemize}
\par\noindent
\underline{Case 2.} The above situation does not hold. Then we set $\modM_{\delta, \beta+1}:=\modM_{\delta, \beta}$
and $w_{\delta, \beta+1}:=w_{\delta, \beta}$.

Having considered the above cases, we set
\begin{itemize}
  \item $\modM_{\delta+1}:=\modM_{\delta, \alpha^*_\delta}$,
  \item $J_\delta:=w_{\delta, \alpha^*_\delta}$,
  \item for $s \in J_\delta$, $\langle  ({\mathfrak e}_s,{\mathbb P}_s,
\bar{\alpha}_s,\bar{t}_s,
\bar{g}_s,h_s,q_s)       : s \in J_\delta  \rangle$ is defined to be $\mathfrak x_s.$
\end{itemize}
This completes our inductive construction of $\cal A_\epsilon$'s for $\epsilon \leq \kappa.$
Finally, we set $\cal A :=\cal A_\kappa$ and $\modM:=\modM_\kappa.$
We are going to show that $\cal A$ is as required.

Clearly we have gotten a weakly semi-nice construction. We now show that it is indeed a semi-nice construction.
Let us prove the property presented in  Definition \ref{nice construction modified}(3)$(G)_1$. Thus suppose that
${\mathfrak e}\in \callE$, ${\bf f}$  is an endomorphism of $\modM$ as
an $\ringR$-module, $\gamma<\kappa$ and for each $n<\omega$ suppose ${\bf g_n}: \modN_n^{\mathfrak e} \to \modM_\gamma$
is a bimodule homomorphism. Suppose by contradiction that the property presented in Definition \ref{nice construction modified}(3)$(G)_1$ fails. Let
\[
{\bf g}:  \cal T_\kappa \to \cal T_\kappa,
\]
where $\cal T_\kappa=\lambda^{<\omega}$, be a one-to-one order preserving (i.e, $\eta \lhd \nu$ iff ${\bf g}(\eta) \lhd {\bf g}(\nu)$) embedding such that for each $\eta \in \cal T_\kappa \setminus \cal T_0, \mathfrak e_{{\bf g}(\eta)} =\mathfrak e$.
Let
\[
\mathcal{B} = \langle B, \in^{\cal B},  F^{\cal B}, G^{\cal B}, (c_n^{\cal B})_{n<\omega}, (P_n^{\cal B})_{n < \omega}          \rangle
\]
be a $\tau = \bigcup\limits_{n<\omega}\tau_n$-structure satisfying the following properties:
\begin{itemize}
  \item $(\cal B, \in^{\cal B})$ expands $(\cal H_{<\aleph_1}(\modM), \in)$,
  \item $F^{\cal B} \restriction \modM={\bf f},$
  \item $G^{\cal B}={\bf g},$
  \item for $n<\omega, ~c_n^{\cal B}=\gamma_n$, where $\gamma_n$ is such that $(n, \mathfrak e, {\bf g}_n)=(n_{\gamma_n}, \mathfrak e_{\gamma_n}, f_{\gamma_n}).$
\end{itemize}
Pick $\gamma^{**} \in \kappa \setminus S$ such that $\gamma^{**}> \gamma$ and for each
$n<\omega, ~{\bf f}({\bf g}_n(x_n^{\mathfrak e})) \in \modM_{\gamma^{**}}$. Recall that $\Rang({\bf g}_n) \subseteq \modM_\gamma \subseteq \modM_{\gamma^{**}}.$
Note that the set
\[
\{ \delta \in S: \gamma^*_\delta=\gamma_\delta   \}
\]
is a stationary set. We are going to use  the black box theorem.
In the light of  Lemma \ref{shelah black box} we observe that  there are $\delta \in S$ and $\beta\in J_{\delta }$ such that
\begin{enumerate}
  \item $\zeta(\beta)=\gamma_\delta,$
  \item $ \gamma^*_\delta=\gamma_\delta,$
  \item $\cal N^\beta \prec \cal B,$ in particular,
  \begin{enumerate}
    \item $f^\beta = {\bf f} \restriction \cal N^\beta,$
    \item for all $n<\omega, c_n^{\cal N^\beta} = c_n^{\cal B}$,
    \item for $n<\omega, {\bf g_n}=f_{c_n^{\mathcal N^\beta}}$.
  \end{enumerate}
\end{enumerate}
Let $\epsilon =\delta+1$. Note that as $\delta \in S$
and $\gamma^*_\delta=\gamma_\delta$, at step $\epsilon$ of the construction we are
at one of the cases (1) or (2). The nontrivial part is to check the property presented in Definition \ref{w.s. nice definition}(n).  If case (1) occurs we are immediately done. So suppose that we are in case (2), and hence, by the construction, we are in one of the following situations:
\begin{enumerate}
  \item[$(\star)_1$] for all $\eta \in \lim(f^\beta)$ and all $i<2$  the type $q_{\delta, \beta}^{\eta, i}$ is realized in $\modM_{\delta, \beta}^{\eta, i}.$

  \item[$(\star)_2$] there is some $\delta' \in S_\delta$ and some $\beta' \in J_{\delta'}$ such that for some
  $\eta \in \lim(f^{\beta'})$ and some $i<2, q_{\delta', \beta'}^{\eta, i}$ is well-defined and is omitted in $\modM_{\delta', \beta'}^{\eta, i},$  but it is realized in $\modM_{\delta, \beta}^{\eta, i}.$

  \item[$(\star)_3$] there is some $\beta'\in w_{\delta, \beta}$,  some
  $\eta \in \lim(f^{\beta'})$ and some $i<2$ such that $\beta' \leq 2^{\aleph_0}+\beta$ and $q_{\delta, \beta'}^{\eta, i}$ is well-defined and is omitted in $\modM_{\delta, \beta'}^{\eta, i},$  but it is realized in $\modM_{\delta, \beta}^{\eta, i}.$
\end{enumerate}
We only consider the case $(\star)_1$ and leave the other cases to the reader, as they are easier to prove. Pick some
$\eta \in \lim(f^\beta)$. For $i<2$ the type
\[
q_{\delta, \beta}^{\eta, i}=\{ \varphi^{\mathfrak e}_{n}(y - z_{\beta, n}^{\eta, i}): n<\omega                \}
\]
is realized in  $\modM_{\delta, \beta}^{\eta, i}$. By our construction, $\modM_{\delta, \beta}^{\eta, i}$
is generated by $\modM_{\delta, \beta} \cup \modP_\beta$ and $h_{\delta, \beta}^{\eta, i}: \modP^{\mathfrak e} \to \modP_\beta$
is onto, thus we can find $y_i \in \modM_{\delta, \beta}$ and $z_i \in \modP^{\mathfrak e}$
such that $q_{\delta, \beta}^{\eta, i}$ is realized by
$y_i+h_{\delta, \beta}^{\eta, i}(z_i)$. Hence for all $n<\omega$,
\[
\modM_{\delta, \beta}^{\eta, i} \models \varphi^{\mathfrak e}_{n}( y_i+h_{\delta, \beta}^{\eta, i}(z_i) - z_{\beta, n}^{\eta, i}).
\]

Recall from  ($\varepsilon$-6) that:
$$ 
g_{\beta, n}^{\eta, i}=i\cdot f_{c_n^{\mathcal N^\beta}}= i \cdot {\bf g_n} \quad \quad\quad (\ast)$$

Hence
\[
 z_{\beta, n}^{\eta, i}=F^{\cal N^\beta} \left(\sum\limits_{\ell < n}(h_{t_\ell}+i\cdot {\bf g}_\ell)(x_\ell^{\mathfrak e}) \right)=
 \sum\limits_{\ell < n} {\bf f}(h_{t_{\ell}}(x_\ell^{\mathfrak e})) + i \cdot \sum\limits_{\ell < n}{\bf f}({\bf g}_\ell(x_\ell^{\mathfrak e})).
\]
It follows for each $i<2 $ that

\[\begin{array}{ll}
 y_i+h_{\delta, \beta}^{\eta, i}(z_i) - z_{\beta, n}^{\eta, i}
&=  y_i+h_{\delta, \beta}^{\eta, i}(z_i) - \sum\limits_{\ell < n} {\bf f}(h_{t_{\ell}}(x_\ell^{\mathfrak e})) - i \cdot \sum\limits_{\ell < n}{\bf f}({\bf g}_\ell(x_\ell^{\mathfrak e})) \\
&\in \varphi_n^{\mathfrak e}(\modM_{\delta, \beta}^{\eta, i}).   \\
\end{array}\]

By applying Lemma \ref{properties of w.s.n.c}(3) to $z_0 \in \modP^{\mathfrak e}$
and $y_0 - z_{\beta, n}^{\eta, 0} \in \modM_{\delta, \beta}$
we can find some $z'_{n, 0} \in \sum\limits_{\ell < \omega}\Rang({\bf h}_\ell^{\mathfrak e})$ equipped
with the following two properties:
\begin{enumerate}
  \item[(4)] $z_0 - z_{n, 0}' \in \varphi_n^{\mathfrak e}(\modP^{\mathfrak e})$,
  \item[(5)] $y_0 - z_{\beta, n}^{\eta, 0}+h_{\delta, \beta}^{\eta, 0}(z_{n, 0}') \in \varphi_n^{\mathfrak e}(\modM_{\delta, \beta})$.
\end{enumerate}
In the same vein, there is
$z_{n, 1}' \in \sum\limits_{\ell < \omega}\Rang({\bf h}_\ell^{\mathfrak e})$ such that
\begin{enumerate}
  \item[(6)] $z_1 - z_{n, 1}' \in \varphi_n^{\mathfrak e}(\modP^{\mathfrak e})$,
  \item[(7)] $y_1 - z_{\beta, n}^{\eta, 1}+h_{\delta, \beta}^{\eta, 1}(z_{n, 1}') \in \varphi_n^{\mathfrak e}(\modM_{\delta, \beta})$.
\end{enumerate}
Given $n<\omega$, pick $k_n < \omega$ such that
\[
z_{n, 0}', z_{n, 1}' \in \sum\limits_{\ell < k_n}\Rang({\bf h}_\ell^{\mathfrak e}).
\]
So we can find $\{z'_{n, 0, \ell}: \ell  < k_n  \}$ and $\{z'_{n, 1, \ell}: \ell  < k_n  \}$ such that for $i<2,~z'_{n, i, \ell} \in \Rang({\bf h}_\ell^{\mathfrak e})$
and $z_{n, i}' = \sum\limits_{\ell < k_n} z'_{n, i, \ell}.$ For $\ell < k_n$ pick some $z_{n, i, \ell} \in \modN_n^{\mathfrak e}$ with
${\bf h}_\ell^{\mathfrak e}(z_{n, i, \ell})=z'_{n, i, \ell}$.

Set $z=z_1-z_0.$ In view of  clauses (4) and (6) above, we have
\[
z \in (z_{n,1}'-z_{n,0}')+ \varphi_n^{\mathfrak e}(\modP^{\mathfrak e})= \sum\limits_{\ell < k_n} {\bf h}_\ell^{\mathfrak e} (z_{n, 1,\ell})- \sum\limits_{\ell < k_n} {\bf h}_\ell^{\mathfrak e} (z_{n, 0,\ell})+ \varphi_n^{\mathfrak e}(\modP^{\mathfrak e})
\]
for all $n$.
Thus part (a) of Definition \ref{nice construction modified}(3)$(G)_1$ is satisfied.

For part (b) of Definition \ref{nice construction modified}(3)$(G)_1$, by items  (5) and (7) we have
\[
\sum\limits_{\ell < n}{\bf f}({\bf g}_\ell(x_\ell^{\mathfrak e})) \in (y_1-y_0)+ \left(h_{\delta, \beta}^{\eta, 1}(z'_{n, 1}) -h_{\delta, \beta}^{\eta, 0}(z'_{n, 0})\right) + \varphi_n^{\mathfrak e}(\modM_{\delta,\beta})\quad(\dagger)
\]
Now let $\modK$ be such that $\modM_{\delta, \beta}=\modM_{\gamma^{**}}\oplus \modK$ and let $\pi: \modM_{\delta, \beta} \to \modK$ be the natural projection. We apply $(\ast)$ for $i=1$ along with the property presented in  Definition \ref{w.s. nice definition}(F)(k) to see:
\begin{enumerate}
\item[(8)] $h_{\delta, \beta}^{\eta, 1}(z_{n,1}') = \sum\limits_{\ell<k_n} h_{\delta, \beta}^{\eta, 1}(z'_{n, 1, \ell})= \sum\limits_{\ell<k_n} h_{\delta, \beta}^{\eta, 1}({\bf h}_\ell^{\mathfrak e}(z_{n, 1, \ell}))= \sum\limits_{\ell<k_n}(h_{t_\ell}+{\bf g}_\ell)(z_{n, 1, \ell})$.
\end{enumerate}
Also, we apply $(\ast)$ for $i=0$ along with the property presented in  Definition \ref{w.s. nice definition}(F)(k) to see:
\begin{enumerate}
\item[(9)] $h_{\delta, \beta}^{\eta, 0}(z_{n,0}') = \sum\limits_{\ell<k_n} h_{\delta, \beta}^{\eta, 0}(z'_{n, 0, \ell})=\sum\limits_{\ell<k_n} h_{\delta, \beta}^{\eta, 1}({\bf h}_\ell^{\mathfrak e}(z_{n, 0, \ell}))=  \sum\limits_{\ell<k_n}h_{t_\ell}(z_{n, 0, \ell}).$\end{enumerate}
 By the choice of $\gamma^{**}$, ${\bf f}({\bf g}_\ell(x_\ell^{\mathfrak e})) \in \modM_{\gamma^{**}}$ and $\rang({\bf g_\ell}) \subseteq \modM_\gamma \subseteq \modM_{\gamma^{**}}$ and hence the following assertions are valid as well:
\begin{enumerate}
\item[(10)] for all $\ell < n, \pi({\bf f}({\bf g}_\ell(x_\ell^{\mathfrak e})))=0$,

\item[(11)] for all $\ell < k_n$ and $i<2,$ $\pi({\bf g}_\ell(z_{n, i, \ell}))=0$,
\end{enumerate}

We may assume without loss  of the generality that $\alpha_{\ell}>\gamma^{**}$ for all $\ell<\omega$, where $\alpha_{\ell}$ is such that $t_{\ell}\in J_{\alpha_\ell}$.
It then follows from Definition \ref{w.s. nice definition}(E) that for each $\ell < \omega,$ $\rang(h_{t_\ell}) \subseteq \modK$
and hence:
\begin{enumerate}
\item[(12)] for all $\ell < k_n, \pi(h_{t_\ell}(z'_{n, 1, \ell}))=h_{t_\ell}(z'_{n, 1, \ell}).$
\end{enumerate}
By setting \begin{enumerate}
	\item[(13)] $y':=y_1-y_0$ and $y:=y'-\pi(y')$
\end{enumerate} in the previous formulas we observe that:
 \[\begin{array}{ll}
\sum\limits_{\ell < n}{\bf f}({\bf g}_\ell(x_\ell^{\mathfrak e}))&\stackrel{(10)}=\sum\limits_{\ell < n}{\bf f}({\bf g}_\ell(x_\ell^{\mathfrak e}))-\pi(\sum\limits_{\ell < n}{\bf f}({\bf g}_\ell(x_\ell^{\mathfrak e})))\\
&\stackrel{(\dagger)}\in (y_1-y_0)+ \left(h_{\delta, \beta}^{\eta, 1}(z'_{n, 1}) -h_{\delta, \beta}^{\eta, 0}(z'_{n, 0})\right) + \varphi_n^{\mathfrak e}(\modM_{\delta,\beta}))\\
& \  \ -\pi (y_1-y_0) -\pi \left(h_{\delta, \beta}^{\eta, 1}(z'_{n, 1}) -h_{\delta, \beta}^{\eta, 0}(z'_{n, 0})\right) + \varphi_n^{\mathfrak e}(\modM_{\delta,\beta}))
\\
&\stackrel{(13)}=y+  \left(h_{\delta, \beta}^{\eta, 1}(z'_{n, 1}) -h_{\delta, \beta}^{\eta, 0}(z'_{n, 0})\right) -\pi \left(h_{\delta, \beta}^{\eta, 1}(z'_{n, 1}) -h_{\delta, \beta}^{\eta, 0}(z'_{n, 0})\right) + \varphi_n^{\mathfrak e}(\modM_{\delta,\beta})\\
&\stackrel{(8+9)}=y+  \sum\limits_{\ell<k_n}(h_{t_\ell}+{\bf g}_\ell)(z_{n, 1, \ell}) - \sum\limits_{\ell<k_n}h_{t_\ell}(z_{n, 0, \ell})
 \\
& \  \  -\pi \sum\limits_{\ell<k_n}(h_{t_\ell}+{\bf g}_\ell)(z_{n, 1, \ell}) +\pi \sum\limits_{\ell<k_n}h_{t_\ell}(z_{n, 0, \ell}) +\varphi_n^{\mathfrak e}(\modM_{\delta,\beta})\\
&\stackrel{(11+12)}=y+ \sum\limits_{\ell<k_n}(h_{t_\ell}+{\bf g}_\ell)(z_{n, 1, \ell}) - \sum\limits_{\ell<k_n}h_{t_\ell}(z_{n, 0, \ell})
 \\
& \  \  - \sum\limits_{\ell<k_n}(h_{t_\ell})(z_{n, 1, \ell}) +\sum\limits_{\ell<k_n}h_{t_\ell}(z_{n, 0, \ell}) +\varphi_n^{\mathfrak e}(\modM_{\delta,\beta})\\
&=y+  \sum\limits_{\ell<k_n}{\bf g}_\ell(z_{n, 1, \ell})+ \varphi_n^{\mathfrak e}(\modM_{\delta,\beta}).
\end{array}\]
This completes the proof of  Definition \ref{nice construction modified}(3)$(G)_1$(b).

We now sketch the proof of property $(G)_2$ of Definition \ref{nice construction modified}(3), as its details are similar to the  above.  Let  ${\mathfrak e}\in \calE^\frakss$ and
for $n<\omega$ and $\alpha\in \kappa\smallsetminus S$.
By ${\bf h}_{\alpha,n}$
we mean an embedding of $\modN^{\mathfrak e}_n$ into
$\modM_\kappa$
\st\ $\modM_\alpha+\Rang ({\bf h}_{\alpha,n})
=\modM_\alpha \oplus \Rang
({\bf h}_{\alpha,n})\leq_{\aleph_0} \modM_\kappa$.

We look at the following set
\[
C:=\{\delta < \kappa: \forall \alpha < \delta~ \forall n<\omega, \Rang({\bf h}_{\alpha, n}) \subseteq \modM_\delta            \},
\]
which $C$ is  a club subset of $\kappa$.\footnote{Club means closed and unbounded.}
Let
\[
{\bf g}:  \cal T_\kappa \to \cal T_\kappa
\]
be such that for all $\eta, \mu \in \lim({\bf g})$ and all $n<\omega, {\bf g}(\eta(n))={\bf g}(\mu(n))$.
Let
\[
\mathcal{B} = \langle B, \in^{\cal B},  F^{\cal B}, G^{\cal B}, (c_n^{\cal B})_{n<\omega}, (P_n^{\cal B})_{n < \omega}          \rangle
\]
be a $\tau = \bigcup\limits_{n<\omega}\tau_n$-structure satisfying the following properties:
\begin{itemize}
  \item $(\cal B, \in^{\cal B})$ expands $(\cal H_{<\aleph_1}(\modM), \in)$,
 \item $G^{\cal B}={\bf g},$
  \item for $n<\omega, ~c_n^{\cal B}=\gamma_n$, where $\gamma_n$ is such that $(n, \mathfrak e, {\bf h}_{\alpha_n, n}-h_{t_n})=(n_{\gamma_n}, \mathfrak e_{\gamma_n}, f_{\gamma_n}),$ where $t_n={\bf g}(\eta(n))$ and $\alpha_n$ is such that $t_n \in J_{\alpha_n}$.
\end{itemize}
In the same vein, we can find $\delta \in S \cap C$, $\beta\in J_{\delta }$ such that
\begin{enumerate}
  \item $\zeta(\beta)=\gamma_\delta,$
  \item $ \gamma^*_\delta=\gamma_\delta,$
  \item $\cal N^\beta \prec \cal B,$ in particular,
  \begin{enumerate}
    \item $f^\beta = {\bf f} \restriction \cal N^\beta,$
    \item for all $n<\omega, c_n^{\cal N^\beta} = c_n^{\cal B}$,
    \item for $n<\omega, {\bf h}_{\alpha_n, n}-h_{t_n}=f_{c_n^{\mathcal N^\beta}}$, where $t_n$ and $\alpha_n$ are defined as above.
  \end{enumerate}
      \item The properties presented in Clause (F) of Definition \ref{w.s. nice definition} are hold.
\end{enumerate}
According to Definition \ref{w.s. nice definition}(F)(k), we have
\[
{\bf h_{\beta}^{\delta}} \circ {\bf h}_n=h_{t_n} + ({\bf h}_{\alpha_n, n}-h_{t_n})={\bf h}_{\alpha_n, n},
\]
as claimed by  $(G)_2$ from Definition \ref{nice construction modified}(3).

In order to prove clause (2) of the theorem,
suppose further that $\lambda$ is a regular cardinal. We show that by shrinking $S$, if necessary, we can assume that the above constructed structure is indeed a nice construction, i.e., we can omit ``semi'' from it. The map $\delta \mapsto \gamma^{**}_\delta$ is a regressive function on $S$. We are going to use the Fodor's lemma. This says that $\delta \mapsto \gamma^{**}_\delta$  is constant on some stationary subset $S_1$ of $S$, thus for some $\gamma^{**}$ and for all $\delta \in S_1, \gamma^{**}_\delta=\gamma^{**}$. Now note that for each $\delta \in S_1$ and each
$s \in J_\delta, \mathfrak e^\delta_s$ is of the form $\mathfrak e_\beta$ for some $\beta < \gamma^{**}$,
so again by Odor's lemma
and for each $s \in J_\delta, \bar g^\delta_s$ is an $\omega$-sequence $\langle g^\delta_{s, n}: n<\omega   \rangle$
where for each $n, g^{\delta}_{s, n}: \modN_n^{\mathfrak{e}_s} \to \modM_{\gamma^{**}}$ is a bimodule homomorphism.
But, the set
\[
\{g: g: \modN_n^{\mathfrak{e}_s} \to \modM_{\gamma^{**}} \text{~is a bimodule homomorphism}\}
\]
has size less than $\lambda$.
Now given an $\delta \in S_1$ for which clause (F) of Definition \ref{w.s. nice definition} holds, the sets
\[
\{ \mathfrak e_s: s \in J_\delta \} ~~\&~~ \{\bar g^\delta_s: s \in J_\delta \}
\]
have size less than $\lambda$. By another application of Fodor's lemma, we can find a stationary subset $S_2$
of $S_1$ such that the sets $\{ \mathfrak e_s: s \in J_\delta \}$ and  $\{\bar g^\delta_s: s \in J_\delta \}
$ are the same for all $\delta\in S_2.$
This concludes the required result.
\end{proof}

\section{Any endomorphism is somewhat definable}

\label{Any endomorphism is somewhat definable}
In this section we shall investigate what ``$\bar \modM=\langle \modM_\alpha: \alpha \leq \kappa \rangle$ is a semi-nice construction''
gives us and look somewhat at various implications.
\begin{Definition}
Given $\alpha < \kappa$,
  $n<\omega$ and $\fucF:\modM_\kappa\to \modM_\kappa  $ a bimodule homomorphism  the notation
$(\Pr)^{n}_\alpha{[}\fucF,{\mathfrak e}{]}$
 stands for  the following  statement: if ${\bf h}:\modN^{\mathfrak e}_{n}\to\modM_\kappa$ is a bimodule homomorphism,
then
$$
{\bf f}({\bf h}(x_n^{\mathfrak e}))\in \modM_\alpha+\varphi^{\mathfrak e}_\ell(\modM_\kappa)+\Rang({\bf h}),
$$ for all $\ell<\omega$.
\end{Definition}

As an application of the results of Section 3, we present the following statement. This plays an essential role in proving some stronger versions of $(\Pr)^{n}_\alpha{[}\fucF,{\mathfrak e}{]}$, see blow.
\begin{lemma}
\label{prnalphafe}
Let $ {\mathscr A}$
be a semi-nice construction for $(\lambda,{\frakss}, S
,\kappa)$ and suppose ${\bf f}:\modM_\kappa\to
\modM_\kappa$ is  an endomorphism of
$\ringR$-modules  and suppose $\kappa\geq\kappa (\calE)$.
Then for every ${\mathfrak e}\in \callE$
there are $\alpha \in \kappa\setminus S$ and
$n(\ast)<\omega$ such that
$(\Pr)^{n(\ast)}_\alpha{[}\fucF,{\mathfrak e}{]}$ holds.
\end{lemma}
\begin{proof}
Assume not. Then for every $\epsilon \in \kappa\setminus S$ and $n<\omega$
we can find a bimodule homomorphism ${\bf h}_{\epsilon,n}: \modN^{\mathfrak e}_n \to \modM_{\kappa}$ such that for some $\ell(\epsilon, n)<\omega$,
$${\bf f} ({\bf h}_{\epsilon,n}( x^{\mathfrak e}_
n))  \notin \modM_\epsilon +\varphi^{\mathfrak e}_{\ell(\epsilon, n)}
(\modM_\kappa)
+\Rang({\bf h}_{\epsilon,n}).
$$
Without loss of generality, we may assume that
 $n<\ell(\epsilon, n).$ The following set
 \[
 E:=\{\zeta < \kappa: \Rang({\bf f} \upharpoonright \modM_\zeta )\subseteq \modM_\zeta \text{~and~}\forall \xi < \zeta ~\forall n<\omega, \Rang({\bf h}_{\xi,n}) \subseteq \modM_\zeta  \}
 \]
 is   a club subset of $\kappa$. Since $S$
is stationary, we have
\[
S \cap \lim \left(E \cap (S^\kappa_{\aleph_0}\setminus S)\right)  \neq \emptyset,
\]
where $\lim$ denote the set of limit points. Let $\zeta$ be in this intersection.
 Then $\cf(\zeta)=\aleph_0$ and there exists an increasing sequence $\langle  \zeta_n: 0<n<\omega   \rangle$ of elements
 of $E \cap (S^\kappa_{\aleph_0}\setminus S)$ cofinal in $\zeta$.

For each $0<\ell<\omega,$ as $\zeta_\ell \notin S,$ we can find a bimodule $\modK_\ell$ such that $\modM_{\zeta_{\ell+1}}=\modM_{\zeta_\ell}\oplus \modK_\ell$. Set also
$\modK_0=\modM_{\zeta_0}$. Thus, $\modM_\zeta=\bigcup\limits_{\ell<\omega}\modM_{\zeta_\ell}=\bigoplus\limits_{\ell<\omega}\modK_\ell.$
Let ${\bf g}^*_\ell: \modM_{\zeta_{\ell+1}} \to \modK_\ell$ be the natural projection, this is well-defined, because  $\modK_\ell$ is a direct summand of $\modM_{\zeta_{\ell+1}}$.
Pick also an infinite subset $\mathcal{U} \subseteq \omega$ such that
\[
\forall n, m \in \mathcal{U} \left( n<m \implies \ell(\zeta_n, n) <m \right).
\]

 For each $n<\omega$, we set $t_n := {\bf h}_{\zeta_n, n}(x^{\mathfrak e}_n) \in \modM_\zeta$. By Definition \ref{nice construction modified}(3)$(G)_1$,
 we can find $y \in \modM_\kappa,~z \in \modP^{\mathfrak e}$
 and a sequence $\{ z_{n, \ell}, z'_{n, \ell}: n \in \mathcal{U}, \ell<\omega  \}$ with $z_{n, \ell} \in \dom ({\bf h}_{\zeta_\ell, \ell})$ such that for each $n \in \mathcal{U}$,  $z_{n, \ell}=0$ for all  large enough $\ell$, say for all $\ell \geq k_n$
and for all large enough $n \in \mathcal{U}$ we have
\begin{enumerate}
\item[(a)]
$z \in \sum\limits_{\ell \in \mathcal{U} \cap k_n}{\bf h}_{\ell}^{\mathfrak e}(z'_{n,\ell})+\varphi_n^{\mathfrak e}(\modP^{\mathfrak e}).$

\item[(b)]
$\sum\limits_{\ell \in \mathcal{U} \cap n}{\bf f}({\bf h}_{\zeta_\ell, \ell}(x_\ell^{\mathfrak e})) \in y +\sum\limits_{\ell \in \mathcal{U} \cap k_n} {\bf h}_{\zeta_\ell, \ell}(z_{n, \ell}) +\varphi_n^{\mathfrak e}(\modM_{\kappa}).$

\end{enumerate}
As $\zeta \notin S, \modM_\zeta \leq_{\aleph_0} \modM_\kappa$ and hence for some $\modK$ we have $y \in \modM_\zeta\oplus\modK$
and $\modM_\zeta\oplus \modK \leq_{\aleph_0} \modM_\kappa$.
So, in clause $(b)$ all elements are in  $\modM_{\zeta_n}\oplus \modK$. This allow us to replace
$\varphi_n^{\mathfrak e}(\modM_{\kappa})$ with $\varphi_n^{\mathfrak e}(\modM_{\zeta_n}\oplus \modK)$
to get the following  new clause:
\begin{enumerate}
\item[$(b')_n$]:
	$\sum\limits_{\ell \in \mathcal{U} \cap n}{\bf f}({\bf h}_{\zeta_\ell, \ell}(x_\ell^{\mathfrak e})) \in y +\sum\limits_{\ell \in \mathcal{U} \cap k_n} {\bf h}_{\zeta_\ell, \ell}(z_{n, \ell}) +\varphi_n^{\mathfrak e}(\modM_{\zeta_n}\oplus \modK).$
	\end{enumerate}

Let $\pi:\modM_\zeta\oplus \modK\to\modM_\zeta$ be the natural projection and let $y':=\pi(y)$.  Recall that
$\pi(\varphi_n^{\mathfrak e}(\modM_{\zeta_n}\oplus \modK))=\varphi_n^{\mathfrak e}(\modM_{\zeta_n})$. By applying $\pi$
along with $(b')_n$  we lead to the following equation:
\begin{enumerate}
	\item[$(b'')_n$]:
	$\sum\limits_{\ell \in \mathcal{U} \cap n}{\bf f}({\bf h}_{\zeta_\ell, \ell}(x_\ell^{\mathfrak e})) \in y' +\sum\limits_{\ell \in \mathcal{U} \cap k_n} {\bf h}_{\zeta_\ell, \ell}(z_{n, \ell}) +\varphi_n^{\mathfrak e}(\modM_{\zeta_n}).$
\end{enumerate}
Pick $n(*) <\omega$ such that $y' \in \modM_{\zeta_{n(*)}}$ and choose successive elements $n < m$ in $\mathcal{U}$ bigger than $n(*)$ which are large enough. Pick $w \in \varphi_m^{\mathfrak e}(\modM_{\zeta_m})$  such that
\begin{enumerate}
	\item $\sum\limits_{\ell \in \mathcal{U} \cap m}{\bf f}({\bf h}_{\zeta_\ell, \ell}(x_\ell^{\mathfrak e})) = y' +\sum\limits_{\ell \in \mathcal{U} \cap k_m} {\bf h}_{\zeta_\ell, \ell}(z_{m, \ell}) +w.$
\end{enumerate}

 Then we have:
\begin{enumerate}

\item[(2)] As $\sum\limits_{\ell \in \mathcal{U} \cap n}{\bf f}({\bf h}_{\zeta_\ell, \ell}(x_\ell^{\mathfrak e})), y' \in \modM_{\zeta_n}$, by (1) we have
\[\begin{array}{ll}
{\bf g}^*_n({\bf h}_{\zeta_n, n}(x_n^{\mathfrak e})) &={\bf g}^*_n\left( \sum\limits_{\ell \in \mathcal{U} \cap m}{\bf f}({\bf h}_{\zeta_\ell, \ell}(x_\ell^{\mathfrak e}))         \right)\\
&= \sum\limits_{\ell \in \mathcal{U} \cap k_m} {\bf g}^*_n\left({\bf h}_{\zeta_\ell, \ell}(z_{m, \ell})\right)+{\bf g}^*_n(w).
\end{array}\]
\end{enumerate}

In particular, we drive the following:
\begin{enumerate}
\item[(3)] $\sum\limits_{\ell \in \mathcal{U} \cap k_m} {\bf g}^*_n\left({\bf h}_{\zeta_\ell, \ell}(z_{m, \ell})\right) \in \Rang({\bf g}^*_n \circ {\bf h}_{\zeta_n, n}  )+\varphi_m^{\mathfrak e}(\modM_{\zeta_m}).$
\end{enumerate}
Due to the definition of ${\bf g}^*_n$ we know:
\begin{enumerate}
\item[(4)]
${\bf f}({\bf h}_{\zeta_n, n}(x_n^{\mathfrak e}))-{\bf g}^*_n\left({\bf f}({\bf h}_{\zeta_n, n}(x_n^{\mathfrak e}))\right) \in \modM_{\zeta_n}.$
\end{enumerate}

Putting all things together we have
\[\begin{array}{ll}
{\bf f}({\bf h}_{\zeta_n, n}(x_n^{\mathfrak e})) &\in {\bf g}_n^*\left({\bf f}({\bf h}_{\zeta_n, n}(x_n^{\mathfrak e}))\right)+\modM_{\zeta_n}\\
&\subseteq \sum\limits_{\ell \in \mathcal{U} \cap k_m} {\bf g}_n^*\left({\bf h}_{\zeta_\ell, \ell}(z_{m, \ell})\right)+\modM_{\zeta_n} +\varphi_m^{\mathfrak e}(\modM_{\zeta_m})\\
& \subseteq \modM_{\zeta_n}+\Rang({\bf g}^*_n \circ {\bf h}_{\zeta_n, n})+\varphi_m^{\mathfrak e}(\modM_{\zeta_m})\\
&=\modM_{\zeta_n}+\Rang( {\bf h}_{\zeta_n, n})+\varphi_m^{\mathfrak e}(\modM_{\zeta_m}).
\end{array}\]
This is a contradiction and the lemma follows.
 \end{proof}

\begin{Definition}
\label{hdsprinciple}
  By $\text{\bf hds}^{\modM_2}_{\modM_1}(h,\modN)$ we mean the following data:
\begin{enumerate}
\item $\modM_1, \modM_2,\modN$ are bimodules,
\item $\modM_1
\subseteq \modM_2$,
\item $h$ is a bimodule homomorphism from $\modN$ into
$\modM_2$,
\item $\modN\in
c\ell_{is}({\cal K})$,
\item $\modM_1+\Rang(h)=\modM_1\oplus\Rang(h)$,
\item $\modM_1+\Rang(h)\leq_{\aleph_0} \modM_2$.
\end{enumerate}
\end{Definition}

\begin{Definition} Let  $\fucF:\modM_\kappa\to \modM_\kappa  $ be an $\ringR$-module homomorphism.  By  $(\Pr^-)^{n}_\alpha[\fucF,{\mathfrak e}]$ we mean  the following statement:
if $\text{\bf hds}^{\modM_\beta}_{\modM_\alpha}
(h,\modN^{\mathfrak e}_{n})$
where
$\alpha<\beta \in \kappa \setminus S$, then $$\fucF(h(x_{n}^{\mathfrak e}))\in
\modM_\alpha+\Rang(h)+
\varphi^{\mathfrak e}_\ell(\modM_\kappa),$$
 for all $\ell<\omega$.
\end{Definition}
The next lemma summarizes the main properties of the above defined statement.
\begin{lemma}
\label{hdsproperties}
\begin{enumerate}
\item If $\text{\bf hds}^{\modM_2}_{\modM_1}(h_1,\modN)$ holds, $h_0$ is a
bimodule homomorphism
from $\modN$ into $\modM_1$ and $h=h_0+h_1$, then
$\text{\bf hds}^{\modM_2}_{\modM_1}(h,\modN)$ holds as well.

\item If $\modM_0\subseteq \modM_1\subseteq \modM_2$ are
bimodules,  $\modM_0\leq_{\aleph_0}\modM_1$
and $\text{\bf hds}^{\modM_2}_{\modM_1}(h,
\modN)$ holds, then so does  $\text{\bf hds}^{\modM_2}_{\modM_0}(h,\modN)$.
\item Assume $\alpha\leq\beta< \kappa$, $\alpha\notin S$ and
$\beta\notin S$.
Then $(\Pr^-)^{n(\ast)}_\alpha[\fucF,{\mathfrak e}]$ implies
$(\Pr^-)^{n(\ast)}_\beta[\fucF,{\mathfrak e}]$.
\item If $(\Pr)^{n(\ast)}_\alpha[{\bf f},{\mathfrak e}]$  holds, then
$(\Pr^-)^{n(\ast)}_{
\alpha}[{\bf f},{\mathfrak e}]$ holds too.
\item Assume $\bar{\modM}=\langle
\modM_\alpha:\alpha\leq\kappa\rangle$ is a weakly semi-nice
construction for $(\lambda,\frakss, S, \kappa)$.  If
$\text{\bf hds}^{\modM_\lambda}_{ \modM_\alpha}(h,\modN)$ holds where
$\alpha<\kappa$, then  for a club of
$\beta\in (\alpha,\kappa)$, also
$\text{\bf hds}^{\modM_\beta}_{\modM_\alpha}(h,\modN)$ holds.
\item If $\text{\bf hds}^{\modM_2}_{\modM_1}(h,\modN)$ holds, then so does
$\text{\bf hds}^{\modM_1+\Rang(h)}_{\modM_1}(h,\modN)$.
\end{enumerate}
\end{lemma}
\begin{proof}
\begin{enumerate}
\item We only need to show that $\modM_1 + \Rang(h)= \modM_1 \oplus \Rang(h) \leq_{\aleph_0} \modM_2$. But, this is clear as
  $\modM_1 + \Rang(h)=\modM_1 + \Rang(h_1)$ and the property $\text{\bf hds}^{\modM_2}_{\modM_1}(h_1,\modN)$ holds.

\item We have
\[
\modM_0 + \Rang(h)=\modM_0 \oplus \Rang(h) \leq_{\aleph_0} \modM_1 \oplus \Rang(h) \leq_{\aleph_0} \modM_2,
\]
  from which the result follows.

\item  This is a combination of  Lemma \ref{properties of w.s.n.c}(1) and clause (2) above.

\item Clear.

\item By our assumption, $h: \modN \to \modM_\kappa$ is a bimodule homomorphism and $\modM_\alpha + \Rang(h)=\modM_\alpha \oplus \Rang(h) \leq_{\aleph_0} \modM_\kappa.$ But then for a club $C \subseteq (\alpha, \kappa)$ and for all $\beta \in C, \modM_\alpha \oplus \Rang(h) \leq_{\aleph_0} \modM_\beta.$ Then $C$ is as required.

\item Clear from the definition.
  \end{enumerate}
\end{proof}
\begin{Definition}
\label{prnalphafezdefinition}Let  $\fucF:\modM_\kappa\to \modM_\kappa  $   be a bimodule homomorphism,
$\alpha < \kappa,~n < \omega$ and let $z \in \modN_n^{\mathfrak e}$. The notation $(\Pr 1)^{n}_{\alpha,z}{[{\bf f},{\mathfrak e}]}$ stands for the following statement: if $h$ is a bimodule homomorphism
from $\modN^{\mathfrak e}_{n}$ into $\modM_\kappa$, then
$${\bf f}(h(x^{\mathfrak e}_{n}))-h(z) \in \modM_\alpha+\varphi^{\mathfrak
e}_\omega(\modM_\kappa).$$

\end{Definition}
\begin{lemma}
\label{prnalphafez}
Let $\langle \modM_\alpha:\alpha\leq\kappa\rangle$
be a weakly
semi-nice
construction for $(\lambda, {\frakss}, S,\kappa)$ and
${\bf f}:\modM_\kappa\to \modM_\kappa$ be an $\ringR$-endomorphism  and let  $\alpha\notin S$  be such that $(\Pr^-)^{n(\ast)}_\alpha[{\bf f},{\mathfrak e}]$ holds. Then
 the property $(\Pr 1)^{n(\ast)}_{\alpha,z}{[{\bf f},{\mathfrak e}]}$ holds for some
 $z\in \modL^{\mathfrak e}_{n(\ast)}[{\cal K}^\frakss]$.
\end{lemma}
\begin{proof} Let $x:=x^{\mathfrak e}_{n(*)}$.
We shall prove the lemma in a sequence of  claims, see Claims \ref{prnclaim1}--\ref{prnclaim4}.
\begin{claim}
\label{prnclaim1}
 Let $\beta\in
(\alpha,\lambda)\setminus S$ be such that
$\text{\bf{hds}}^{\modM_\beta}_{\modM_\alpha}(h,
\modN^{\mathfrak e}_{n(\ast)})$ holds.  Then
\[\fucF(h(x))\in \modM_\alpha+\Rang(h)+\varphi^{\mathfrak
e}_\omega(\modM_\kappa).\]
\end{claim}
\begin{proof}
Let $\modN:=\Rang(h)$. We know that
$\modM_\alpha+\modN=\modM_\alpha\oplus \modN
\leq_{\aleph_0} \modM_\beta$ and
$\modM_\beta\leq_{\aleph_0}\modM_\kappa$. Combining these yields that
 $\modM_\alpha\oplus \modN\leq_{\aleph_0} \modM_\kappa$.
In view of Definition \ref{ads definition} and Lemma \ref{modm is AEC}(6), there is $\modK\subseteq
\modM_\kappa$
such that $$(\modM_\alpha\oplus
\modN)+\modK=
\modM_\alpha
\oplus \modN \oplus \modK \leq_{\aleph_0}
\modM_\kappa.$$
Then, we may assume that $\fucF(h(x))\in\modM_\alpha\oplus \modN\oplus \modK$. Since
$\modM_\alpha\oplus \modN\oplus
\modK\leq_{\aleph_0}\modM_\kappa$, and in the light of Lemma \ref{comparing two orders}
we observe that $\modM_\alpha\oplus \modN
\oplus \modK\leq^{pr}_{\varphi^{\mathfrak e}_n}\modM_\kappa$ holds for all $n<\omega$. According to
the property $(\Pr^-)^{n(*)}_\alpha[{\bf f},{\mathfrak e}]$, we know $${\bf f}(h(x))\in \modM_\alpha+\modN+\varphi^{\mathfrak
	e}_n(\modM_\kappa),$$
for all $n<\omega.$ We use these to find elements $x_{1,n}\in \modM_\alpha$ and $
x_{2,n}\in \modN$ such that
$$\fucF(h(x))-x_{1,n}-x_{2,n}\in
\varphi^{\mathfrak e}_n(\modM_\kappa).$$ It follows from
$\modM_\alpha\oplus \modN
\oplus \modK\leq^{pr}_{\varphi^{\mathfrak e}_n}\modM_\kappa$ that
\[\begin{array}{ll}
{\bf f}(h(x))& =x_{1,n}+x_{2,n}+(\fucF(h(x))-x_{1,n}-x_{2,n})\\
&\in\modM_\alpha+\modN+\varphi^{\mathfrak e}_n(\modM_\alpha
\oplus\modN\oplus \modK).
\end{array}\]

Let $g$ be the projection from $\modM_\alpha\oplus \modN\oplus
\modK$ onto
$\modK$. Recall that its kernel is $\modM_\alpha\oplus \modN$.
Let $z^1_n\in \modM_\alpha$,
$z^2_n\in \modN$ and
$z^3_n\in
\varphi_n(\modM_\alpha \oplus \modN \oplus \modK)$ be such that
${\bf f}(h(x))=z^1_n+z^2_n+z^3_n$. Let us evaluate $g$ on both sides of this formula to obtain:
$$
g({\bf f}(h(x)))=g(z^1_n) + g(z^2_n) +g(z^3_n) =z^3_n \in \varphi^{\mathfrak e}_n(\modM_\alpha\oplus \modN\oplus
\modK)\subseteq
\varphi^{\mathfrak e}_n(\modM_\kappa).
$$

As this holds for each $n$, we have
$$
g({\bf f}(h(x))) \in\bigcap\limits_{n<\omega} \varphi^{\mathfrak e}_n(\modM_\kappa).
$$
So,
\[\begin{array}{ll}
{\bf f}(h(x))& =
\left({\bf f}(h(x))-g({\bf f}(h(x)))\right)+(g({\bf f}(h(x)))\\&\in (\modM_\alpha\oplus
\modN)+\bigcap_{n<\omega}\varphi_n (\modM_\kappa)\\
&=\quad \modM_\alpha+\Rang(h)+\varphi^{\mathfrak e}_\omega(\modM_\kappa).
  \end{array}\]
  The claim follows.
  \end{proof}
\begin{claim}
\label{prnclaim2}
For $i=1,2$, let $\modN^\ast_i$ and $\beta_i$ be such that $\modM_\alpha\oplus
\modN^\ast_i\leq_{\aleph_0} \modM_{\beta_i}$ and  $\alpha<
\beta_i<\kappa$ where $\beta_i$
is not in $S$. Let $h_i:\modN^{\mathfrak e}_{n(\ast)}\to\modN^\ast_i$ be an
isomorphism and $z_i\in
\modN^{\mathfrak e}_{n(\ast)}$ be such that
$
{\bf f}(h_i(x))-h_i(z_i)\in \modM_\alpha+\varphi^{\mathfrak e}_\omega
(\modM_\kappa).
$
Then
$
z_1\equiv z_2\ \mod\ \varphi^{\mathfrak e}_\omega(\modN^{\mathfrak e}_{n(\ast)})
$.
\end{claim}
\begin{proof}
Let  $\beta\in \kappa \setminus S$ be large enough \st\
$\beta>\beta_1$ and
$\beta>\beta_2$.
Let $\modN^\ast_3$ be isomorphic to $\modN^{\mathfrak e}_{n(\ast)}$ and
such that for any large enough $\gamma<\kappa$,
$$
\modM_\beta+\modN^\ast_3=\modM_\beta\oplus
\modN^\ast_3\leq_{\aleph_0}
\modM_\gamma.
$$
Let $h_3$ be an isomorphism from $\modN^{\mathfrak e}_{n(\ast)}$ onto
$\modN^\ast_3$. In the light of Claim \ref{prnclaim1} we deduce that
$$
{\bf f}(h_3(x))-h_3(z_3)\in \modM_\alpha+\varphi^{\mathfrak e}_\omega(\modM_\kappa),
$$where
$z_3\in \modN^{\mathfrak e}_{n(\ast)}$.
By the transitivity of the equivalence relation,  it is enough to show that $z_3\equiv z_1$ and $z_3\equiv z_2\ \mod\
\varphi^{\mathfrak e}_\omega (\modN^{\mathfrak e}_{n(\ast)})$
in $\modN^{\mathfrak
e}_{n(\ast)}$; and by the
symmetry it is enough to prove $z_3\equiv z_1$ $\mod\
\varphi^{\mathfrak e}_\omega (\modN^{\mathfrak e}_{n(\ast)})$.

Pick $\gamma > \beta$ large enough. Then
$
\modM_\alpha\oplus \modN^\ast_1\oplus \modN^\ast_3\leq_{\aleph_0}
\modM_\gamma.
$
Let $\modN^\ast_4=\{h_1(z)-h_3(z): z\in \modN^{\mathfrak e}_{n(\ast)}\}$ and
define $h_4:
\modN^{\mathfrak e}_{n(\ast)}\longrightarrow \modM_{\beta+1}$ by
$$
h_4(z)=h_1(z)-h_3(z).
$$
Due to its definition, we know $\modN^\ast_4$ is a sub-bimodule of $\modM_\alpha \oplus
\modN^*_1 \oplus \modN^*_3$
and hence of $\modM_\kappa$. Also,
$h_4$ is an isomorphism from $\modN^{\mathfrak e}_{n(\ast)}$ onto
$\modN^\ast_4$, and $\modM_\alpha
\oplus \modN^\ast_1\oplus \modN^\ast_3=\modM_\alpha\oplus
\modN^\ast_1\oplus \modN^\ast_4$.
Thus, we have
\begin{enumerate}
\item[$(\ast)_1$]\quad $\modM_\alpha\oplus \modN^\ast_4\leq_{\aleph_0}
\modM_\alpha\oplus \modN^\ast_1\oplus
\modN^\ast_3\leq_{\aleph_0}\modM_\gamma$.
\end{enumerate}
Now modulo $\modM_\alpha+\varphi^{\mathfrak e}_\omega(\modM_\kappa)$, we have
\begin{enumerate}
\item[$(\ast)_2$]\quad ${\bf f}(h_4(x))={\bf f}(h_1(x)-h_3(x))={\bf f}(h_1(x))-
{\bf f}(h_3(x))\equiv h_1(z_1)-h_3(z_3)$.
\end{enumerate}
Next, by Claim \ref{prnclaim1} and $(\ast)_1$, we have
\begin{enumerate}
\item[$(\ast)_3$]\quad ${\bf f}(h_4(x))\in\Rang(h_4)+(\modM_\alpha+
\varphi^{\mathfrak e}_\omega(\modM_\kappa))$.
\end{enumerate}
So
\begin{enumerate}
\item[$(\ast)_4$]\quad $h_1(z_1)-h_3(z_3)
\in\Rang(h_4)+(\modM_\alpha+\varphi^{\mathfrak
e}_\omega (\modM_\kappa))$.
\end{enumerate}
We combine $(\ast)_4$ along with the definition of $h_4$ to deduce that
$$(h_1(z_1)-h_3(z_3)- h_4(z)\in \modM_\alpha+
\varphi^{\mathfrak e}_\omega (\modM_\kappa),$$
for some $z\in \modN^{\mathfrak
	e}_{n(\ast)}$. Recall that $h_4(z)= h_1(z)-h_3(z)$. It yields that
\[(h_1(z_1)-h_3(z_3))-(h_1(z)-h_3(z))\in \modM_\alpha+\varphi^{\mathfrak e}_\omega
(\modM_\kappa).\]
In other words,
\[h_1(z_1-z) - h_3(z_3-z)\in \modM_\alpha+\varphi^{\mathfrak
e}_\omega(\modM_\kappa).\]
Therefor, there is $y \in \modM_\alpha$ such that $h_1(z_1-z) - h_3(z_3-z)-y\in
\varphi^{\mathfrak e}_\omega(\modM_\kappa)$. As $\modM_\alpha\oplus
\modN^\ast_1\oplus
\modN^\ast_3 \leq_{\aleph_0}\modM_\kappa$, we deduce that
\[h_1(z_1-z) - h_3(z_3-z)-y\in\varphi^{\mathfrak e}_\omega(\modM_\alpha\oplus
\modN^\ast_1
\oplus \modN^\ast_3).\]

Recall that $h_1(z_1-z)\in \modN_1,h_3(z_3-z)\in \modN^\ast_3$ and  $y\in
\modM_\alpha$. We conclude from Claim
\ref{pr properties}(5) that
 $$\varphi^{\mathfrak e}_\omega(\modM_\alpha)\oplus\varphi^{\mathfrak e}_\omega(
\modN^\ast_1)
\oplus \varphi^{\mathfrak e}_\omega(\modN^\ast_3)\cong\varphi^{\mathfrak e}_\omega(\modM_\alpha\oplus
\modN^\ast_1
\oplus \modN^\ast_3).$$
Hence$$
h_i(z_i-z)\in h_i'' \left(\varphi^{\mathfrak e}_\omega
(\modN^{\mathfrak e}_{n(\ast)})
\right) \quad \mbox{ for } i=1,3,
$$i.e., $z_i-z\in\varphi^{\mathfrak e}_\omega(
\modN^{\mathfrak e}_{n(\ast)}).$
It then follows that
\[z_1-z_3=(z_1-z)-(z_3-z)\in\varphi^{\mathfrak e}_\omega
(\modN^{\mathfrak e}_{n(\ast)}).\]
So $z_1-z_3\in \varphi^{\mathfrak e}_\omega(\modN^{\mathfrak e}_{n(\ast)})$,
which finishes the proof of the claim.
\end{proof}
\par \noindent
\begin{claim}
\label{prnclaim3}
There is $z\in \modN^{\mathfrak e}_{n(\ast)}$ such that if $h:\modN^{\mathfrak e}_{n(\ast)}\to\modM_\kappa$ is a  bimodule homomorphism, then
$
{\bf f}(h(z))-h(z)\in \modM_\alpha+\varphi^{\mathfrak e}_\omega
(\modM^{\mathfrak e}_\kappa).$
\end{claim}
\begin{proof}
Let $\modN_0^\ast$, $\beta_0$ and $h: \modN^{\mathfrak e}_{n(\ast)} \to \modN_0^\ast$ be as  Claim \ref{prnclaim2}. In the light of Claim  \ref{prnclaim1}, there is $z\in \modN^{\mathfrak e}_{n(\ast)}$ which
satisfies the above
requirement for this $h$. We show that $z$ is as required.
Suppose not and let $h_0$ be a counterexample, i.e.,
\[
{\bf f}(h_0(z))-h_0(z)\notin \modM_\alpha+\varphi^{\mathfrak e}_\omega
(\modM^{\mathfrak e}_\kappa).
\]
 Choose $\beta$
such that
$\beta_0 < \beta \in \kappa \setminus  S$ and $\Rang(h_0)\subseteq
\modM_\beta$. Such a $\beta$ exists as  $\kappa=\cf(\kappa)\geq \kappa(\calE)$.
Let $h_1$ be an isomorphism from
 $\modN^{\mathfrak e}_{n(\ast)}$ onto some
$\modN^\ast_1$ such that $\modM_\beta\oplus
\modN^\ast_1\leq_{\aleph_0}\modM_\gamma$ for
some $\gamma\in(\beta,\kappa)\setminus S$. So
$$
{\bf f}(h_1(z))-h_1(z)\in \modM_\alpha+\varphi^{\mathfrak e}_\omega(\modM_\kappa).
$$
Let $h_2: \modN^{\mathfrak e}_{n(\ast)}\longrightarrow
\modM_\kappa$ be defined by
$$
h_2(z)=h_1(z)-h_0(z).
 $$
Easily $h_2$ is a bimodule homomorphism. Set   $\modN^\ast_2=:\Rang(h_2)$.  By the assumptions on
$\modN^\ast_1$ and $h_1$, we deduce that
$$
\modM_\beta\oplus \modN^\ast_1=\modM_\beta\oplus
\modN^\ast_2\leq_{\aleph_0} \modM_\kappa.
$$
In the light of Claim \ref{prnclaim2}, we see that
$$
\fucF(h_2(x))-h_2(z) \in \modM_\alpha+\varphi^{\mathfrak e}_\omega
(\modM_\kappa).
$$
We apply this to conclude that
$$
\begin{array}{ll}
\fucF(h_0(z))& =
\fucF(h_1(z)-h_2(z))\\&=\fucF(h_1(z))-\fucF(h_2(z))\\&\in h_1(z)-h_2(z)+\modM_\alpha+
\varphi^{\mathfrak e}_\omega (\modM_\kappa)\\
&=h_0(z)+\modM_\alpha+\varphi^{\mathfrak e}_\omega(\modM_\kappa),
\end{array}
$$
i.e.,
\[
{\bf f}(h_0(z))-h_0(z)\in \modM_\alpha+\varphi^{\mathfrak e}_\omega
(\modM^{\mathfrak e}_\kappa),
\]
which contradicts our initial assumption on $h_0$. The claim follows.
\end{proof}
\begin{claim}
\label{prnclaim4}
Let $z$ be as  Claim \ref{prnclaim3}. Then
 $z\in \modL^{\mathfrak e}_{n(\ast)}[{\cal K}]$.
\end{claim}
\begin{proof}
Let $h_1,h_2:\modN^{\mathfrak e}_{n(\ast)}\to\modN\in {\cal K}$ be such that $h_1(x^{\mathfrak e}_{n(\ast)})-
h_2(x^{\mathfrak e}_{n(\ast)})
\in\varphi^{\mathfrak e}_\omega(\modN)$ and $\fucF(h_i(x^{\mathfrak e}_{n(*)})) =
h_i(z) \mod \modM_\alpha+\varphi^{\mathfrak e}_\omega
(\modM_\lambda)$. By the definition of  $\modL^{\mathfrak e}_{n(\ast)}[{\cal K}]$,
it is sufficient to show $h_1(z)- h_2(z) \in \varphi^{\mathfrak e}_\omega
 (\modN)$.
To this end, choose $\gamma\in
 \kappa\setminus S$ large enough and an embedding $h$
of $\modN$ into $\modM_\kappa$ such that
$\modM_\gamma+\Rang(h)=\modM_\gamma\oplus\Rang(h)\leq_{\aleph_0}
\modM_\kappa$. Let $i=1,2$ and
note that $h \circ h_i$ is a bimodule  \homo\ from
$\modN^{\mathfrak e}_{n(*)}$
into $\modM_\kappa$.
Thanks to Claim \ref{prnclaim3} we observe that
$$\fucF(h(h_i(x^{\mathfrak e}_{n(*)}))) - \fucF(h(h_i(z))) \in \modM_\alpha+
\varphi^{\mathfrak e}_\omega (\modM_\kappa).$$
So by subtracting the above relations for $i=1, 2$ we have
\[\fucF(h(h_1-h_2(x^{\mathfrak e}_{n(\ast)})))-\fucF(h(h_1-h_2(z))) \in
\modM_\alpha + \varphi^{\mathfrak e}_\omega(\modM_\kappa).\]
Also, $h_1(x^{\mathfrak e}_{n(*)})= h_2(x^{\mathfrak e}_{n(*)})
\mod \varphi^{\mathfrak e}_\omega (\modN)$. It follows that $h(h_1-h_2(z)) \in
\modM_\alpha+\varphi^{\mathfrak e}_\omega
(\modM_\kappa)$. We conclude from $\modM_\alpha + \Rang(h) = \modM_\alpha
\oplus \Rang(h)$ that $$h(h_1-h_2(z)) \in \varphi^{\mathfrak e}_\omega (\Rang(h)).$$Since $h$
is an embedding, $(h_1-h_2)(z)\in\varphi^{\mathfrak e}_\omega (\modN).$
 So
$h_1(z)-h_2(z)= (h_1-h_2)(z)\in\varphi^{\mathfrak e}_\omega(\modN)$. By definition, $z\in \modL^{\mathfrak e}_{n(\ast)}[{\cal K}]$ as required.
\end{proof}
In sum, Lemma \ref{prnalphafez} follows.
\end{proof}

\begin{lemma}
\label{2.4a}
Let $\bar{\modM}$ be a semi-nice construction for $(\lambda,
\frakss, S,\kappa)$ and $\fucF$ be an $\ringR$-endomorphism of $\modM_\kappa$. Then for some $\alpha<\kappa,
~n(*)<\omega$ and $z\in\modL^{\mathfrak e}_{n(*)}$ the statement
$(\Pr 1)^{n(*)}_{\alpha,z}[\fucF, \mathfrak e]$ holds.
\end{lemma}
\begin{proof}
In the light of Lemma \ref{prnalphafe}, we can find $\alpha \in \kappa\setminus S$ and $n(*) < \omega$ such that $(\Pr)^{n(*)}_{\alpha}[\fucF, \mathfrak e]$
holds. According to Lemma \ref{prnalphafez}, there exists $z\in\modL^{\mathfrak e}_{n(*)}$ such that
$(\Pr 1)^{n(*)}_{\alpha,z}[\fucF, \mathfrak e]$ holds, as claimed.
\end{proof}

\begin{Remark}
\label{2.5}
The property $(\Pr 1)^{n(\ast)}_{\alpha,z}[\fucF,{\mathfrak e}]$
is almost what is required, only the ``error term'' $\modM_\alpha$
 is too large.
 However, before we do this, we note that for the solution of
Kaplansky test problems, as done later in Section \ref{Kaplansky test problems}, this improvement is immaterial as we just
divide by a stronger ideal, i.e., we allow to divide by a submodule of bigger
cardinality.
\end{Remark}

\begin{Definition}
\label{Gnsequence}
Assume $\bar \modM=\langle \modM_\alpha:\alpha\leq\kappa\rangle$
is a weakly semi-nice construction for $(\lambda,\frakss,S,\kappa)$, $\fucF$ is an ${\bf R}$-endomorphism of $\modM_\kappa$,  ${\mathfrak e}\in \callE^\frakss$ and $n(\ast)<\omega$. Let also
 $\bar \groupG:= \langle \groupG_n: n\geq n(*) \rangle$ be a sequence of additive subgroups of $\varphi^{\mathfrak
e}_{n(\ast)}(\modM_\kappa)$.
\begin{enumerate}

\item
 We say $\bar \groupG$ is $\bar \varphi$-appropriate for
$\modM$ if   $\groupG_n\subseteq \varphi^{\mathfrak e}_n (\modM_\kappa)$ for
all $n \geq n(*).$
\item
We say $\bar \groupG$ is compact  for $(\bar{\varphi}^{\mathfrak e},n(*))$ in $\modM_\kappa$,  if it is $\bar \varphi$-appropriate  and for any $z^*_\ell\in \groupG_\ell$ with $\ell
\geq n(\ast)$, there is $z^*\in \groupG_{n(*)}$ \st\
$$z^*-\sum\limits_{\ell=n(*)}^{n} z^*_\ell\in \varphi_{n+1} (\modM_\kappa). \footnote{So in appropriate sense, $\sum\limits_{\ell\geq n(*)} z_\ell$
exists; of
course we can increase $n(*)$.} $$

\item We say
 $\groupG_m$ is
$(\cal K, \bar\varphi)$-finitary in $\modM$, if $\groupG_m \subseteq \sum\limits_{\ell < n} \modK_\ell + \varphi_\omega(\modM_\kappa)$ for some finite $n<\omega$
and  $\modK_\ell \in c\ell_{\text{is}}(\cal K)$ such that $\sum\limits_{\ell < n} \modK_\ell \leq_{\aleph_0} \modM_\alpha$ for $\alpha$
large enough in $\kappa \setminus S$.

\item We say $\bar \groupG$ is $(\cal K, \bar\varphi)$-finitary in $\modM_\kappa$ if   $~\groupG_m$ is
$(\cal K, \bar\varphi)$-finitary in $\modM$  for some $m \geq n(\ast)$.

\item We say $\bar \groupG$ is non-trivial if  $\mathbb{G}_m \nsubseteq \varphi_\omega^{\mathfrak e}(\modM_\kappa)$ for all $m\geq n(\ast)$.
\end{enumerate}
\end{Definition}

We need to present another definition:

\begin{Definition}
	\label{strong} Assume $\bar \modM=\langle \modM_\alpha:\alpha\leq\kappa\rangle$
	is a weakly semi-nice construction for $(\lambda,\frakss,S,\kappa)$. Let  $\fucF$ be any ${\bf R}$-endomorphism of $\modM_\kappa$ and ${\mathfrak e}\in \callE^\frakss$.  We replace weakly by strongly
	from Definition
	\ref{w.s. nice definition} provided the following  properties denoted by $(G)^+$, are valid:
\begin{enumerate} 
\item there is  some $\alpha<\kappa$ such that $|\modM_\alpha||<\lambda$;
\item there is  some $n(*)<\omega$ and
 $z\in
\modN^{\mathfrak e}_{n(\ast)}$ such that for every bimodule \homo\ $h:\modN^{\mathfrak e}_{n(*)}
\rightarrow \modM$ we have $$ \fucF h(x^{\mathfrak e}_{n(*)})- h(z)
\in \modM_\alpha+\varphi_\omega (\modM_\kappa).$$
 \end{enumerate}
\end{Definition}

\begin{Definition} In the previous definition, we replace strongly by ${\rm strong}^+$ if
$\modM_\alpha$ is replaced by $$\modM_* \oplus \modK \leq_{\aleph_0}
\modM_\kappa$$ for some $\modK\in c\ell_{\text{is}}(\cal K).$\end{Definition}

\begin{Lemma}
\label{groupsGn}
Assume  $\kappa=\lambda$ and
$\bar\modM=\langle \modM_\alpha:
\alpha\leq\kappa\rangle$ is
a semi-nice construction for
$(\lambda, {\frakss},S, \kappa)$ such that $\alpha<\kappa,~
||\modM_\alpha||<\lambda$.
Suppose ${\mathfrak e}\in {\callE}$ and $\fucF$ is an $\ringR$-endomorphism of
$\modM_\kappa$, for some
$n(\ast)<\omega$, $\alpha(\ast)\in\kappa\setminus S$ and $z\in \modL^{\mathfrak e}_{n(\ast)}[{\cal K}]$ such that the property
$(\Pr 1)^{n(\ast)}_{\alpha(\ast),z}
[\fucF,{\mathfrak e}]$ holds.\footnote{Such a thing exists by Lemma \ref{2.4a}.}
Then there exists a decreasing sequence $\bar{\mathbb G}^\ast=\langle {\mathbb
G}^\ast_n:n\geq n(\ast)\rangle$ of additive subgroups of $\varphi^{\mathfrak
e}_{n(\ast)}(\modM_\kappa)$, satisfying ${\mathbb G}^\ast_n\subseteq
\varphi^{\mathfrak e}_n(\modM_\kappa)$ and equip with the following properties:
\begin{enumerate}
\item for every $n\geq n(\ast)$ and every bimodule-homomorphism
$h:\modN^{\mathfrak e}_n
\longrightarrow \modM_\kappa$,\begin{enumerate}\item
$\fucF(h(x^{\mathfrak e}_n))-h(z_n)\in {\mathbb G}^\ast_n$ where
$z_n:=g^{\mathfrak e}_{n(\ast),n}(z)$,
 \item ${\mathbb G}_n\subseteq \modM_{\alpha(\ast)}+
\varphi_\omega^{\mathfrak e} (\modM_\kappa)$.
\end{enumerate}

\item $\bar {\mathbb G}^\ast$ is compact for
$\bar{\varphi}^{\mathfrak e}$.
\end{enumerate}
\end{Lemma}
\begin{proof}
For every $n\geq n(\ast)$ we define ${\mathbb G}^\ast_n$ by
\[
{\mathbb G}^\ast_n:=\big\{ \fucF(h(x^{\mathfrak e}_n))-h(z_n): h: \modN^{\mathfrak e}_n\to \modM_\kappa\mbox{
 is a bimodule homomorphism} \big\}.
\]
It is easily seen that
${\mathbb G}^\ast_n \subseteq
\modM_{\alpha(\ast)}+\varphi^{\mathfrak e}_\omega
(\modM_\kappa)$. Indeed let $h: \modN^{\mathfrak e}_n\to \modM_\kappa$ be
  a bimodule homomorphism. Then $h'=h \circ g^{\mathfrak e}_{n(\ast),n}: \modN^{\mathfrak e}_{n(\ast)}\to \modM_\kappa$ is a bimodule
   homomorphism. In view of  $(\Pr 1)^{n(\ast)}_{\alpha(\ast),z}[\fucF,{\mathfrak e}]$ we observe that
\[
 \fucF(h(x^{\mathfrak e}_n))-h(z_n)= \fucF(h'(x^{\mathfrak e}_{n(\ast)}))-h'(z) \in \modM_{\alpha(\ast)}+\varphi^{\mathfrak e}_\omega
(\modM_\kappa).
\]
Also as $x^{\mathfrak e}_n, z \in \varphi^{\mathfrak e}(\modN^{\mathfrak e}_n)$,
we have $\fucF(h(x^{\mathfrak e}_n))-h(z_n) \in \varphi^{\mathfrak e}_n(\modM_\kappa)$. Clearly,
${\mathbb G}^\ast_n$ is an additive
subgroup of $\varphi^{\mathfrak e}_n(\modM_\kappa)$ and the sequence $\bar{\mathbb G}$ is decreasing.

Let us prove clause (2). Suppose $z^\ast_\ell\in {\mathbb G}^\ast_\ell$ for
$n(\ast)\leq\ell<\omega$, so for some bimodule homomorphism
$h_\ell:\modN^{\mathfrak e}_\ell\longrightarrow \modM_{\kappa}$ we have
$$z^\ast_\ell=\fucF(h_\ell(x_n^{\mathfrak e}))-h_\ell(z_\ell).$$
Let $\alpha(0)$ with $\alpha(\ast) <\alpha(0)< \kappa$ be such that $\alpha(0)\notin S$,
and for  $n(\ast) \leq \ell,$ $\Rang(h_\ell)
\subseteq \modM_{\alpha(0)}$ and
$\fucF(h_\ell(x^{\mathfrak e}_\ell)) \in \modM_{\alpha(0)}$. Note that such $\alpha(0)$ necessarily exists as $\kappa=
\cf(\kappa)\geq\kappa(\calE)+\aleph_0$. We need the following claim:

\begin{claim}
\label{existencealphaast}
For each $n \geq n(\ast)$ and $\beta>  \alpha (1)$ with $\beta\in\kappa
\setminus S$, there are  $\gamma$  satisfying
$\beta<\gamma\in\kappa
\setminus S$,  some embedding $h_{\beta, n}:\modN^{\mathfrak e}_n
\longrightarrow\modM_\gamma$ and some $\modK_{\beta, n} \in c\ell^{\aleph_0}_{is} ({\cal K})$ such that the following two items hold:
\begin{enumerate}
\item $\modM_\beta\oplus\Rang(h_{\beta,n})\oplus \modK_{\beta, n} \leq_{\aleph_0}
\modM_\gamma,$
\item $\fucF(h_{\beta, n}(x^{\mathfrak e}_n ))\in \Rang(h_{\beta,n})\oplus \modK_{\beta, n}$.
\end{enumerate}
In particular,
$\fucF(h_{\beta, n}(x^{\mathfrak e}_n ))- h_{\beta, n}(z_n) \in \varphi_\omega^{\mathfrak e}(\modM_\kappa)$.
\end{claim}
\begin{proof} Fix $n$ and $\beta$ as in the claim. For every $\gamma$
satisfying $\beta <\gamma\in
\kappa\setminus S$, let $h_{\gamma}:\modN^{\mathfrak e}_n
\longrightarrow \modM_{\gamma+1}$ and $\modK^0_{\gamma}\in
c\ell_{is} ({\cal K})$ be such that $h_{\gamma}$ is a bimodule embedding and
\[\fucF(h_{\gamma}(x_n^{\mathfrak e}))\in \modM_\gamma\oplus\Rang(h_{\gamma})
\oplus \modK^0_{\gamma}
\leq_{\aleph_0} \modM_\kappa.
\]
Let $\epsilon_\gamma>\gamma$ be in $\kappa\setminus S$ such that $\fucF$
maps $\modM_{\epsilon_\gamma}$ into $\modM_{\epsilon_\gamma}$
and
\[\modM_\gamma\oplus\Rang(h_{\gamma})\oplus
\modK^0_{\gamma} \leq_{\aleph_0}
\modM_{\epsilon_\gamma}.\]
Let $\fucF(h_{\gamma}(x_n^{\mathfrak e}))=
z^1_{\gamma}+z^2_{\gamma}+z^3_{\gamma}$, where $z^1_{\gamma}\in
\modM_\gamma$, $z^2_{\gamma}\in\Rang(h_\gamma)$ and
$z^3_{\gamma}\in \modK^0_\gamma$. By
Fodor's lemma for some $z$ and for a stationary set $T \subseteq \kappa \setminus S$ we have
\begin{itemize}
  \item $\min(T) > \beta,$
  \item $z^1_\gamma=z$, for all $\gamma \in T.$
\end{itemize}
Let $\gamma(1)$, $\gamma(2) \in T$ be such that  $\epsilon_{\gamma(1)}<\gamma(2)$. Now, the following
\begin{itemize}
	\item [(i)] $\gamma:=\epsilon_{\gamma(2)},$
	\item [(ii)] $h_{\beta,n}: =h_{\gamma(2)}-h_{\gamma(1)}$ and
	\item [(iii)]  $\modK_{\beta, n}:=\modK^0_{\gamma(1)}\oplus
	\modK^0_{\gamma(2)}\oplus\Rang(h_{\gamma(1)})$
\end{itemize}
 are as required.

The particular case follows by the
choose of
$\alpha(\ast)$.
\end{proof}
Let us complete the proof of Lemma \ref{groupsGn}. To this end, we look at the following club of $\kappa$:
\[
A:=\{\beta<\kappa: \beta > \alpha(0) \text{~and~}
\fucF''(\modM_\beta)\subseteq \modM_\beta\}.
\]
Let $n \geq n(\ast)$ and $\beta \in A \cap (\kappa \setminus S).$ According to Claim \ref{existencealphaast}, we can find
some $\gamma_\beta>\beta$ a bimodule embedding
$h_{\beta,n}:\modN^{\mathfrak e}_n\longrightarrow
\modM_{\gamma_\beta}$, and $\modK_{\beta, n} \in c\ell^{\aleph_0}_{is} ({\cal K})$ such that
\begin{itemize}
\item $\modM_\beta\oplus\Rang(h_{\beta,n})\oplus \modK_{\beta, n} \leq_{\aleph_0}
\modM_{\gamma_\beta},$
\item $\fucF(h_{\beta, n}(x^{\mathfrak e}_n ))\in \Rang(h_{\beta,n})\oplus \modK_{\beta, n}$.
\end{itemize}
Set ${\cal U}=:\{\ell:n(\ast)\leq\ell<\omega
\}$ and ${\mathfrak e}^\ast={\mathfrak e}\restriction {\cal U}$, so ${\mathfrak
e}^*\in{\callE}^{\frakss}$ as well. For $\ell \in {\cal U}$ and $\beta \in A \cap (\kappa \setminus S),$ set
$$h'_{\beta,\ell}=h_{\beta,\ell}+h_\ell.$$
We are going to use the property presented in part $(G)_2$ of Definition \ref{nice construction modified}(3). In this regard,  we can find an embedding ${\bf h}: \modP^{\mathfrak e^{\ast}} \to \modM_\kappa$
and an increasing sequence $\langle \epsilon_n: n <\omega    \rangle$ of elements of $A \cap (\kappa \setminus S)$
such that $\epsilon=\sup\limits_{n<\omega}\epsilon_n \in A \cap (\kappa \setminus S)$ and
for each $n \in {\cal U}, h'_{\epsilon_n, n}={\bf h} \circ h_n^{\mathfrak{e}^\ast}$.
Then
\[
 \fucF(h'_{\epsilon_\ell, \ell}(x^{\mathfrak e}_\ell))-h'_{\epsilon_\ell, \ell}(z_\ell)=\fucF(h_{\epsilon_\ell, \ell}(x^{\mathfrak e}_\ell))-h_{\epsilon_\ell, \ell}(z_\ell) - z^\ast_\ell \in z^\ast_\ell   + \varphi^{\mathfrak e}_{\omega}(\modM_\kappa).
\]
Let $x:=  {\bf f}_{n(\ast)}^{\mathfrak{e}}(x_1^{\mathfrak{e}^{\ast}})$ and $z':= {\bf f}_{n(\ast)}^{\mathfrak{e}}(z)$ where ${\bf f}_{n(\ast)}^{\mathfrak{e}}$
were defined in Lemma \ref{Pe bimodules defined}.
It then follows that\begin{itemize}
	\item $x - \sum\limits_{\ell=n(\ast)}^{\ell=n-1}h_\ell^{\mathfrak e}(x_\ell^{\mathfrak e}) \in \varphi_n^{\mathfrak e}(\modP^{\mathfrak e})$, and
	\item $z' - \sum\limits_{\ell=n(\ast)}^{\ell=n-1}h_\ell^{\mathfrak e}(z_\ell) \in \varphi_n^{\mathfrak e}(\modP^{\mathfrak e}).$
	\end{itemize}
We deduce from these memberships that
\[
{\bf h}(x) - \sum\limits_{\ell=n(\ast)}^{\ell=n-1}{\bf h}(h_\ell^{\mathfrak e}(x_\ell^{\mathfrak e})) \in \varphi_n^{\mathfrak e}(\modM_\kappa)
\]
and
\[
{\bf h}(z') - \sum\limits_{\ell=n(\ast)}^{\ell=n-1}{\bf h}(h_\ell^{\mathfrak e}(z_\ell)) \in \varphi_n^{\mathfrak e}(\modM_\kappa).
\]
As $\bf f$ is an endomorphism, we have
\[
{\bf f}({\bf h}(x)) - \sum\limits_{\ell=n(\ast)}^{\ell=n-1}{\bf f}({\bf h}(h_\ell^{\mathfrak e}(x_\ell^{\mathfrak e}))) \in \varphi_n^{\mathfrak e}(\modM_\kappa).
\]
We plug this in the previous formula and deduce that
\[
 {\bf f}({\bf h}(x)) - {\bf h}(z')- \sum\limits_{\ell=n(\ast)}^{\ell=n-1}({\bf f}({\bf h}(h_\ell^{\mathfrak e}(x_\ell^{\mathfrak e})))- {\bf h}(h_\ell^{\mathfrak e}(z_\ell)) )\in \varphi_n^{\mathfrak e}(\modM_\kappa).
\]
Note that for some $\modK \in c\ell(\cal K), \modM_\epsilon= \modM_{\alpha(0)} \oplus \modP^{\mathfrak e} \oplus \modK$. Let
\[
\pi:  \modM_\epsilon \to \modM_{\alpha(0)}
\]
be the projection map and set
\[
z^\ast:= \pi({\bf f}({\bf h}(x))) - \pi({\bf h}(z')).
\]
We apply $h'_{\epsilon_\ell, \ell}={\bf h} \circ h_\ell^{\mathfrak{e}^\ast}$ along with the previous observation to see:

	\[\begin{array}{ll}
{\bf f}({\bf h}(h_\ell^{\mathfrak e}(x_\ell^{\mathfrak e})))- {\bf h}(h_\ell^{\mathfrak e}(z_\ell))&=\fucF(h'_{\epsilon_\ell, \ell}(x^{\mathfrak e}_\ell))-h'_{\epsilon_\ell, \ell}(z_\ell)\\
&= \fucF(h_{\epsilon_\ell, \ell}(x^{\mathfrak e}_\ell))-h_{\epsilon_\ell, \ell}(z_\ell) - z^\ast_\ell.
\end{array}\]
This yields that
	\[\begin{array}{ll}
\pi({\bf f}({\bf h}(h_\ell^{\mathfrak e}(x_\ell^{\mathfrak e}))))- \pi({\bf h}(h_\ell^{\mathfrak e}(z_\ell)))&=  \pi(\fucF(h_{\epsilon_\ell, \ell}(x^{\mathfrak e}_\ell)))-\pi(h_{\epsilon_\ell, \ell}(z_\ell)) - \pi(z^\ast_\ell)\\
&= z^\ast_\ell.
\end{array}\]
Therefore,

	\[\begin{array}{ll}
z^\ast-\sum\limits^{n-1}_{\ell=n(\ast)}z^*_\ell &=  \pi({\bf f}({\bf h}(x))) - \pi({\bf h}(z'))-\sum\limits_{\ell=n(\ast)}^{\ell=n-1}\pi( {\bf f}({\bf h}(h_\ell^{\mathfrak e}(x_\ell^{\mathfrak e}))))- \pi( {\bf h}(h_\ell^{\mathfrak e}(z_\ell)))\\
&\in \varphi_n^{\mathfrak e}(\modM_\kappa).
\end{array}\]

So, $z^\ast$ is as required.
\end{proof}

Let $z\in \modN^{\mathfrak e}_{n(\ast)}$. Following the above lemma, let us define,
 ${\groupG} ^\frakss_{n,z} [\bar \modM]$ as
$\langle \groupG_n: n\geq n(*)
\rangle$ where
$$\groupG_n:=\{ \fucF(h(x^{\mathfrak e}_n))-h(g^{\mathfrak e}_{n(*),n}(z)):  h \in \textmd{Hom}(\modN^{\mathfrak e}_n,\modM_\kappa ) \}.$$
We now show that under extra assumptions we can get a better ${\mathbb G}_n$'s.
\begin{lemma}
\label{moreonGn}
Let  $\bar \modM=\langle \modM_\alpha:\alpha\leq\kappa\rangle$
be a weakly semi-nice construction for $(\lambda,\frakss,S,\kappa)$,
 $\fucF$ be an ${\bf R}$-endomorphism,  ${\mathfrak e}\in \callE^\frakss$ and $n(\ast)<\omega$.
\begin{enumerate}
\item If $\modK=\mathop{{\bigoplus}}\limits_{t\in I} \modK_t$
(as $\ringR$-modules), $\bar \groupG={\groupG} ^\frakss_{n,z} [\bar \modM]$
 is decreasing and
compact for
$(\bar\varphi^{\mathfrak e},n(\ast))$ in $\modK$ over
$\modK_0$, then for
some finite $J\subseteq I$ and $m<\omega$:
$$
{\mathbb G}_m\subseteq\mathop{{\bigoplus}}\limits_{t\in J}
\modK_t+\varphi^{\mathfrak e}_\omega(\modK).
$$

\item Suppose $\modK=\sum\limits_{t\in I}
\modK_t$ (as $\ringR$--modules) and
\begin{enumerate}
\item[] if $J_1$, $J_2$ are finite subsets of $I$, $n<\omega$,
$y_\ell\in\sum\limits_{t\in J_\ell}\modK_t$ for
$\ell=1,2$,

\item[$(\ast):$]and if $y_1-y_2\in
\varphi^{\mathfrak e}_n(\modK)$, then for some
$y_3\in\sum\limits_t
\{\modK_t:t\in J_1\cap J_2\}$

we have
$y_1-y_3\in\varphi^{\mathfrak e}_n(\modK)$ and
$y_3-y_2\in \varphi^{\mathfrak e}_n(\modK)$.
\end{enumerate}
Then for
some finite $J\subseteq I$ and $m<\omega$:
$$
{\mathbb G}_m\subseteq\mathop{{\bigoplus}}\limits_{t\in J}
\modK_t+\varphi^{\mathfrak e}_\omega(\modK).
$$

\item Suppose $\bar {\mathbb G}$ is   compact for
$(\bar\varphi^{\mathfrak e},n(\ast))$ in
$\modK$ (as $\ringR$-modules), $h:\modK\longrightarrow
\modK'$ is an $\ringR$-homomorphism and for
$h(x)\in\varphi^{\mathfrak e}_\ell(\modK')
\setminus\varphi^{\mathfrak e}_{\ell+1}(\modK')$
there exists $y\in\varphi^{\mathfrak e}_\ell(\modK)
\setminus\varphi^{\mathfrak e}_{\ell+1}(\modK)$ such that
$h(y)=h(x)$.
Then ${h''}(\bar{\mathbb G})=\langle {h''}({\mathbb G}_n):n \geq n(\ast)\rangle$ is
compact for $(\bar\varphi^{\mathfrak e},n(\ast))$ in $\modK'$.

\item If ${\mathbb G}\subseteq \modK$, $\modK=
\mathop{{\bigoplus}}\limits^n_{t=1} \modK_t$ and
the projection of ${\mathbb G}$ to each $\modK^t$ is
$({\cal K},\bar\varphi)$-finitary,
then ${\mathbb G}$ in $(\cal K, \bar\varphi)$-finitary.

\item If $\modK_0\subseteq
\modK_1\leq^{ads}_{{\cal K},\aleph_0}
\modK_2$ and $\bar {\mathbb G}$ is $(\bar\varphi^{\mathfrak e},n(*))$-compact in $\modK_2$ over $\modK_0$,
then $\langle {\mathbb G}_n\cap \modK_1:n \geq n(*)\rangle$ is
$(\bar\varphi^{\mathfrak e},n(*))$-compact in $\modK_1$ over $\modK_0$.
\end{enumerate}
\end{lemma}
\begin{proof}
(1) Let us first suppose that the index set $I$ is countable. Suppose by contradiction that for all finite $J \subseteq I$
and $m<\omega$,
\[
{\mathbb G}_m \nsubseteq \mathop{{\bigoplus}}\limits_{t\in J}
\modK_t+\varphi^{\mathfrak e}_\omega(\modK).
\]
We   argue by induction on   $\ell\geq n(\ast)$ to find $z_\ell$,
$J_\ell$ and $n_\ell$ such that the following properties hold:
\begin{itemize}
\item [i)] $J_\ell$ is a finite subset of $I$,
\item  [ii)] $J_\ell
\subseteq J_{\ell+1}$,

\item  [iii)] $I=\bigcup\limits_{\ell<\omega}J_\ell,$

\item  [iv)] $z_\ell\in {\mathbb G}
_\ell\setminus\left(\mathop{{\bigoplus}}
\limits_{t\in J_\ell}K_t+\varphi^{\mathfrak e}_\omega(\modK)\right)$,

\item [v)] $z_\ell \notin \mathop{{\bigoplus}}
\limits_{t\in J_\ell}K_t+\varphi^{\mathfrak e}_{n_{\ell+1}}(\modK)$,

\item  [vi)] $z_\ell\in \mathop{{\bigoplus}}\limits_{t\in J_{\ell+1}}\modK_t$.
\end{itemize}
Let $\cal U$ be an infinite subset of $\omega$ such that $0 \in \cal U$ and
\[
\ell < m \mbox{ and } m \in \cal U \implies n_{\ell+1} < m.
\]
Define $\langle z^*_\ell: \ell < \omega  \rangle$ by $z^*_\ell=z_\ell$ if $\ell \in \cal U$ and $z^*_\ell=0$ otherwise.
By compactness, we can find $z^\ast \in \groupG_{n(\ast)}$ such that for each
$\ell \geq n(\ast)$,
\[
z^\ast - \sum\limits_{i=n(\ast)}^\ell z^*_i \in \varphi^{\mathfrak e}_{\ell+1}(\modK).
\]
Let $J \subseteq I$ be a finite set such that $z^\ast \in \bigoplus\limits_{t\in J}\modK_t$. Pick $m \in \cal U$
such that $J \subseteq J_m$ and set $n$ be the least element of $\cal U$ above $m$. Then $z_m=z^*_m$ and
\[
z^\ast - \sum\limits_{i=n(\ast)}^m z^*_i = z^\ast - \sum\limits_{i=n(\ast)}^{n-1} z^*_i\in \varphi^{\mathfrak e}_{n}(\modK).
\]
An easy contradiction.

Let us now show that we can get the result for arbitrary $I$. Thus suppose $I$ is uncountable and suppose that the conclusion of the lemma fails
for it. Construct $z_\ell, J_\ell, n_\ell$ as before and set $\bar I=\bigcup\limits_{\ell=n(\ast)}^{\omega}J_\ell.$ Then $\bar I$ is countable. Set
$\bar \modK=\bigoplus\limits_{t\in \bar I}\modK_t$
and $\bar\groupG_n= \groupG_n \cap \bar\modK$. It then follows that the conclusion fails for this case, contradicting the above argument.

(2). This is similar to $(1)$.

(3). Suppose $h(x_\ell) \in h''(\groupG_\ell)$, for $\ell \geq n(\ast)$. By our assumption, we can find a sequence $\langle y_\ell: \ell \geq n(\ast) \rangle$ such that
\begin{itemize}
  \item $y_\ell \in \groupG_\ell$,
  \item $h(y_\ell)=h(x_\ell)$,
  \item $h(x_\ell) \in \varphi^{\mathfrak e}_\ell(\modK')
\setminus\varphi^{\mathfrak e}_{\ell+1}(\modK') \implies y_\ell \in \varphi^{\mathfrak e}_\ell(\modK)
\setminus\varphi^{\mathfrak e}_{\ell+1}(\modK)$.
\end{itemize}
Let $z \in \groupG_{n(\ast)}$ be such that for each $n \geq n(\ast), z - \sum\limits_{\ell=n(\ast)}^n y_\ell \in \varphi^{\mathfrak e}_{n+1}(\modK)$. This implies that
\[
h(z) - \sum\limits_{\ell=n(\ast)}^n h(y_\ell) \in \varphi^{\mathfrak e}_{n+1}(\modK).
\]
Thus $h(z) \in h''(\groupG_{n(\ast)})$ is as required.

(4). For $t=1, \cdots, t$ let $\pi_t: \modK \to \modK_t$ be the projection map to $\modK_t$. We have
\[
\pi_t''(\groupG) \subseteq \sum\limits_{i< m_t} \modK_i^t + \varphi^{\mathfrak e}_\omega(\modK),
\]
where each $\modK_i^t$ is in $c\ell_{\text{is}}(\modK)$. Then
\[
\groupG = \sum\limits_{t=1}^n \pi_t''(\groupG) \subseteq    \sum\limits_{t=1}^n \left( \sum\limits_{i< m_t} \modK_i^t + \varphi^{\mathfrak e}_\omega(\modK) \right) \subseteq \sum\limits_{t=1}^n  \sum\limits_{i< m_t} \modK_i^t + \varphi^{\mathfrak e}_\omega(\modK).
\]
We are done.

(5). For each $\ell \geq n(\ast)$,
we peak
$z_\ell \in \groupG_\ell \cap \modK_1$. Due to our assumption, there exists $z^\ast \in \modK_2$ such that  $ z^\ast - \sum\limits_{\ell=n(\ast)}^n z_\ell \in  \varphi^{\mathfrak e}_{n+1}(\modK_2)$  for all $n \geq n(\ast)$. Consequently,
\[
\modK_2 \models~\exists z^\ast \bigwedge\limits_{n=n(\ast)}^\omega \varphi^{\mathfrak e}_{n+1}(z^\ast - \sum\limits_{\ell=n(\ast)}^n z_\ell ).
\]
As $\modK_1\leq^{ads}_{{\cal K},\aleph_0}
\modK_2$ and in the light of Lemma \ref{comparing two orders}(1) we have
\[
\modK_1 \models~\exists z^\ast \bigwedge\limits_{n=n(\ast)}^\omega \varphi^{\mathfrak e}_{n+1}(z^\ast - \sum\limits_{\ell=n(\ast)}^n z_\ell ).
\]
Let $z^\ast \in \modK_1$ be  witness this. Then, $z^\ast - \sum\limits_{\ell=n(\ast)}^n z_\ell \in  \varphi^{\mathfrak e}_{n+1}(\modK_1)$
 for all $n \geq n(\ast),$ as required.
\end{proof}

\begin{Remark}
\label{2.7A}\begin{itemize}
	\item[i)] In Lemma \ref{moreonGn}(3), we can weaken the assumption on $h$ to:
	for some $\eta\in {}^\omega\omega$ diverging to infinity
	if $\ell\geq n(\ast)$ and $h(x)\in\varphi^{\mathfrak e}_\ell
	(\modK')\setminus\varphi^{\mathfrak e}_{\ell+1}(\modK')$, then
$h(x)=h(y)$ for some $y\in\varphi^{\mathfrak e}_{n(\ast)}
	(\modK) \setminus
	\varphi^{\mathfrak e}_{\eta(\ell)}(\modK).$
	\item [ii)] Note that if $h$ is a projection, then it satisfies in the presented condition from the first item.
\end{itemize}
\end{Remark}

\begin{lemma}
\label{2.9}
Let $\ringR$, $\ringS$ and every $\modN \in {\cal K}$
be of  cardinality $<2^{\aleph_0}$ and let
${\mathfrak e}\in {\callE}$.  Then  there is no non-trivial compact
 $\bar {\mathbb G}$ for $\bar \varphi^{\mathfrak e}$
in any ${\cal K}$-bimodule $\modM$.
\end{lemma}
\begin{proof} Let
	$\langle {\mathbb G}_n:n \geq n_0 \rangle$ be $(\bar\varphi^{\mathfrak e},n_0)$-compact
	in
	$\modM$. According to Definition \ref{Gnsequence} (4) we need to show that
	${\mathbb G}_m\subseteq\varphi^{\mathfrak e}_\omega(\modM)$
	for	some $m$.
  Suppose on the contradiction that 	${\mathbb G}_m\nsubseteq\varphi^{\mathfrak e}_\omega(\modM)$     for all $m\geq n_0$.  Hence,  $\groupG_m \nsubseteq \varphi^{\mathfrak e}_{\ell_m }(\modM)$
  for some $\ell_m>n_0$. Pick $z_m \in \groupG_m \setminus \varphi^{\mathfrak e}_{\ell_m}(\modM)$, for $m \geq n_0$.
For any infinite $\cal U \subseteq \omega \setminus n_0$, find $z_{\cal U} \in \groupG_{n_0}$ such that
  \[
 ~z_{\cal U} - \sum\{z_m: m \in \cal U \cap [n_0, n]   \} \in \varphi_{n+1}^{\mathfrak e}(\modM).
  \] for all  $ n \in \cal U$.
 Let  $\cal U_1$ and $ \cal U_2$ be two subsets of  $\omega$ with finite intersection property. Then $z_{\cal U_1} \neq z_{\cal U_2}.$ It follows that
  \[
  2^{\aleph_0} \leq ||\groupG_{n_0}|| \leq ||\modM|| < 2^{\aleph_0},
  \]
  which is impossible.
\end{proof}

\begin{lemma}
\label{gncanbefinitary}
Under the same assumptions as in Lemma \ref{groupsGn},  ${\groupG} ^\frakss_{n,z} [\bar \modM]=\langle \groupG_n: n\geq n(\ast)
\rangle$ is $(\cal K, \bar\varphi^{\mathfrak e})$-finitary in
$\modM_\kappa$. Furthermore, if for $\modN \in \cal K,$ there is no non-trivial $\bar{\mathbb{L}}= \langle \mathbb{L}_n: n \geq n(\ast) \rangle$ compact for $(\bar\varphi^{\mathfrak e}, n(\ast))$
in $\modM_\ast \oplus \modN$, then, by increasing $n(\ast)$, we can take ${\groupG} ^\frakss_{n,z} [\bar \modM]=\bar 0,$ i.e., $\groupG_n=0$ for all $n \geq n(\ast)$.
\end{lemma}
\begin{proof}
  Pick $\alpha(\ast) \in \kappa\setminus S$ such that $z\in \modL^{\mathfrak e}_{n(\ast)}[{\cal K}] \subseteq \modM_{\alpha(\ast)}$.
  We use the assumption  $\modM_{\alpha(\ast)} \in c\ell(\cal K)$ along with  Lemma \ref{moreonGn}(1) to
   find $m<\omega$ and a finite subset $\{\modK_0, \cdots, \modK_{n-1}\}$ of $c\ell_{\text{is}}(\cal K)$ such that
  $\modK_0, \cdots, \modK_{n-1}$ are direct summands of $\modM_{\alpha(\ast)}$ and
  \[
  \groupG_m \subseteq \sum\limits_{\ell  < n}\modK_\ell + \varphi_\omega^{\mathfrak e}(\modM_\kappa).
  \]
Recall from  Lemma \ref{properties of w.s.n.c}(1) that   $\sum\limits_{\ell  < n}\modK_\ell \leq_{\aleph_0} \modM_{\alpha(\ast)} \leq_{\aleph_0}\modM_\alpha$, where $\alpha > \alpha(\ast)$. Thus
  ${\groupG} ^\frakss_{n,z} [\bar \modM]=\langle \groupG_n: n\geq n(\ast)
\rangle$ is $(\cal K, \bar\varphi^{\mathfrak e})$-finitary in
$\modM_\kappa$.

 Now suppose that for each  $\modN \in \cal K,$ there is no non-trivial $\bar{\mathbb{L}}= \langle \mathbb{L}_n: n \geq n(\ast) \rangle$ compact for $(\bar\varphi^{\mathfrak e}, n(\ast))$
in $\modM_\ast \oplus \modN$.    Suppose by contradiction that ${\groupG} ^\frakss_{n,z} [\bar \modM]\neq \bar 0.$
Let $\alpha(\ast)$, $\{\modK_0, \cdots, \modK_{n-1}\}\subset c\ell_{\text{is}}(\cal K)$
and  $m$ be as above.   By increasing $n(\ast)$ we may assume that $m=n(\ast)$.
 Then
$\bigoplus_{\ell  < n}\modK_\ell$ is a direct summand of $\modM_{\alpha(\ast)}$.  Suppose $\modM_{\alpha(\ast)} = \bigoplus_{\ell  < n}\modK_\ell\oplus \modM$ and we look at the natural projection map $\pi: \modM_{\alpha(\ast)} \to \modM$. In view of Lemma \ref{moreonGn}(3) (and Remark \ref{2.7A}) we observe that
$\pi''({\groupG} ^\frakss_{n,z} [\bar \modM])$ is non-trivial and compact for $(\bar\varphi^{\mathfrak e}, n(\ast))$
in $\modM_\ast \oplus \modN$,
which is a contradiction.
\end{proof}

\begin{Remark}
\label{2.10}
{\rm
\begin{enumerate}
\item[(1)]  Suppose for every ${\cal K}$-bimodule
$\modM$ and ${\mathfrak e}\in
{\callE}$ and ${\mathbb G}_n\subseteq \modM$ for
$n\geq n_0$, if
$\langle {\mathbb G}_n:n \geq n_0 \rangle$ is $(\bar\varphi^{\mathfrak e},n_0)$-compact
in
$\modM$, then for
some $m$, ${\mathbb G}_m\subseteq\varphi^{\mathfrak e}_\omega(\modM)$. Recall that  Lemma \ref{2.9} presents a situation for which this property holds.
In view of Lemma
\ref{groupsGn}  we
can choose ${\mathbb G}_{n(\ast)}=0$. In particular, the
``error term'' disappears, i.e., for
every endomorphism  $\fucF$
of $\modM_\lambda$ as an $\ringR$--module, for some $m$ we
have $\fucF\lceil \varphi^{\mathfrak e}_m (\modM_\lambda)/
\varphi^{\mathfrak e}_\omega
(\modM_\lambda)$ is equal to
$ {\bf h}^{{\mathfrak e},m}_{\modM_\lambda,z}$ (for its definition, see Definition \ref{hzen} below).
\item[(2)] If $\ringR$, $\ringS$ have cardinality $<2^{\aleph_0}$, we have interesting
such ${\cal K}$'s, for example ${\cal K}$ the family of finitely generated
finitely presented bimodules.
\end{enumerate}
}
\end{Remark}

\section{More specific rings and families ${\callE}$}
\label{More specific rings and families}

In this section we introduce some special rings that play an important role in our solution of Kaplansky test problems. We also
specify some specific elements of ${\callE}$ that we work with them later.
Let us start by extending the notion of pure semisimple for a pair of rings.
\begin{Definition}
	\label{starbar}
	Given a bimodule $\modM$ and a sequence $\bar\varphi$ of formulas, the notation  $(\ast)^{\bar{\varphi},\modM}_{\aleph_0,\aleph_0}$ stands for
 the following two assumptions:
 	\begin{enumerate}
 	\item[(a)] $\varphi_n=\varphi_n(x)$ is in ${\mathcal L}^{cpe}_{\aleph_0,\aleph_0}(\tau_{\ringR})$, and
 	\item[(b)]  the \sq\
 	$\langle\varphi_n(\modM):n<\omega\rangle$ is strictly decreasing.
 \end{enumerate}
\end{Definition}
Note that if  $\bar{\varphi}$ is as above, then it is $(\aleph_0,\aleph_0)$-adequate.
Also for simplicity, we can assume that  $\varphi_{n+1}(x)\vdash\varphi_n(x)$ holds for all $n<\omega$.
\begin{Definition}
	\label{purely semisimpledef}
	The pair $(\ringR,\ringS)$ of rings
	is called purely semisimple if
	for some bimodule $\modM^\ast$
	and a sequence
	$\bar{\varphi}=\langle\varphi_n(x):n<\omega\rangle$, the property $(\ast)^{\bar{\varphi},\modM^\ast}_{\aleph_0,\aleph_0}$ holds.
\end{Definition}
Thanks to
Theorem \ref{shelahpure}
 a ring $\ringR$   is not purely semisimple
if and only if the pair
$(\ringR, \ringR)$ is purely semisimple.
\begin{Definition}
	\label{c.1}
	\begin{enumerate}
		\item The sequence $\bar{\varphi}:=\langle\varphi_n(x):n<\omega\rangle$ is called  very nice if
		it is as in Definition \ref{starbar} and for some
		$m_n,k_\ell<\omega$ and $a_\ell,b_{\ell,i}\in \ringR$ we have
		\begin{enumerate}
			\item[(a)] $\varphi_n(x)=(\exists y_0,\ldots,y_{k_n-1})[\bigwedge\limits_{
				\ell=0}^{m_n-1} a_\ell x_\ell=\sum\limits_{i<k_\ell} b_{\ell,i} y_i]$,
			\item[(b)] $m_n<m_{n+1}$, $k_\ell\leq k_{\ell+1}$.
		\end{enumerate}
		\item Let   $\bar{\varphi}^1:=\langle\varphi_n^1(x):n<\omega\rangle$  and $\bar{\varphi}^2:=\langle\varphi_n^2(x):n<\omega\rangle$  be two sequences of formulas. By   $\bar{\varphi}^1\leq \bar{\varphi}^2$ we mean
		\[(\forall n<\omega)(\exists m<\omega)[\varphi^2_m(x)\vdash
		\varphi^1_n(x)].\]
		\item  Let   $\bar{\varphi}^1$ and $\bar{\varphi}^2$  be two sequences of formulas. We say $\bar{\varphi}^1$ and $\bar{\varphi}^2$ are equivalent if
		$\bar{\varphi}^1\leq \bar{\varphi}^2$ and $\bar{\varphi}^2\leq \bar{\varphi}^1$.
		
		\item  ${\mathfrak e} \in \callE$ is called $\kappa$-simple if for each
		$n$ there are a set $X\subseteq \modN^{\mathfrak e}_n$ of cardinality
		$<\kappa$ and a set $\Sigma$ of
		$<\kappa$ equations
		from $\mathcal L_{\aleph_0, \aleph_0}(\tau_{\ringR})$ with parameters from
		$X$
		\st\ $\modN^{\mathfrak e}_n$ is  generated by $X$ freely except the
		equations in $\Sigma$. We call $X$ a witness.
	\end{enumerate}
\end{Definition}

\begin{lemma}Adopt the above notation. The following assertions are true:
	\label{c.2}
	\begin{enumerate}
		\item Suppose $(*)^{\bar{\varphi},\modM^*}_{\aleph_0,\aleph_0}$ holds. Then  there is a very nice sequence $\bar{\varphi}'$ equivalent to $\bar{\varphi}$.
		\item If the $\varphi_n$'s are from infinitary logic, the same thing holds,
		only $m_n$, $k_\ell$ may be infinite but for each $\ell$ the set $\{i:
		b_{\ell,i}\neq 0\}$ is finite.
	\end{enumerate}
\end{lemma}

\begin{proof} We only prove (1), as clause (2) can be proved in a similar way. Thus assume $(\ast)^{\bar\varphi,\modM^\ast}_{\aleph_0,\aleph_0}$ holds. According to Lemma \ref{reducing to simple formula},
	we can assume that each $\varphi_n$ is a simple formula, so it is of the form
	\[\varphi_n(x)=\left(\exists y_0,\ldots,y_{k_n-1}\right)\left(
	\bigwedge\limits_{\ell=0}^{m_n-1} a^n_\ell x=\sum\limits_{i < k_n}
	b^n_{\ell,i} y_i\right),\]
	where $a^n_\ell,b^n_{\ell,i}$ are members of $\ringR$,  $k_n,m_n$ are
	natural numbers.
After replacing $\varphi_n$ with $\bigwedge\limits_{\ell
		\leq n}\varphi_\ell$, if necessary, we may assume that the sequence is decreasing in the sense that  $\varphi_{n+1}(x)\vdash\varphi_n(x)$,
	for each $n$. Also \wolog\, and by taking $b^n_{\ell,i}=0_\ringR$, we can assume that  $k_\ell<k_{\ell+1}$ and $m_n<m_{n+1}$. Finally, note that we can even get a better sequence by taking  $\varphi_0(x)=\exists y_0(x=y_0)$ (so
	$m_0=1$, $a_0=1_\ringR$, $k_0=1$ and $b_{0,0}=1$). This completes the proof.
\end{proof}
\begin{remark}
	\label{veryspecialphin}
Let   $\bar\varphi$ be a   very nice sequence.         According to Lemma \ref{c.2} and its proof, we assume from now on that $\varphi_0$ is of the form  $\varphi_0(x)=\exists y_0(x=y_0)$.
\end{remark}
We now assign to each very nice sequence $\bar\varphi$, an element  ${\mathfrak e}(\bar\varphi) \in {\callE}_{\aleph_0,\aleph_0}$
as follows.
\begin{Definition}
	\label{c.3}
	Suppose $\bar\varphi$ is a very nice sequence.  Then
	$$
	{\mathfrak e}(\bar\varphi)=\langle \modN_n,x_n,g_n:n<\omega\rangle
	$$
	is defined as follows:
	\begin{itemize}
		\item $\modN_n$ is the $(\ringR,\ringS)$
		bimodule which is   generated  by
		$$
		\{x_n\}\cup\left\lbrace y_{n,i}: i<k_{m_n-1}\right\rbrace
		$$
		freely except  to the equations
		$$
		\{ a_{\ell} x_n=\sum\limits_{i<k_\ell} b_{\ell,i} y_{n,i}:\ell<m_n
		\}.
		$$
	In other words, $ \modN_n = (\ringR x \ringS \bigoplus\limits_{i< k_{m_n-1}} \ringR y_{n, i} \ringS) / \modK$,
		where $\modK$ is the bimodule generated by $\langle   a_{\ell} x_n-\sum\limits_{i<k_\ell} b_{\ell,i} y_{n,i}:\ell<m_n     \rangle$.
		
		\item  $g_n: \modN_n \to \modN_{n+1}$ is defined so that $g_n(x_n)=
		x_{n+1}$ and $g_n(y_{n,i})=y_{n+1,i}$ for $i<k_{m_n-1}$.
	\end{itemize}
	We call $\mathfrak e$ simple if it is of the form ${\mathfrak e}(\bar\varphi)$ for some very nice $\bar\varphi.$
\end{Definition}
Note that by Remark \ref{veryspecialphin}, for each $n$, $x_n=y_{n, 0}.$
Then next lemma shows that ${\mathfrak e}(\bar\varphi) \in {\callE}_{\aleph_0,\aleph_0}$
\begin{lemma}
	\label{verynicegivesE00}
	Let $\bar\varphi$  be very nice and set ${\mathfrak e}:={\mathfrak e}(\bar\varphi)$. The following assertions are valid:
	\begin{enumerate}
		\item $x_n\in\varphi_n(\modN_n)$.
		\item Let $\modM$ be a bimodule. Then $x^\ast\in\varphi_n(\modM)$  if and only if
		for some bimodule homomorphism $h:\modN_n\to\modM$ we have $h(x_n) =x^\ast$.
		\item
		$x_n\notin\varphi_{n+1}(\modN_n)$.
		\item Each $g_n$ is a bimodule homomorphism.
	\end{enumerate}
\end{lemma}
\begin{proof}
	Clauses (1), (3) and (4) are trivial. Clause (2)
	follows from Lemma \ref{formula vs hom}.
\end{proof}
We will frequently use the following simple observation without any mention of it.
\begin{lemma}
	\label{simplephiequalpsi}
	Let ${\mathfrak e}\in\calE^\frakss$
	be simple. Then for every $n<\omega,$ $\psi^{\mathfrak e}_n$ and $\varphi^{\mathfrak e}_n$
	are equivalent. In other words, if $\modM^\frakss_*\leq_{\aleph_0}\modM$ and $x\in \modM$ then
	$\modM\models$``$\psi^{\mathfrak e}_n(x)\leftrightarrow\varphi^{\mathfrak e}_n(x)$''.
\end{lemma}
We now define some very special bimodules.
\begin{definition}\label{tr}
	Let $n<\omega$.
	\begin{enumerate}
		\item Let $\modN'_n$ be the bimodule   generated by $x_n$, $y'_{n,i}$ and
		$y''_{n,i}$ for $i<k_{m_n-1}$ freely  subject to the following relations for
		$\ell<m_n$:\begin{enumerate}
			\item $a_\ell x=\sum_{i<k_\ell} b_{\ell,i} y'_{n,i}$
	\item $ a_\ell x=
		\sum_{i<k_n} b_{\ell,i} y''_{n,i}.$\end{enumerate}	
		\item	Let $\modN^\ell_n$ for $\ell=1, 2$, be the sub-bimodule of $\modN'_n$ generated by:
\[\begin{array}{ll}
		\{x_n\}\cup\{y'_{n,i}: i<k_{m_n-1}\}&\qquad\mbox{for}\quad\ell = 1,\\
		\{x_n\}\cup\{y''_{n,i}:i<k_{m_n-1}\}&\qquad\mbox{for}\quad\ell = 2.
		\end{array}\]
		\item	Let $f^\ell_n:\modN_n\longrightarrow \modN^\ell_n$ be the bimodule homomorphisms
		defined by the following assignments:

			\begin{enumerate}
			\item $f^\ell_n(x_n)=x_{n},$
			\item $f^1_n(y_{n,i})=y'_{n,i}$ and
			\item $f^2_n(y_{n,i})=
			y''_{n,i}.$
			\end{enumerate}
		\item  $\modL^{\tr, \bar\varphi}_n:=
		\{z\in\varphi_n(\modN_n):f^1_n(z)-f^2_n(z)\in
		\varphi_\omega (\modN'_n)\}.$ 
	\end{enumerate}	Clearly, it is an Abelian subgroup of $\modN_n$.
\end{definition}
Let $\mathfrak e \in \callE$ and suppose  $\bar\varphi$ is an adequate sequence for $\mathfrak e$.
Also, let $({\bf h_1},{\bf h_2})$  be a pair of bimodule homomorphisms  from  $\modN_n^{\mathfrak e}$ to $\modM$. Recall from Definition \ref{subgroups Ln} that
\[\modL^{{\mathfrak
		e},\bar{\varphi},{\bf h_1},{\bf h_2}}_n=\{z\in\varphi_n(\modN^{\mathfrak e}_n): {\bf h_1}(z)={\bf h_2}(z)\ \mod \
\varphi_\omega(\modM)\}.\]
In the case $\frakss$ is a context, we recall
$$\modL^{\mathfrak e}_n[\frakss]=\bigcap\limits_{\modM \in \cal K\cup \{\modM_\ast\}}\big\{\modL^{\mathfrak e,\bar\varphi^{\mathfrak e},{\bf h_1}, {\bf h_2}}_n:  {\bf h_1}, {\bf h_2} \in  \textmd{Hom}(\modN_n^{\mathfrak e}, \modM) \text{~and~}{\bf h_1}(x_n^{\mathfrak e})={\bf h_2}(x_n^{\mathfrak e})\big\}.$$
In the next lemma we show that the bimodule $\modN'$ and the homomorphisms $f^1_n, f^2_n$ are sufficient to determine
$\modL^{\mathfrak e(\bar\varphi)}_n[\frakss]$ provided $\bar\varphi$ is  very nice.
\begin{lemma}
	\label{c.4}
	Let $\frakss$ be a nice context, $\bar\varphi$ is very nice and
	${\mathfrak e}:={\mathfrak e}(\bar\varphi) \in \callE^{\frakss}$. Then
	$\modL^{{\mathfrak e}(\bar \varphi)}_n[\frakss]=\modL^{\rm tr, \bar\varphi}_n$.
\end{lemma}

\begin{proof}
	Let $z\in\modL^{{\mathfrak e}(\bar \varphi)}_n[\frakss]$.
	Since $f^1_n,f^2_n:\modN_n\longrightarrow \modN^\ell_n$ satisfy $f^1_n(x_n)=f^2_n(x_n)$, thus  it follows from Definition \ref{tr}
	that
	$z\in\modL^{\rm tr, \bar\varphi}_n$.

	In order to prove $\modL^{\rm tr, \bar\varphi}_n\subseteq\modL^{{\mathfrak e}(\bar \varphi)}_n[\frakss]$,
	let  ${\bf h_1}, {\bf h_2}:
	\modN^{\mathfrak
		e}_n\to\modM \in {\cal K}\cup \{\modM_\ast\}$ be such that ${\bf h_1}(x_n) ={\bf h_2}(x_n)$.
	Note that
	$\modM  \models\varphi^{\mathfrak e}_{n}({\bf h}_1(x_n))$ and $\modM  \models\varphi^{\mathfrak e}_{n}({ \bf h}_2(x_n))$ ,
	thus we can find ${\bf h}_1':\modN^1_n\to\modM$ and ${\bf h}_2':\modN^2_n\to\modM$ such that
	${\bf h}_1=f^1_n  \circ {\bf h}_1'$ and ${\bf h}_2=f^2_n \circ {\bf h}_2'$. Take some $z\in\modL^{\rm tr, \bar\varphi}_n$. By definition,
	$z\in\varphi^{\mathfrak e}_{n}(\modN_n)$ and $f^1_n(z)-f^2_n(z)\in
	\varphi_\omega (\modN'_n)$. But then
	${\bf h}_1(z)-{\bf h}_2(z)\in\varphi^{\mathfrak e}_{\omega}(\modM)$. From this,
	$z\in\modL^{{\mathfrak
			e},\bar{\varphi},{\bf h_1},{\bf h_2}}_n$. Since ${\bf h_1},{\bf h_2}$ were be arbitrary,
	$z\in	\modL^{{\mathfrak e}(\bar \varphi)}_n[\frakss]$.
\end{proof}

So, if $\frakss$ is a context whose $\callE^\frakss$ consists of simple $\mathfrak e$'s  and if $\modM \in \cal K^{\frakss}$, then every $\ringR$-endomorphism is
in some sense definable, i.e., by fixing  ${\mathfrak e}\in
\calE^\frakss$ and restricting ourselves to $\varphi^{\mathfrak e}_{n}
(\modM)$ for large enough $n,$ modulo $\varphi^{\mathfrak e}_n
(\modM)$, it is determined by some $z\in\modL^{\mathfrak e}_{n}
[{\cal K}^\frakss]$. However, not every such $z$ may really occur.
We  try to formalize this in Lemma
\ref{definabilityhzen}.

\begin{notation}From now on we fix a context  $\frakss=(\mathcal K, \modM_*, \calE, \ringR, \ringS, \ringT)$.\end{notation}

\begin{Definition}
	\label{hzen}
	Suppose ${\mathfrak e}\in {\callE}$,
	$n<\omega$, $z\in \modL^{\mathfrak e}_n$
	and let $\modM$ be a bimodule  which $\leq^{ads}_{{\cal K},\aleph_0}$-extends
	$\modM_\ast$.
	We define ${\bf h}_{\modM,z}^{{\mathfrak e},n}$   as a morphism from the additive group $\psi^{\mathfrak e}_n(\modM)/
	\varphi^{\mathfrak
		e}_\omega(\modM)$\footnote{Pedantically we should write $\psi^{\mathfrak e}_n(\modM)/\varphi^{\mathfrak e}_\omega
		(\modM)\cap \psi^{\mathfrak e}_n (\modM).$} to itself so that for every bimodule   homomorphism   $h:\modN_n^{\mathfrak e}\to\modM$
	\[
	{\bf h}_{\modM,z}^{{\mathfrak e},n}\left(h(x_n^{\mathfrak e})+\varphi^{\mathfrak e}_\omega(\modM)\right):=h(z)+
	\varphi_\omega(\modM).
	\]
\end{Definition}
\begin{remark}
	By Lemma \ref{formula vs hom}, every $w \in \psi^{\mathfrak e}_n(\modM)$ is of the form $h(x_n^{\mathfrak e})$, for some $h:\modN_n^{\mathfrak e}\to\modM$ as above. Also note that as  $z \in \modL^{\mathfrak e}_n$, if $h':\modN_n^{\mathfrak e} \to \modM$  is another bimodule
	homomorphism	such that $h'(x_n^{\mathfrak e})=h(x_n^{\mathfrak e})$, then $h'(z) =h(z)$ mod $\varphi^{\mathfrak e}_\omega(\modM)$.
	This shows that ${\bf h}^{\mathfrak e, n}_{\modM, z}$  is well-defined, and does not depend on the choice of ${h}$.	
\end{remark}
\begin{Definition}
	\label{2.11} Let ${\mathfrak e}\in {\callE}$,
	$n<\omega$, $z\in \modL^{\mathfrak e}_n$
	and let $\modM$ be a bimodule  such that
	$\modM_\ast \leq_{\aleph_0}\modM$. Then
	\begin{enumerate}
		\item $z \in \modL^{\mathfrak e}_n$ is called  $(\frakss,{\mathfrak e},n)$-nice
		if when $h:\modN^{\mathfrak e}_n\longrightarrow \modM$ is a bimodule homomorphism, $m \geq n$
		and   $\modM\models\psi^{\mathfrak e}_{m}(h(x^{\mathfrak e}_{n}))$,
		then $\modM\models\psi^{\mathfrak e}_{m}(h(z))$.
		
		
		\item $z \in \modL^{\mathfrak e}_n$ is called $(\frakss,
		{\mathfrak e})$-nice if
		it is  $(\frakss,{\mathfrak e},n)$-nice for every $n$.
		
		\item $z \in \modL^{\mathfrak e}_n$ is called weakly
		$(\frakss,{\mathfrak e})$-nice if
		there is an infinite ${\cal U} \subseteq \omega$  such that $z$
		is $(\frakss,{\mathfrak e}, n)$-nice
		for every $n\in{\cal U}$.
		
	\end{enumerate}
	We remove  $(\frakss,
	{\mathfrak e})$, when it is clear from the context.
\end{Definition}
We now define a subgroup of $\modL^{{\mathfrak e},n}_n$.
\begin{Definition}
	For each $\mathfrak e \in \callE$ and $n<\omega$, we define
	$$\modL^{{\mathfrak e},\ast}_n:=\{z \in \modL^{{\mathfrak e}}_n: z \text{~is ~}(\frakss,
	{\mathfrak e})\text{-nice}\}.$$
\end{Definition}

\begin{lemma}
	\label{2.11A}Let ${\mathfrak e}$ be simple and $\modM_\kappa$ be strongly nice
	for $(\lambda,\frakss, S, \kappa)$ and let $\fucF$ be an
	$\ringR $-endomorphism of $\modM_\kappa$. The following assertions hold:
	\begin{enumerate}
		\item Let $z$ be as Lemma \ref{prnalphafez}.   Then
		$z\in\modL^{{\mathfrak e},\ast}_n$.
		\item For some $n<\omega$ and
		$z\in\modL^{{\mathfrak e},*}_n$   there is $\modK'\in
		{\rm cl_{is}} ({\cal K})$ such that
		$\modM':=\modM_*\oplus \modK'\leq_{\aleph_0}\modM_\kappa $ and $$x \in
		\psi^{\mathfrak e}_n (\modM_\kappa) \Rightarrow   {\bf f}(x)+\varphi^{\mathfrak e}_n
		(\modM_\kappa)+\modM'={\bf h}^{{\mathfrak e},n}_{\modM_\kappa,z} (x)+\varphi^{\mathfrak e}_n
		(\modM_\kappa)+\modM'.$$
	\end{enumerate}
\end{lemma}

\begin{proof}
	$(1)$: Since ${\mathfrak e}$ is simple, there is a very nice sequence
	$\overline{\varphi}$ such that  ${\mathfrak e}={\mathfrak e}(\overline{\varphi})$.
	Since $z\in\modL^{{\mathfrak e}}_n$ we know that
	$z\in\varphi_n(\modN^{{\mathfrak e}}_n)$ and $f^1_n(z)-f^2_n(z)\in
	\varphi_\omega (\modN'_n)$ where
	$f^\ell_n:\modN_n\longrightarrow \modN^\ell_n$ are   bimodule homomorphisms defined in Definition \ref{tr}(3).
	Now suppose ${\bf h}:\modN_n\to\modM_\kappa$ is a   bimodule homomorphism, $m\geq n$ and suppose
	$$\modM_\kappa \models\psi^{\mathfrak e}_{m}({\bf h}(x^{\mathfrak e}_n)).$$
	We are going to show that $$\modM_\kappa \models\psi^{\mathfrak e}_{m}({\bf h}(z)).$$
	Recall that ${\bf f}({\bf h}(x_n))-{\bf h}(z)\in \modM_\alpha+\psi^{\mathfrak e}_{\omega}(\modM_\kappa)$.
	This gives an element $y\in\modM_\alpha$ such that $${\bf f}({\bf h}(x_n))-{\bf h}(z)+y\in\psi^{\mathfrak e}_{\omega}(\modM_\kappa).$$
	In particular, $\modM_\kappa \models\psi^{\mathfrak e}_{m}({\bf f}({\bf h}(x^{\mathfrak e}_n))-{\bf h}(z)+y)$.
	As $\overline{\varphi}$ is very simple,
	$$\modM_\kappa \models\varphi^{\mathfrak e}_{m}({\bf f}({\bf h}(x^{\mathfrak e}_n))-{\bf h}(z)+y).$$
	We conclude from this that there is a bimodule homomorphism $H_1:\modN_n^{\mathfrak e}\to\modM_\kappa$ such that
	$H_1(x_m)={\bf f}({\bf h}(x_n))-{\bf h}(z)+y$.  Since $\modM_\kappa \models\psi^{\mathfrak e}_{m}({\bf h}(x^{\mathfrak e}_n))$,
	we have  $\modM_\kappa \models\psi^{\mathfrak e}_{m}({\bf f}({\bf h}(x^{\mathfrak e}_n)))$. This gives a bimodule homomorphism $H_2:\modN^{\mathfrak e}_n\to\modM_\kappa$ such that
	$H_2(x^{\mathfrak e}_m)={\bf f}({\bf h}(x^{\mathfrak e}_n))$.  Define
	$K:\modN^{\mathfrak e}_n\to\modM_\kappa$ by  $K:=H_2-H_1$. Clearly, $K(x^{\mathfrak e}_m)={\bf h}(z)+y$. So,
	$$\modM_\kappa \models\varphi^{\mathfrak e}_{m}({\bf h}(z)+y).$$
	Take $\modM\leq_{\aleph_0}\modM_* $ be such that ${\bf h}(z)\in\modM$ and $\modM \oplus \modM_\alpha\leq_{\aleph_0}\modM_\kappa $. Then
	$$\modM \oplus \modM_\alpha \models\psi^{\mathfrak e}_{m}({\bf h}(z)+y).$$
	So, $\modM   \models\psi^{\mathfrak e}_{m}({\bf h}(z))$ and then $  \modM_\kappa \models\psi^{\mathfrak e}_{m}({\bf h}(z))$, as required.
	
	$(2)$: In view of Lemmas \ref{prnalphafe} and \ref{prnalphafez}, there are $n<\omega$ and $\alpha\in\kappa\setminus S$
	such that the property $(\Pr 1)^{n}_{\alpha,z}[\fucF,{\mathfrak e}]$
	holds. Since ${\mathfrak e}$ is simple, by Lemma \ref{simplephiequalpsi} we have
	$\psi^{\mathfrak e}_n (\modM)\equiv\psi^{\mathfrak e}_n (\modM)$.
	
	Now let $x \in
	\psi^{\mathfrak e}_n (\modM_\kappa)$. Then $x \in
	\varphi^{\mathfrak e}_n (\modM_\kappa) $.  This give us a  bimodule homomorphism ${\bf h}:\modN^{\mathfrak e}_{n}\to\modM_\kappa$
	such that ${\bf h}(x^{\mathfrak e}_n)=x$.
	Since $(\Pr 1)^{n}_{\alpha,z}[\fucF,{\mathfrak e}]$ holds, thus
	we have
	$$
	{\bf f}(x)-{\bf h}(z)={\bf f}({\bf h}(x^{\mathfrak e}_n))-{\bf h}(z)\in \modM_\alpha+\varphi^{\mathfrak e}_\omega(\modM_\kappa).
	$$
	Also, recall that $${\bf h}^{{\mathfrak e},n}_{\modM,z} (x+\varphi^{\mathfrak e}_\omega(\modM_\kappa))={\bf h}(z)+\varphi^{\mathfrak e}_\omega(\modM_\kappa).$$
	This completes the argument.
\end{proof}
\begin{remark}In Lemma \ref{2.11A}, if $||\modM_*||<\lambda$ and $||\modK||<\lambda$
	for every $\modK\in \modK$,
	then $||\modM'||<\lambda= ||\modM_\kappa||$.
\end{remark}

\begin{Definition}
	\label{2.12}
	\begin{enumerate}
		\item Let $\modM$ be an  $\ringR$-module, $\mathfrak e \in \callE$ and $n<\omega.$
		\begin{enumerate}
			\item For $\fucF\in \End(\modM)$, we set $\hat{\fucF}_n:=\fucF\lceil\varphi^{\mathfrak e}_n(\modM)/
			\varphi^{\mathfrak e}_\omega (\modM).$
			\item
			Let
			$\End^{{\mathfrak e},n}(\modM):=\ \big\{\hat{\fucF}_n:\fucF\in \End(\modM)\big\}$\footnote{Recall that
				$\End(\modM):=\End_{\ringR}(\modM)$ is the ring of endomorphisms of $\modM$.}.
			
			\item  Let $\Upsilon_{n, \lambda}(\modM)$ be  the family of all $\fucF \in \End (\modM)$
			such that for some $A\subseteq \modM$ of
			cardinality $<\lambda$ we have
			$
			\Rang(\fucF\rest \varphi^{\mathfrak e}_n(\modM))
			\subseteq\{x+\varphi^{\mathfrak e}_\omega
			(\modM):x\in\varphi^{\mathfrak e}_n(\langle A\rangle_\modM)\}$. 	\item
			 $\End^{{\mathfrak e},n}_{<\lambda}(\modM):=\ \big
			\{\hat{\fucF}_n\in \End^{{\mathfrak e},n}
			(\modM):\fucF \in \Upsilon_{n, \lambda}(\modM)\},$
			which is  a two-sided ideal of $\End^{{\mathfrak e},n}(\modM)$.
			\item Let $n\leq  m$. The notation ${\bf h}^{{\mathfrak e},n,m}_{<\lambda}[\modM]$ stands for the
			natural map from
			$
			\End^{{\mathfrak e},n}_{<\lambda}
			(\modM)$ to $\End^{{\mathfrak e},m}_{<\lambda}(\modM)$.
			In particular, $\{
			\End^{{\mathfrak e},n}_{<\lambda}(\modM);{\bf h}^{{\mathfrak e},n,m}_{<\lambda}[\modM]\}$
			is a directed system. \item $\End^{{\mathfrak e},\omega}_{<\lambda}(\modM):=\varinjlim(	\cdots\longrightarrow	 \End^{{\mathfrak e},n}_{<\lambda}(\modM)  \longrightarrow		\End^{{\mathfrak e},n+1}_{<\lambda}(\modM)\longrightarrow\cdots).$
\item We denote the natural maps:
			$${\bf h}^{{\mathfrak e},n, \omega}_{<\lambda}[\modM] :
			\End^{{\mathfrak e},n}_{<\lambda}
			(\modM) \to\End^{{\mathfrak e},\omega}_{<\lambda}(\modM).$$
		\end{enumerate}
		\item  Suppose in addition  that    $\modM$ is an $(\ringR,\ringS)$-bimodule.
		\begin{enumerate}
			\item For any $\fucF\in \End_{\ringR}(\modM)$, we assign $\hat{\fucF}_n:=\fucF\lceil\varphi^{\mathfrak e}_n(\modM)/
			\varphi^{\mathfrak e}_\omega (\modM).$\footnote{There should be no confusion with clause (1)(a) above, as here we are talking about a bimodule $\modM.$}
			\item The ring of $\ringR$-endomorphisms of $\modM$ induces the following ring$$\End^{{\mathfrak e},n,*}(\modM):=\{\hat{\fucF}_n:\fucF\in
			\End_\ringR (\modM)\}.$$
			\item Let $\Upsilon^*_{n, \lambda}(\modM)$ be  the family of all $\fucF \in \End_{\ringR}(\modM)$
			such that for some $A\subseteq \modM$ of
			cardinality $<\lambda$ we have $\Rang(\fucF\rest \psi^{\mathfrak e}_n
			(\modM))\subseteq \{x+\varphi^{\mathfrak e}_\omega (\modM):
			x\in \varphi^{\mathfrak e}_n(\langle A\rangle_\modM)\}$. 	\item
			 $\End_{<\lambda}^{{\mathfrak e},n,*}:= \{\hat{\fucF}_n:\fucF\in\Upsilon^*_{n, \lambda}(\modM)\}.$
		\end{enumerate}
	\end{enumerate}
\end{Definition}
It is easily seen that all the above defined notions are rings.
We now define expansions of $\varphi^{\mathfrak e}_n(\modM)/\varphi^{\mathfrak e}_\omega(\modM)$.
\begin{Definition}
	Let $\modM, \mathfrak e$ and $n$ be as above.
	\begin{enumerate}
		\item The notation
		${\mathfrak B}^{\mathfrak e}_{n}(\modM)$ stands for
		$\varphi^{\mathfrak e}_n(\modM)/\varphi^{\mathfrak e}_\omega(\modM)$
		expanded by the finitary relations definable by formulas in $\cal L_{\infty, \omega}^{\text{pe}}(\tau_{\ringR})$
		(so actually even
		if we use this notation for a bimodule $\modM$, it counts only as an
		$\ringR$-module).
		\item Similarly, we define ${}^+{\mathfrak B}^{\mathfrak e}_n(\modM)$, where  p.e. formulas are
		replaced by ``formulas preserved by direct sums''.
		\item For a bimodule  $\modM$ we define ${\mathfrak B}^{\mathfrak e}_n(\modM)$ and $
		{}^+{\mathfrak B}^{\mathfrak e}_n(\modM)$ similarly let restricting ourselves
		to $\psi^{\mathfrak e}_n(\modM)$.
	\end{enumerate}
\end{Definition}
So,
\[
{\mathfrak B}^{\mathfrak e}_{n}(\modM)=(\varphi^{\mathfrak e}_n(\modM)/\varphi^{\mathfrak e}_\omega(\modM), \langle R  \rangle_{R \in \cal R}),
\]
where $\cal R$ consists of all finitary relations defined by a formula from
$\cal L_{\infty, \omega}^{\text{pe}}(\tau_{\ringR})$. Similarly
\[
{}^+{\mathfrak B}^{\mathfrak e}_{n}(\modM)=(\varphi^{\mathfrak e}_n(\modM)/\varphi^{\mathfrak e}_\omega(\modM), \langle R  \rangle_{R \in {}^+\cal R}),
\]
where ${}^+\cal R$ consists of all finitary relations defined which are defined by a formula from
$\cal L_{\infty, \omega}(\tau_{\ringR})$ which is preserved under direct sums.
Since, by Lemma \ref{formulas and direct sum}, pe-formulas are preserved by direct limits, ${}^+{\mathfrak B}^{\mathfrak e}_{n}(\modM)$ expands ${\mathfrak B}^{\mathfrak e}_{n}(\modM)$.

\begin{lemma}
	\label{2.13}The following assertions are hold:
	\begin{enumerate}
		\item Adopt the notation of Definition \ref{2.12}(1). Then  $\End^{{\mathfrak e},n}(\modM)$ is a ring with $1$ and
		$\End^{{\mathfrak e},n}_{<\lambda}(\modM)$ is a  two-sided ideal of $\End^{{\mathfrak e},n}(\modM)$.
		Note that this ideal may be proper, i.e.,  $1\notin \End^{{\mathfrak e},n}_{<\lambda}(\modM)$ or not.
		Suppose in addition $\modM$
		is a bimodule, then $\ringS$ is mapped naturally.
		\item $\End^{{\mathfrak e},n}_{<\lambda}(\modM)$ is a two-sided subideal of $\End^{{\mathfrak
				e},n}_{<\mu}(\modM)$ when $\lambda<\mu$.
		\item
		$\End^{{\mathfrak e},n}_{<\|\modM\|^+}(\modM)=
		\End^{\mathfrak e,n}(\modM)$.
		\item If $\modM_1,\modM_2$ are $\ringR$-modules and ${\bf h}$ is an $\ringR$-homomorphism from $\modM_1$ to
		$\modM_2$, then  ${\bf h}$
		induces a homomorphism  from
		${\mathfrak B}^{\mathfrak e}_n(\modM_1)$ into ${\mathfrak B}^{\mathfrak e}_n(\modM_2)$.
	\end{enumerate}
\end{lemma}
\begin{proof}
	(1).
	It is clear that all the defined notions are rings (not necessarily with $1$).
	Now suppose $\modM$ is a bimodule. Then one can easily show that for each
	$s\in \ringS$,
	$s$ defines
	an $\ringR$-endomorphism of $\modM$ by
	$x\rightarrow x s.$
	
	Items (2) and (3) are clear. To prove (4), first note that, as ${\bf h}(\varphi_n^{\mathfrak e}(\modM_1)) \subseteq \varphi_n^{\mathfrak e}(\modM_2)$,
	\[
	\hat{\bf h}_n= {\bf h} \restriction \varphi_n^{\mathfrak e}(\modM_1) / \varphi_\omega^{\mathfrak e}(\modM_1):\varphi_n^{\mathfrak e}(\modM_1) / \varphi_\omega^{\mathfrak e}(\modM_1) \to \varphi_n^{\mathfrak e}(\modM_2) / \varphi_\omega^{\mathfrak e}(\modM_2)
	\]
	is well-defined.
	Now let $\varphi(\nu_1, \cdots, \nu_n) \in \cal L_{\infty, \omega}^{\text{pe}}(\tau_{\ringR}),$
	and for $\ell=1, 2$ set
	\[
	R^\ell=\{\langle x_1, \cdots, x_n \rangle \in \modM_\ell: \modM_\ell \models \varphi(x_1,  \cdots, x_n)             \}.
	\]
Clearly ${\bf h}(R^1)=R^2$,
	and hence $\hat{\bf h}_n(R^1+\varphi_\omega(\modM_1))=R^2+\varphi_\omega(\modM_2)$.
	The result follows immediately.
\end{proof}
\begin{remark}
	\label{fixinghathn}
	If ${\bf h}$ is an $\ringR$-homomorphism from $\modM_1$ to
	$\modM_2$ and $n<\omega$, then we use $\hat{\bf h}_n$ for the homomorphism from ${\mathfrak B}^{\mathfrak e}_n(\modM_1)$ into ${\mathfrak B}^{\mathfrak e}_n(\modM_2)$ from Lemma \ref{2.13}(4).
\end{remark}

We now define some rings derived from the ring of
$\ringR$-endomorphism of bimodules. Before doing that we need the following lemma which shows
that the maps ${\bf h}^{{\mathfrak e},n}_{\modM,z}$ from Definition \ref{hzen} are definable.

\begin{lemma}
	\label{definabilityhzen}
	Suppose $\modM$ is a bimodule, $\mathfrak e \in \callE$ is simple and $n<\omega$.
	\begin{enumerate}
		\item Suppose ${\bf h}: \modM_1 \to \modM_2$ is an $\ringR$-homomorphism, $\mathfrak e \in \callE$, $n<\omega$
		and $z \in \modL^{\mathfrak e}_n$. Then ${\bf h}^{\mathfrak e, n}_{\modM_1, z} \circ \hat{\bf h}_n=\hat{\bf h}_n \circ {\bf h}^{\mathfrak e, n}_{\modM_2, z}$.
		
		\item Suppose $z \in \modL^{\mathfrak e}_n$ is $n$-nice and $m>n.$ Then there is $y \in \modN_m^{\mathfrak e}$ such that for every bimodule
		$\modM,$
		\[
		{\bf h}^{\mathfrak e, m}_{\modM, y}={\bf h}^{\mathfrak e, n}_{\modM, z} \restriction (\psi_n(\modM) / \varphi_\omega(\modM)).
		\]
		
		\item Suppose $\psi(x, y) \in \cal L_{\infty, \omega}^{\rm pe}(\tau_{(\ringR, \ringS)})$ is such that
		\begin{enumerate}
			\item[(i)] $\psi_n^{\mathfrak e}(x) \vdash \exists y \psi(x, y)$,
			\item[(ii)] $\psi(x, y) \vdash \psi_n^{\mathfrak e}(x)~ \wedge ~\psi_n^{\mathfrak e}(y)$,
			\item[(iii)] $\psi(x, y_1) \wedge \psi(x, y_2) \vdash \psi_\ell^{\mathfrak e}(y_1-y_2),$ for all $\ell < \omega$.
		\end{enumerate}
		Then there is some  $z \in \modL^{\mathfrak e}_n$ such that
		$(\star)^n_{\psi, z}:$ for every bimodule $\modM$ and  $x, y \in \psi_n(\modM)$,
		\[
		\modM \models \psi(x, y) \iff {\bf h}^{\mathfrak e, n}_{\modM, z}(x+\varphi^{\mathfrak e}_\omega(\modM))= y + \varphi^{\mathfrak e}_\omega(\modM).
		\]
		\item For every  $z \in \modL^{\mathfrak e}_n$, there exists
		$\psi(x, y) \in \cal L_{\infty, \omega}^{\rm pe}(\tau_{(\ringR, \ringS)})$ satisfying the above conditions, such that $(\star)^n_{\psi, z}$ holds.
		
		\item If $z_1, z_2 \in \modL^{\mathfrak e}_n,$ then for some $z_3 \in \modL^{\mathfrak e}_n$ and for all bimodule  $\modM,$
		\[
		{\bf h}^{\mathfrak e, n}_{\modM, z_3}={\bf h}^{\mathfrak e, n}_{\modM, z_1} \circ {\bf h}^{\mathfrak e, n}_{\modM, z_2}.
		\]
		Furthermore,
		\[
		{\bf h}^{\mathfrak e, n}_{\modM, z_1} \pm {\bf h}^{\mathfrak e, n}_{\modM, z_2}= {\bf h}^{\mathfrak e, n}_{\modM, z_1 \pm z_2}.
		\]
		
		\item If $z \in \modL^{\mathfrak e}_n$ and ${\bf h}^{\mathfrak e, n}_{\modM, z}$ is a bijection, then for some $z'\in \modL^{\mathfrak e}_n$
		and for all bimodule  $\modM,$  ${\bf h}^{\mathfrak e, n}_{\modM, z'}$ is the inverse of ${\bf h}^{\mathfrak e, n}_{\modM, z}$.
		
		\item If $\modN^{\mathfrak e}_n$ is finitely presented, then the formula $\psi(x, y)$ is first order.
		Furthermore, if $\modN^{\mathfrak e}_n$ is
		 generated by $\{y_i:i<m_n\}$ freely  except
		equations involving $r\in \ringR$ only
		(no $s\in \ringS$), and $z\in\sum\limits_{i<
			m_n} \ringR y_i$, then
		$\psi(x, y) \in \cal L_{\omega, \omega}^{\rm pe}(\tau_{\ringR})$.
	\end{enumerate}
\end{lemma}
\begin{proof}
	(1). Let ${\bf g}: \modN_n^{\mathfrak e} \to \modM_1$ be a bimodule homomorphism and let $x={\bf g}(x_n^{\mathfrak e})$. Then
	\[
	{\bf h}^{\mathfrak e, n}_{\modM_2, z}(\hat{\bf h}_n(x+\varphi_\omega^{\mathfrak e}(\modM_1)))={\bf h}^{\mathfrak e, n}_{\modM_2, z}({\bf h}(x)+\varphi_\omega^{\mathfrak e}(\modM_2)))={\bf h}({\bf g}(z)) + \varphi_\omega^{\mathfrak e}(\modM_2).
	\]
	In a similar way, we have
	\[
	\hat{\bf h}_n ({\bf h}^{\mathfrak e, n}_{\modM_1, z}(x)) = \hat{\bf h}_n ({\bf g}(z)+\varphi_\omega^{\mathfrak e}(\modM_1))=
	{\bf h}({\bf g}(z))+\varphi_\omega^{\mathfrak e}(\modM_2)).
	\]
	The result follows immediately.
	
	(2). Set $y=g_{n, m}(z)$. It is easily seen that $y$ is as required.
	
	(3). Let $z$ be such that $\psi_n^{\mathfrak e}(x_n^{\mathfrak e}) \vdash  \psi(x_n^{\mathfrak e}, z)$. Such a $z$
	exists by (i), and it is unique mod $\varphi_\omega^{\mathfrak e}(\modM)$ by (iii).
	Now define ${\bf h}^{\mathfrak e, n}_{\modM, z}$ as given. By items (i) and (iii), ${\bf h}^{\mathfrak e, n}_{\modM, z}$
	is a well-defined function and by item (ii) we deduce that
	\begin{enumerate}
	\item[ ] $	\dom({\bf h}^{\mathfrak e, n}_{\modM, z})= \psi_n^{\mathfrak e}(\modM)/\varphi_\omega^{\mathfrak e}(\modM),$
	\item[ ] $\Rang({\bf h}^{\mathfrak e, n}_{\modM, z})= \psi_n^{\mathfrak e}(\modM)/\varphi_\omega^{\mathfrak e}(\modM).$
\end{enumerate}
	So, we are done.
	
	(4). Suppose  $z \in \modL^{\mathfrak e}_n$ is given, and let $\modN_n^{\mathfrak e}$
	be as Definition \ref{c.3}. Then for some $r^*, r^*_{i} \in \ringR$ and $s^*, s^*_i \in \ringS,$ for $i< k_{m_n-1}$, we have
	\[
	z=r^* x_n s^* + \sum\limits_{i< k_{m_n-1}} r^*_i y_{n, i} s^*_i\quad(+)
	\]
	Let $\psi(x, y)$ be the formula
	\[
	\psi_n^{\mathfrak e}(x)~ \wedge ~\psi_n^{\mathfrak e}(y) \wedge y=r^* x s^*.
	\]
	We are going to show that $\psi(x, y)$ is as required. Clauses (i)-(iii) are clearly satisfied.
	To prove $(\star)^n_{\psi, z}$, let $\modM$ be a bimodule and $x, y \in \psi_n^{\mathfrak e}(\modM)$.
	Suppose first that $\psi(x, y)$ holds. Let also ${\bf g}: \modN_n^{\mathfrak e} \to \modM$ be a bimodule homomorphism,
	defined on generators $x_n, y_{n, i}$ by
	\[
	{\bf g}(x_n)= x, \qquad {\bf g}(y_{n, i})=0.
	\]
In view of $(+)$ we observe that
	\[
	{\bf g}(z)= r^* x s^*=y.
	\]
	According to th definition
	\[
	{\bf h}^{\mathfrak e, n}_{\modM, z}(x+\varphi^{\mathfrak e}_\omega(\modM))= y = \varphi^{\mathfrak e}_\omega(\modM).
	\]
	
	Conversely, suppose that ${\bf h}^{\mathfrak e, n}_{\modM, z}(x+\varphi^{\mathfrak e}_\omega(\modM))= y = \varphi^{\mathfrak e}_\omega(\modM).$
	Let ${\bf g}$ be defined as above, so that ${\bf g}(x_n)=x$. Now
	it is clear that ${\bf g}(z)= r^* x s^*=y$,
	and thus $\psi(x, y)$ holds.
	
	Clauses (5) and  (6) follows from a combination of (4) and (5). Clause (7) is clear as well.
	The lemma follows.
\end{proof}
The above lemma allows us to define ${\bf h}^{\mathfrak e, n}_{z}$, independent of the choice of the bimodule $\modM$.
\begin{Definition}
	\label{def rings}
	Suppose $\mathfrak e \in \callE$ is simple as witness by $X$ (see Definition \ref{c.1}(4)) and $n<\omega.$
	\begin{enumerate}
		\item Let $\ringDE^{\mathfrak e}_n$ be the following ring, whose universe is:
		\[
		\{ {\bf h}^{{\mathfrak e},n}_z: z\in \modL^{{\mathfrak e}}_{n} \},
		\]
		such that:
		\begin{enumerate}
			\item ${\bf h}^{{\mathfrak e},n}_{z_1}={\bf h}^{{\mathfrak e},n}_{z_2}$ if and only if
			$z_1-z_2\in\varphi^{\mathfrak e}_\omega(\modN^{\mathfrak e}_n)$,
			\item ${\bf h}^{{\mathfrak e},n}_{z_1}\pm {\bf h}^{{\mathfrak e},n}_{z_2}={\bf h}^{{\mathfrak e},n}_{z_1 \pm z_2}$,
			\item ${\bf h}^{\mathfrak e, n}_{z_1} \circ {\bf h}^{\mathfrak e, n}_{z_2}={\bf h}^{\mathfrak e, n}_{z_3}$,
			where $z_3$ is as in Lemma \ref{definabilityhzen}(5) ( it is
			unique modulo $\varphi_\omega^{\mathfrak e}(\modN^{\mathfrak e}_n)$),
			
			\item the zero is ${\bf h}^{{\mathfrak e},n}_0$, the one is
			${\bf h}_{x^{\mathfrak e}_n}^{{\mathfrak e},n}.$
		\end{enumerate}
		Note that $\ringDE^n$
		is embedded into the endomorphism ring of
		$\psi^{\mathfrak e}_n(\modN^{\mathfrak e}_n)/\varphi^{\mathfrak e}_\omega(\modN^{\mathfrak e}_n)$
		as an Abelian
		group.
		
		\item $\ringDe^{\mathfrak e}_n:=\{{\bf h}_z^{{\mathfrak e},n}\in
		\ringDE^{\mathfrak e}_n:z
		\in\sum\{\ringR x:x\in X\}\}$.

		\item $\ringdE^{\mathfrak e}_n:=\big\{{\bf h}^{{\mathfrak e},n}_z\in
		\ringDE^{\mathfrak e}_n:
		{\bf h}_{\modM,z}^{{\mathfrak e},n}$ is an endomorphism of
		${\mathfrak B}^{\mathfrak e}_n(\modM)$ for every bimodule
		$\modM$ such that $\modM_\ast \leq_{\aleph_0}\modM\big\}$,
		
		\item
		$\ringdE^{\mathfrak e}_{n,\ast}:=
		\{{\bf h}^{{\mathfrak e},n}_z \in \ringdE^{\mathfrak e}_n: z$ is $n$-nice$\}$.
		\item $\ringde^{\mathfrak e}_n=:\ringDe^{\mathfrak e}_n\cap
		\ringdE^{\mathfrak e}_n$,
		$\ringde^{\mathfrak e}_{n,\ast}=:\ringDe^{\mathfrak e}_n\cap
		\ringdE^{\mathfrak e}_{n,\ast}$, so they two depend on $X$ which
		witnesses that ${\mathfrak e}$ is simple.
		\item $\ringDE^{\mathfrak e}_n(\ringR)$ is
		$\ringDE^{\mathfrak e}_n$, when we choose
		$\ringS=\ringT=\Cent(\ringR)$;
		similarly for the others.
	\end{enumerate}
\end{Definition}

In the following diagram we display these rings. By $A\to B$ we mean $A$ is a subring of $B$:

$$
\begin{tikzpicture}\tiny{
	\matrix (m) [matrix of math nodes, row sep=3em,
	column sep=3em]{
 	& \ringde^{\mathfrak e}_{n}& & \ringDe^{\mathfrak e}_{n}\\
		\ringdE^{\mathfrak e}_{n } & & \ringDE^{\mathfrak e}_{n} & \\
		& \ringde^{\mathfrak e}_{n+1} & & \ringDe^{\mathfrak e}_{n+1} \\
		\ringdE^{\mathfrak e}_{n+1} & & \ringDE^{\mathfrak e}_{n+1} & \\};
	\path[-stealth]
	(m-1-2) edge (m-1-4) edge (m-2-1)
	edge [densely dotted] (m-3-2)
	(m-1-4) edge (m-3-4) edge (m-2-3)
	(m-2-1) edge [-,line width=6pt,draw=white] (m-2-3)
	edge (m-2-3) edge (m-4-1)
	(m-3-2) edge [densely dotted] (m-3-4)
	edge [densely dotted] (m-4-1)
	(m-4-1) edge (m-4-3)
	(m-3-4) edge (m-4-3)
	(m-2-3) edge [-,line width=6pt,draw=white] (m-4-3)
	edge (m-4-3);}
\end{tikzpicture}
$$

\begin{lemma}
	\label{pr rings}The following assertions are hold:
	\begin{enumerate}
		\item $\ringDE^{\mathfrak e}_n$ is a ring. Furthermore if $\mathfrak e$ is simple, then $\ringDe^{\mathfrak e}_n$
		and $\ringdE^{\mathfrak e}_n$ and
		$\ringde^{\mathfrak e}_n$
		are subrings of $\ringDE^{\mathfrak e}_n$ (all have the unit $1={\bf h}^{{\mathfrak e},n}_{x^{\mathfrak e}_n}$
		and zero
		${}^n{\bf h}^{\mathfrak e}_0$, and extend $\ringT$).
		\item $\ringDe^{\mathfrak e}_n$, $\ringdE^{\mathfrak e}_n$ commute. In particular,
		$\ringde^{\mathfrak e}_n$ is
		commutative.
		\item There is a natural homomorphism from
		$\ringDE^{\mathfrak e}_n$ to $\ringDE^{\mathfrak
			e}_{n+1}$ ($n<\omega$), the direct limit is denoted by $\ringDE^{\mathfrak
			e}$. Similarly for $\ringDe^{\mathfrak e}_n, \ringdE^{\mathfrak e}_{n}$ and
		$\ringde^{\mathfrak e}_n$,
		
		\item The ring $\ringS$ is
		naturally mapped into $\ringdE^{\mathfrak e}_n$,
		\item The ring
		 $\ringde^{\mathfrak e}_n$ is naturally embedded into
		$\ringde^{\mathfrak e}_{n+1}$ and $\ringDE^{\mathfrak e}_n$ into
		$\ringDE^{\mathfrak e}_{n+1}$.
		\item The Abelian group $\psi^{\mathfrak e}_n(\modM)/\varphi^{\mathfrak e}_\omega
		(\modM)$ is  equipped with a
		module structure over $\ringDE^{\mathfrak e}_n$ and it is naturally a
		$(\ringDe^{\mathfrak e}_n,
		\ringdE^{\mathfrak e}_n)$--bi module (with $\ringde^{\mathfrak e}_n$
		playing the role of $\ringT$).
	\end{enumerate}
\end{lemma}

\begin{proof}
	(1). This is clear.
	
	(2).  Suppose $z, w \in \modL^{{\mathfrak e}}_{n}$ are such that $ {\bf h}_{z}^{\mathfrak e, n} \in \ringDe^{\mathfrak e}_n$
	and ${\bf h}_{w}^{\mathfrak e, n} \in \ringdE^{\mathfrak e}_n$. Without loss of generality, $z=rx$ for some $r \in \ringR$ and $x \in X$.
	Furthermore, since  ${\bf h}_{z}^{\mathfrak e, n} \in \ringDe^{\mathfrak e}_n$, it preserves $pe$-definable relations,
	so we can assume without loss of generality that  $X=\{ x_n\} \cup \{ y_{n, i}: i < k_{m_n-1}    \},$
	the canonical sequence is taken from  Definition \ref{c.3}.
	Let also
	$w= r^* x_n s^* + \sum\limits_{i < k_{m_n-1}} r_i y_{n, i} s_i$. We have to show that
	$${\bf h}_{z}^{\mathfrak e, n} \circ {\bf h}_{w}^{\mathfrak e, n}={\bf h}_{w}^{\mathfrak e, n} \circ {\bf h}_{z}^{\mathfrak e, n}.$$
	\\
	{\bf {Case 1:}} $z=r y_{n, i}$ for some $i < k_{m_n-1}.$
	
	 In this case, let $\psi_z(x, y)=\psi_n^{\mathfrak e}(x) \wedge \psi_n^{\mathfrak e}(y) \wedge y=0$ and  $\psi_w(x, y)=\psi_n^{\mathfrak e}(x) \wedge \psi_n^{\mathfrak e}(y) \wedge y=r^* x s^*$.
	It follows from the proof of Lemma \ref{definabilityhzen}(4) that, for each $\modM \geq_{\aleph_0} \modM_*$ and $x \in \psi_n^{\mathfrak e}(\modM)$,
	\[
	{\bf h}_{\modM, z}^{\mathfrak e, n}(x+ \varphi_\omega^{\mathfrak e}(\modM))= y+ \varphi_\omega^{\mathfrak e}(\modM) \iff \psi_z(x, y)
	\]
	and
	\[
	{\bf h}_{\modM, w}^{\mathfrak e, n}(x+ \varphi_\omega^{\mathfrak e}(\modM))= y+ \varphi_\omega^{\mathfrak e}(\modM) \iff \psi_w(x, y).
	\]
	
	In particular, we have
	
	$${\bf h}_{\modM, w}^{\mathfrak e, n}({\bf h}_{\modM, z}^{\mathfrak e, n}(x+ \varphi_\omega^{\mathfrak e}(\modM)))=y+ \varphi_\omega^{\mathfrak e}(\modM)$$
	$$\Updownarrow$$
	$$ \exists v \left( \psi_z(x, v) \wedge \psi_w(v, y)              \right) $$
	$$\Updownarrow$$
	$$\exists v \left(   \psi_n^{\mathfrak e}(x) \wedge \psi_n^{\mathfrak e}(v) \wedge \psi_n^{\mathfrak e}(y)  \wedge v=0 \wedge y = r^* v s^*         \right).$$
	Similarly, we have
	$${\bf h}_{\modM, z}^{\mathfrak e, n}({\bf h}_{\modM, w}^{\mathfrak e, n}(x+ \varphi_\omega^{\mathfrak e}(\modM)))=y'+ \varphi_\omega^{\mathfrak e}(\modM)$$
	$$\Updownarrow$$
	$$ \exists v \left( \psi_w(x, v) \wedge \psi_z(v, y')              \right) $$
	$$\Updownarrow$$
	$$\exists v \left(   \psi_n^{\mathfrak e}(x) \wedge \psi_n^{\mathfrak e}(v) \wedge \psi_n^{\mathfrak e}(y')  \wedge v= r^* x s^* \wedge y' =0\right).$$
	It follows from  the above equations that $y=y'=0$, and thus the equality follows.
	\\
	{\bf{Case 2:}} $z=r x_n$.

	In this case, for each $x \in \modM$, ${\bf h}_{\modM, z}^{\mathfrak e, n}(x+ \varphi_\omega^{\mathfrak e}(\modM))=rx+ \varphi_\omega^{\mathfrak e}(\modM)$. Set ${\bf h}_{\modM, w}^{\mathfrak e, n}(x+ \varphi_\omega^{\mathfrak e}(\modM))=y+\varphi_\omega^{\mathfrak e}(\modM).$ Then
	\[
	{\bf h}_{\modM, w}^{\mathfrak e, n}({\bf h}_{\modM, z}^{\mathfrak e, n}(x+ \varphi_\omega^{\mathfrak e}(\modM)))= {\bf h}_{\modM, w}^{\mathfrak e, n}(rx+ \varphi_\omega^{\mathfrak e}(\modM)) = r {\bf h}_{\modM, w}^{\mathfrak e, n}(x+ \varphi_\omega^{\mathfrak e}(\modM))=ry+\varphi_\omega^{\mathfrak e}(\modM),
	\]
	and
	\[
	{\bf h}_{\modM, z}^{\mathfrak e, n}({\bf h}_{\modM, w}^{\mathfrak e, n}(x+ \varphi_\omega^{\mathfrak e}(\modM)))= {\bf h}_{\modM, z}^{\mathfrak e, n}(y+\varphi_\omega^{\mathfrak e}(\modM))=ry+\varphi_\omega^{\mathfrak e}(\modM).
	\]
	The equality follows in this case as well.

	(3).
	Recall from Lemma \ref{definabilityhzen}(2) that the assignment ${\bf h}^{\mathfrak e, n}_{\modM, z}\mapsto{\bf h}^{\mathfrak e, n+1}_{\modM, g_n(z)}$ defines an embedding
	map $F_n:\ringDE^{\mathfrak e}_n\to \ringDE^{\mathfrak e}_{n+1}$. This yields a directed system $\{\ringDE^{\mathfrak e}_n\}_{n\geq 1}$. Now, we define$$\ringDE^{\mathfrak e}:=\varinjlim(\ringDE^{\mathfrak e}_0\stackrel{F_0}\longrightarrow\ringDE^{\mathfrak e}_1\longrightarrow\cdots\longrightarrow\ringDE^{\mathfrak e}_n\stackrel{F_n}\longrightarrow\ringDE^{\mathfrak e}_{n+1}\longrightarrow\cdots).$$
	Similarly, we define the rings $\ringDe^{\mathfrak e}, \ringdE^{\mathfrak e}$
	$\ringdE^{\mathfrak e}_{\ast}$ and $ \ringde^{\mathfrak e}$
	by taking them  as a direct limit of the corresponding directed system.
	In all cases they depend on $\frakss$ of course.

	(4).	To each $s\in \ringS$ we assign ${\bf f}_s\in\End(\modM)$ defined by ${\bf f}_s(x)=xs$.
	The assignment $s\mapsto {\bf f}_s$ defines a  map
	$\ringS\longrightarrow\End(\modM)$ which is an embedding. For each $n<\omega$, this induces a map
	$$(\hat{{\bf f}_s})_n:\frac{\varphi^{\mathfrak e}_n(\modM)}{\varphi^{\mathfrak e}_\omega
		(\modM)}\longrightarrow \frac{\varphi^{\mathfrak e}_n(\modM)}{\varphi^{\mathfrak e}_\omega
		(\modM)}.$$
	We may regard this as an endomorphism of ${\mathfrak B}^{\mathfrak e}_n(\modM)$.
	Now, let $z:=x_ns$. Then $(\hat{{\bf f}_s})_n= {\bf h}^{{\mathfrak e},n}_z$.
	We proved that the assignment $s\mapsto (\hat{{\bf f}_s})_n$ induces a map
	$\rho_n:S\to \ringdE^{\mathfrak e}_n$. Denote the natural map
	$\ringdE^{\mathfrak e}_n\to\ringdE^{\mathfrak e}_{n+1} $ by $H_n$.
	Now we look at the following commutative diagram:
	
	\[
	\begin{CD}
 \ringS @>=>> \ringS @>=>>\ldots@>=>>\ringS	@>=>> \ringS @>=>>  \ldots\\
  @VV\rho_{1} V@V \rho_2VV @V \ldots VV   @VV\rho_{n-1} V@V \rho_nVV   \\
  \ringdE^{\mathfrak e}_{1} @>H_{1}>> \ringdE^{\mathfrak e}_2 @>H_{2}>> \ldots  @> >> \ringdE^{\mathfrak e}_{n-1} @>H_{n-1}>> \ringdE^{\mathfrak e}_n @>H_{n} >> \ldots
	\end{CD}
	\]
	
		\medskip
 Taking direct limits of these directed systems, lead us  to a natural map $$\rho:\ringS=\varinjlim \ringS\to\varinjlim \ringdE^{\mathfrak e}_n=\ringdE^{\mathfrak e},$$ as claimed.

	(5). This is similar to (4).
	
	(6). The assignment
	$$(m+\varphi^{\mathfrak e}_\omega
	(\modM),{\bf h}^{\mathfrak e, n}_{\modM, z})\mapsto{\bf h}^{\mathfrak e, n}_{\modM, z}(m+\varphi^{\mathfrak e}_\omega
	(\modM))$$
	defines the scaler
	multiplication $\frac{\psi^{\mathfrak e}_n(\modM)}{\varphi^{\mathfrak e}_\omega
		(\modM)}\times\ringDE^{\mathfrak e}_n\to\frac{\psi^{\mathfrak e}_n(\modM)}{\varphi^{\mathfrak e}_\omega
		(\modM)}$. Now, let $m+\varphi^{\mathfrak e}_\omega(\modM)\in\frac{\psi^{\mathfrak e}_n(\modM)}{\varphi^{\mathfrak e}_\omega
		(\modM)}, {\bf h}^{\mathfrak e, n}_{\modM, z}\in\ringDe^{\mathfrak e}_n$ and ${\bf h}^{\mathfrak e, n}_{\modM, w}\in\ringdE^{\mathfrak e}_n$.
	The assignment $$({\bf h}^{\mathfrak e, n}_{\modM, z}, m+\varphi^{\mathfrak e}_\omega(\modM), {\bf h}^{\mathfrak e, n}_{\modM, w}) \mapsto{\bf h}^{\mathfrak e, n}_{\modM, w}({\bf h}^{\mathfrak e, n}_{\modM, z}(m+\varphi^{\mathfrak e}_\omega
	(\modM)))$$defines the scaler
	multiplication $$\ringDe^{\mathfrak e}_n\times\frac{\psi^{\mathfrak e}_n(\modM)}{\varphi^{\mathfrak e}_\omega
		(\modM)}\times	\ringdE^{\mathfrak e}_n\to\frac{\psi^{\mathfrak e}_n(\modM)}{\varphi^{\mathfrak e}_\omega
		(\modM)}.$$Since $\ringde^{\mathfrak e}_n=	\ringdE^{\mathfrak e}_n \cap	\ringDe^{\mathfrak e}_n$ is commutative,
	this induces the desired bimodule structure on $\psi^{\mathfrak e}_n(\modM)/\varphi^{\mathfrak e}_\omega
	(\modM)$.
\end{proof}

The following diagram summarizes the relation between the above rings, where
by $A\to B$  we mean $A$ is a subring of $B$:
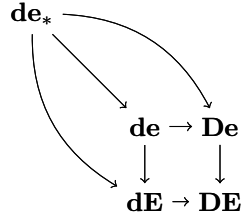
$$
\begin{tikzpicture}
\node (P) {$\ringde$};
\node (B) [right of=P] {$\ringDe$};
\node (A) [below of=P] {$\ringdE$};
\node (C) [below of=B] {$\ringDE$};
\node (P1) [node distance=1.5cm, left of=P, above of=P] {$\ringde_{\ast}$};
\draw[->] (P) to node { } (B);
\draw[->, ] (P) to node { } (A);
\draw[->, ] (A) to node  { } (C);
\draw[->, ] (B) to node { } (C);
\draw[->, bend right] (P1) to node [ ] {} (A);
\draw[->, bend left] (P1) to node {$ $} (B);
\draw[->, ] (P1) to node {$ $} (P);

\node (P) {$\ringde$};
\node (B) [right of=P] {$\ringDe$};
\node (A) [below of=P] {$\ringdE $};
\node (C) [below of=B] {$\ringDE $};
\node (P1) [node distance=1.5cm, left of=P, above of=P] {$\ringde_{\ast}$};
\draw[->,] (P) to node { } (B);
\draw[->,] (P) to node { } (A);
\draw[->, ] (A) to node  { } (C);
\draw[->, ] (B) to node { } (C);
\draw[->, bend right] (P1) to node [ ] {} (A);
\draw[->, bend left] (P1) to node {$ $} (B);
\draw[->, ] (P1) to node {$ $} (P);

\end{tikzpicture}$$

The following lemma says that for example for strongly
semi-nice construction $\bar \modM$
we have some control over $\End_\ringR(\modM_\kappa)$; note
that it only says it is not too large, but we have the freedom to choose the
ring $\ringS$ in order to make $\End(\modM_\lambda)$ have some elements with desirable
properties.

\begin{Lemma}
	\label{2.16}
	Suppose $\bar \modM=\langle \modM_\alpha:\alpha\leq\kappa\rangle$
	is  a strongly semi-nice
	construction for  $(\lambda,\frakss, S, \kappa)$,
	every ${\mathfrak e}\in {\callE}$
	is simple  and $\modM:=\modM_\kappa$. The following assertions are hold:
	
	\begin{enumerate}
		\item[(i)] If $(\Pr 1)^n_{\alpha,z}[\fucF,{\mathfrak e}]$ holds, then
		${\bf h}^{{\mathfrak e},n}_z$ is an endomorphism of ${\mathfrak B}^{\mathfrak
			e}_n(\modM)$. Furthermore ${\bf h}^{{\mathfrak e},n}_z \in \ringdE^{\mathfrak e}_{n,\ast}$.
		\item[(ii)]
		If $(\Pr 1)^n_{\alpha,z}[\fucF,{\mathfrak e}]$ holds
		and $\fucF$ is an automorphism
		of $\modM$, then ${\bf h}^{{\mathfrak e},n}_{\modM,z}$ is an
		automorphism of ${\mathfrak B}^{\mathfrak
			e}_n(\modM)$ and even of
		${}^+  {\mathfrak B}
		^{\mathfrak e}_n(\modM)$.
		\item[(iii)]
		$\End^{{\mathfrak e},\omega}(\modM)/
		\End^{{\mathfrak e},\omega}_{<
			\lambda}(\modM)$  embedded into
		$\ringdE^{\mathfrak e}$.
	Suppose in addition that $\lambda=\kappa$ and let $\ringS$ be a  subring   of $\End^{{\mathfrak e},\omega}
		(\modM)
		/\End^{{\mathfrak e},\omega}_{<
			\lambda}(\modM)$  of \power\ $<\lambda$. There is a
		club $C$ of $\kappa$ such that for any  $\alpha\in C\setminus S$
	  large enough, the ring
		$\ringS$ is embedded into
		$\End^{{\mathfrak e},\omega}(\modM/\modM_\alpha)$.
		\item[(iv)] Let ${\bf h}^{{\mathfrak e},n}[\modM]$
		denote the natural map
		$\End^{{\mathfrak e},n}
		(\modM) \to\End^{{\mathfrak e},\omega}(\modM) $ and set\footnote{See Definition \ref{2.12} for the definition of $\hat{\bf f}_n.$}
		\[\begin{array}{lr}		
		{\bf E}_n:=\big\{{\bf h}^{{\mathfrak e},n}[\modM](\hat{\bf f}_n): \fucF\in \End
		(\modM)\mbox{ and there are }
		z_n(\fucF)\in \modL^{\mathfrak e}_n \mbox{ and }\\
		\qquad\qquad\qquad\qquad\qquad\alpha_n(\fucF)<\kappa \text{ such that }
		(\Pr 1)^n_{\alpha_n(\fucF),z_n(\fucF)}[\fucF,{\mathfrak e}] \big\}.
		\end{array}\]		
		Then
		\begin{itemize}
			\item	$\End^{{\mathfrak e},\omega}(\modM)
			=\bigcup\limits_{n<
				\omega}{\bf E}_n$,
			\item ${\bf E}_n\subseteq {\bf E}_{n+1}$,
			\item  $z_n(\fucF)$
			is unique modulo $\varphi^{\mathfrak e}_\omega(\modN^{\mathfrak e}_n)$.
		\end{itemize}	
		\item[(v)] ${\bf E}_n$ is a subring of
		$\End^{\mathfrak e,\omega}(\modM)$ and the
		mapping $\hat{\fucF}_n\mapsto {\bf h}^{\mathfrak e, n}_{z_n(\fucF)}$
		is a homomorphism from
		\[\begin{array}{lr}
		\{\hat{\fucF}_n:&\fucF\in
		\End(\modM)\mbox{ and
		}(\Pr 1)^n_{\alpha_n(\fucF),z_n(\fucF)}\mbox{ for some }\ \\
		&\alpha_n(\fucF)<\kappa,\ z_n(\fucF)\in \modL^\tr_n\}
		\end{array}\]
		into $\ringdE^{\mathfrak e}_n$ with kernel
		$\End^{{\mathfrak e},n}_{<\lambda}(\modM)$, i.e.,
		$$
		\{\fucF\in \End^{{\mathfrak e},n}(\modM):z_n
		(\fucF)\in\varphi^{\mathfrak e}_\omega(\modN_n)\}.$$
		\item[(vi)] The ring $\ringS$ is naturally mapped into
		$\End(\modM)$, for
		each $\alpha\leq\omega$, there is a natural
		homomorphism from $\End(
		\modM)$ to $\End^{{\mathfrak e},\alpha}
		(\modM)$, where for
		$\alpha<\omega$ has a natural mapping to $\ringdE$.
		In particular,  $\ringS$ is naturally mapped
		into $dE^{\mathfrak e}$.
	\end{enumerate}
\end{Lemma}

\begin{proof}
	(i). 	 In view of Lemma \ref{2.13}(ii), ${\bf f}$
	induces a homomorphism from
	${\mathfrak B}^{\mathfrak e}_n(\modM_1)\to{\mathfrak B}^{\mathfrak e}_n(\modM_2)$,
and we  conclude from Lemma \ref{2.11A} that $z\in\modL^{{\mathfrak e},\ast}_n$.
	By definition, ${\bf h}^{{\mathfrak e},n}_z \in \ringdE^{\mathfrak e}_{n,\ast}$.
	
	(ii). For some formula $\psi(x,y)\in\mathcal{L}_{\mu,\kappa}^{pe}(\tau_{(\ringR, \ringS)})$, for all $\modM$, and all $x,y\in\psi_{n}(\modM)$ we
	have  $${\bf h}^{{\mathfrak e},n}_{\modM,z}(x+\varphi_{\omega}(\modM))=y+\varphi_{\omega}(\modM)\Leftrightarrow\modM
	\models\psi(x,y).$$
	Now, we define $\psi'(x,y)$ by the following role
	$$\modM \models \psi'(x,y)\Leftrightarrow\modM
	\models\psi(y,x).$$
	Since ${\bf f}$ is an automorphism, $\psi'(x,y)$ satisfies  the assumptions (i)-(iii)  of Lemma \ref{definabilityhzen}(3),
	and hence for some $z'\in\modL^{{\mathfrak e}}_n$ and all $x,y\in\psi_{n}(\modM)$,
	$${\bf h}^{{\mathfrak e},n}_{\modM,z'}(x+\varphi_{\omega}(\modM))=y+\varphi_{\omega}(\modM)\Leftrightarrow\modM
	\models\psi'(x,y).$$It is now clear that
	${\bf h}^{{\mathfrak e},n}_{\modM,z'}$ is the inverse of ${\bf h}^{{\mathfrak e},n}_{\modM,z}$. In particular,  ${\bf h}^{{\mathfrak e},n}_{\modM,z}$ is an
	automorphism of ${\mathfrak B}^{\mathfrak
		e}_n(\modM)$.
	
	(iii)$+$(iv)$+$(v). For each $n$, set
	\[\begin{array}{lr}
	{\bf F}_n=\{\hat{\fucF}_n:&\fucF\in
	\End(\modM)\mbox{ and
	}(\Pr 1)^n_{\alpha_n(\fucF),z_n(\fucF)}\mbox{ for some }\ \\
	&\alpha_n(\fucF)<\kappa,\ z_n(\fucF)\in \modL^\tr_n\}.
	\end{array}\]
	
	We will show that:
	\begin{enumerate}
		\item $\End^{{\mathfrak e},\omega}(\modM_\kappa)
		=\bigcup\limits_{n<\omega}{\bf E}_n$.
		\item ${\bf E}_n$ is a subring of
		$\End^{\mathfrak e,\omega}(\modM)$.
		
		\item the
		assignment $\hat{\fucF}_n\mapsto {\bf h}^{\mathfrak e, n}_{z_n(\fucF)}$
		yields a homomorphism
		$\varrho_n:{\bf F}_n\to\ringdE^{\mathfrak e}_n$ with kernel
		$$
		\{\hat{\fucF}_n\in \End^{{\mathfrak e},n}(\modM):z_n
		(\fucF)\in\varphi^{\mathfrak e}_\omega(\modN_n)\}
		.$$
		\item The  assignment ${\bf h}^{\mathfrak e,n}[\modM](\hat{\fucF}_n)\mapsto {\bf h}^{\mathfrak e, n}_{z_n(\fucF)}$ induces a homomorphism ${\bf E}_n\to \ringdE^{\mathfrak e}_n$.
	\end{enumerate}
	To see (1), clearly we have   $\bigcup\limits_{n<
		\omega}{\bf E}_n\subset\End^{{\mathfrak e},\omega}(\modM_\kappa)$. To see the reverse inclusion take ${\bf g}\in\End^{{\mathfrak e},\omega}(\modM_\kappa)$. Then for some $n<
	\omega$ we have ${\bf g}={\bf h}^{\mathfrak e, n}[\modM](\hat{\fucF}_n)$ where $\fucF\in\End^{{\mathfrak e},n}(\modM) $.
	Here, we are going to use the a strongly semi-nice construction. In view of this assumption, we can find
	$ \alpha_n(\fucF)<\lambda$ and $ z_n(\fucF)\in \modL^\tr_n$ such that the property $(\Pr 1)^n_{\alpha_n(\fucF),z_n(\fucF)} $ holds. By definition ${\bf g}\in {\bf E}_n$. This completes the proof of $(1)$.
	
	$(2)$. This is clear and we leave it to the reader.
	
	$(3)$. Note that $z_n(\fucF_1\pm\fucF_2)=z_n(\fucF_1)\pm z_n(\fucF_2)$ modulo $\varphi_\omega(\modN_n)$. This shows that $\varrho_n$ is additive:
	$$\varrho_n(\fucF_1 \pm \fucF_2)={\bf h}^n_{z_n(\fucF_1 \pm \fucF_2)} ={\bf h}^n_{z_n(\fucF_1)}\pm {\bf h}^n_{z_n(\fucF_2)}= \varrho_n(\fucF_1)+\varrho_n(\fucF_2).$$Similarly, for any $r\in R$ we have
	$ \varrho_n(r\fucF )= \varrho_n(r\fucF )$. Here, we compute   $\ker(\varrho_n)$: $$\varrho_n(\fucF )=0\Leftrightarrow {\bf h}^n_{z_n(\fucF)}={\bf h}^n_{0}\Leftrightarrow z_n(\fucF)\in\varphi_\omega(\modN_n),$$
i.e., $\ker(\varrho_n)=\{\hat{\bf f}_n\in {\bf F}_n:z_n(\fucF)\in\varphi_\omega(\modN_n)\}$, as claimed.
	
	$(4)$. This is trivial.
	
	Items (iv) and (v) follow  immediately. The assignment
	\[
	{\bf h}^{\mathfrak e, n}[\modM](\hat{\bf f}_n) \mapsto {\bf h}^{\mathfrak e, n}_{z_n(\fucF)}
	\]
	defines a homomorphism from
	$\End^{{\mathfrak e},\omega}(\modM)$ into
	$\ringdE^{\mathfrak e}$ with kernel
	\[
	\bigcup\limits_{n<\omega} \{{\bf h}^{\mathfrak e, n}[\modM](\hat{\bf f}_n)\in \End^{{\mathfrak e},\omega}(\modM):z_n
	(\fucF)\in\varphi^{\mathfrak e}_\omega(\modN_n)\},
	\] which is included in $\End^{{\mathfrak e},\omega}_{<
		\lambda}(\modM).$
	In sum, we have a natural embedding from the ring
	$\End^{{\mathfrak e},\omega}(\modM)/
	\End^{{\mathfrak e},\omega}_{<
		\lambda}(\modM)$ into
	$\ringdE^{\mathfrak e}$.

	It remains to prove the moreover part of (iii). To this end, let
	$\ringS$ be any subring of $\End^{{\mathfrak e},\omega}
	(\modM)
	/\End^{{\mathfrak e},\omega}_{<
		\lambda}(\modM)$  of \power\ $<\lambda$.
	For each $s \in \ringS$ find $n_s<\omega$ and ${\bf f}_s \in \End(\modM)$ such that
	$s={\bf h}^{\mathfrak e, n}[\modM](\hat{({\bf f}_s)}_{n_s})+\End^{{\mathfrak e},\omega}_{<
		\lambda}(\modM)$. We look at the following club of $\kappa$:
	\[
	C_s:=\{\alpha < \kappa: \Rang({\bf f}_s \restriction \modM_\alpha) \subseteq \modM_\alpha   \}.
	\]
For each
	$\alpha \in C_{s}\setminus S$, we define an  endomorphism  $\theta(s): \modM / \modM_\alpha \to \modM/\modM_\alpha$
	by the following role:
	$$\theta(s)(x+\modM_\alpha)={\bf f}_s(x)+\modM_\alpha.$$
	It is now natural to take $\bigcap\limits_{s \in \ringS}C_s$ as our required club, but we need to do a little more.
	Indeed, it is not clear that if $\theta$ is a homomorphism, as it may not preserve  addition or multiplication.
	To handle this, we shrink the above intersection further as follows.
	Given $s_1, \cdots, s_m \in \ringS$,  we have
	$${\bf h}^{\mathfrak e, n}[\modM]((\hat{{\bf f}}_{s_1 +\cdots+ s_m})_{n_{s_1 +\cdots+ s_m}})- {\bf h}^{\mathfrak e, n}[\modM]((\hat{{\bf f}}_{s_1})_{n_{s_1}})- \cdots -{\bf h}^{\mathfrak e, n}[\modM]((\hat{{\bf f}}_{s_m})_{n_{s_m}}) \in \End^{{\mathfrak e},\omega}_{<
		\lambda}(\modM)$$
	and
	$${\bf h}^{\mathfrak e, n}[\modM]((\hat{{\bf f}}_{s_1 \cdots s_m})_{n_{s_1 \cdots s_m}})- \left({\bf h}^{\mathfrak e, n}[\modM]((\hat{{\bf f}}_{s_1})_{n_{s_1}})\circ \cdots \circ {\bf h}^{\mathfrak e, n}[\modM]((\hat{{\bf f}}_{s_m})_{n_{s_m}})\right) \in \End^{{\mathfrak e},\omega}_{<
		\lambda}(\modM).$$
	Thus we can find $\alpha(s_1, \cdots, s_m) < \kappa$ such that
	\[
	\Rang(\left({\bf f}_{s_1 +\cdots+ s_m}-{\bf f}_{s_1}-\cdots -{\bf f}_{s_m}\right) \restriction \psi^{\mathfrak e}(\modM)) \subseteq \modM_{\alpha(s_1, \cdots, s_m)}
	\]
	and
	\[
	\Rang(\left({\bf f}_{s_1 \cdots s_m}-\left({\bf f}_{s_1}\circ \cdots \circ {\bf f}_{s_m}\right)\right) \restriction \psi^{\mathfrak e}(\modM)) \subseteq \modM_{\alpha(s_1, \cdots, s_m)}.
	\]
	Let $\alpha(\ast)=\sup\{\alpha(s_1, \cdots, s_m): m<\omega,~s_1, \cdots, s_m \in \ringS    \}< \kappa$.
	Then $C=\bigcap\limits_{s \in \ringS}C_{s}\setminus (\alpha(\ast)+1)$ is as required. The key point is that for  $s_1, \cdots, s_m \in S,$
	and $\alpha \in C,$ modulo $\modM_\alpha$, we have
	\[
	{\bf f}_{s_1 +\cdots+ s_m}= {\bf f}_{s_1}+\cdots +{\bf f}_{s_m}
	\]
	and
	\[
	{\bf f}_{s_1 \cdots s_m}={\bf f}_{s_1}\circ \cdots \circ {\bf f}_{s_m}.
	\]
	(vi). Let $s\in S$. The notation ${\bf f}_s$ stands
	for the multiplication map by $s$, e.g., ${\bf f}_s\in\End(\modM)$ and it is defined by ${\bf f}_s(x)=xs$.
	The assignment $s\mapsto {\bf f}_s$ defines an embedding
	$S\longrightarrow\End(\modM)$. Let $\alpha<\omega$.
	The map ${\bf f}\mapsto \hat{{\bf f}}_n$ yields a natural homomorphism $\End(\modM)\to\End^{{\mathfrak e},n}_{}(\modM) $. For $\alpha:=\omega$, as  $\End^{{\mathfrak e},\omega}(\modM)$
	is the direct limit of $\langle
	\End^{{\mathfrak e},n}(\modM):n<\omega\rangle$,
	and by the above argument, we have a natural homomorphism $\End(\modM)\to\End^{{\mathfrak e},\omega}(\modM) $. Finally,
	the assignment ${\bf h}^{\mathfrak e, n}[\modM](\hat{\bf f}_n)\mapsto{\bf h}^n_{z_n(\fucF)}$ defines a homomorphism $\End^{{\mathfrak e},\omega}(\modM)\to\ringdE $.
\end{proof}

\begin{definition} By   $\End^{{\mathfrak e},n}_{cpt}(\modM)$ we mean $$\left\lbrace h\in
\End^{{\mathfrak e},n}:
\mbox{the range of $h$ is compact for }({\mathfrak e},n)\mbox{ in }\modM_\lambda
\right\rbrace.$$
\end{definition}

\begin{remark}

 In Lemma \ref{2.16} we can replace
	$\End^{{\mathfrak e},n}_{<\lambda}(\modM)$
with
	$\End^{{\mathfrak e},n}_{cpt}(\modM)$ and drive  the analogue statement.
	Since this has no role in this paper, we leave the routine modification to the reader.
\end{remark}

\begin{lemma}
	\label{3.1}
	Let $\frakss:=(\cal K,\modM_\ast,{\callE},\ringR,  \ringS,  \ringT)$ be a $\lambda$-context,
	where $\lambda$ is a regular cardinal such that $\lambda=\lambda^{\aleph_0}>|\ringR|+|\ringS|+\aleph_0
	+||\modM_*||$ and  for all $\alpha<\lambda,~\alpha^{\aleph_0}
	<\lambda$.
	Then  there
	is a bi-module $\modM \geq_{\aleph_0} \modM_\ast$
	satisfying
	$\|\modM\|=\lambda=|\psi^{\mathfrak e}_n (\modM)/\varphi^{\mathfrak
		e}_\omega (\modM)|$ such that $\modM$  has few direct decompositions in the
	following sense:
	\begin{enumerate}
		\item[(i)] if $\modM=\mathop{{\bigoplus}}\limits_{t\in J}
		\modM_t$ and ${\mathfrak e}\in {\callE}$,  then
		for all but finitely many
		$t\in J$ we have  $\bigvee\limits_{n} [\psi^{\mathfrak e}_n (\modM_t)
		\subseteq \varphi^{\mathfrak e}_\omega (\modM_t)]$. In particular,
		\[\bigvee\limits_n [\varphi^{\mathfrak
			e}_n (\modM_t)=\varphi^{\mathfrak e}_\omega(\modM_t)],\]
	provided  ${\mathfrak e}$ is simple.
		\item[(ii)] Assume in addition that $|\ringR|+|\ringS|<2^{\aleph_0}$ and  that  $\modM=\modK_\alpha\oplus
		\modP_\alpha$ for $\alpha<(|\ringR|+|\ringS|+\aleph_0)^+$.   Then
		for some
		$\alpha_0<\alpha_1$ and some $n<\omega$ we have $\psi^{\mathfrak e}_n (\modK_{\alpha_\ell})\subseteq \varphi^{\mathfrak e}_n
		(\modK_{\alpha_1})+\varphi^{\mathfrak e}_\omega (\modM), $ where $\ell<2$.  In particular, if  ${\mathfrak e}$ is simple, then
		$$
		\varphi^{\mathfrak e}_n(\modK_{\alpha_1})+\varphi^{\mathfrak e}_\omega(\modM)=
		\varphi^{\mathfrak e}_n(\modK_{\alpha_{2}})+
		\varphi^{\mathfrak e}_\omega(\modM).
		$$
		
		\item[(iii)] $\End^{\mathfrak e, \omega}(\modM)/\End^{{\mathfrak e},\omega}_{<\lambda}(\modM)$ has
		cardinality $\leq |\ringR|+|\ringS|+\aleph_0$.
	\end{enumerate}
\end{lemma}

\begin{proof}
	Let	$\langle \modM_\alpha : \alpha \leq \lambda\rangle$ be a strongly semi-nice construction for
	$\frakss$	where $S\subset S^{\lambda}_{\aleph_0}$ is stationary
	such that $S^{\lambda}_{\aleph_0}\setminus S$ is stationary  as well. Let $\modM:=\modM_{\lambda}$. We show $\modM$ is as required.

	(i). Suppose not. Let
	$\modM=\mathop{{\bigoplus}}\limits_{t\in J}
	\modM_t$ and ${\mathfrak e}\in {\callE}$ be a counterexample  to the claim.
	Without loss of the generality,
	$J=\omega$ and that for each $n<\omega,~ \psi^{\mathfrak e}_n (\modM_n)
	\nsubseteq \varphi^{\mathfrak e}_\omega (\modM_n)$.
	Define ${\bf f}:\modM\to \modM$ in such a way that
	$$
	{\bf f}(x)=\left\{\begin{array}{ll}
	0 &x\in \modM_{2n},\\
	x &x\in \modM_{2n+1}
	\end{array}\right.
	$$
In other words, ${\bf f}$ is the natural projection from $\modM$ onto $\bigoplus\limits_{n<\omega}\modM_{2n+1}$.
	Then for some $n(\ast)<\omega$, $\alpha<\lambda$ and  $ z\in\modL^{\mathfrak e}_{n(\ast)}$
	the property $(\Pr 1)^{n(\ast)}_{\alpha, z}[{\bf {\bf f}},{\mathfrak e}]$
	holds.
	Let $\bar{\mathbb G}^\ast=\langle {\mathbb
		G}^\ast_n:n\geq n(\ast)\rangle$ be a decreasing sequence of additive subgroups of $\varphi^{\mathfrak
		e}_{n(\ast)}(\modM)$ as in Lemma \ref{groupsGn}. In particular,
	\begin{enumerate}
		\item  If $n\geq n(\ast)$, ${\bf h}:\modN^{\mathfrak e}_n\to\modM$ and $z_n:=g_{n(\ast),n}(z)$, then
		${\bf f}({\bf h}(x^{\mathfrak e}_n))-{\bf h}(z_n)\in{\mathbb G}_n^\ast$.
		
		\item If $z^*_\ell \in \mathbb{G}^\ast_\ell,$ for $\ell \geq n(\ast),$ then there exists $z^* \in \mathbb{G}^\ast_{n(\ast)}$
		such that $z^* - \sum\limits_{\ell=n(\ast)}^{n}z_\ell^* \in \varphi^{\mathfrak e}_{n+1}(\modM)$.
	\end{enumerate}
Due to Lemma \ref{moreonGn}(1) we know that there are $k,m<\omega$ such that
	\begin{enumerate}
		\item[(3)] $
		{\mathbb G}^\ast_m\subseteq\mathop{{\bigoplus}}\limits_{\ell<k}\modM_\ell+\varphi^{\mathfrak e}_\omega(\modM).
		$
	\end{enumerate}

	For each $n> n(\ast)+m+k$ pick some $y_n \in \psi^{\mathfrak e}_{n} (\modM_{n})\setminus \varphi_\omega^{\mathfrak e}(\modM_n)$.
	Let also
	${\bf h}_n:\modN^{\mathfrak e}_{n}\to\modM_{n}$ be a bimodule homomorphism such that $y_n={\bf h}_n(x_{n}^{\mathfrak e})$.
		We choose $n$ large enough.
	The assignment $$x\mapsto ({\bf h}_n(x),{\bf h}_{n+1}(g_n(x))	\in \modM_n\oplus\modM_{n+1}\subset\modM$$ defines a
	map ${\bf h}:\modN_{n}\to \modM$. This  unifies ${\bf h}_{n+1} $ and ${\bf h}_{n } $.
By symmetry, we may assume that $n$ is even.
	It then follows from clause (3) that
\[
{\bf h} (z_n)=  ({\bf h}_n(z_n),{\bf h}_{n+1}(z_{n+1}))\in \varphi_\omega^{\mathfrak e}(\modM_n)\subset\varphi_\omega^{\mathfrak e}(\modM)
\]
and
\[
 y_n-{\bf h}(z_{n+1}) = y_n-({\bf h}_{n+1}(z_{n+1}),{\bf h}_{n+2}(z_{n+2})) \in \varphi_\omega^{\mathfrak e}(\modM_n)\subset\varphi_\omega^{\mathfrak e}(\modM).
\]
It follows from the first one that  ${\bf h}_n(z_n)$ and ${\bf h}_{n+1}(z_{n+1})$ are in $\varphi_\omega^{\mathfrak e}(\modM)$.
By repetition,  ${\bf h}_{n+2}(z_{n+2})$ is in $\varphi_\omega^{\mathfrak e}(\modM)$.
We plug these in the second containment to see
$$y_n\in\varphi_\omega^{\mathfrak e}(\modM).$$	
This is a   contradiction that we searched for it.
	
Here, we assume that  $\mathfrak e$ is simple. According to Lemma \ref{simplephiequalpsi},  $\psi^{\mathfrak e}_n(\modM_t)=\varphi^{\mathfrak e}_n(\modM_t)$ for all  $n$ and $t\in J$. Hence for the $n$ as chosen above, we have
	\[
	\varphi^{\mathfrak e}_n(\modM_t)=\varphi^{\mathfrak e}_\omega(\modM_t),
	\]
	as required.
	
	(ii). For each $\alpha<(|\ringR|+|\ringS|+\aleph_0)^+$, let ${\bf f}_\alpha$ be the projection onto $\modK_\alpha$, i.e.,
	$$
	{\bf f}_{\alpha}(x)=\left\{\begin{array}{ll}
	x &x\in \modK_{\alpha},\\
	0 &x\in \modP_{\alpha}
	\end{array}\right.
	$$
	Pick $n_\ast(\alpha)<\omega$,  $ z_\ast(\alpha)\in\modL^{\mathfrak e}_{n_\ast(\alpha)}$ and
	$\beta_\ast(\alpha)$ such that  $(\Pr 1)^{n_\ast(\alpha)}_{\beta_\ast(\alpha),z_\ast(\alpha)}[{\bf f}_\alpha, \mathfrak e]$
	holds. We combine  Lemma  \ref{groupsGn} along with Lemma \ref{moreonGn}(1) to  find $m(\alpha)<\omega$ and a compact and decreasing sequence  $\langle \bar{\mathbb G}^\ast_{\ell}: \ell \geq m(\alpha) \rangle$ of additive subgroups of $\varphi^{\mathfrak
		e}_{n(\ast)}(\modM)$. Furthermore, in the light of Lemma \ref{2.9} we can assume that $\mathbb{G}^\ast_{m(\alpha)} \subseteq \varphi^{\mathfrak e}_\omega(\modM).$ There exists  a stationary set $S$ of $\alpha<(|\ringR|+|\ringS|+\aleph_0)^+$
	such that for some $n<\omega$ and some fixed $z \in \modL^{\mathfrak e}_n$,	and for all $\alpha \in S,$ we have $n_\ast(\alpha)=n$ and  $z_\ast(\alpha)=z$. Let
	$\alpha_0 < \alpha_1$ be in $S$ and $\ell< 2$. It follows that:
	
	\[\begin{array}{ll}
	y\in\psi^{\mathfrak e}_n (\modK_{\alpha_{\ell}})&\Rightarrow \exists {\bf h}:\modN^{\mathfrak e}_n\to  \modK_{\alpha_{\ell}},\quad {\bf h}(x^{\mathfrak e}_n)=y\\
	&\Rightarrow
	{\bf f}(y)-{\bf h}(z_n)\in{\mathbb G}^\ast_n\subset\varphi^{\mathfrak e}_\omega(\modM)\\
	&\Rightarrow  y -{\bf h}(z_n)\in\varphi^{\mathfrak e}_\omega(\modM)\\
	&\Rightarrow y \in\varphi^{\mathfrak e}_n(\modK_{\alpha_{\ell}})+ \varphi^{\mathfrak e}_\omega(\modM).
	\end{array}\]
	We proved that $\psi^{\mathfrak e}_n (\modK_{\alpha_{\ell}}) \subseteq \varphi^{\mathfrak e}_n(\modK_{\alpha_{\ell}})+\varphi^{\mathfrak e}_\omega(\modM)$.

	Now suppose that	${\mathfrak e}$ is simple. In view of Lemma \ref{simplephiequalpsi},$
	\psi^{\mathfrak e}_n (\modK_{\alpha_{\ell}})=\varphi^{\mathfrak e}_n (\modK_{\alpha_{\ell}})$.
	It then follows that
	$$
	\varphi^{\mathfrak e}_n(\modK_{\alpha_\ell})+\varphi^{\mathfrak e}_\omega(\modM)=
	\varphi^{\mathfrak e}_n(\modK_{\alpha_{1-\ell}})+
	\varphi^{\mathfrak e}_\omega(\modM).
	$$
	
	(iii). We apply the notation
introduced in Lemma \ref{2.16}. Recall that
	$E_n$ is defined by the natural image of
 \[\begin{array}{lr}F_n:=
 \{\fucF\lceil(\varphi_n^{\mathfrak e} (\modM_\lambda)/
 \varphi^{\mathfrak e}_\omega (\modM_\lambda):&\fucF\in
 \End(\modM)\mbox{ and
 }(\Pr 1)^n_{\alpha_n(\fucF),z_n(\fucF)}\mbox{ for some }\ \\
 &\alpha_n(F)<\lambda,\ z_n(\fucF)\in \modL^\tr_n\}
 \end{array}\]
in 	$\End^{\bar\varphi,\omega}(\modM)$. Also,
 the
mapping $\fucF\mapsto \bf{h}^n_{z_n(\fucF)}$
induces a homomorphism from
$F_n$
into $\ringdE^{\mathfrak e}_n$ with kernel
$\End^{{\mathfrak e},n}_{<\lambda}(\modM)$.
Since 	$\End^{{\mathfrak e},\omega}(\modM)
=\bigcup\limits_{n<
	\omega}{\bf E}_n$, it is enough to show that
$\{ \bf{h}^n_{z_n(\fucF)}:z_n(\fucF)\in \modL^\tr_n\}$ is of cardinality at most
$\leq |\ringR|+|\ringS|+\aleph_0$. Thus it is enough
 to show that the cardinality of $\modN^{{\mathfrak e}}_n$ is  at most
$\leq |\ringR|+|\ringS|+\aleph_0$. This holds, because
 $\modN^{{\mathfrak e}}_n$ as an $(\ringR,\ringS)$ bimodule is finitely generated.	
\end{proof}

\begin{Remark}
	\label{3.1A}
	If we omit ``$|\ringR|+|\ringS|<2^{\aleph_0}$'', we get by the same proof weaker
	conclusions: with an ``error term'' which is included in a finitely
	generated bimodule.
\end{Remark}
The following is supposed to be used together with any of the later lemmas
here as its conclusion is in  their assumptions.
\begin{Lemma}
	\label{3.2}
 Let $\ringR$ be a ring which is not purely semisimple and let $\ringS:=\ringT:=\langle
		1\rangle_\ringR$. Let ${\cal K}={\cal K}[\ringR,\mu]$
		be the family of $(\leq \mu_1)$ generated and $(\leq\theta)$
		presented bimodules where $\mu_1^\theta \leq \mu$ and let $\modM_*$
		be a bimodule of cardinality $<\mu$. Finally, let
		$\callE:=\callE_{\ringR,\mu}$ be the set of $(<\kappa)$-simple
		non trivial ${\mathfrak e}\in \callE_{(\ringR,\ringS)}$ wherever
		$(||\ringR||+\aleph_0)^{<\kappa}\leq\mu$.  The following assertions are true.
		\begin{enumerate}
			\item[(a)] There is ${\mathfrak e}\in {\callE}$ which is very nice.
			\item[(b)] If $(|\ringR|+\aleph_0)^{\aleph_0}\leq\mu$
			then ${\mathfrak e}({\bar\varphi})
			\in {\callE}$ for every very nice $\bar\varphi$.
			\item[(c)] $\frakss=({\cal K},\modM_*,\callE,\ringR,\ringS,
			\ringT)$ is a non trivial context which is simple.
		\end{enumerate}
	\end{Lemma}

\begin{proof}
 Since $\ringR$ is not purely semisimple, and in the light of Theorem \ref{shelahpure}, we can find a sequence $\bar\varphi=\langle \varphi_n(x): n<\omega \rangle$
	as in Definition \ref{purely semisimpledef}.  Then $\mathfrak e=\mathfrak e(\bar\varphi) \in \callE$ is very nice. This confirms (a).
	Items (b) and (c) are clear.
\end{proof}

\begin{Lemma}
	\label{3.2b}
Suppose $\ringR$ is a ring, $T$ is a complete first order theory of
		$\ringR$-modules which is not superstable, $\ringS=\ringT=
		\langle 1\rangle_\ringR$ and $\mu\geq |\ringR|
		+\kappa$.
		Then  there is a family
		${\cal K}\cup\{\modM_\ast\}$ of $\ringR$-modules with $\mu$ members  such that:
		\begin{enumerate}
			\item[(a)] $\modM_\ast\oplus\mathop{{\bigoplus}}\limits_{t\in I}
			\modM_t$ is model
			of $T$ whenever $\modM_t\in {\cal K}$, moreover $$\modM_\ast \prec_{\mathcal{L}({\tau_{\ringR}})}
			\modM_*\oplus\mathop{{\bigoplus}}\limits_{t \in I} \modM_t.$$
			
			\item[(b)] For any $\modN\in {\cal K}$ we have $ ||\modN||
			\leq \mu$.
			
			\item[(c)] If $2^{\mu_1}\leq\mu$ and $\mu_1\geq ||\ringR||
			+\kappa$, then every model $\modN$ of $T$ satisfying
			$$\modM_* \prec_{\mathcal{L}({\tau_{\ringR}})} \modM_*\oplus \modN$$ belongs to ${\cal K}$.
			
			\item[(d)] Let  $2^{\mu_1}\leq \mu$ and $\mu_1\geq ||\ringR||+
			\aleph_0$. Then  every appropriate \sq\ $\langle \modN_n,
			g_n,x_n:n<\omega\rangle$ with $||\modN_n||\leq
			\mu_1$ belongs to ${\cal K}$.
			\item[(e)]  Let  $\calE$ be  the set of $(<\aleph_0)$-simple
			non trivial ${\mathfrak e}\in \callE_{(\ringR,\ringS)}$ such that each $\modN_n^{\mathfrak e}$ is in $\cal K$.
			Then $\frakss=({\cal K},\modM_*,\callE,\ringR,\ringS,
			\ringT)$ is a $\mu$-context; note that a bimodule for $\frakss$ is
			just an $\ringR$-module.
		\end{enumerate}
\end{Lemma}

\begin{proof}
 Let  $\modM_\ast $
	be any $\aleph_1$-saturated model of
	$T$   of size $\leq \mu$ and set
	$${\mathcal K}:=\{\modN: \modN
	\mbox{ is an $\ringR$--module such that
	} \modM_\ast
	\prec_{{\mathcal L} (\tau_\ringR)}
	\modM_\ast
	\oplus
	\modN \mbox{ and } 2^{\| \modN \|} \leq \mu\}.$$
	It is easily seen that $\cal K$ is as required. Since this has no role in the paper, we leave the routine details to the reader.
\end{proof}

The following easy lemma plays an essential role in the sequel:

\begin{lemma}
	\label{solfree}
	Let  $ \frakss$ be a $\lambda$-context and
	$\ringS$ be a free $\ringT$-module with a base $\{c^\ast_\beta:\beta<\alpha\}$.
	Let $\bar{\varphi}$ be very nice,
	$ \varphi_n=
	(\exists y_0,\ldots, y_{k_n-1}) \varphi'_n$ where
	$$\varphi'_n=\bigwedge\limits^{m_n-1}_{\ell=0}
	[a^n_\ell x-\sum\limits_{i=0}^{k_n-1}
	b^n_{\ell,i} y_i].$$
	Let  ${\mathfrak e}={\mathfrak e}({\bar \varphi})
	\in \calE^\frakss$ and $\modN_n=
	\modN^{\mathfrak e}_n$ (see Definition \ref{c.3}). The following assertions hold:
	\begin{enumerate}
		\item Let $\modN_{n,0}$ be the $\ringR$-submodule of
		$\modN_n$  generated by $\{x,y_i: i<
		k_{m_n-1}\}$.
		Then  $\modN_n$ is the direct sum $\sum\limits_{\beta<\alpha}
		\modN_{n,\beta}$ and $\modN_{n,0}\stackrel{h_\beta}{\cong} \modN_{n,\beta}$, as
		$\ringR$-modules, where $\modN_{n,\beta}$
		is the $\ringR$-module   generated  by $\{x c^\ast_\beta\}\cup \{y_ic^\ast_\beta: i<k_{m_n-1}\}$ freely except the
		equations $\varphi'_n$ and  $h_0$ is the identity.
		\item As a $\ringT$-module, $\varphi_n(\modN_n)/\varphi_\omega(\modN_n)=\sum\limits_{\beta<
			\alpha}\varphi_n(\modN_{n,\beta})/\varphi_\omega(\modN_{n,
			\beta})$.
		\item For any $z\in \modL^{{\mathfrak e}}_n$, there
		are $z_\beta\in
		\modN_{n,0} \cap \modL^{{\mathfrak e}}_n\cap \varphi^{\mathfrak e}_n
		(\modN_{n,0})$
		such that $z=\sum\limits_{\beta<\alpha} h_\beta(z_\beta)$. Also,
		${\bf h}_z^{\mathfrak e, n}=\sum\limits_{\beta< \alpha}{\bf h}_{h_\beta(z_\beta)}^{\mathfrak e, n}$. In particular,  $z$ is $n$-nice if and only if each $z_\beta$ is
		$n$-nice.
		\item The rings $\ringde^{\mathfrak e}_n$ and  $\ringS$
		generate $\ringdE^{\mathfrak e}_n$. In fact each element of $\ringdE^{\mathfrak e}_n$ has the form
		$\sum\limits_{\beta<\alpha}x_\beta b_\beta$, where $x_\beta\in \ringde^{\mathfrak e}_n$.  Also we have
		$\ringdE^{\mathfrak e}_n=\ringde^{\mathfrak e}_n\mathop{{\otimes}}\limits_\ringT \ringS$.
		\item Let $\mathcal{I}_n$ be a maximal ideal of $\ringde^{\mathfrak e}_n$.
		Then	${\bf D}_n:=\ringde^{\mathfrak e}_n/\mathcal{I}_n$  is a field.
		\item Let   $\ringT':=\ringT/(\mathcal{I}_n\cap \ringT)$,
		$\ringS':={\bf S}/(\mathcal{I}_n\cap \ringT)$ and let $\modM$   be from a strongly semi-nice construction.
		Then  any set of equations on $\ringS$ which has a solution in
		$\End_\ringR(\modM)$
		has a solution in ${\bf D}_n\mathop{{\otimes}}\limits_{\ringT'} \ringS'$.
	\end{enumerate}
\end{lemma}

\begin{proof}
	$(1)$. Since 	$\ringS$ is free as a $\ringT$-module with the base $\{c^\ast_\beta:\beta<\alpha\}$, we have	$\ringS=\bigoplus_{\beta<\alpha}\ringT c^\ast_\beta$.
	Now, we apply this through the following natural identifications of  $\ringR$-modules:
	
		\[\begin{array}{ll}
	\ringR x \ringS\oplus\bigoplus_{i<k_m-1}\ringR y_i\ringS&=\bigoplus_{\beta<\alpha}\ringR xc^\ast_\beta\oplus\bigoplus_{i<k_m-1}\bigoplus_{\beta<\alpha}\ringR y_ic^\ast_\beta\\
	&=\bigoplus_{\beta<\alpha}\left(\ringR xc^\ast_\beta\oplus\bigoplus_{i<k_m-1}\ringR y_ic^\ast_\beta\right).
	\end{array}\]

It turns out  from the previous displayed identification that  $$\modN_n\cong\bigoplus_{\beta<\alpha}
	\modN_{n,\beta},$$ as an $\ringR$-module.
	Also for any $y\in \modN_{n,0}$, we set $h_\beta(y):=yc_\beta$. This yields an $\ringR$-module isomorphism $\modN_{n,0}\stackrel{h_\beta}{\cong} \modN_{n,\beta}$.
	
	$(2)$. We apply Lemma
	\ref{pr properties}
	along with (1) to conclude that $\varphi_n(\modN_n)=\bigoplus_{\beta<
		\alpha}\varphi_n(\modN_{n,\beta})$. Consequently, $\varphi_n(\modN_n)/\varphi_\omega(\modN_n)=\bigoplus_{\beta<
		\alpha}\varphi_n(\modN_{n,\beta})/\varphi_\omega(\modN_{n,
		\beta})$.
	
	$(3)$.  Let $z\in \modL^{{\mathfrak e}}_n$. Then,
	$$z\in\varphi_n(\modN_n)=\bigoplus_{\beta<
		\alpha}\varphi^{\mathfrak e}_n(\modN_{n,\beta})\cong\bigoplus_{\beta<
		\alpha}h_\beta''(\varphi^{\mathfrak e}_n(\modN_{n,0})).
	$$	Let $z_\beta\in
	\modN_{n,0} $ be such that
	$z=\sum\limits_{\beta<\alpha} h_\beta(z_\beta)$. It is evident that $z_\beta\in
	\modN_{n,0} \cap \modL^{\mathfrak e}_n\cap \varphi^{\mathfrak e}_n
	(\modN_{n,0})$.
	
	Now suppose that ${\bf g}: \modN_n \to \modM$ is a bimodule homomorphism. Then
	\[\begin{array}{ll}
	{\bf h}_z^{\mathfrak e, n}({\bf g}(x_n)+\varphi_\omega(\modM))&={\bf g}(z)+\varphi_\omega(\modM)\\
	&=\left(\sum\limits_{\beta<\alpha}{\bf g}(h_\beta(z_\beta))\right) + \varphi_\omega(\modM)\\
	&=\sum\limits_{\beta<\alpha} \left({\bf g}(h_\beta(z_\beta)) + \varphi_\omega(\modM)\right)\\
	&=\sum\limits_{\beta<\alpha}{\bf h}_{h_\beta(z_\beta)}^{\mathfrak e, n}({\bf g}(x_n)+\varphi_\omega(\modM)).
	\end{array}\]
	It follows that ${\bf h}_z^{\mathfrak e, n}=\sum\limits_{\beta< \alpha}{\bf h}_{h_\beta(z_\beta)}^{\mathfrak e, n}$.

	$(4)$. 	 Let  $x\in \ringde^{\mathfrak e}_n$   and  $s\in \ringS$. Then $xs\in \ringdE^{\mathfrak e}_n$. This implies
	the existence of  a $\ringT$-linear map $\ringde^{\mathfrak e}_n\times\ringS\to \ringdE^{\mathfrak e}_n$.
	By the  universal property of tensor products,
	there is a map  $f:\ringde^{\mathfrak e}_n\otimes_{\ringT}\ringS\to \ringdE^{\mathfrak e}_n$.
	Now, let $z\in\ringdE^{\mathfrak e}_n$. By clause (3), there are
	$z_\beta\in\ringde^{\mathfrak e}_n$ and $s_\beta\in \ringS$ such that
	$z=\sum\limits_\beta z_\beta s_\beta$. Moreover,
	by its proof, we know that such a presentation is unique.
	This shows that $f$ is an isomorphism. Up to this identification,
	$\ringdE^{\mathfrak e}_n=\ringde^{\mathfrak e}_n\mathop{{\otimes}}\limits_\ringT \ringS$.
	Thanks to Lemma \ref{pr rings} we know the rings $\ringde^{\mathfrak e}_n$ and  $\ringS$   commute with each other.

	$(5)$. As  $\mathcal{I}_n$ is a maximal ideal, it is clear that  ${\bf D}_n$  is a division ring. Since the ring $\ringde^{\mathfrak e}_n$ is commutative (see Lemma \ref{pr rings}(2)) we deduce that ${\bf D}_n$ is a field.

	$(6)$. Let
	\begin{enumerate}
		\item[$(\ast)_1$] \qquad\qquad\qquad \qquad\qquad\qquad $\sum_js_{ij}X_i^j=0,$
	\end{enumerate} be a system of polynomial
	equations with parameters $s_{ij}\in S$ and indeterminates $\{X_i\}$. Suppose
	these equations have a solution $\fucF\in\End(\modM)$. This means that
	\begin{enumerate}
		\item[$(\ast)_2$] \qquad\qquad\qquad \qquad\qquad\qquad $\sum_js_{ij}\fucF^{j}=0.$
	\end{enumerate}
	There are
	$z_n(\fucF)\in \modL^{\mathfrak e}_n$ and $ \alpha_n(\fucF)<\lambda$ such that
	$(\Pr 1)^n_{\alpha_n(\fucF),z_n(\fucF)}[\fucF,{\mathfrak e}]$.
	Let
	$f_n:=\hat{{\bf f}}_n$.
	Then 	
	\begin{enumerate}
		\item[$(\ast)_3$] \qquad\qquad\qquad \qquad\qquad\qquad $\sum_js_{ij}f^{j}_n=0.$
	\end{enumerate}
	Recall from Lemma \ref{2.16}(v) that the natural
	mapping $\varrho_n: \hat{\fucF}_n\mapsto {\bf h}^n_{z_n(\fucF)}$
	is a homomorphism from
	\[\begin{array}{lr}
	\{\hat{\fucF}_n:&\fucF\in
	\End(\modM)\mbox{ and
	}(\Pr 1)^n_{\alpha_n(\fucF),z_n(\fucF)}\mbox{ holds for some }\ \\
	&\alpha_n(\fucF)<\lambda,\ z_n(\fucF)\in \modL^\tr_n\}
	\end{array}\]
	into $\ringdE^{\mathfrak e}_n$ with kernel included in
	$\End^{\mathfrak e,n}_{<\lambda}(\modM)$.
	Let
	$$\pi:\End^{{\mathfrak e},n}(\modM)\to\frac{\End^{{\mathfrak e},n}(\modM)}{\End^{{\mathfrak e},n}_{<\lambda}(\modM)}$$
	be the canonical map and let   $g_n=\pi({\bf h}^n_{z_n(\fucF)})$.
	Applying $\pi \circ \varrho_n$ to both sides of $(\ast)_3$, we get
	\begin{enumerate}
		\item[$(\ast)_4$] \qquad\qquad\qquad \qquad\qquad\qquad	$\sum_js_{ij}g^{j}_n=0.$
	\end{enumerate}
	Since there is an embedding
	$$\rho_n:	\frac{\End^{{\mathfrak e},n}(\modM)}{\End^{{\mathfrak e},n}_{<\lambda}(\modM)}\hookrightarrow\ringdE^{\mathfrak e}_n,$$
	by setting $e_n:=\rho_n(g_{n}),$ we have
	\begin{enumerate}
		\item[$(\ast)_5$] \qquad\qquad\qquad \qquad\qquad\qquad	$\sum_js_{ij}e^{j}_n=0$.
	\end{enumerate}

In view of (4) we see that	$\ringdE^{\mathfrak e}_n=\ringde^{\mathfrak e}_n\mathop{{\otimes}}\limits_\ringT \ringS$.
	Let $$\sigma:\ringde^{\mathfrak e}_n\mathop{{\otimes}}\limits_\ringT \ringS\to   {\bf D}_n\mathop{{\otimes}}\limits_{\ringT'} \ringS'$$
	be the natural map induced by $\ringde^{\mathfrak e}_n\twoheadrightarrow{\bf D}_n$ and $\ringS\twoheadrightarrow\ringS'$.
	Set $t_n:=\sigma(e_n)$. By applying $\sigma$ to $(\ast)_5$
	we have
	\begin{enumerate}
		\item[$(\ast)_6$] \qquad\qquad\qquad \qquad\qquad\qquad	$\sum_js_{ij}t^{j}_n=0$.
	\end{enumerate}
	This essentially says that the polynomial  equations from $(\ast)_1$ with parameters in $S$
	have a solution in ${\bf D}_n\mathop{{\otimes}}\limits_{\ringT'} \ringS'$ as well.
	This is what we want to prove.
\end{proof}

In what follows we will use the following two consequences
of Lemma \ref{solfree}:
\begin{corollary}
	\label{3.5} Suppose that the following three items holds:
	\begin{enumerate}
		\item[(a)]   $\ringR$ is a ring   which is not pure semisimple and let $\ringT$ be the subring of  $\ringR$  generated by $1$,
i.e.,  $\ringT\cong {\mathbb Z}/ n{\mathbb Z}$, where $n:=\Char(\ringR)$  which is not
		necessarily prime.
		\item[(b)]  $\ringS$ is a ring containing  $\ringT$ such that $(\ringS,+)$ is a free
		$\ringT$-module and suppose that for every $s\in \ringS\setminus \{0_\ringS\}$
		for some $\modN\in {\cal K}\cup \{\modM_*\}$  we have
		$\modN s \not= \{0_\modN\}$, and
		\item[(c)] $\lambda=\cf (\lambda)>||\ringR||+||\ringS||+\aleph_0$
		and $\alpha<\lambda\Rightarrow |\alpha|^{\aleph_0}<\lambda$.
	\end{enumerate}
	Then  we can find an $\ringR$-module $\modM$ of \power\
	$\lambda$, and a
	homomorphism $\bf{h}$ from $\ringS$ into $\End(\modM)$ such that:
	\begin{enumerate}
		\item[(d)] $\Ker({\bf h})=\{0\}$.
		\item[(e)] If $\Sigma$  is a set of equations with parameters in
		$\ringS$ such that ${\bf h}(\Sigma)$ is solvable in
		$\End_\ringR(\modM)$,  then
 $\Sigma$ is solvable in ${\bf D}\otimes \ringS$ for some field ${\bf D}$.
		\item[(f)] If $s\in \ringS\setminus\{0_{\ringS}\}$ and
		 $\modN\in{\cal K}$ is such that $\modN s \not= \{0_\modN\}$, then
		the image of $\modM$ under
		${\bf h}(s)$  has cardinality $\lambda$.
	\end{enumerate}
\end{corollary}

\begin{corollary}
	\label{3.6}
	Suppose $\ringS$ is a ring extending $\mathbb Z$ such that $(\ringS,+)$ is free,
and let $\ringR$ be a ring which is not pure semisimple. Let   ${\bf D}$ be a
field such that  $p:=\Char({\bf D})|\Char(\ringR)$  and set
${\mathbb Z}_p:={\mathbb Z}/ p{\mathbb Z}$ if $p>0$,
and  ${\mathbb Z}_p:={\mathbb Z}$ otherwise.
Let $\Sigma$ be a set of
equations over $\ringS$ which is not solvable in ${\bf D}\mathop{{\otimes}}
_{{\mathbb Z}_p}(\ringS/ p\ringS)$.	Then for $\modM$ which is strongly nicely constructed,
$\Sigma$ is not
solvable in $\End(\modM)$.
\end{corollary}

Recall that a module is $\aleph_0$-free if each of its finitely
generated submodules are free. This yields the following statement:

\begin{Remark}
	\label{3.6A}
	{\rm
		In Corollary   \ref{3.6}, if $(\ringS,+)$ is  an $\aleph_0$-free
		$\ringT$-module,  the similar conclusions are hold.
	}
\end{Remark}

\section{All things together: Kaplansky test problems}
\label{Kaplansky test problems}
We are now ready to answer Kaplansky test problems.

\begin{notation}For the rest of this section we assume that
 $\ringR$ is a ring which  is not pure semisimple.
\end{notation}

\begin{theorem}
	\label{3.3}
	Let $\frakss$ be a  non trivial  $\lambda$-context with
	$\ringS=\ringT=\langle 1_\ringR\rangle_\ringR$ where
	$\lambda=\cf(\lambda)>|\alpha|^{\aleph_0}+||\frakss||$
	for $\alpha<\lambda$. Then there are
		$\ringR$-modules $\modM$,
		$\modM_1$ and $\modM_2$ of \power\ $\lambda$ such that:
	\begin{enumerate}
		\item
		$\modM\oplus \modM_1\cong \modM\oplus \modM_2$, and
		$\modM_1\not\cong \modM_2$.
\item   $\modM_1$, $\modM_2$, $\modM_1
		\oplus \modM$ and $\modM_2\oplus \modM$ are $\leq_{\aleph_0}$--extensions of
		$\modM^\ast$.
		\item $\modM_1\equiv_{\mathcal{L}_{\infty,
				\lambda}} \modM_2$.
	\end{enumerate}
\end{theorem}
\begin{remark}Adopt the notation of Theorem \ref{3.3}.
\begin{enumerate}
\item [(a)] Note that (1) becomes trivial if we remove the ``of \power\ $\lambda$''. To see this, take
 $\modM$,
		$\modM_1$ and $\modM_2$ to be free $\ringR$-modules with $$\| \modM \| > \| \modM_1 \| > \| \modM_2 \| \geq |\ringR|+\aleph_0.$$
\item [(b)]  Recall that $\modM_1\equiv_{\mathcal{L}_{\infty,
				\lambda}} \modM_2$ means for every sentence $\sigma \in \mathcal{L}_{\infty,
				\lambda},$
\[
\modM_1 \models \sigma \iff \modM_2 \models \sigma.
\]

\end{enumerate}
\end{remark}

\begin{proof}
	$(1)$. Let $\ringT$ be the subring of $\ringR$
	which 1 (the unit) generates.
	
		\textbf{Step A)}:\quad Here, we  introduce the auxiliary ring $\ringS$:
	
	Let $\ringS:=\frac{\ringT\langle{\cal X},{\cal W}_1,
		{\cal Y},{\cal W}_2\rangle}{I}$ where $\ringT\langle{\cal X},{\cal W}_1,
	{\cal Y},{\cal W}_2\rangle$  is the skew polynomial ring  in  noncommuting variables $\{{\cal X},{\cal W}_1,
	{\cal Y},{\cal W}_2\}$  with coefficients in the commutative ring $\ringT$, and $I$ is its  two-sided ideal
	generated by:
\begin{description}
\item[$(\ast)$]
${\cal X}{\cal X}={\cal X}$,\\
${\cal Y}{\cal Y}={\cal Y}$,\\
${\cal X}{\cal W}_1{\cal W}_2={\cal X} $,\\
${\cal Y}{\cal W}_2{\cal W}_1={\cal Y}$,\\
${\cal X}{\cal W}_1{\cal Y}={\cal X}{\cal W}_1$,\\
$(1-{\cal X})(1-{\cal Y})=1-{\cal X}$,\\
${\cal Y}{\cal X}={\cal Y}$,\\
${\cal Y}{\cal W}_2{\cal X}={\cal Y}{\cal W}_2	$.
\end{description}
In other words, $\ringS$ is the ring generated by $\ringT\cup \{{\cal X},{\cal W}_1,
{\cal Y},{\cal W}_2\}$ extending $\ringT$  freely except equations $(\ast)$
(to understand these equations see the definition of $\modM^\otimes$
as a bimodule below).

\textbf{Step B)}:\quad Let $\alpha<\beta<\gamma$ be
additively indecomposable ordinals\footnote{Recall that an ordinal $\gamma$ is additively indecomposable if for all ordinals
$\alpha, \beta < \gamma$ we have $\alpha+\beta < \gamma.$} and let $\modM$ be  an $\ringR$-module. We define a new
 bimodule $\modM^\otimes$ related to $\modM$ and ordinals $\alpha, \beta, \gamma$.

	For $i<\gamma$, let $\modM\stackrel{h_i}{\cong}
	\modM^\otimes_i$  (where $\modM_i^\otimes$ is an
	$\ringR$-module) and set $\modM^\otimes:=\mathop{{\bigoplus}}
	\limits_{i<\gamma} \modM_i^\otimes$.
	We expand $\modM^\otimes$ to an $(\ringR,\ringS)$-bimodule. To this end, we take $x\in \modM^\otimes_i$.
	Due to the axioms of bimodules, it is enough  to define $\{h_i(x){\cal X},h_i(x){\cal Y},h_i(x){\cal W}_1,h_i(x){\cal W}_2\}$. We define these via the following rules:
	$$
h_i(x){\cal X}:=\left\{\begin{array}{ll}
	h_i(x) &\qquad i\geq\alpha,\\
	0    &\qquad i<\alpha,
	\end{array} \right.
	$$
	$$
	h_i(x){\cal Y}:=\left\{\begin{array}{ll}
h_i(x) &\qquad i\geq\beta,\\
	0    &\qquad i<\beta,
	\end{array}\right.
	$$
	$$
h_i(x){\cal W}_1:=\left\{\begin{array}{ll}
h_j(x) &\mbox{ if for some }\epsilon,\ i=\alpha+\epsilon<\gamma,\
	j=\beta+\epsilon<\gamma,\\
	0    &\mbox{ otherwise,}
	\end{array}\right.
	$$and
	$$
h_i(x){\cal W}_2:=\left\{\begin{array}{ll}
h_j(x) &\mbox{ if for some }\epsilon,\ i=\beta+\epsilon<\gamma,\
	j=\alpha+\epsilon<\gamma,\\
	0    &\mbox{ otherwise.}
	\end{array}\right.
	$$
	
	
	Let $\bar \modM=\langle
	\modM_\alpha:\alpha\leq\kappa \rangle$ be a strongly semi-nice
	construction. Recall that semi-nice
	construction is a consequence of Section 3, and its strong form
	was constructed in Section 4.
	
	 Let
	$\modP:=\modM_\kappa$ and let
	${}_\ringR\modP$ be $\modP$ as an $\ringR$-module.
	\medskip

		\textbf{Step C)}:\quad
	Here, we define the
	$\ringR$-modules $\modM$,
	$\modM_1$ and $\modM_2$ of \power\ $\lambda$ such that
	$\modM\oplus \modM_1\cong \modM\oplus \modM_2$.
	
	To this end, recall from the second step that every element of $\ringS$ may be considered as an
	endomorphism of ${}_\ringR\modP$.
	Set ${}_\ringR\modM^1:=({}_\ringR\modP){\cal X}$  and ${}_\ringR\modM_1:=
	({}_\ringR\modP)(1-{\cal X})$.
	We conclude from the formula ${\cal X}{\cal X}-{\cal X}=0$
	that  ${}_\ringR\modM^1\cap{}_\ringR\modM_1=0$. Let us  use from the formula
	${\cal X}(1-{\cal X})=1$ that
	$${}_\ringR\modP={}_\ringR\modP({\cal X}+(1-{\cal X}))={}_\ringR\modP {\cal X}+(1-{\cal X}){}_\ringR\modP={}_\ringR\modM^1+{}_\ringR\modM_1={}_\ringR\modM^1
	\oplus {}_\ringR\modM_1.$$
	Let  $_\ringR\modM^2:=
	({}_\ringR\modP){\cal Y}$ and $_\ringR\modM_2:=
	({}_\ringR\modP)(1-{\cal Y})$.
In the same vein, the above formulas lead us to the following decomposition $${}_\ringR\modP={}_\ringR\modM^2
	\oplus{}_\ringR\modM_2.$$

	In view of the equation ${\cal X}{\cal W}_1
	{\cal Y}={\cal X}{\cal W}_1$, we have
	$$\modM^1{\cal W}_1={}_\ringR\modP {\cal X}  {\cal W}_1 ={}_\ringR\modP  {\cal X}{\cal W}_1
	{\cal Y}   \subset{}_\ringR\modP
	{\cal Y}=\modM^2.$$
This yields   a homomorphism from $\modM^1$ to $\modM^2$, defined by the help of the following assignment
\[
a \mapsto a {\cal W}_1.
\]
	Similarly,  ${\cal W}_2$ provides a
	homomorphism from $\modM^2$ onto $\modM^1$, defined via
\[
b \mapsto b {\cal W}_2.
\]	
Thanks to the equations ${\cal Y}{\cal W}_2{\cal W}_1={\cal Y}$ and ${\cal Y}{\cal X}={\cal Y}$,
	it is easily seen that the multiplication maps by ${\cal W}_1$ and ${\cal W}_2$
	are inverse to each  other. This provides an
	isomorphism from $\modM^1$ onto $\modM^2$, so let ${}_\ringR\modM:={}_\ringR \modM^1\cong {}_\ringR
	\modM^2$.

	\medskip
	
	\textbf{Step D)}:\quad One has ${}_\ringR\modM_1\not\cong {}_\ringR\modM_2$.

	Assume towards  contradiction that ${}_\ringR\modM_1\cong {}_\ringR\modM_2$. Thus there are $f_1:{}_\ringR\modM_1\to {}_\ringR\modM_2$
	and  $f_2:{}_\ringR\modM_2\to {}_\ringR\modM_1$ such that $f_1f_2=1$ and $f_2f_1=1$.
Recall that ${}_\ringR\modP={}_\ringR\modM^1
	\oplus {}_\ringR\modM_1 = {}_\ringR\modM^2
	\oplus {}_\ringR\modM_2$. Define ${\cal Z}_1\in\End_\ringR({}_\ringR\modP)$ by applying the following assignment
\[
(a,b)\in{}_\ringR\modM^1
	\oplus {}_\ringR\modM_1 \mapsto (a {\cal W}_1, f_1(b)) \in {}_\ringR\modM^2
	\oplus {}_\ringR\modM_2.
\]
In the same vein,  define ${\cal Z}_2\in\End_\ringR({}_\ringR\modP)$ via
\[
(a,b)\in{}_\ringR\modM^2
	\oplus {}_\ringR\modM_2 \mapsto (a {\cal W}_2, f_2(b)) \in {}_\ringR\modM^1
	\oplus {}_\ringR\modM_1.
\]
Clearly, ${\cal Z}_1 {\cal Z}_2
	={\cal Z}_2 {\cal Z}_1=1=\id_\modP$.
	It is also easy to check that:
	\[\begin{array}{ll}
	{\cal X}{\cal Z}_1={\cal X}{\cal Z}_1{\cal Y},\\   (1-{\cal X}){\cal Z}_1
	=(1-{\cal X}){\cal Z}_1(1-{\cal Y}),\\
	{\cal Y}{\cal Z}_2={\cal Y}{\cal Z}_2{\cal X},\\  (1-{\cal Y}){\cal Z}_2
	=(1-{\cal Y}){\cal Z}_2(1-{\cal X}).
	\end{array}\]
	We use just one very simple
	non trivial ${\mathfrak e}\in {\callE}^{\frakss}$. According to Lemma \ref{groupsGn}, there are
	$n(\ast)<\omega$, $z_1$ and $z_2\in \modL^{{\mathfrak e},\ast}_{n(\ast)}$
	such that the
	equations above hold in the endomorphism ring of Abelian group
	$\varphi^{\mathfrak e}_{n(\ast)}(\modM)/\varphi^{\mathfrak e}_\omega(\modM)$
	for any
	bimodule $\modM$,\footnote{Recall from Lemma \ref{simplephiequalpsi} that $\varphi^{\mathfrak e}_n(\modM)\equiv \psi^{\mathfrak e}_n(\modM)$ holds for all $n<\omega.$} when we replace ${\cal Z}_1$ (resp. ${\cal Z}_2$)
	by ${\bf h}^{{\mathfrak e},n(\ast)}_{\modM,z_1}$ (resp.
	${\bf h}^{{\mathfrak e},n(\ast)}_{\modM,z_2}$) and interpret
	${\cal X},{\cal Y},{\cal W}_1,{\cal W}_2 \in \ringS$ naturally. This holds in particular
	for the bimodule $\modM^\otimes$
	we  defined in Step B).
So, the following equations hold:
\[\begin{array}{ll}
  \mathcal{X}{\bf h}^{{\mathfrak e},n}_{\modM^\otimes,z_1}=\mathcal{X}{\bf h}^{{\mathfrak e},n}_{\modM^\otimes,z_1}\mathcal{Y}, \\
 (1-\mathcal{X}){\bf h}^{{\mathfrak e},n}_{\modM^\otimes,z_1}
	=(1-\mathcal{X}){\bf h}^{{\mathfrak e},n}_{\modM^\otimes,z_1}(1-\mathcal{Y}), \\
\mathcal{Y}{\bf h}^{{\mathfrak e},n}_{\modM^\otimes,z_2}=	\mathcal{Y}{\bf h}^{{\mathfrak e},n}_{\modM^\otimes,z_2}\mathcal{X},\\
	(1-\mathcal{Y}){\bf h}^{{\mathfrak e},n}_{\modM^\otimes,z_2}
	=(1-\mathcal{Y}){\bf h}^{{\mathfrak e},n}_{\modM^\otimes,z_2}(1-\mathcal{X}).
\end{array}\]

These equations in turn define  certain decompositions of $\varphi^{\mathfrak e}_{n}(\modM^\otimes)/\varphi^{\mathfrak e}_\omega(\modM^\otimes)$ which yield to the following isomorphism
	$$
	\frac{\varphi^{\mathfrak e}_{n}\bigl(\sum_{i<\beta}
		\modM^\otimes_i\bigl)}{\varphi^{\mathfrak
			e}_\omega \bigl(\sum_{i<\beta}\modM^\otimes_i\bigr)}
	\stackrel{\cong}\longrightarrow
	\frac{\varphi^{\mathfrak e}_{n}\bigl(\sum_{i<\alpha}\modM^\otimes_i\bigl)}{
		\varphi^{\mathfrak e}_\omega\bigl(\sum_{i<\alpha}\modM^\otimes_i\bigr)}.
	$$
	The cardinality of left (resp. right) hand side is $|\beta| $ (resp. $|\alpha| $). Thus if we choose   $|\beta| >|\alpha| $, we get a contradiction that we searched for it.
	
$(2)$. This is in the proof of $(1)$.

$ (3)$. This follows from (2) and Lemma \ref{comparing two orders}.
\end{proof}

We now prove the existence of $\ringR$-modules with the Corner pathology. We present such a thing by applying Corollary  \ref{3.5}.

\begin{theorem}
	\label{3.7}	Let $m(\ast)>2$ and let $\lambda>|\ringR|$ be a cardinal of the form
	$\lambda=\left(\mu^{\aleph_0}
	\right)^+$. 	 Then  there is an $\ringR$-module $\modM$ of cardinality
	$\lambda$
	such that:
	$$
	\modM^n\cong \modM\quad\emph{ if and only if }\quad m(\ast)-1\mbox{ divides } n-1.
	$$
\end{theorem}

\begin{proof}
	We divide the proof into nine steps. 

	\textbf{Step A)}	 We first introduce
	 rings	  $\ringS_0$ and $\ringS$. The ring $\ringS_0$   is incredibly easy compared to  $\ringS$
and $\ringS$ is essentially the ring of endomorphisms we would like.

	To this end, let $\ringT$ be the subring of
	$\ringR$ which 1
	generates. Let $\ringS_0$ be the ring extending $\ringT$
	generated by $\{{\cal X}_0,\ldots, {\cal X}_{m(\ast)-1},{\cal W},
	{\cal Z}\}$ freely except the equations:
	\begin{description}
		\item[$(\ast)_1$] ${\cal X}^2_\ell={\cal X}_\ell$,\\
		${\cal X}_\ell {\cal X}_m=0$ ($\ell\neq m$),\\
		$1={\cal X}_0+\ldots+ {\cal X}_{m(\ast)-1}$,\\
		${\cal X}_\ell {\cal W}{\cal X}_m=0$ for $\ell+1\neq m\ \mod\ m(\ast)$,\\
		${\cal W}^{m(\ast)}=1$,\\
		${\cal Z}^2=1$,\\
		${\cal X}_0{\cal Z}(1-X_0)={\cal X}_0{\cal Z}$,\\
		$(1-{\cal X}_0){\cal Z}{\cal X}_0=(1-{\cal X}_0){\cal Z}$.
	\end{description}
	The meaning of these equations will become clear when we use them.
	
	Similarly, we define $\ringS$, but in addition we require $\sigma=0$, where $\sigma$ is a term in the language of rings,
	when $_{\bf D }\modM^\ast\sigma=0$
	for every field $\bf D$ and where ${}_{\bf D}\modM^*$ is  the $({\bf D},\ringS_0)$-
	bimodule as
	defined below. Now ``$\ringS$
	is a free $\ringT$-module'' will be proved later.
	
	For an integer $m$, the notation $[-\infty,m)$ stands for $\{n:n \mbox{ is an integer }<m\}$ and if
	$\eta\in {}^{[-\infty,m)}\omega$, then we set  $\eta\restriction k:=\eta
	\restriction [-\infty,\Min\{m,k\})$.
	
We look at the following sets:
	\[\begin{array}{lll}
	W_0:=&\big\{\eta:&\eta\mbox{ is a function with domain of the form }[-\infty,
	n)\\
	& &\mbox{and range }\subseteq\{1,\ldots,m(\ast)-1\},\mbox{ and such that}\\
	& &\mbox{for every small enough }m\in {\mathbb Z},\ \eta(m)=1\big\},
	\end{array}\]
	and   $W_1:=W_0\times \{0, \ldots ,m(\ast)-1\}$.

	Let $\bf D$ be a field   such that if $\ringT$  is finite and of cardinality $n$, then
 $\Char({\bf D})$ divides $n$.
	So ${\bf D} \otimes \ringS$ is the ring extending ${\bf D}$ by adding
	${\bf D},{\cal X}_0,\ldots,{\cal X}_{m(\ast)-1}, {\cal W}, {\cal Z}$
	as non commuting variables freely except satisfying the equation in $(\ast)_1$
	and if $||\ringT||$ is finite we divide $\ringS$ by $p
	\ringS$ where $p:=\Char({\bf D})$.
	So there is a homomorphism $g_{\bf D}$ from $\ringS$ to ${\bf D}\otimes\ringS$
	\st\ $\{0_{\ringS}\}= \cap \{ \Ker (g_{\bf D}):{\bf D} \emph{ as above}\}$.

	Let $\modM^\ast={}_{\bf D}\modM^\ast$ be the left
	$\bf D$-module freely generated by
	\[\{x_{\eta,\ell}: \eta\in W_0,\ \ell<m(\ast)\}.\]
	We make ${}_{\bf D}\modM^\ast$ to a right
	$({\bf D}\mathop{{\otimes}} \ringS_0)$-module by defining $x z$ for $x\in
	{}_{{\bf D}}\modM^*$ and $z\in\ringS_0$. It is enough to deal with
	$z\in
	\{{\cal X}_m:m<m(*)\}\cup\{{\cal Z},{\cal W}\}$. Let $x=
	\sum\limits_{\eta,\ell} a_{\eta,\ell} x_{\eta,\ell}$ where
	$(\eta,\ell)$ vary on
	$W_0$, $a_{\eta,\ell}\in{\bf D}$ and $\{(\eta,\ell):a_{\eta,\ell}\neq 0\}$
	is finite. We define
$$(\sum\limits_{\eta,\ell} a_{\eta,\ell}x_{\eta,\ell})z:=\sum\limits_{\eta,
		\ell}a_{\eta,\ell}(x_{\eta,\ell} z).$$
Here, we define the action
	of $\{ {\cal X}_m,{\cal Z}, {\cal W} \}$ on $x_{\eta,\ell}$ as follows:
	\[\begin{array}{rcl}
	x_{\eta,\ell}{\cal X}_m&:=&\left\{\begin{array}{ll}
	x_{\eta,\ell}&\mbox{if }\ell=m,\\
	0 &\mbox{if }\ell\neq m,
	\end{array}\right.\\
	\\
	x_{\eta,\ell}{\cal Z} &:=&\left\{\begin{array}{ll}
	x_{\eta\conc\langle\ell\rangle,0} &\mbox{if }\ell>0,\\
	x_{\eta\restriction [-\infty,n-1),\eta(n-1)}&\mbox{if }\ell=0,\mbox{
		and } (-\infty,n)=\Dom(\eta),
	\end{array}\right.\\
	\\and\\
	x_{\eta,\ell}{\cal W} &:=&x_{\eta,m}\ \mbox{ when }m=\ell+1\ \mod\ m(\ast).
	\end{array}\]
	We get a $({\bf D},\ringS)$-bimodule as the identities in the
	definition of $\ringS_0$ and $\ringS$
	holds. If ${\bf D}=\ringT$, some by inspection (those of $(\ast)_1$), the rest by the
	choice of $\ringS$. If ${\bf D}\neq \ringT$, by the restriction on ${\bf D}$.
	Let
$$_{\bf D}\modM^\ast_\ell:=\{\sum\limits_\eta d_{\eta,\ell} x_{\eta, \ell}: \
	\eta\in W_0 \mbox{ and } d_{\eta, \ell}\in {\bf D}\}.$$
So clearly ${}_{\bf D}\modM^\ast=
	\mathop{{\bigoplus}}\limits_{\ell=0}^{m(\ast)-1} {}_{\bf D} \modM^\ast_\ell$.

	\textbf{Step B)}  In this step we introduce an $\ringR$-module $\modM$  such that
	$
	\modM^{ m(\ast)-1}\cong \modM
	$.

	Let $\modP$ be the $(\ringR, \ringS)$-bimodule  constructed in
Lemma	\ref{prnalphafez} for $\lambda$. Here, we used the existence of seminice construction and its strong version.

  We look at  $\modP_\ell:=\modP{\cal X}_\ell$.   Let $i\neq j$. We use the formula ${\cal X}_i{\cal X}_j=0$  to observe that $\modP_i\cap  \modP_j=0$.
	Since
	$\sum\limits^{m(\ast)-1}_{\ell=0} {\cal X}_\ell=1$ we have
	$${}_{\ringR}
	\modP=\sum\limits^{m(\ast)-1}_{\ell=0} \modP {\cal X}_\ell=\mathop{{\bigoplus}}\limits^{m(\ast)-1}_{\ell=0}
	{}_{\ringR}\modP_\ell.$$
Suppose $\ell+1\neq m\ \mod\ m(\ast)$. We combine the formula   ${\cal X}_\ell {\cal W}{\cal X}_m=0$     with the formula $\sum\limits^{m(\ast)-1}_{\ell=0} {\cal X}_\ell=1 $ to observe that
	$$  \modP_\ell{\cal W}=\modP{\cal X}_\ell{\cal W}({\cal X}_0+\ldots+{\cal X}_{m(\ast)})= \modP_\ell{\cal W}{\cal X}_{\ell+1}\subset \modP_{\ell+1
	},$$
	i.e.,  ${\cal W}:{}_\ringR\modP_\ell\twoheadrightarrow {}_\ringR\modP_{\ell+1
	}$ is surjective.
	Here, we use the  relation ${\cal W}^{m(\ast)}=1$ to equip ${\cal W}$ as  an embedding  from
	$\modP$ onto $\modP$. Hence, ${\cal W}\restriction {}_\ringR\modP_\ell$
	is an isomorphism
	from ${}_\ringR\modP_\ell$ onto ${}_\ringR\modP_{\ell+1}$.
	So $${}_\ringR\modP_{m(\ast)-1}\cong\ldots\cong {}_\ringR\modP_1\cong {}_\ringR\modP_0.$$
	Thanks to the formulas   ${\cal X}_0{\cal Z}(1-X_0)={\cal X}_0{\cal Z}$ and $\sum\limits^{m(\ast)-1}_{\ell=0} {\cal X}_\ell=1 $ we conclude that
	${\cal Z}$ maps ${}_\ringR\modP_0$ into
	$$ \modP {\cal X}_0{\cal Z}=\modP {\cal X}_0{\cal Z}(1-X_0)=\modP {\cal X}_0{\cal Z}( {\cal X}_1+\ldots+{\cal X}_{m(\ast)-1})
	\subset\mathop{{\bigoplus}}\limits^{m(\ast)-1}_{\ell=1} \modP  {\cal X}_\ell
	=\mathop{{\bigoplus}}\limits^{m(\ast)-1}_{\ell=1} \modP_\ell\cong
	\modP_0.
	$$
	By the same vein, ${\cal Z}$ maps  $
	\mathop{{\bigoplus}}\limits^{m(\ast)-1}_{\ell=1} \modP_\ell
	$   into     ${}_\ringR\modP_0$. We are going to use the formula ${\cal Z}^2=1$ to    exemplifies ${\cal Z}$ with the following isomorphism
	$$
	\mathop{{\bigoplus}}\limits^{m(\ast)-1}_{\ell=1}{}_\ringR\modP_\ell\cong
	{}_\ringR\modP_0.
	$$
But, we have just shown
	$$
	\mathop{{\bigoplus}}\limits^{m(\ast)-1}_{\ell=1}{}_\ringR\modP_\ell\cong
	({}_\ringR\modP_0)^{m(\ast)-1}.
	$$This completes the proof
	of Step B).

	So, it is enough to show
	\begin{description}
		\item[$(*)_2$] \qquad\qquad\qquad $1<k<m(\ast)-1\quad\Rightarrow\quad {}_\ringR
		\modP^k_0\not\cong {}_{\ringR}\modP_0$.
	\end{description}
	Assume $k$ is a counterexample. The desired contradiction will be presented in
	Step I), see below. To this end we need some preliminaries.

	\textbf{Step C)} There exist a field ${\bf D}$ and  ${\cal Y}\in
	{\bf D}\mathop{{\otimes}}\limits_\ringT \ringS$ satisfying the
	following equations:
	\begin{description}
		\item[$(*)_3$ ] ${\cal Y}\restriction {}_{\bf D}\modM^\ast_0$ is an
		isomorphism from
		${}_{\bf D}\modM^\ast_0$ onto $\mathop{{\bigoplus}}\limits_{\ell=1}^k
		{}_{\bf D}\modM^\ast_\ell$,
		
		${\cal Y}\restriction \mathop{{\bigoplus}}\limits_{\ell=1}^k
		{}_{\bf D}\modM^\ast_\ell$ is an
		isomorphism from $\mathop{{\bigoplus}}\limits_{\ell=1}^k
		{}_{\bf D}\modM^\ast_\ell$ onto ${}_{\bf D}\modM^\ast_0$,
		
		${\cal Y}\restriction\mathop{{\bigoplus}}\limits^{m(\ast)}_{\ell=k+1}
		{}_{\bf D}\modM^\ast_\ell$ is the identity, and
		
		${\cal Y}^2=1$.
	\end{description}

	Indeed, recall from  ${}_\ringR
	\modP^k_0\cong {}_{\ringR}\modP_0$ that  $_\ringR\modP$  is equipped with an
	endomorphism $\fucF$ such that:
	\begin{description}
		\item[$(*)_4$ ] $\fucF\restriction {}_\ringR\modP_0$ is an
		isomorphism from
		${}_\ringR\modP_0$ onto ${}_\ringR\modP_1\oplus\ldots\oplus
		{}_\ringR\modP_k$,
		
		$\fucF\restriction (_\ringR\modP_1\oplus\ldots\oplus{}_\ringR
		\modP_k)$ is an
		isomorphism from $_\ringR\modP_1\oplus\ldots
		\oplus {}_\ringR\modP_k$ onto $_\ringR\modP_0$,
		
		$\fucF\restriction{}_\ringR\modP_j$, for $k<j<m(\ast)$, is the identity, and
		
		$\fucF^2=1$.
	\end{description}
Assuming $(\ringS,+)$ is a free $\ringT$-module,  according to Corollary \ref{3.5},	there is  a
	field $\bf D$   with the property that $p:=\Char({\bf D})$ divides $n:=\Char(\ringR)$,
	when $n>0$,  and there exists ${\cal Y}\in
	{\bf D}\mathop{{\otimes}}\limits_\ringT \ringS$ satisfying the
	the desired equations.
	This completes the proof of Step C).

	In sum, we are left to show that $(\ringS,+)$ is a free $\ringT$-module.
	Note that each of $\{1_S, {\cal X}_0,\ldots, {\cal X}_{m(\ast)}, {\cal W}, {\cal Z\}}$
	map each generator $x_{\eta,\ell}$
	to another generator or zero, so this applies to any composition of
	them, in fact, in quite specific way for $\rho$, $\nu\in
	\{1,\ldots,m (\ast)-1\}^{< \omega}$, $k$, $m<\omega$.
	 Let
	${\cal Y}^{k,m}_{\rho,\nu}$ be a
	monomial operator in the generators of $\ringS$ such that:
	$$
	x_{\eta,\ell}{\cal Y}^{k,m}_{\rho,\nu}=\left\{\begin{array}{ll}
	x_{\eta_1\conc\nu,m} & \mbox{ if } \ell=k, \mbox{ and } \eta=\eta_1\conc\rho \mbox{  for some }\eta_1
	\in W_0\\
	0 & \mbox{  otherwise.}
	\end{array} \right.
	$$
In order to define these operators,  first we define $Y^{k,k}_{\rho,\langle\rangle}$ and ${\cal Y}^{k,k}_{\langle\rangle,\nu}$
 by induction on $\lg(\rho)$ and $\lg(\nu)$ respectively.  Also, negative powers of ${\cal W}$
	can be defined, because ${\cal W}$ is invertible. Now,
	we let $${\cal Y}^{k,m}_{\rho,\nu}:=Y^{k,k}_{\rho,
		\langle\rangle}
	{\cal Y}_{\langle\rangle,\nu}^{k,k} {\cal W}^{m-k}.$$

	\textbf{Step D)}:  Set $\Omega=\{{\cal Y}^{k,m}_{\rho,\nu}:\ \rho, \nu\in
		[1, m(\ast))^{< \omega},\ k, m \in [0,
		m(\ast))$ and if
		$\lg(\rho)>0$ and $\lg
		(\nu)>0$, then $\rho (0)= \nu(0)]\}.$ Then
	\begin{enumerate}
		\item[$\boxtimes_1$] \qquad\qquad\qquad\qquad $\Omega$ generates $\ringS$ as a $\ringT$-module
		\end{enumerate}	
			First 		note that if ${\cal Y}^{k,m}_{\rho_1,\nu_1}$ is not in the family,
		then for some $(\varrho,\rho, \nu)$
		we have $Y^{k,m}_{\rho ,\nu }$
		belongs to the family
		and $(\rho,\nu)=(\varrho^\frown\rho_1,\varrho\conc\nu_1)$. Hence
		${\cal Y}^{k,m}_{\rho,\nu}={\cal Y}^{k,m}_{\rho_1,\nu_1}$. So, $\Omega$ generates all of ${\cal Y}^{k,m}_{\rho,\nu}$'s and thus we can use them.

		Let $\ringS'$ be the sub-$\ringT$-module of $\ringS$ generated by
		$\ringT$ and the family above.
		Now $\ringT\subseteq \ringS'$ because
\begin{enumerate}
\item[$(\dagger)_1$] \qquad\qquad\qquad  $1=\sum\{Y_{\langle\rangle,\langle\rangle}^{m,m}:
		m =0,\ldots,m(*)-1\}.$
\end{enumerate}
To see this, we evaluate  both sides at $x_{\eta,\ell}$.
		The right hand side of   $(\dagger)_1$
		is equal to
		
		\[\begin{array}{ll}
		x_{\eta,\ell}\sum\{Y_{\langle\rangle,\langle\rangle}^{m,m}:
		m =0,\ldots,m(*)-1\}&=x_{\eta^{\frown}\langle\rangle,\ell}\sum\{Y_{\langle\rangle,\langle\rangle}^{m,m}:
		m =0,\ldots,m(*)-1\}\\
		&=x_{\eta^{\frown}\langle\rangle,\ell}Y_{\langle\rangle,\langle\rangle}^{\ell,\ell}\\
		&=x_{\eta^{\frown}\langle\rangle,\ell}\\
		&=x_{\eta,\ell}.
		\end{array}\]	
		Since the left hand side of $(\dagger)_1$ is the identity, we get the desired equality.
		
		Next we show
		${\cal X}_m\in \ringS'$ for each $m<m(*)$. It suffices to show that
\begin{enumerate}
\item[$(\dagger)_2$] \qquad\qquad\qquad\qquad\qquad\qquad  ${\cal X}_m=
		{\cal Y}^{m,m}_{\langle\rangle,\langle\rangle}.$
\end{enumerate}
	In order to see   $(\dagger)_2$, we  evaluate both sides at $x_{\eta,\ell}$.
		Let $n=\lg(\eta)+1$.  First, assume that $\ell=m$. Then,  the right hand side  of $(\dagger)_2$
		is equal to  $$x_{\eta,\ell}{\cal Y}^{m,m}_{\langle\rangle,\langle\rangle}
		=x_{\eta,m}{\cal Y}^{m,m}_{\langle\rangle,\langle\rangle}=x_{\eta,m},$$ which is equal to $x_{\eta,\ell} {\cal X}_m$. Now, we show the claim when $\ell\neq m$. In this case both sides of $(\dagger)_2$ are equal to zero, e.g., $(\dagger)_2$ is valid.

	In order to show ${\cal W}\in\ringS'$ we bring the following claim:
\begin{enumerate}
\item[$(\dagger)_3$] \qquad\qquad\qquad\qquad ${\cal W}=\sum\{
		{\cal Y}^{\ell,m}_{\langle\rangle,\langle\rangle}:\ell
		=m+1 \mod\ m(*)\}.$
\end{enumerate}
As before, it is enough to  evaluate  both sides of   $(\dagger)_3$ at $x_{\eta,\ell}$.
		Let $m$ be such that $\ell\equiv_{m(\ast)} m+1$.
		The right hand side of $(\dagger)_3$
		is equal to
	\[\begin{array}{ll}
	x_{\eta,\ell}\sum\{
	{\cal Y}^{\ell,m}_{\langle\rangle,\langle\rangle}:\ell
	=m+1 \mod\ m(*)\}
		&=x_{\eta,\ell}{\cal Y}^{\ell,m}\\
		&=x_{\eta,m}.
		\end{array}\]
By definition, this is equal to the left hand side of  $(\dagger)_3$.

		Finally, we claim that ${\cal Z}\in\ringS'$. Indeed,  it suffices to prove that
\begin{enumerate}
\item[$(\dagger)_4$] ${\cal Z}=
		\{{\cal Y}^{0,m}_{\langle m\rangle,\langle\rangle}:m=1,\ldots,m(*)-1\}+
		\sum\{{\cal Y}^{m,0}_{\langle\rangle,\langle m\rangle}:m=1,\ldots, m(*)-1\}.$
\end{enumerate}
 Again we  evaluate  both sides of $(\dagger)_4$ at $x_{\eta,\ell}$.
		Let $n=\lg(\eta)+1$ and first suppose that $\ell=0.$ Then,  the right hand side of $(\dagger)_4$
		is equal to
		\[\begin{array}{ll}
		x_{\eta,0}\{{\cal Y}^{0,m}_{\langle m\rangle,\langle\rangle}:m=1,\ldots,m(*)-1\}
		&+x_{\eta,0}\sum\{{\cal Y}^{m,0}_{\langle\rangle,\langle m\rangle}:m=1,\ldots, m(*)-1\}\\
		&=
		x_{\eta,0}\{{\cal Y}^{0,m}_{\langle m\rangle,\langle\rangle}:m=1,\ldots,m(*)-1\}
		\\
		&=x_{\eta{\restriction}{n-1}\conc\langle  \eta(n)\rangle,0}{\cal Y}^{0,\eta(n)}\\
		&=x_{\eta{\restriction}{n-1}\conc\langle  \eta(n)\rangle,\eta(n)}.
		\end{array}\]
		This is equal to the left hand side of $(\dagger)_4$. Suppose now that $\ell>0$.
		Then the right hand side
		is equal to  $x_{\eta \conc\langle  \ell\rangle,0}$. By definition this is equal to the left hand side  of $(\dagger)_4$.

	So far,  we have proved that the subring $\ringS'$  includes $\ringT,{\cal X}_m (m<m(*)), {\cal Z}$
		and ${\cal W}$ and is a sub $\ringT$-module. To finish this step, it suffices
		to prove that  $\ringS'$ is closed under product. For this, it is enough to prove
		that $\ringS'$ includes the product ${\cal Y}^{k^1,m^1}_{\rho_1,\nu_1}
		{\cal Y}^{k^2,m^2}_{\rho_2,\nu_2}$ of any two members of  $\Omega$.
		
		Now if $m^1\neq k^2$ or no one of $\nu_1,\rho_2$ is an end-segment
		of the other then this is not the case. So easily the product is
		$Y^{k^1,m^2}_{\rho_3,\nu_3}$ where:
\begin{itemize}
\item if	$\nu_1=\rho_2$, then $(\rho_3,\nu_3)=(\rho_1,\nu_2)$,
\item if
		$\nu_1=(\varrho^1)^{\frown} \rho_2$, then $(\rho_3,\nu_3)=(\rho_1,(\varrho^1)^{\frown}
		\nu_2)$,
\item if $\rho_2=(\varrho^1)^{\frown}\nu_1$, then $(\rho_3,\nu_3)=
		((\varrho^1)^{\frown} \nu_1,\nu_2)$.
\end{itemize}
 So $\boxtimes_1$ holds, which completes the proof
		of Step D).

		\textbf{Step E)}  The set $\Omega$ from Step D)
		is a free basis of $\ringS$ as a $\ringT$-module.
		
		Assume that ${\cal X}=\sum\{a^{k,m}_{\nu,\rho}
		{\cal Y}^{k,m}_{\nu,\rho}:\rho,k,\nu,m\}=0$ for some finite set
		$\{(\rho,\nu,k,m):a^{k,m}_{\nu,\rho}\neq 0\}$.
We have to show that $a^{k,m}_{\nu,\rho}=0$, for all such $\rho,k,\nu,m$.

	Let $k^*<\omega$ and
		$\rho^*\in \{0,\ldots,m(*)-1\}^{< \omega}$ be such that $\lg(\rho^*) >\sup\{\lg
		(\rho):a^{k,n}_{\rho,\nu}\neq 0\}$. We look at $x_{\rho^*,
			k^*}{\cal X}$. Let $A(k^*, \rho^*)$ be the set of all $\nu$ in the finite sequence above where $\rho^*=\rho^*_\nu$$^{\frown} \nu$,
for some $\rho^*_\nu.$
Then
		
		\[\begin{array}{ll}
		x_{\rho^*,
			k^*}{\cal X}&=
\sum\limits_{\nu, \rho, k, m} a^{k,m}_{\nu,\rho}(x_{\rho^*,k^*}
		{\cal Y}^{k,m}_{\nu,\rho}) \\
&=\sum\limits_{\nu, \rho, m}\{a^{k^*,m}_{\nu,\rho} (x_{\rho^*,k^*}
		{\cal Y}^{k^*,m}_{\nu,\rho}):\nu \in A(k^*, \rho^*)\}\\
		&=\sum\limits_{\nu, \rho, m} \{a^{k^*,m}_{\nu,\rho}
		x_{(\rho^*_\nu)^{\frown} \rho, m}:\nu \in A(k^*, \rho^*)\}\\
&=0.
		\end{array}\]
Recall that $\modM^\ast$ is the left $\bf D$-module freely generated by $x_{\eta, \ell}$'s, so as we are free in choosing $k^*, \rho^*,$
we can easily show that
 for any
	tuple	$(\rho,\nu,k,n)$ in the finite set above, for some suitable choice of  $\rho^*$,
		$a^{k,n}_{\rho,\nu}$ is the only component of $x_{(\rho^*_\nu)^{\frown} \rho, m}$ in the sum above, and hence $a^{k,n}_{\rho,\nu}=0$. This completes the proof of Step E).
%

		Before we continue, let us introduce some notations and definitions.
	Clearly,  ${\cal Y}$ can be represented as a  finite  sum of the form:
		$$
		{\cal Y}=\sum\{d^{k,m}_{\rho,\nu}
		{\cal Y}^{k,m}_{\rho,\nu}: k,m,\rho,\nu\}.
		$$
	 Choose $n(*)$
		large enough such that $n(*)>\Max\{\lg(\rho),\lg(\nu):d^{k,m}_{\rho,\nu}
		\neq 0\}$.
		For any $u\subseteq \{(\eta,\ell):\eta\in W_0,\ \ell<m(\ast)\}$, let $\modN_u$ be
		the sub ${\bf D}$-module of ${}_{\bf D}\modM^\ast$ generated by
		$\{x_{\eta,\ell}:(\eta,\ell)
		\in u\}$.
		
		For $\eta\in w_0$ and $\ell=1,\ldots,m(\ast)$ set
		${}_{\bf D} \modM^\ast_{[1,k]}:=
		\mathop{{\bigoplus}}\limits^k_{j=1} \modM^\ast_\ell$,
		and let

		\[\begin{array}{ll}
		w_{\eta,\ell}&:=\{(\nu,m)|(\nu,m)=(\eta,\ell)\mbox{ or }
		\eta^\frown\langle\ell\rangle\trianglelefteq\nu \mbox{ and } m<m(*)\},
		\\u_\ell
		&:=\{(\eta,m)| m=\ell,\ \eta\in w_0\},
		\\w^m_{\eta,\ell},
		&:=\ w_{\eta,\ell}\cap u_m, \\
		w^{[1,n]}_{\eta,\ell}
		&:=w_{\eta,\ell}\cap\bigcup\limits_{m\in [1,n]}u_m.
		\end{array}\]

	\textbf{Step F)}
	For any $\eta\in w_0$ and $\ell<m(\ast)$, there is a finite subset $u\subseteq
	w^0_{\eta,\ell}$ such that
	\begin{enumerate}
		\item[$(\alpha)$] ${\cal Y}$ maps $\modN_{w^0_{\eta,\ell}\setminus u}$
		into $\modN_{w_{\eta,\ell}}$
		in fact into $\modN_{w^{[1,k]}_{\eta,\ell}}$
		\item[$(\beta)$] if $v$ is finite and $u\subseteq v\subseteq
		w^0_{\eta,\ell}$ then
		$\modN_{w^0_{\eta,\ell}\setminus v} {\cal Y}$ is a ${\bf D}$-vector subspace of $\modN_{w^0_{\eta,\ell}} {\cal Y}$
		of cofinite dimension.
		\item[$(\gamma)$] for any finite subsets
		$v_1$, $v_2$ of $w^0_{\eta,\ell}$
		extending $u$  the following equalities are true:
		\[\begin{array}{l}
		\dim(\modN_{w^0_{\eta,\ell}}/\modN_{w^0_{\eta,\ell}\setminus v_1})-
		\dim(\modN_{w^{[1,k]}_{\eta,\ell}}/(\modN_{w^0_{\eta,\ell}
			\setminus v_1}{\cal Y}))\\
		=\dim (\modN_{w^0_{\eta,\ell}}/\modN_{w^0_{\eta,\ell}\setminus v_2})-\dim
		(\modN_{w^{[1,k]}_{\eta,\ell}}/(\modN_{w^0_{\eta,\ell}
			\setminus v_2}{\cal Y})),
		\end{array}\] where the dimensions are computed as ${\bf D}$-vector spaces.
	\end{enumerate}

	Indeed, the case $\ell=0$ is easy.
	Recalling the representation of ${\cal Y}$ and the choice of
	$n(\ast)<\omega$,   if we set
	$$
	u=\{(\nu,m)\in w_{\eta,\ell}:\ |\Dom(\nu)\setminus\Dom(\eta)|<n(\ast)\},
	$$
	then we have
	$$
	(\nu,m)\in w_{\eta,\ell}\setminus u\quad\Rightarrow \quad x_{\nu,m} {\cal Y}\in
	\modN_{w_{\eta,\ell}}.
	$$
	Note that $u$ is finite and clause $(\alpha)$ holds.
	As ${\cal Y} {\cal Y}=1$
	etc., we can show
	that $(\beta)$ holds. For checking the
	equalities in clause $(\gamma)$, let $v_3:=v_1\cup v_2$. Due to the transitivity of
	equality, it is enough to
	prove the required equality for  pairs $(v_1,  v_3)$ and  $(v_2, v_3)$. This enable us
	to assume
	without loss of generality  that  $v_1\subseteq v_2$. Also, we stipulate $v_0=\emptyset$.
	
	Now
	$$
	\modN_{w^0_{\eta,\ell}\setminus v_2}\subseteq \modN_{w^0_{\eta,\ell}
		\setminus v_1}\subseteq \modN_{w^0_{\eta,\ell}\setminus w_0}=
	\modN_{w^0_{\eta,\ell}},
	$$
	and
	$$
	\dim (\modN_{w^0_{\eta,\ell}}/\modN_{w^0_{\eta,\ell}
		\setminus v_i})=|v_i|.
	$$
	Manipulating the equalities in clause $(\alpha)$, they are equivalent to
	\begin{enumerate}
	\item[(*)]\qquad\qquad\qquad	 $\dim(\modN_{w^0_{\eta,\ell}}/
		\modN_{w^0_{\eta,\ell}\setminus
			v_1})-\dim( \modN_{w^0_{\eta,\ell}}/
		\modN_{w^0_{\eta,\ell}\setminus v_2})=$
		
		\qquad\qquad$\dim(\modN_{w^{[1,k]}_{\eta,\ell}}/(\modN_{w^0_{\eta,\ell}
			\setminus v_1}{\cal Y}))-
		\dim(\modN_{w^{[1,k]}_{\eta,\ell}}/(
		\modN_{w^0_{\eta,\ell}\setminus v_2}{\cal Y})).$
	\end{enumerate}
	Since $\modN_{w^0_{\eta,\ell}\setminus v_2}\subseteq
	\modN_{w^0_{\eta,\ell}
		\setminus v_1} \subseteq \modN_{w^0_{\eta,\ell}}$,  the left hand side
of $(*)$ is
	equal to $$\dim(\modN_{w^0_{\eta,\ell}\setminus v_1}/
	\modN_{w^0_{\eta,\ell} \setminus
		v_2}).$$ Also, $\modN_{w^0_{\eta,\ell}\setminus v_2}{\cal Y}\subseteq
	\modN_{w^0_{\eta,\ell}
		\setminus v_1}{\cal Y} \subseteq \modN_{w^{[1,k]}_{\eta,\ell}}$
	(the first inclusion holds
	as $\modN_{w^0_{\eta,\ell}\setminus v_2}\subseteq
	\modN_{w^0_{\eta,\ell}\setminus
		v_1}$ and the second one holds by clause $(\beta)$ and $u\subseteq v_1\cap v_2$).
	Hence the right hand side of
	$(\ast)$ is equal to $$\dim(\modN_{w^0_{\eta,\ell}\setminus v_1}
	{\cal Y}/ \modN_{w^0_{\eta,
			\ell}\setminus v_2}{\cal Y}).$$ Hence  $(\ast)$ is reduced in proving
	\begin{enumerate}
		\item[$(\ast)'$] \qquad\qquad $\dim(\modN_{w^0_{\eta,\ell}\setminus v_1}/
		\modN_{w^0_{\eta,\ell}
			\setminus v_2})=\dim(\modN_{w^0_{\eta,\ell}\setminus v_1}/
		(\modN_{w^0_{\eta,\ell}
			\setminus v_2}{\cal Y}))$.
	\end{enumerate}
	Since ${\cal Y}$ is an isomorphism from $\modN_{w^0_{\eta,\ell}
		\setminus v_2}$ onto $\modN_{w^0_{\eta,\ell}
		\setminus v_2}{\cal Y}$,
	it implies the validity of $(\ast)'$.
	This completes the proof.
	\medskip

 Let ${\bf n}_{\eta,\ell}$
	be the natural number so
	that for every large enough finite subset $v\subseteq w_{\eta,\ell}$
	$$
	{\bf n}_{\eta,\ell}:=\dim\left(\frac{\modN_{w^0_{\eta,\ell}}}{\modN_{w^0_{\eta,\ell}\setminus v}}\right)-
	\dim\left(\frac{\modN_{w^{[1,
				k]}_{\eta,\ell}}}{\modN_{w^0_{\eta,\ell}\setminus v} {\cal Y} }\right).
	$$
In view of Step F), this is a well-defined notion, which does not depend on the choice of $v$.	
	\medskip
	
	\textbf{Step G)} If $\eta$, $\nu\in w_0$ and $\ell\in
	\{0,1,\ldots,m(\ast)-
	1\}$, then ${\bf n}_{\eta,\ell}={\bf n}_{\nu,\ell}$.
	
	To see this, we  define a function $f:w_{\eta,\ell}\longrightarrow w_{\nu,\ell}$ by
	$f(\eta \conc\rho,m)=
	(\nu\conc\rho,m)$ for any
	$\rho\in\{0,\ldots,m(*)-1\}^{<\omega}$.
	This function is one-to-one and onto,
	and it induces an isomorphism from
	$\modN_{w^{[0,m(\ast))}_{\eta,\ell}}$ onto
	$\modN_{w_{\nu,\ell}^{[0,m(\ast))}}$. It almost commutes with all of our
	operations (the problems are
	``near'' $\eta$ and $\nu$). So choose $v_1\subseteq
	w^0_{\eta,\ell}$ large enough, as
	required in the definition of ${\bf n}_{\eta,\ell}$, ${\bf n}_{\nu,\ell}$
	and to
	make $\modN_{w^0_{\eta,\ell}\setminus v_1}
	{\cal Y}\subseteq \modN_{w^{[1,k]}_{\eta,\ell}}$, and let
	$v_2=f(v_1)$. Now check the desired claim.
	
	\medskip
	So,  we shall write
	${\bf n}_\ell$ instead of ${\bf n}_{\nu,\ell}$. By Step G), this is well-defined.
	
	\textbf{Step H)} The following equations are valid:

	$$
	\ideallI_\ell =\left\{\begin{array}{ll}
	0 &\mbox{if }  \ell\in [1,k]  \\
	\ideallI_0+\ideallI_1+\ldots+\ideallI_{m(\ast)-1} &\mbox{if }  \ell\in[k+1,m(\ast))\emph{  or  }\ell=0
	\end{array}\right.
	$$

	\medskip
	
To prove this, first note that $w_{\eta,\ell}=\{(\eta,\ell)\}\cup
	\bigcup\limits_{m<m(*)} w_{\eta\conc\langle\ell\rangle,m}$. In order to apply Step F), we fix a pair 	$(\eta,\ell)$
	and  we choose
	a finite  subset $v \subseteq w^0_{\eta,\ell}$ large enough as there.
Let $(\eta\conc\langle\ell\rangle,m)$ be
 \st
	$$
	\{(\eta,\ell)\}\cup \{(\eta\conc\langle\ell\rangle,m):\ m<m(\ast)\}
	\subseteq v.
	$$
	Now we compute $\ideallI_{\eta,\ell}$  which is equal to $\ideallI_\ell$,
	and $\ideallI_{\eta\conc\langle\ell \rangle,m}$   which is equal to $=\ideallI_m$
	for $m<m(*)$ and we shall get the equality.

Now set
$${\bf n}^1_{v,\eta,\ell}:=\dim\left(\frac{\modN_{w^0_{\eta,\ell}}}{
	\modN_{w^0_{\eta,\ell}\setminus v}}\right)-\sum\limits_{m<m(*)}\dim\left(\frac{
	\modN_{w^0_{\eta\conc \langle\ell\rangle,m}}}{\modN_{w^0_{\eta\conc
			\langle\ell\rangle,m}\setminus v}}\right).$$
By definition, we have $\modN_{w^0_{\eta,\ell}}=\bigoplus\limits_{m<m(*)}
\modN_{w^0_{\eta\conc
		\langle\ell\rangle,m}}\oplus {\bf D} x_{\eta,\ell}$. Using this equality, we can easily see that	
	\begin{enumerate}
		\item[$(\otimes_1)$] ${\bf n}^1_{v,\eta,\ell}$ is equal to $1$
		when $\eta\in w_0, \ell<m(\ast)$
		and $v\subseteq w_0$ is finite large enough (we can replace $v$ by
		$v\cap w^0_{\eta,\ell}$ of course).
	\end{enumerate}

In particular,  the following definition makes sense:
	\begin{enumerate}
		\item[$(\otimes_2)$] $\ideallI^2_{v,\eta,\ell}:=\dim
		\left(\frac{\modN_{w^{[1,k]}_{\eta,\ell}}}
		{\modN_{w^{[1,k]}_{\eta,\ell}\setminus v}}\right)-\sum\limits_{m<m(*)}\dim
		\left(\frac{\modN_{w^{[1,k]}_{\eta\conc
					\langle\ell\rangle,m}}}{\modN_{w^{[1,k]}_{\eta\conc
					\langle\ell\rangle,m}\setminus v}}\right).$
	\end{enumerate}
Suppose  $\eta\in w_0,\ell<m(\ast)$ and  $v\subseteq w_0$
is finite and large enough.
	In view of definition, and as  the argument of $(\otimes_1)$ we present the following three  implications:
	\par\noindent
	$(\otimes_{2.1})\quad\ell=0\ \Rightarrow\  {\bf n}^2_{v,\eta,\ell}=0$,
	\par\noindent
	$(\otimes_{2.2})\quad\ell\in [1,k]\ \Rightarrow\  {\bf n}^2_{v,\eta,\ell}=1$,
	\par\noindent
	$(\otimes_{2.3})\quad\ell\in [k+1,m(*)) \Rightarrow\ {\bf n}^2_{v,\eta,\ell}=0$.
	
	Next, we show that
	\begin{enumerate}
		\item[$(\otimes_3)$] ${\bf n}^1_{v,\eta,\ell}-{\bf n}^2_{v,\eta,\ell}
		=\ideallI_\ell-\sum\limits_{m<m(*)}\ideallI_m$.
	\end{enumerate}
Indeed, we have
	\[\begin{array}{ll}
	{\bf n}^1_{v,\eta,\ell}-{\bf n}^2_{\nu,\eta,\ell}&= \rdim(\frac{ \modN_{w^0_{\eta,\ell}}}{\modN_{w^0_{\eta,\ell}\setminus v}})-
	\rdim (\frac{\modN_{w^{[1,k]}_{\eta,\ell}}}{w^{[1,k]}_{\eta,\ell}\setminus v})\\
	&-\sum\limits_{m<n(*)}   \rdim
	(\frac{\modN_{w^0_{\eta\conc\langle\ell\rangle,m}}}
	{\modN_{w^0_{\eta^{\conc} {\langle\ell\rangle,m}}\smallsetminus v}})\\
	&+\sum\limits_{m<n(*)}  \rdim
	(\frac{\modN_{w^{0}_{\eta^{\conc} {\langle\ell\rangle,m}}   }}{\modN_{w^{[1,k]}_{\eta^{\conc} {\langle\ell\rangle,m}}\smallsetminus v}})\\
&={\bf n}_{\eta,\ell}
	-\sum\limits_{\ell<m(*)} {\bf n}_{\eta^{\conc} {\langle\ell\rangle,m}}\\
&={\bf n}_\ell-\sum\limits_{m<m(*)}{\bf n}_m.
	\end{array}\]

	Now we combine $(\otimes_1)$, $(\otimes_{2.1})$ and $(\otimes_{2.2})$
  along with $(\otimes_{2.3})$ to deduce the following formula.

\begin{enumerate}	\item[$(\otimes_4)$]$
{\bf n}^1_{v,\eta,\ell}-{\bf n}^2_{v,\eta,\ell}=\left\{\begin{array}{ll}
0 &\mbox{if }  \ell\in [1,k]  \\
1&\mbox{otherwise }
\end{array}\right.
$
	\end{enumerate}
	Step H) follows from   $(\otimes_3)+(\otimes_4)$.

	\medskip

	\textbf{Step I)}  In this  step, we drive our desired contradiction.
To see this, recall from Step H) that
	\[\begin{array}{ll}
	\sum\limits^{m(\ast)-1}_{\ell=0}\ideallI_\ell&=({\bf n}_0+\ldots+
	{\bf n}_{m(*)-1})\\
&+\sum\limits^k_{\ell=1}(\ideallI_0+
	\ldots+\ideallI_{m(\ast)-1}-1)+\sum\limits^{m(\ast)-1}_{\ell=k+1}
	(\ideallI_0+\ldots+
	\ideallI_{m(\ast)-1})\\
	&=m(\ast)\sum\limits^{m(\ast)-1}_{\ell=0} \ideallI_\ell-k.
	\end{array}\]
	
	\noindent So,
	$$
	-(m(\ast)-1)\left(\sum^{m(\ast)-1}_{\ell=1} \ideallI_\ell\right)=-k,
	$$i.e.,
	$$
	\sum^{m(\ast)-1}_{\ell=1} \ideallI_\ell= \frac{k}{m(\ast)-1}.
	$$
	Now, the left hand side is an integer, while as $1<k<m(\ast)-1$,
	 the right hand  side is not an integer, a contradiction.

	This completes the proof.
\end{proof}

An additional outcome is a slight improvement of the above pathological property.

\begin{corollary}
	\label{3.8}
	Assume $\lambda=(\mu^{\aleph_0})^+ > |\ringR|$  and let $0<m_1<m_2-1$. Then there is an $\ringR$-module $\modM$ of cardinality $\lambda$ such
	that:
	$$
	\modM^{n^1}\cong \modM^{n^2}\mbox{ if and only if   } m_1<n^1\ \&\ m_1\leq n^2\ \&\ [(m_2-m_1)|
	(n^1-n^2)]
	$$
	( $|$ means divides).
\end{corollary}

\begin{proof}
	Let $m(\ast):=m_2-m_1$ and let $\ringS_1$ be the ring constructed in the proof of
	Theorem \ref{3.7}  with respect to $m(\ast)$. Let also $\modM^\ast$ be the
	$(\ringR,\ringS_1)$-bimodule constructed there. In particular, it is equipped with a structure of bimodule over $(\ringR,\ringS_1)$. Choose cardinals
	$$
	\lambda_{m_1}>\lambda_{m_1-1}>\ldots >\lambda_0>
	\|\modM^*\|+\|\ringR\|+\|\ringS_1\|+\aleph_0.
	$$
	Let $I:=\bigcup\limits_{m\leq m_1}\prod\limits_{\ell=0}^{m-1}\lambda_\ell$
	and define a function
	$f:I\longrightarrow I$ by the following rules:
$$
	f(\eta):=\left\{\begin{array}{ll}
	\eta\restriction k & \mbox{ if } \eta \in I, \mbox{ and } \lg(\eta)=k+1\\
	\eta & \mbox{ if } \eta=\langle\rangle.
	\end{array} \right.
	$$

	\noindent Let $\modM^\otimes$ be a bimodule which is
	$\mathop{\oplus}\limits_{\eta\in I}
	\modM^\otimes_\eta$ with $\modM^\otimes_\eta\cong \modM^\ast$
	for each $\eta\in I$ and  let $h_\eta: \modM^\ast
	\longrightarrow \modM^\otimes_\eta$ be such an isomorphism.
	
	For every endomorphism ${\cal Y}$ of $\modM^\ast$
	we define an endomorphism ${\cal Y}^\otimes\in\End_{\ringR}(\modM^\otimes)$
	of $\modM^\otimes$ as an $\ringR$-module as follows. We are going  to define the action of
	${\cal Y}^\otimes$ on each $\modM^\otimes_\eta$ ($\eta\in I$). To this end, we take $x\in
	\modM^\otimes_\eta$ and define the action via:
	$$
	x {\cal Y}^\otimes:=h_{f(\eta)}\left((h^{-1}_\eta(x){\cal Y})\right).
	$$
	Let $\ringS^\otimes$ be the subring of the ring of endomorphisms of
	$\modM^\otimes$ as an $\ringR$-module generated by
	$$
	\{1^\otimes, {\cal X}^\otimes_0, {\cal X}^\otimes_1, \ldots , {\cal X}^
	\otimes_{m(\ast)}, {\cal W}^\otimes, {\cal Z}^\otimes\}.\footnote{Each one is a member of $\ringS$ and as is
	such an endomorphism of $\modM^\otimes$ as
	an $\ringR$-module.}
	$$
	 Note also that $1^\otimes$ is not a unit in $\ringS^\otimes$.
	Now we continue as in the proof of Theorem \ref{3.3}, with $\modM^\otimes$
	here instead of $\modM$ there, and $\ringS^\otimes$ here instead of $\ringS$ there. We leave the details to the reader.
\end{proof}

\begin{theorem}
	\label{4.7} Let $\lambda=(\mu^{\aleph_0})^+>|\ringR|$ be a regular cardinal.
	Then there are $\ringR$-modules $\modM_1$
		and $\modM_2$ of cardinality $\lambda$ such that:
		\begin{enumerate}
	\item	$\modM_1$, $\modM_2$ are not isomorphic,
		
	\item	$\modM_1$ is isomorphic to a direct summand of $\modM_2$,
		
	\item	$\modM_2$ is isomorphic to a direct summand of $\modM_1$.
	\end{enumerate}
\end{theorem}

\begin{proof} Let $\ringT$ be the subring of $\ringR$ which $1$ generates.
	As before, we need to choose a ring $\ringS$
	(essentially
	the ring of
	endomorphisms
	we would like).

\textbf{Step A)}: To make things  easier, we first introduce
a ring	  $\ringS_0$ which is easy compared to  $\ringS$.

Let $A_1$ (resp. $A_{-1})$  be the set of even
	(resp. odd) integers and let $f$ be the following function:
	$$
	f(i):=\left\{\begin{array}{ll}
	i+1 & \mbox{if }i\geq 0,\\
	i-1 & \mbox{if }i<0.
	\end{array}\right.
	$$
	Thus, $f$ maps $A_1$ (resp. $A_{-1}$) into $A_{-1}$ (resp. $A_1$). Also, $A_1\setminus
	\Rang (f\restriction A_{-1})=\{0\}$ and $A_{-1}\setminus\Rang (f\restriction
	A_1)=\{-1\}$. Let $i$  vary on the
	integers. Let $\ringS_0$ be the ring extending $\ringT$
	generated freely by $\{{\cal X}_1, {\cal X}_{-1}, {\cal W}_1,
	{\cal W}_{-1}, {\cal Z}_1, {\cal Z}_{-1}\}$. Let
	${\bf D}$ be a field \st\ if $\|\ringT\|$ is finite, then the
	characteristic of ${\bf D}$ is finite and divides $\|\ringT\|$. Then, we set
	$\ringS_{\bf D}^*={\bf D}\mathop{{\otimes}}\limits_\ringT \ringS_0$ (see Theorem \ref{3.7}).
	
	We are going to define a right $({\bf D}\mathop{{\otimes}}\limits_\ringT \ringS_0)$--
	module  $\modM^\ast_{{\bf D}^\ast}$.
	We first define a left ${\bf D}$-module structure over $\modM:=\sum
	\{{\bf D} x_i:i\in{\mathbb Z}\}$ with
	$b(\sum\limits_i a_ix_i)=\sum\limits_i(ba_i)x_i$ where $a_i$ and  $b$ are in ${\bf D}$. As mentioned, we
	like to make $\modM$ a right $({\bf D}\mathop{{\otimes}}\limits_\ringT \ringS_0)$--module which
	will be called $\modM^*_{\bf D}$.
	Let  $x\in \modM$ and $c\in
	{\bf D}\mathop{{\otimes}}\limits_\ringT \ringS_0$.
	To define  $xc$, as
	${\bf D}$ and $\ringS_0$ commute, it is enough to define
	it for $x=x_i$ and $c\in\{{\cal X}_1,{\cal X}_{-1},{\cal W}_1,
	{\cal W}_{-1}, {\cal Z}_1, {\cal Z}_{-1}\}$.
	So let
	$$
	x_i{\cal X}_1:=\left\{\begin{array}{ll}
	x_i &\mbox{if } i\in A_1,\\
	0 &\mbox{if } i\in A_{-1},
	\end{array} \right.
	$$
	$$
	x_i{\cal X}_{-1}:=\left\{\begin{array}{ll}
	0 &\mbox{if } i\in A_1,\\
	x_i &\mbox{if } i\in A_{-1},
	\end{array}\right.
	$$
	$$
	x_i{\cal W}_1:=x_{f(i)},
	$$
	$$
	x_i{\cal W}_{-1}:=\left\{\begin{array}{ll}
	x_{f^{-1}(i)} &\mbox{if } i\in{\rang(f)}  ,\\
	0 &\mbox{otherwise,}
	\end{array}\right.
	$$
	$$
	x_i{\cal Z}_1:=\left\{\begin{array}{ll}
	x_i &\mbox{if } i\in A_1\cap ({\mathbb Z}\setminus \{0\}),\\
	0 &\mbox{otherwise,}
	\end{array}\right.
	$$and
	$$
	x_i{\cal Z}_{-1}:=\left\{\begin{array}{ll}
	x_i &\mbox{if } i\in A_{-1}\cap ({\mathbb Z}\setminus\{0\}),\\
	0 &\mbox{otherwise.}
	\end{array}\right.
	$$
	These make $\modM^\ast_{\bf D}$ a right
	$\ringS^*_{\bf D}$--module. Let $g^\ast_{{\bf D}}$ be the natural ring \homo\
	from $\ringS_0$  into $\ringS^*_{\bf D}$. By using
	$g^\ast_{{\bf D}}$, the $\ringS^*_{\bf D}$--module $\modM^\ast_{\bf D}$ becomes a right $\ringS_0$-module.
	\medskip

\textbf{Step B)}:\quad In this step we define the auxiliary ring $\ringS$.

Let $\ringS$ be the ring
	(with 1, associative
	but not necessarily commutative) extending $\ringT$
	generated by ${\cal X}_1$, ${\cal X}_{-1}$,
	${\cal W}_1$, ${\cal W}_{-1}$, ${\cal Z}_1$, ${\cal Z}_{-1}$
	freely except the following equations (to
	understand them see below):
	\begin{enumerate}
		\item[$(\ast)$]
		$\sigma=0$ if $\sigma$ is a term\footnote{i.e. in the language
			of rings, in the variables ${\cal X}_1,{\cal X}_{-1},{\cal W}_1,{\cal W}_{-1},
			{\cal Z}_1,{\cal Z}_{-1}$},
		$\modM^\ast_{\bf D}\sigma=0$ for $\modM^\ast_{\bf D}$ as defined in Step A),
		for every field ${\bf D}$ such that if  $\ringT$ is of finite cardinality $n$, then
	 $\Char({\bf D})|n$.
	\end{enumerate}
	
	Let us state some of the relations, which follow from $(\ast)$, and we will use them more in the rest of our argument.
	
	\begin{description}
		\item[$(\star)_1 $] ${\cal X}^2_1={\cal X}_1$,\\
		${\cal X}^2_{-1}={\cal X}_{-1}$,\\ ${\cal X}_1+{\cal X}_{-1}=1$,\\
		${\cal X}_1{\cal X}_{-1}= {\cal X}_{-1}{\cal X}_1=0$  \\
		${\cal Z}^2_1={\cal Z}_1$, ${\cal Z}_1{\cal X}_1={\cal Z}_1=
		{\cal X}_1{\cal Z}_1$ \\
		${\cal X}_{-1}{\cal W}_1=
		({\cal X}_{-1}{\cal W}_1)
		{\cal X}_1{\cal Z}_1$
		\\${\cal X}_1{\cal Z}_1{\cal W}_{-1}=
		({\cal X}_1{\cal Z}_1{\cal W}_{-1}){\cal X}_{-1}$
		\\${\cal X}_{-1}{\cal W}_1{\cal W}_{-1}={\cal X}_{-1}$ and\\
		${\cal X}_1{\cal Z}_1{\cal W}_{-1}{\cal W}_1
		={\cal Z}_1={\cal X}_1{\cal Z}_1$.
	\end{description}
	
 Let us show, for example, that
	${\cal X}_1+{\cal X}_{-1}=1$. We need to show the left hand side is the identity
	map when we consider it as an endomorphism of $\modM^\ast_{\bf D}$. To
	this end, we evaluate ${\cal X}_1+{\cal X}_{-1}$  at any generator of $\modM^\ast_{\bf D}$, say  $ x_i$. Recall that
	$$
	x_i{\cal X}_1=\left\{\begin{array}{ll}
	x_i &\mbox{if } i\in A_1,\\
	0 &\mbox{if } i\in A_{-1},
	\end{array} \right.
	$$
	$$
	x_i{\cal X}_{-1}=\left\{\begin{array}{ll}
	0 &\mbox{if } i\in A_1,\\
	x_i &\mbox{if } i\in A_{-1},
	\end{array}\right.
	$$Thus, $x_i({\cal X}_1+{\cal X}_{-1})=x_i$, as claimed. The other relations will follow in the same way.

	\textbf{Step C)}: Here, we introduce the $\ringR$-modules $\modM_1$
	and $\modM_2$ of cardinality $\lambda$ such that:\begin{enumerate}
		\item[(i)] 	$\modM_1$ is isomorphic to a direct summand of $\modM_2$ and
		\item[(ii)]	$\modM_2$ is isomorphic to a direct summand of $\modM_1$.
	\end{enumerate}

	Let $\langle \modM_\alpha:\alpha\leq\lambda\rangle$ be a strongly semi-nice construction for
	$(\lambda,{\frakss},S, \lambda)$ and let $\modM:=\modM_\lambda$. Recall that such a thing exists.
	Let $\modM_1:=\modM {\cal X}_1$ and
	$\modM_{-1}:=\modM   {\cal X}_{-1}$.
	Since ${\cal X}_1{\cal X}_{-1}= {\cal X}_{-1}{\cal X}_1=0$, we have
	$\modM_1\cap\modM_{-1}=0$.
	Thanks to the formula ${\cal X}_1+{\cal X}_{-1}=1$,
	$$\modM=\modM({\cal X}_1+{\cal X}_{-1})=\modM {\cal X}_1+\modM{\cal X}_{-1}=\modM_1+ \modM_{-1}=\modM_1\oplus \modM_{-1}.$$
	Then
	$\modM_1$ and $\modM_{-1}$ are equipped with $\ringR$-module structure, and there is the identification  $\modM=\modM_1\oplus \modM_{-1}$.
	We
	shall show that $\modM_1$, $\modM_{-1}$ are as required in
	Theorem \ref{4.7} (with respect to $\modM_1$ and $\modM_2$).

The relations ${\cal Z}^2_1={\cal Z}_1$ and ${\cal Z}_1{\cal X}_1={\cal Z}_1=
	{\cal X}_1{\cal Z}_1$ imply that $$\modM_1=\modM_1(1-{\cal Z}_1)\oplus
	\modM_1{\cal Z}_1,$$
	i.e., $\modM_1{\cal Z}_1$
	is a direct summand of $\modM_1$.
	Since ${\cal X}_{-1}{\cal W}_1=
	({\cal X}_{-1}{\cal W}_1)
	{\cal X}_1{\cal Z}_1$, we have
	$$\modM_{-1}{\cal W}_1=\modM   {\cal X}_{-1}{\cal W}_1=\modM({\cal X}_{-1}{\cal W}_1){\cal X}_1{\cal Z}_1\subset\modM {\cal X}_1{\cal Z}_1=\modM_1{\cal Z}_1.$$Thus,
	${\cal W}_{1}$ maps $\modM_{-1}$
	into $\modM_1{\cal Z}_1$.
	By using the formula ${\cal X}_1{\cal Z}_1{\cal W}_{-1}=
	({\cal X}_1{\cal Z}_1{\cal W}_{-1}){\cal X}_{-1}$ we observe that
	$$ \modM_1{\cal Z}_1{\cal W}_{-1}=\modM {\cal X}_1{\cal Z}_1{\cal W}_{-1}=\modM {\cal X}_1{\cal Z}_1{\cal W}_{-1} {\cal X}_{-1}\subset\modM   {\cal X}_{-1}=\modM_{-1}.$$
	In other words,
	${\cal W}_{-1}$ maps $\modM_1{\cal Z}_1$
	into $\modM_{-1}$. We are going to combine the formula ${\cal X}_{-1}{\cal W}_1{\cal W}_{-1}={\cal X}_{-1}$ along with 
	${\cal X}_1{\cal Z}_1{\cal W}_{-1}{\cal W}_1
	={\cal Z}_1={\cal X}_1{\cal Z}_1$ to deduce that
	the multiplication maps by ${\cal W}_{1}$  and ${\cal W}_{-1}$  are the inverse of each other. This implies that $\modM_{-1}$ is
	isomorphic to a direct summand of $\modM_1$ (as left $\ringR$-modules).
		Similarly, we obtain: $$\modM_{-1}=\modM_{-1}(1-{\cal Z}_{-1})\oplus \modM_{-1}
	{\cal Z}_{-1}.$$So, $\modM_{-1}{\cal Z}_{-1}$ is
	a direct summand of $\modM_{-1}$ and $\modM_{-1}{\cal Z}_{-1}$
	is isomorphic to $\modM_1$. By the same argument, $\modM_1$ is isomorphic
	to a direct summand of $\modM_2$  as left $\ringR$-modules. In summary, we showed that
\begin{enumerate}
		\item[$\bullet$] 		 $\modM_1\cong \modM_1(1-{\cal Z}_1)\oplus \modM_{-1}$ and
		\item[$\bullet$]	$\modM_{-1}\cong \modM_{-1}(1-{\cal Z}_{-1})\oplus
		\modM_1.$
	\end{enumerate}
	This completes the proof of Step C).

	It remains to show that $\modM_1\not\cong \modM_{-1}$. Suppose on the contrary
	that  they are isomorphic, and we  get a contradiction, which is presented at Step H)  below.
	Let $\mathfrak e \in \calE^{\frakss}$ be simple.

		\textbf{Step D)}:	There is a solution     ${\cal Y}\in\ringdE^{\mathfrak e}_n$ to  the following equations:
		\begin{description}
			\item[$(\ast)_2$]\qquad ${\cal X}_1{\cal Y}{\cal X}_{-1}={\cal X}_1{\cal Y}$,
			\quad ${\cal X}_{-1}{\cal Y}{\cal X}_1={\cal X}_{-1}{\cal Y}$,
			\quad ${\cal Y}{\cal Y}=1$.
		\end{description}

To see this, recall that  $\modM_1\cong \modM_{-1}$
and $\modM=\modM_1\oplus \modM_{-1}$.
Let $h$ be an isomorphism from $\modM_1$ onto $\modM_{-1}$.
	 Define $\fucF: \modM \to \modM$ by
	$$(a,b) \in \modM=\modM_1\oplus \modM_{-1} \implies \fucF(a,b)=(h^{-1}(b),h(a)).$$
	So,
	$\fucF\in
	\End_\ringR(\modM)$ and it satisfies $\fucF\restriction
	\modM_1=h$ and $\fucF\restriction \modM_{-1}=
	h^{-1}$. The member in $\ringdE^{\mathfrak e}_n$ which $\fucF$ induces   solves the equations in $(\ast)_2$.  This completes the proof of Step D).
	\medskip
	
In the light of Corollary \ref{3.5}, it is enough to prove
	the following two items:
	\begin{enumerate}
		\item[(a)] in ${\bf D}\mathop{{\otimes}}\limits_\ringT \ringS_0$ there is
		no solution to
		$(\ast)_2$, in particular, there is no such ${\cal Y}$. Note that $\ringS^*_{{\bf D}}$
		have the same characteristic as ${\bf D}$.
		\item[(b)] $\ringS_0$ is a free $\ringT$--module.
	\end{enumerate}
	Clearly $\ringS$ is a $\ringT$-module, generated by the set of monomials in
	\[\{{\cal X}_1, {\cal X}_{-1}, {\cal W}_1, {\cal W}_{-1}, {\cal Z}_1,
	{\cal Z}_{-1}\}.\]
	Our aim  is to show that $\ringS$ is a free $\ringT$-module;
	in fact we shall exhibit explicitly
	a free basis. For $\ell\in\{1,-1\}$, $k\in {\mathbb Z}$, $n\geq 0$, $n\geq
	-k$, we define an endomorphism $\fucF^\ell_{k,n}$
	of $\modM^\ast_{\bf D}$ by
	$$
 \fucF^\ell_{k,n}(	x_i):=\left\{\begin{array}{ll}
	x_{f^k(i)}&\mbox{  if } f^{-n}(i)\mbox{ is well defined and }
	x_i\in A_\ell,\\
	0 &\mbox{{otherwise}}
	\end{array}\right.
	$$
It is easy to see that it is an endomorphism of
	$\modM^\ast_{\bf D}$ as a left ${\bf D}$-module. We will
	define   a monomial
	${\cal Y}^\ell_{k,n}$, in the following way. For every monomial $\sigma$
	 let $\sigma^0$
	 be $1=\id_{\modM^\ast_{\bf D}}$ and remember $n\geq -k$ so $n+k\geq 0$. Now set
	$$
	{\cal Y}^\ell_{k,n}={\cal X}_\ell({\cal W}_{-1})^n{\cal W}^{n+k}_1.
	$$
	It is easy to see that the operation of ${\cal Y}^\ell_{k,n}$
	on $\modM^\ast_{\bf D}$ by right multiplication,
	is equal to $\fucF^\ell_{k,n}$.
	
	Let
$$\mathcal{G}:=\{{\cal Y}^\ell_{k,n}:(\ell, k, n)\in w\}$$
	where
	$
	w=\{(\ell,k,n):\ \ell\in \{1,-1\},\ k\in {\mathbb Z},\ n\geq 0,\ k+n\geq
	0\}.
	$
	
In the next step, we show that
	$
	\mathcal{G}
	$
	generates $\ringS$ as a $\ringT$-module.
	\medskip
	
\textbf{Step E)}:\quad The set
	$ \mathcal{G}$
	generates $\ringS$ as a $\ringT$-module.
	
	It is enough to show that for every monomial $\sigma$, some equation $\sigma=
	\sum a^\ell_{n,k}{\cal Y}^\ell_{k,n}$ holds in $\ringS$,
	 where $\{(\ell,n,k):a^\ell_{n,
		k}\neq 0\}$ is finite and $a^\ell_{k,n}\in \ringT$, i.e., it holds in the ring of
	endomorphism of $\modM^\ast_{\bf D}$. We prove this by induction on the length of the monomial $\sigma$.
	
	If the length is zero, $\sigma$ is 1. Recall from $(\star)_1$ that $1={\cal X}_1+{\cal X}_{-1}$.
	By definition, ${\cal X}_\ell={\cal Y}^\ell_{0,0}$; so $1={\cal Y}^1_{0,0}+{\cal Y}^{-1}_
	{0,0}$ as required.
	
	If the length is $>0$, by the induction hypothesis it is enough to prove the following:
	\begin{enumerate}
		\item[$(\ast)$] \quad Let $\tau\in\{{\cal X}_1,{\cal X}_{-1},
		{\cal W}_1,{\cal W}_{-1},{\cal Z}_1, {\cal Z}_{-1}\}$. Then
	 ${\cal Y}^{\ell(\ast)}_{k(\ast),n(\ast)}\tau$
		is equal to some
		$\sum\limits_{\ell,k,n} a^\ell_{k,n}{\cal Y}^\ell_{k,n}$.
	\end{enumerate}
	
	Indeed, it is enough to check equality on the generators of
	$\modM^\ast_{\bf D}$, that is the $x_i$'s. The proof     of $(\ast)$ is  divided
	into three cases:

	\noindent\underline{Case 1:}\quad ${\cal Y}^{\ell(\ast)}_{k(\ast),n
		(\ast)}{\cal W}_\ell$ is:
	\[\begin{array}{rcl}
	{\cal Y}^{\ell(\ast)}_{k(\ast)+1,n(\ast)}& \mbox{ if }  &\ell=1\\
	{\cal Y}^{\ell(\ast)}_{k(\ast)-1,n(\ast)}& \mbox{ if }  &\ell=-1 \mbox{ and }\
	k(\ast)+ n(\ast)>0,\\
	{\cal Y}^{\ell(\ast)}_{k(\ast)-1,n(\ast)+1}&  \mbox{ if  }  &\ell=-1 \mbox{ and }\
	k(\ast)+ n(\ast)=0.
	\end{array}\]
	
	First assume that $\ell=1$. Then  \[\begin{array}{ll}
	{\cal Y}^{\ell(\ast)}_{k(\ast),n
		(\ast)}{\cal W}_\ell&={\cal X}_{1}({\cal W}_{-1})^n{\cal W}^{n+k}_1{\cal W}_1\\
	&={\cal X}_{1}({\cal W}_{-1})^n{\cal W}^{n+k+1}\\
	&={\cal Y}^{\ell(\ast)}_{k(\ast)+1,n(\ast)}.
	\end{array}\]
	Now assume that $\ell=-1$ and $k(\ast)+ n(\ast)>0$.  If $i\in\rang(f)$,
	then $x_i{\cal W}_1{\cal W}_{-1}=x_i$. From
	this, $$x_i{\cal Y}^{\ell(\ast)}_{k(\ast),n
		(\ast)}{\cal W}_\ell=x_i{\cal Y}^{\ell(\ast)}_{k(\ast)-1,n(\ast)}.$$
	If $i\notin\rang(f)$, then $i\in\{0,1\}$. It is easy to see that
	$$x_0{\cal Y}^{\ell(\ast)}_{k(\ast),n
		(\ast)}{\cal W}_{-1}=x_0{\cal Y}^{\ell(\ast)}_{k(\ast)-1,n(\ast)}$$
	and $$x_{-1}{\cal Y}^{\ell(\ast)}_{k(\ast),n
		(\ast)}{\cal W}_{-1}=x{-1}{\cal Y}^{\ell(\ast)}_{k(\ast)-1,n(\ast)}.$$
	Thus, the functions  ${\cal Y}^{\ell(\ast)}_{k(\ast),n
		(\ast)}{\cal W}_{-1}$ and ${\cal Y}^{\ell(\ast)}_{k(\ast)-1,n(\ast)}$ are the same.
	Finally, assume that $\ell=-1$ and $k(\ast)+ n(\ast)=0$. Then
	\[\begin{array}{ll}
	{\cal Y}^{\ell(\ast)}_{k(\ast),n
		(\ast)}{\cal W}_{-1}&={\cal X}_{-1}({\cal W}_{-1})^n{\cal W}^{n+k}_1{\cal W}_{-1}\\
	&={\cal X}_{-1}({\cal W}_{-1})^n {\cal W}_{-1}\\
	&={\cal X}_{-1}({\cal W}_{-1})^{n+1}\\
	&={\cal Y}^{\ell(\ast)}_{k(\ast)-1,n(\ast)+1}.
	\end{array}\]
	This completes the argument.
	\medskip

	\noindent\underline{Case 2:}\quad ${\cal Y}^{\ell(\ast)}_{k(\ast),n(\ast)}
	{\cal X}_\ell$ is:
	$$\mbox{zero }\quad
	\mbox{  if }\quad
	[\ell(\ast)=\ell  \iff  k(\ast)\mbox{ odd}],
	$$
	$$
	{\cal Y}^{\ell(\ast)}_{k(\ast), n(\ast)}
	\mbox{if }\quad
	[\ell(\ast)=\ell  \iff   k(\ast)\mbox{ even}].
	$$
	The proof of this is similar to Case 1.
	\medskip	

	\noindent\underline{Case 3:}\quad ${\cal Y}^{\ell(\ast)}_{k(\ast),n(\ast)}
	{\cal Z}_\ell$ is:
	\[\begin{array}{rcl}
	{\cal Y}^{\ell(\ast)}_{k(\ast),n(\ast)}& \mbox{ if }   &n(\ast)+k
	(\ast)>0\mbox{
		and }[\ell(\ast)=\ell  \iff    k(\ast)\mbox{ even}],\\
	Y^{\ell(\ast)}_{k(\ast),n(\ast)+1}&  \mbox{ if  }   &n(\ast)+k(\ast)=0
	\mbox{ and }[\ell(\ast)=\ell  \iff    k(\ast)\mbox{ even}],\\
	\mbox{zero}&  \mbox{ if  }  &[\ell(\ast)=\ell \iff   k(\ast)\mbox{ odd}].
	\end{array}\]
  It is enough to check equality on the generators of
	$\modM^\ast_\fieldD$, that is the $x_i$'s. The proof is again similar to the proof of Case 1.
	\medskip
	
\textbf{Step F)}:\quad The set $\mathcal{G}=\{{\cal Y}^\ell_{k,n}:(\ell, k, n)\in w\}$
	generates $\ringS$ freely as a $\ringT$-module.
	
	In the light of Step E) we see that $\mathcal{G}$ generates $\ringS$ as a $\ringT$-module. Towards contradiction
	suppose that $0=\sum\{a^\ell_{k,n}{\cal Y}^\ell_{k,n}:(\ell,k,n)\in w\}$ in $\ringS$
	where $w\subseteq w^\ast$ is finite, $a^\ell_{k,n} \in \ringT$ and not
	all of them are zero. If $n=|\ringT|$ is finite, we take the field $\bf D$ such that  $\Char({\bf D})|n$, and  some $a^\ell_{k,n}$
	is not zero in ${\bf D}$. Hence, $0=\sum\{a^\ell_{k,n}:
	(\ell,k,n)\in w\}$  as an endomorphism of $\modM^\ast_{\bf D}$,
	a left ${\bf D}$-module where $a^\ell_{k,n}\in
	{\bf D}$ and $ w\subseteq w^\ast$ is  finite. We
	shall prove that $a^\ell_{k,n}=0$ for every $(\ell,k,n)\in w$.
	
	If $i\in A_1$ and $i\geq 0$, then
	\[\begin{array}{rlc}
	0&=x_i(\sum\limits_{(\ell,k,n)\in w} a^\ell_{k,n}{\cal Y}^\ell_{k,n})\\
	&=\sum\limits_{(\ell,k,n)\in w} a^\ell_{k,n}(x_i
	{\cal Y}^\ell_{k,n})\\&=
	\sum\limits_{ (\ell,k,n)\in w}\{a^\ell_{k,n}x_{i+k}:\ell=1,\ \mbox{ and }n\leq i\}
	\\
	&=\sum\limits_{j\geq 0}(\sum\limits_{(1,k,n)\in w}\{a^1_{k,n}:\  \ i\geq n,\
	i+k=j\}) x_j\\&=
	\sum\limits_{j\geq 0}(\sum\limits_{(1,j-i,n)\in w}\{a^1_{j-i, n}:\ i\geq n,\
	\})x_j.
	\end{array}\]
	Hence for every $i\in A_1$, $i\geq 0$ and $j\geq 0$ we have
	\begin{description}
		\item[$(\ast)^a_{i,j}$] \qquad\qquad $0=\sum\{a^1_{j-i, n}:\ n\geq 0,\ n\leq i \mbox{ and
		}n+(j-i)\geq 0\}$.
	\end{description}
	Similarly, for $i\in A_{-1}$, $i\geq 0$ (equivalently, $i>0$ as $i\in
	A_{-1}\ \Rightarrow\ i\neq 0$) and $j\geq 0$ we can prove
	\begin{description}
		\item[$(\ast)^b_{i,j}$] \qquad\qquad $0=\sum\{a^{-1}_{j-i,n}:\ n\geq 0,\ n\leq i\mbox{ and
		}n+(j-i)\geq 0\}$.
	\end{description}
	Similarly, for $i\in A_1$, $i<0$,
	\[\begin{array}{rlc}
	0&=x_i(\sum\limits_{(\ell,k,n)\in w}a^\ell_{k,n}
	{\cal Y}^\ell_{k,n})\\&=
	\sum\limits_{(\ell,k,n)\in w} a^\ell_{k,n}(x_i {\cal Y}^\ell_{k,n})\\
	&=\sum\limits_{(1,k,n)\in w}\{a^1_{k,n}x_{i-k}: \mbox{  } -i>n\}\\
	&=\sum\limits_{j<0}(\sum_{ (1,i-j,n)\in w}\{a^1_{i-j,n}: \mbox{  }
	n<-i\})x_j.
	\end{array}\]
	Hence for every $i\in A_1$, $i<0$ and $j<0$ we have
	\begin{description}
		\item[$(\ast)^c_{i,j}$] \qquad $0=\sum\{a^1_{i-j,n}: n\geq 0$ and $n+(i-j)\geq 0$
		and $n<-i\}$.
	\end{description}
	Similarly, for every $i\in A_{-1}$, $i<0$ and $j<0$
	\begin{description}
		\item[$(\ast)^d_{i,j}$] \qquad $0=\sum\{a^{-1}_{i-j,n}:\ n\geq 0\mbox{ and }n+(i-j)
		\geq 0\mbox{ and } n<-i\}$.
	\end{description}
	Choose, if possible, $(k,m)$ such that:
	\begin{enumerate}
		\item $(1,k,m)$ belongs to $w$,
		\item $a^1_{k,m}\neq 0$,
		\item $m$ is minimal under (1)+(2).
	\end{enumerate}
	Note that $m\geq 0$ by the definition of
	$w$. First assume that $m$ is even.  Let $i=m$ and $j=i+k$. So $i\in A_1$ (being even), $i\geq 0$ and
	$j=m+k$ is $\geq 0$ as $(1,k,m)\in w$. In the equation $(\ast)^a_{i,j}$
	the term $a^1_{k,m}$ appears in the sum, and for every other term $a^1_{k_1,
		m_1}$ which appears in the sum, we have $m_1< m$ (and $k_1=k$), and hence by
	(3) above is zero.
	So it follows that $a^1_{k,m}$ is zero, a contradiction.
	
	If $m$ is odd, we get a similar contradiction using $(\ast)^c_{i,j}$. Let
	$i=-m-1$ and $j=i-k$. Note that $m\geq 0$, hence $i<0$ and $i$ is even as $m$ is odd,
	so $i\in A_1$. Also $j=i-k=-m-1-k\leq -1<0$ (recalling $k+m\geq 0$ as
	$(1,k,m)\in w^*$). In the equation
	$(\ast)^c_{i,j}$, the term $a^1_{i-j,n}=a^1_{k,n}$ appears in the sum if and only if
	$0\leq n< -i=m+1$, and $n+(i-j)=n+k\geq 0$ (but if the later fails,
	$a^1_{k,m}$ is not defined) so $a^1_{k,m}$ appears, and if another term
	$a^1_{k_1,m_1}$ occurs then $m_1\leq m$ (and $k_1=k$). Hence $m_1<m$ and so
	$a^1_{k_1,m_1}=0$. Necessarily $a^1_{k,m}$ is zero, a contradiction.
	
	So $a^1_{k,n}=0$ whenever it is defined.
	
	Similarly $a^{-1}_{k,n}=0$ whenever it is defined (use $(\ast)^b_{i,j}+
	(\ast)^d_{i,j}$). So, $\ringS_0$ is a free module over $\ringT$, as required.
	\medskip
	
\textbf{Step G)}:\quad
	In this step we get a contradiction that
	we searched for it.  This will show that $\modM_1\not\cong \modM_{-1}$.
	
	To this end, recall
	   that there are    finitely  many nonzero $a^\ell_{k,n}\in
	{\bf D}$
	such that
	\begin{description}
		\item[$\boxtimes$]\qquad ${\cal Y}=\sum\left\{a^\ell_{k,n}
		{\cal Y}^\ell_{k,n}:\ n\geq
		0\mbox{ and }k+n\geq 0\mbox{ and }\ell\in \{1,-1\}\right\}.$
	\end{description}
	Let $n(\ast)<\omega$ be such that
	$$
	a^\ell_{k,n}\neq 0\quad\Rightarrow\quad |k|,n<n(\ast).
	$$
	For $\ell=1,-1$ let
	\[\begin{array}{lr}
	\modM^\pos_\ell:=\big\{\sum\limits_{i\geq 0} d_i x_i:&d_i\in {\bf D},\mbox{ all but finitely many of } d_i's \mbox{  are zero }\ \ \\
	&\mbox{ and }d_i\neq 0\ \Rightarrow\ i\in A_\ell\big\},\\
	\modM^\nega_\ell:=\big\{\sum\limits_{i<0}d_i x_i:& d_i\in
	{\bf D},\mbox{ all but finitely many of } d_i's \mbox{  are zero }\ \ \\
	&\mbox{ and }d_i\neq 0\ \Rightarrow\ i\in A_\ell\big\}.
	\end{array}\]
	Clearly as a ${\bf D}$-module
	$$
	\modM^\ast_{\bf D}=\modM^\pos_1\oplus \modM^\pos_{-1}\oplus \modM^\nega_1\oplus \modM^\nega_{-1}.
	$$
	Let ${\cal Y}^\pos_\ell:= {\cal Y}\restriction \modM^\pos_\ell$ and
	${\cal Y}^\nega_\ell:= {\cal Y}\restriction
	\modM^\nega_\ell$ for $\ell\in \{1,-1\}$.
	
	Now each ${\cal Y}^\ell_{k,n}$ maps $\modM^\pos=\modM^\pos_1\oplus
	\modM^\pos_{-1}$ to itself,
	and $\modM^\nega=\modM^\nega_1\oplus \modM^\nega_{-1}$ to itself, and hence by
	$\boxtimes$ above also ${\cal Y}$ does it. According to $(\ast)_2$ from Step B) we have
	${\cal X}_1{\cal Y}{\cal X}_{-1}={\cal X}_1{\cal Y}$. Hence
	$$\modM_1{\cal Y}=\modM{\cal X}_1{\cal Y}=\modM{\cal X}_1{\cal Y}{\cal X}_{-1}\subset\modM{\cal X}_{-1}=\modM_{-1},$$
	i.e., ${\cal Y}$ maps $\modM_1$ to $\modM_{-1}$. Thus, by the
	previous sentence, ${\cal Y}$ maps $\modM^\pos_1$ into
	$\modM^\pos_{-1}$, and $\modM^\nega_1$
	into $\modM^\nega_{-1}$, i.e., ${\cal Y}^\pos_1$ (resp. ${\cal Y}^\nega_1$) is into
	$\modM^\pos_{-1}$  (resp. $\modM^\nega_{-1}$).
	
	Similarly by $(\ast)_2$ we have ${\cal X}_{-1}{\cal Y}{\cal X}_1=
	{\cal X}_{-1}{\cal Y}$, hence ${\cal Y}$ maps
	$\modM^\pos_{-1}$ into $\modM^\pos_1$ and $\modM^\nega_{-1}$
	into $\modM^\nega_1$. Also, the
	mappings ${\cal Y}^\pos_1$, ${\cal Y}^\pos_{-1}$, ${\cal Y}^\nega_1$,
	${\cal Y}^\nega_{-1}$ are
	endomorphisms of ${\bf D}$-modules. As ${\cal Y}^2=1$ (again by $(\ast)_2$)
	we conclude that
	${\cal Y}^\pos_1$ and ${\cal Y}^\pos_{-1}$ are the inverse of each
	other, so both of them are
	isomorphisms. Similarly for ${\cal Y}^\nega_1$ and ${\cal Y}^\nega_{-1}$.
	
	Let
	\[\begin{array}{lr}
	\modM^\stp_1:=\{\sum\limits_{i>0}d_ix_i:&d_i\in {\bf D},\mbox{ all
		but finitely many $d_i$'s are zero }\ \\
	&\mbox{and }d_i\neq 0\ \Rightarrow\ i\in A_1\}.
	\end{array}\]
	Clearly $\modM^\stp_1$ is a sub ${\bf D}$-module of $\modM^\pos_1$.
Note that  $x_0\in \modM^\pos_1\setminus \modM^\stp_1$. This yields the
 	difference between $\modM^\stp_1$ and $\modM^\pos_1$.
	
	Let
	\[\begin{array}{lr}
	\modN:=\{\sum\limits_{i>n(\ast)} d_ix_i:&d_i\in {\bf D},
	\mbox{ all but finitely many of } d_i's \mbox{  are zero }\ \\
	&\mbox{and }d_i\neq 0\ \Rightarrow\ i\in A_1\}.
	\end{array}\]
	
	Let $H^\pos:\modM_1^\stp\longrightarrow \modM_1^\nega$
	be defined by $x_i H^\pos=
	x_{-i}$ and $H^\nega:\modM^\nega_{-1}\longrightarrow \modM^\stp_{-1}$
	be defined by
	$x_iH^\nega=x_{-i}$. Both of them are isomorphisms  of
	${\bf D}$-modules.  Note that ${\cal Y}^\pos_1$ is an isomorphism from
	$\modM^\pos_1$ onto $\modM^\pos_{-1}$ and
	$H^\pos {\cal Y}^\nega_1H^\nega$ is an
	isomorphism from $\modM^\stp_1$ onto $\modM^\pos_{-1}$. Note that $$\modM^\stp_1
	\stackrel{H^\pos}{\longrightarrow}\modM^\nega_1
	\stackrel{{\cal Y}^\nega_1}{\longrightarrow}\modM^\nega_{-1}
	\stackrel{H^\nega}{\longrightarrow} \modM^\pos_{-1}.$$
	We show that
	$$
	{\cal Y}^\pos_1\restriction \modN= (H^\pos {\cal Y}^\nega_1 H^\nega)
	\restriction \modN.
	$$
	To see this, it is enough to check equality on the generators of
	$\modN$, that is over the $x_i$'s where $i$ is even and it is larger than $n(\ast)$.
	In particular, by choosing  $n(\ast)$ large enough, we may assume that $i\gg 0$.
	Recall that $${\cal Y}=\sum\left\{a^\ell_{k,n}
	{\cal Y}^\ell_{k,n}:\ n\geq
	0\mbox{ and }k+n\geq 0\mbox{ and }\ell\in \{1,-1\}\right\}.$$
	By  definition $x_i{\cal X}_{-1}=0$.  Then
	\[\begin{array}{ll}
	x_i(H^\pos {\cal Y}^\nega_1 H^\nega)&=x_{-i}(\sum\{a^\ell_{k,n}
	{\cal Y}^\ell_{k,n}\})H^\nega \\
	&=\sum a^\ell_{k,n}x_{-i}(
	{\cal X}_1 {\cal W}_{-1} ^n{\cal W}^{n+k}_1 )H^\nega \\
	&= \sum a^1_{k,n}( x_{-i}{\cal W}_{-1} ^n{\cal W}^{n+k}_1 )H^\nega\\
	&= \sum a^1_{k,n}( x_{f^{-n}(-i)} {\cal W}^{n+k}_1 )H^\nega \\
	&=\sum a^1_{k,n}( x_{f^{n+k}(f^{-n}(-i))} )H^\nega \\
	&=\sum a^1_{k,n}( x_{f^{k}(-i)} )H^\nega \\
	&=  \sum a^1_{k,n}( x_{-i-k)} )H^\nega  \\
	&= \sum a^1_{k,n}  x_{i+k}.
	\end{array}\]
	On
	the other hand,
	$(x_i){\cal Y}^\pos_1\restriction\modN=\sum a^1_{k,n}  x_{i+k}$. Thus we have
	${\cal Y}^\pos_1\restriction \modN= (H^\pos {\cal Y}^\nega_1 H^\nega)
	\restriction \modN$, as claimed.
	
	Let $\modN^\ast:=\rang(  {\cal Y}^\pos_1\restriction \modN)$. Then
	$\modN^\ast=\rang( (H^\pos {\cal Y}^\nega_1H^\nega)\restriction \modN)$.
	So, as ${\cal Y}^\pos_1$ is an
	isomorphism from $\modM^\pos_1$ onto $\modM^\pos_{-1}$,
	and $\modN\subseteq \modM^\pos_1$ we
	have that $\modN^\ast$ is a sub $\bf D$-module of
	$\modM^\pos_{-1}$. This means that $\modM^\pos_{-1}/
	\modN^\ast$ is isomorphic to $\modM^\pos_1/\modN$ (as $\bf D$-modules).
	
	But $H^\pos {\cal Y}^\nega_1H^\nega$ is an isomorphism from
	$\modM^\stp_1$ onto
	$\modM^\pos_{-1}$ and $\modN\subseteq \modM^\stp_1$,
	and it maps $\modN$ onto $\modN^\ast$
	(see above), so $\modM^\stp_1/\modN$ is isomorphic to
	$\modM^\pos_{-1}/\modN^\ast$. By
	the previous paragraph we get $$\modM^\stp_1/ \modN\cong\modM^\pos_{-1}/
	\modN^\ast
	\cong\modM^\pos_1/\modN.$$
	
	On the one hand, $\modM^\pos_1/\modN$ is a free
	$\bf D$-module, as $\{x_{2i}+\modN
	:0\leq 2i\leq n(\ast)
	\}$ is a free basis for it and also $\modM^\stp_1/\modN$
	is a free $\bf D$-module with the base $\{x_{2i}+\modN:
	0<2i\leq n(\ast)\}$. On the other hand, the number of their generators  differ by $1$. This is a contradiction that we searched for it.

The theorem follows.
\end{proof}

\end{document}

%% file: t3totex.tex
\newdimen\theight
\def \Column{%
             \vadjust{\setbox0=\hbox{\sevenrm\quad\quad tcol}%
             \theight=\ht0
             \advance\theight by \dp0    \advance\theight by \lineskip
             \kern -\theight \vbox to \theight{\rightline{\rlap{\box0}}%
             \vss}%
             }}%

\catcode`\@=11
\def\qed{\ifhmode\unskip\nobreak\fi\ifmmode\ifinner\else\hskip5\p@\fi\fi
 \hbox{\hskip5\p@\vrule width4\p@ height6\p@ depth1.5\p@\hskip\p@}}
\catcode`@=12 

\newcount\notenumber
\def\clearnotenumber{\notenumber=0}
\def\note{\global\advance\notenumber by 1
 \footnote{$^{\the\notenumber}$}}

\clearnotenumber